\documentclass[reqno,11pt,a4paper]{amsart}

\usepackage{amsmath,amssymb,amsthm,graphicx,mathrsfs,url}
\usepackage{mathrsfs}
\usepackage{enumitem}
\usepackage{mathtools}
\usepackage[utf8]{inputenc}
\usepackage[dvipsnames]{xcolor}
\usepackage[colorlinks=true,linkcolor=Red,citecolor=Green]{hyperref}
\hypersetup{pdfstartview=XYZ}
\usepackage{a4wide}
\usepackage{tikz}

\theoremstyle{plain}
\newtheorem{introthm}{Theorem}

\newtheorem{theo}{Theorem}
\newtheorem{lemma}{Lemma}[section]
\newtheorem{prop}[lemma]{Proposition}
\newtheorem{cor}[lemma]{Corollary}

\theoremstyle{definition}
\newtheorem{definition}[lemma]{Definition}

\theoremstyle{remark}
\newtheorem{rem}{Remark}[section]
 
\numberwithin{equation}{section}

\newcommand{\C}{\mathbb{C}}
\newcommand{\R}{\mathbb{R}}

\newcommand{\Z}{\mathbb{Z}}
\newcommand{\TT}{\mathbb{T}}
\newcommand{\D}{\mathcal{D}}
\newcommand{\N}{\mathbb{N}}
\newcommand{\OO}{\mathcal{O}}
\newcommand{\eps}{\varepsilon}

\newcommand{\mc}{\mathcal}

\newcommand{\norm}[1]{\left\Vert#1\right\Vert}
\newcommand{\la}{\lambda}

\newcommand{\pl}{\partial}
\newcommand{\x}{\times}

\newcommand{\til}{\widetilde}
\newcommand{\bbar}{\overline}
\newcommand{\cjd}{\rangle}
\newcommand{\cjg}{\langle}

\let\Im=\Imag

\let\Re=\Real
\DeclareMathOperator{\sgn}{sgn}
\DeclareMathOperator{\SL}{SL}
\DeclareMathOperator{\supp}{supp}

\def\Ddots{\mathinner{\mkern1mu\raise\p@
\vbox{\kern7\p@\hbox{.}}\mkern2mu
\raise4\p@\hbox{.}\mkern2mu\raise7\p@\hbox{.}\mkern1mu}}
\makeatother

\def\P{\mathcal{P}}

\def\T{\mathcal{T}}
\def\H{\mathbb{H}}

\renewcommand{\d}{{\rm d}}
\newcommand{\E}{{\mathcal E}}
\newcommand{\M}{{\mc{M}}}
\renewcommand{\O}{{\mathcal O}}
\newcommand{\Xbf}{{\mathbf X}}
\newcommand{\0}{{\rm 0}}
\renewcommand{\epsilon}{\vararepsilon}

\newcommand{\bq}{\begin{equation}}
\newcommand{\eq}{\end{equation}}
\newcommand{\bqn}{\begin{equation*}}
\newcommand{\eqn}{\end{equation*}}

\newcommand{\Cinft}{{C^{\infty}}}
\newcommand{\CT}{{C^{\infty}_c}}
\newcommand{\GL}{\mathrm{GL}}
\newcommand{\SO}{\mathrm{SO}}

\newcommand{\AdS}{{\rm AdS}}

\renewcommand{\det}{\mathrm{det}\,}
\renewcommand{\Im}{\mathrm{Im}\,}
\renewcommand{\Re}{\mathrm{Re}\,}

\newcommand{\gam}{\Gamma\backslash}

\newcommand{\map}[4]{\left\lbrace \begin{array}{ccc} #1 & \to & #2 \\ #3 & \mapsto & #4 \end{array} \right.}

\newcommand{\set}[2]{\left\lbrace #1 \,\middle|\, #2 \right\rbrace}

\author[B.~Delarue]{Benjamin Delarue}
\email{bdelarue@math.upb.de}
\address{Universität Paderborn, Warburger Str.~100, 33098 Paderborn, Germany}

\author[C. Guillarmou]{Colin Guillarmou}
\email{colin.guillarmou@universite-paris-saclay.fr}
\address{Université Paris-Saclay, CNRS, Laboratoire de mathématiques d’Orsay, 91405, Orsay, France}

\author[D.~Monclair]{Daniel Monclair}
\email{daniel.monclair@universite-paris-saclay.fr}
\address{Université Paris-Saclay, CNRS, Laboratoire de mathématiques d’Orsay, 91405, Orsay, France}

\title[Spectra of Lorentzian quasi-Fuchsian manifolds]{Spectra of Lorentzian quasi-Fuchsian manifolds}

\begin{document}

\begin{abstract} A three-dimensional quasi-Fuchsian Lorentzian manifold $M$ is a globally hyperbolic spacetime diffeomorphic to $\Sigma\times (-1,1)$ for a closed orientable surface $\Sigma$ of genus $\geq 2$. It is the quotient $M=\gam \Omega_\Gamma$ of an open set $\Omega_\Gamma\subset \AdS_3$ by a discrete group $\Gamma$ of isometries of $\AdS_3$ which is a particular example of an Anosov representation of $\pi_1(\Sigma)$. We first show that the spacelike geodesic flow of $M$ is Axiom A, has a discrete Ruelle resonance spectrum with associated (co-)resonant states, and that the Poincar\'e series for $\Gamma$ extend meromorphically to $\C$. This is then used to prove that there is a natural notion of resolvent of the pseudo-Riemannian Laplacian $\Box$ of $M$, which is meromorphic on $\C$ with poles of finite rank, defining a notion of quantum resonances and quantum resonant states related to the Ruelle resonances and (co-)resonant states by a quantum-classical correspondence. This initiates the spectral study of convex co-compact pseudo-Riemannian locally symmetric spaces.
\end{abstract}

\maketitle

\section{Introduction}

Three-dimensional Lorentzian quasi-Fuchsian manifolds can be viewed as Lorentzian analogues of quasi-Fuchsian hyperbolic $3$-manifolds.  The latter are complete infinite volume Riemannian manifolds, obtained as a quotient $\Gamma\backslash \H^3$ of the hyperbolic $3$-space $\H^3$ by a Kleinian group  $\Gamma\subset {\rm SL}(2,\C)$ which is a quasiconformal deformation of a co-compact Fuchsian subgroup  $\Gamma^0\subset \SL(2,\R)$ inside $\SL(2,\C)$. Moreover, $\Gamma\backslash \H^3$ is  diffeomorphic to a cylinder $(\Gamma^0\backslash \H^2) \times (-1,1)$. 
 In the Lorentzian case, the hyperbolic space $\H^3$ is replaced by the Anti-de Sitter space $\AdS_3:=\{x\in\R^4\, |\,  q(x)=-1\}$ where $q(x):=-x_1^2-x_2^2+x_3^2+x_4^2$, equipped with the Lorentzian metric induced by $q$ restricted to the tangent space $T\AdS_3$.
Lorentzian quasi-Fuchsian manifolds are constructed by considering quasi-Fuchsian subgroups (see Definition \ref{def - AdS quasi fuchsian group})
$\Gamma\subset {\rm SL}(2,\R)\times {\rm SL}(2,\R)$ of isometries of $\AdS_3$ acting properly discontinuously on a  domain $\Omega_\Gamma\subset \AdS_3$: the quotient $\Gamma\backslash \Omega_\Gamma$ inherits a metric from the Lorentzian metric of $\AdS_3$ that makes it a globally hyperbolic Lorentzian manifold diffeomorphic to a cylinder $(\Gamma^0\backslash \H^2) \times (-1,1)$ for some $\Gamma^0\subset {\rm SL}(2,\R)$, but the  manifold $\Gamma\backslash \Omega_\Gamma$ is not complete. 
\color{black}Quasi-Fuchsian manifolds are fundamental objects of study not only for their own sake but also as toy models for more general Lorentzian manifolds. In spite of this, their theory lacks the abundance of results that are available in the analogous Riemannian case (see e.g.\ \cite{Patterson,Sullivan,mazzeo-melrose87,Guillope-Zworski_Annals,BuOl,JoSB,PaPe,GrZw,GuillarmouAJM,Borthwick, Vasy})\color{black}, partly because Lorentzian geometry is in many respects more complicated than Riemannian geometry. In this paper we define a first notion of discrete spectra for three-dimensional Lorentzian quasi-Fuchsian manifolds, both at the level of the geodesic flow and of the pseudo-Riemannian Laplacian. 

Before discussing our situation, let us briefly recall known results on the spectral geometry of 
rank one Riemannian locally symmetric spaces. Such hyperbolic manifolds  have been intensively studied, not only from a geometric perspective but also because they possess a very rich chaotic dynamical system given by the geodesic flow, which is closely related to the spectrum of the Laplacian and its eigenfunctions.  The most famous example is Selberg's trace formula, which states that for compact or finite volume hyperbolic manifolds, 
the spectrum of the Laplacian is intimately related to the lengths of periodic geodesics \cite{Hejhal, Bu,BuOl_Book}. 
This correspondence can be understood via the representation theory of the co-compact or finite co-volume subgroups of isometries of the hyperbolic spaces. 

For hyperbolic manifolds of infinite volume, the situation is much more complicated as the Laplacian typically has essential spectrum with infinite multiplicity. A natural class of infinite volume  hyperbolic manifolds that have been the subject of extensive research since the 1980's is the class of convex co-compact hyperbolic manifolds \cite{Patterson,Sullivan,mazzeo-melrose87,Guillope-Zworski_Annals,Borthwick,BuOl,PaPe,GuillarmouAJM}. Such Riemannian manifolds $(M,g)$ can be compactified, their geodesic flow satisfies the Axiom A of Smale and the spectral/scattering theory of the Laplacian is now well-understood \cite{mazzeo-melrose87,JoSB,GrZw, Vasy}. In that case, one can define a notion of discrete spectrum for the Laplacian $\Delta_g$, called \emph{quantum resonance spectrum}, as the poles of the meromorphic extension of the resolvent $(\Delta_g-c(\la))^{-1}$ for some appropriate polynomial $c(\la)$ of degree $2$ \cite{mazzeo-melrose87}. The Selberg zeta function admits a meromorphic extension to the complex plane and its zeros/poles are given by the quantum resonances, plus some topological poles at $\la\in -\N$  \cite{BuOl,PaPe}. 

More recent developments using harmonic and microlocal analysis gave 
a novel perspective on these questions by defining a notion of discrete spectrum for the geodesic vector field $X$, viewed as a differential operator acting on functions or tensors. Namely, it has been shown (\cite{BL07,FS11,DZ16} in the Anosov case and \cite{DG16} in the Axiom A case) \color{black}  
that the resolvent $(X+\la)^{-1}$ of the operator $X$ acting on functions on the unit tangent bundle $T^1M$  admits a meromorphic extension to $\la\in \C$, the poles are called \emph{Ruelle resonances} and the associated (generalized) eigenfunctions are called \emph{Ruelle resonant states}. 
The discrete spectrum captures the fine dynamical properties, for example it encodes the SRB measure and the speed of decay of correlations. 

These resonances are also \color{black} poles of a dynamical zeta function defined similarly as Selberg's by a product over periodic geodesics, but involving different coefficients \cite{DZ16}. 
In the case of hyperbolic manifolds, there is a correspondence between Ruelle resonances and poles of the resolvent of Laplacians acting on tensors, as well as between the eigenfunctions \cite{dfg,GHW18, GHW18a,Had18}. This is interpreted as a classical-quantum correspondence.  

In the case of a Lorentzian manifold $(M,g)$, the pseudo-Riemannian Laplacian $\Box_g$ of the Lorentzian metric $g$ is not elliptic and typically has essential spectrum, thus there is no clear notion of a good spectral theory to define a discrete spectrum of $\Box_g$. \color{black} Concerning the geodesic flow, the part of the flow acting on the spacelike tangent bundle $T^1M=\{(x,v)\in TM\,|\, g(v)=1 \}$ is ``hyperbolic'' in some sense, but it is not Axiom A if $M$ is compact because its non-wandering set is non-compact. However, the universal cover, which  is a symmetric space, behaves \color{black} nicely:  for example, by inspecting the case of the anti-de Sitter space $\AdS_3$ equipped with its natural Lorentzian metric $g$, one can see that the resolvent $(\Box_g+\la(\la+2))^{-1}$ of $\Box_g$ admits a meromorphic extension to $\la\in \C$  by explicitly  writing down its integral kernel (see Section \ref{sec:Laplacian}). 

It is thus natural to wonder what the situation looks like in the intermediate case, namely the  Lorentzian convex co-compact locally symmetric spaces, both at the level of the spectral theory of the spacelike geodesic flow and of the pseudo-Riemannian Laplacian $\Box_g$.  The simplest non-elementary examples of such spaces are the quasi-Fuchsian Lorentzian manifolds. 

This paper establishes several spectral properties for quasi-Fuchsian Lorentzian manifolds that are analogous to those known for Riemannian convex co-compact hyperbolic manifolds.
More precisely, we prove that for a quasi-Fuchsian Lorentzian manifold $M=\Gamma\backslash \Omega_\Gamma$ equipped with its natural Lorentzian metric $g$ induced from $\AdS_3$:
\color{black}

\begin{enumerate}[leftmargin=*]
\item[(I)] The incomplete geodesic flow on the space-like unit tangent bundle $T^1M$ extends to a complete flow on a smooth manifold $\M\supset T^1M$ and we prove, using \cite{DG16}, that the resolvent 
\[R_X(\lambda)=(X+\la)^{-1}=\int_0^\infty e^{-t(X+\la)}dt : \CT(\M)\to \D'(\M)\] 
of the flow generator $X$ continues meromorphically from ${\rm Re}(\la)>0$ to $\lambda\in \C$, with poles forming a discrete spectrum of \emph{Ruelle resonances}. The elements in the range of the residue at a resonance $\lambda_0\in \C$ form a  finite-dimensional space and those in $\ker (X+\la_0) $ are called the \emph{Ruelle resonant states}.

\item [(II)] The Poincar\'e series $$\mc{D}_\la(x,x')=\sum_{\gamma\in \Gamma_{x,x'}^c}e^{-\la d(x,\gamma x')}$$ for $x,x'\in \Omega_\Gamma$ is holomorphic in ${\rm Re}(\la)>1$  and extends meromorphically to $\C$. Here $\Gamma_{x,x'}^c$ is the set of $\gamma\in 
\Gamma$ such that $x$ and $\gamma x'$ are joined by a (necessarily unique) spacelike geodesic and the distance $d(x,\gamma x')$ is the length of this geodesic.

\item[(III)] The pseudo-Riemannian Laplacian $\Box_g$ on $M$ admits a holomorphic resolvent in ${\rm Re}(\la)>-1$
\[R_{\Box_g}(\lambda) :\CT(M)\to \D'(M)\]
inverting $(\Box_g+\la(\la+2))$,  which continues meromorphically to $\lambda\in \C$. Its poles form a discrete spectrum of \emph{quantum resonances}.  The elements in the range of the residue at a resonance $\lambda_0\in \C$ form a  finite-dimensional space and those in $\ker (\Box_g+\la_0(\la_0+2)) $ are called the \emph{quantum resonant states}.
Moreover, we prove a version of quantum-classical correspondence between the two resolvents of $\Box_g$ and of $X$ and analyse the first poles. 
\end{enumerate}

While we take inspiration from similar results in Riemannian cases \cite{dfg,GHW18a}, we emphasise that the approach we take here is different, 
since we do not have a good spectral theory of $\Box_g$ at hand. Our philosophy is to use the microlocal method of \cite{DG16} to analyse the resolvent of the geodesic flow and to use it to first construct the Poincar\'e series, and finally the resolvent of $\Box_g$ using the Poincar\'e series.

Let us now explain our main results in more detail.

\subsection{Main results} 
The three-dimensional anti-de Sitter space $$\AdS_3=\{x\in \R^4\,|\, q(x)=-1\},\qquad q(x):=-x_1^2-x_2^2+x_3^2+x_4^2 $$ equipped with the Lorentzian 
metric given by $q|_{\AdS_3}$ is isometric to the Lie group $G=\SL(2,\R)$ equipped with its Killing Lorentzian metric (up to a normalisation constant).
The group $G\times G$ acts on $G$ as isometries by 
\[ (h_1,h_2)\cdot h=h_1hh_2^{-1}\]
 with isotropy group given by the diagonal subgroup $G_d:=\{ (h,h)\in G\times G\}$. 
The Riemannian analogue of $\AdS_3$ is the hyperbolic 3-space $\H^3$, which has $\SL(2,\C)$ as isometry group and ${\rm SU}(2)$ as isotropy group.
 Since  ${\rm SU}(2)$ is compact, every discrete subgroup of $\SL(2,\C)$ acts properly discontinuously from the left on $\H^3=\SL(2,\C)/{\rm SU}(2)$, in particular the convex co-compact subgroups. 
 In contrast, the diagonal subgroup $G_d$ of $G\times G$ is not compact and it is a very strong condition on a discrete subgroup $\Gamma\subset G\times G$ to act  properly discontinuously on $\mathrm{AdS}_3=(G\times G)/G_d$ \cite{BenoistProperness,KobayashiProperness}.

Instead, a quasi-Fuchsian group $\Gamma$ (defined in Definition \ref{def - AdS quasi fuchsian group}) has a non-empty $\Gamma$-invariant proper open subset $\Omega_\Gamma\subset \mathrm{AdS}_3$\footnote{In the main text we write $\O_+(\Lambda_\Gamma)$ instead of $\Omega_\Gamma$, which is a simplified notation for this introduction.} on which $\Gamma$ acts properly discontinuously and freely, defining the quasi-Fuchsian manifold $M:=\gam \Omega_\Gamma$. Here $\Omega_\Gamma$ is defined as the invisible domain (see Definition \ref{def - invisible domain})  of the limit set $\Lambda_\Gamma$ of $\Gamma$, which is a $\Gamma$-invariant acausal circle in the conformal boundary of $\AdS_3$ defined purely in terms of $\Gamma$, thus $M$ depends only on $\Gamma$. We refer to Section \ref{QFsubgroups} for precise definitions and properties.

These \emph{quasi-Fuchsian} groups, studied initially by Mess \cite{mess2007}\color{black},  are precisely the groups $\Gamma\subset G\times G$ of the form $\Gamma=\set{(\rho_1(\gamma),\rho_2(\gamma))}{\gamma\in \pi_1(\Sigma) }$, where $\rho_1,\rho_2:\pi_1(\Sigma)\to G= \SL(2,\R)$ are the holonomy representations of the fundamental group of a closed orientable surface $\Sigma$ of genus $\geq 2$ equipped with two hyperbolic metrics. Such marked metrics form the Fricke (or Teichmüller) space of $\Sigma$, and quasi-Fuchsian subgroups play an important role as an entry to the world of higher Teichmüller theory \cite{AnnaICM}. However, our focus does not lie on this aspect, as we shall consider just one quasi-Fuchsian group at a time rather than families of such groups. 
A quasi-Fuchsian group $\Gamma$ is called \emph{Fuchsian} if the representations $\rho_1,\rho_2$ from Mess' description are conjugate. In this case, $\Gamma$ is conjugate to the embedding of a co-compact torsion-free discrete subgroup of $\SL(2,\R)$ via the diagonal map $\SL(2,\R)\to\SL(2,\R)\times \SL(2,\R)$. 

The \emph{space-like unit tangent bundle} $T^1M\subset TM$ of a quasi-Fuchsian manifold $M=\gam \Omega_\Gamma$ is formed by all tangent vectors of length $1$ with respect to the Lorentzian metric $g$ on $M$. It is invariant under the geodesic flow. Analogously, on anti-de Sitter space we define $T^1\mathrm{AdS}_3\subset T\mathrm{AdS}_3$; notice that each fiber $T^1_x\AdS_3$ is a non-compact surface diffeomorphic to $\R\times \mathbb{S}^1$\color{black}. The space-like geodesics in $T^1\mathrm{AdS}_3$ are precisely those geodesics connecting two points in the conformal boundary $\partial\mathrm{AdS}_3=\mathbb{S}^1\times \mathbb{S}^1$. This corresponds to the Riemannian situation, where all geodesics in the unit tangent bundle $T^1\H^3$ connect two points in  $\partial\H^3=\mathbb{S}^2$. From a dynamical systems point of view, there is no essential difference between the space-like geodesic flow on $T^1\mathrm{AdS}_3$ and the Riemannian geodesic flow on $T^1\H^3$,  see Section \ref{sec:spacelike_flow}.

\color{black}Remarkably, when passing  to convex co-compact quotients of $\H^3$ and quasi-Fuchsian quotients of $\Omega_\Gamma\subset\mathrm{AdS}_3$, respectively, the geodesic and space-like geodesic flows on the quotients again turn out to have analogous dynamical properties: Both are so-called Axiom A systems (after \cite{smale67}), which means that the non-wandering set is compact and agrees with the closure of the union of all periodic geodesics, and the flow is uniformly hyperbolic on this set. For convex co-compact hyperbolic (Riemannian) manifolds,  this is a classical fact (see e.g. \cite{DG16} in a more general setting) while surprisingly in the Lorentzian quasi-Fuchsian case it seems to be a very recent observation: Apart from the general works \cite{Delarue_Monclair_Sanders, Delarue_Monclair_SandersII}, we are not aware of any  references.

Fix a  three-dimensional quasi-Fuchsian Lorentzian manifold $M:=\gam \Omega_\Gamma$. As a first result preparing the main results, we mention that Proposition \ref{prop - spacelike geodesic flow axiom A} provides a self-contained proof of the Axiom A property of the space-like geodesic flow $\varphi_t$ on $T^1M=T^1(\gam \Omega_\Gamma)$. One issue here is that the flow on  $T^1M$ is not complete: The flowout 
\[\widetilde \M:=\bigcup_{t\in \R}\varphi_t(T^1\Omega_\Gamma)\subset T^1\mathrm{AdS}_3\] 
is strictly larger than $T^1\Omega_\Gamma$, and it is not immediately clear whether $\Gamma$ acts properly discontinuously and freely on $\widetilde \M$. This is however true, as shown in Proposition \ref{prop - discontinuity domain unit spacelike tangent bundle}. Thus, 
\[\M:=\gam \widetilde \M\]
is a smooth manifold to which the complete flow $\varphi_t$ descends, extending the incomplete space-like geodesic flow on $T^1M\subset \M$. The actual statement of Proposition \ref{prop - spacelike geodesic flow axiom A} is now that $(\M,\varphi_t)$ is an Axiom A system. We call it the \emph{extended space-like geodesic flow}. The non-wandering set $\mathcal K$ of this flow, on which the essential dynamics takes place, is contained in $T^1M$.  

From a dynamical systems perspective, knowing that $\varphi_t$ is Axiom A makes it just as ``good'' as the geodesic flow of a Riemannian convex co-compact hyperbolic manifold (or any other Axiom A flow). Perhaps surprisingly, this insensitivity to the metric signature has the effect that the spectral properties of the generator $X$ of $\varphi_t$ are more easily accessible for us  than those of the Laplacian $\Box_g$ on $M$. We will therefore begin with the results on $X$. 

\subsubsection{Spectral theory of the generator of the extended space-like geodesic flow}

We use a simplified presentation; details being found in Section \ref{sec:resonancesstandalone}. 
First, observe that for $\lambda\in \C$ with $\Re \lambda >0$, the operator $\int_0^\infty e^{-t(X+\la)}\d t$ is a natural right inverse of $(X+\la)$: more precisely, it is defined by
\bqn
R_{X}(\lambda): L^\infty(\M)\to L^\infty(\M),\quad 
(R_{X}(\lambda)f)(x):= \int_0^\infty e^{-t\lambda} f(\varphi_{-t}(x))\d t
\eqn
and is continuous as an operator\color{black}
\bqn
R_{X}(\lambda): \CT(\mc{M})\to \mc{D}'(\M),
\eqn
where $\mc{D}'(\M)$ denotes the space of distributions on $\M$. This defines a holomorphic family of operators on $\{\Re \lambda>0\}\subset \C$ which are resolvents for the operator $-X$ in the sense that
\bqn
(X + \lambda)(R_{X}(\lambda)f)=R_{X}(\lambda)((X + \lambda)f)=f\qquad \forall\,f\in\CT(\mc{M}),
\eqn
$(X + \lambda)$ mapping $\mc{D}'(\mc{M})\to \mc{D}'(\mc{M})$ (resp.\ $\CT(\mc{M})\to\CT(\mc{M})$) on the left (resp.\ on the right). The resolvent $R_X(\la)$ propagates supports forward along the flow.

\begin{introthm}\label{introthm:resolvent}
The resolvent $R_{X}(\lambda):C_c^\infty(\mc{M})\to \mc{D}'(\mc{M})$ admits a meromorphic continuation from ${\rm Re}(\la)>0$ to  $\la\in \C$ with poles of finite rank. It has a simple pole at $\la_0:=h_{\rm top}-2$, where $h_{\rm top}$ is the topological entropy of the flow $\varphi_t$ on the non-wandering set $\mc{K}$ and the critical exponent of $\Gamma$, and no other 
pole in ${\rm Re}(\la)>\la_0-\eps$ for some $\eps>0$. The residue at $\lambda_0$ is a rank $1$ operator of the form
\[ \Pi_{\la_0}:={\rm Res}_{\la_0} R_{X}(\lambda)= u \otimes v\]
where $u$ and $v$ are respectively measures supported on the backward ($\mc{K}_-$) and forward ($\mc{K}_+$)  trapped sets defined by\color{black} 
\[{\mathcal K}_\pm:=\{x\in \M\,|\,\overline{\{\varphi_{\pm t}(x)\,|\,t\in [0,\infty)\}}\subset \M\text{ \emph{is compact}}\}.\] 
The measures $u,v$ can be expressed in terms of the Patterson-Sullivan measure on the limit set of $\Gamma$ (see Theorem \ref{theo:Poincare_series}).\color{black}
\end{introthm} 
More generally, we consider a $\C$-vector bundle $\mathcal E$ over $\M$ and a lift of $X$ to a first order differential operator $\mathbf{X}$ acting on smooth sections of $\E$, which contains the case of the Lie derivative $\mc{L}_X$ of differential forms. The poles of the meromorphic resolvent are called \emph{Ruelle resonances} of $\mathbf{X}$. They form a discrete spectrum in the complex plane intrinsic to $\mathbf{X}$, and each Ruelle resonance comes with finite-dimensional spaces of (generalized) \emph{resonant states}. For simplicity, let us discuss them in this introduction only in the case $\mathbf{X}=X$, with more precise and detailed statements found in Theorem \ref{thm:resmain1}. In particular, the latter implies that, given a Ruelle resonance $\lambda_0\in \C$ of $X$, the resolvent $R_{X}(\lambda)$ is of the form
\bqn
R_{X}(\lambda)
 = R_H(\lambda) + \sum_{j=1}^{J(\lambda_0)} \frac{(-1)^{j-1}(X + \lambda_0)^{j-1} \Pi_{\lambda_0}}{(\lambda-\lambda_0)^j},
\eqn
where $J(\lambda_0)\in \N$, $R_H(\lambda):\CT(\mc{M})\to \D'(\mc{M})$ is holomorphic near $\lambda_0$, and $\Pi_{\lambda_0}: \CT(\mc{M})\to \D'(\mc{M})$
is a finite rank operator whose image is given by the space $\mathrm{Res}_{X}^{J(\lambda_0)}(\lambda_0)$ where
\[
\forall n\in \N, \,\, \mathrm{Res}_{X}^{n}(\lambda_0):=\{u\in \D'(\mc{M})\,|\, \supp u\subset {\mathcal K}_-,\mathrm{WF}(u)\subset E_u^\ast, (X + \lambda_0)^{n} u=0 \}
\]
are spaces of \emph{generalized resonant states} of $X$ at $\lambda_0$ (the resonant states are those in 
$\mathrm{Res}_{X}^{1}(\lambda_0)$). \color{black}
Here   
 $\mathrm{WF}(u)$ denotes the wavefront set of the distribution $u$, and $E_u^\ast\subset T^\ast \M|_{{\mathcal K}_-}$ is the annihilator of the unstable and the neutral subbundles provided by the uniform hyperbolicity of the flow $\varphi_t$, see Lemma \ref{lem:Gammaextension}.

\subsubsection{Poincaré series}\label{introsec:poincare} Let us now summarize the main results of Section \ref{sec:poincare_series}. Given two points $x,x'\in \Omega_\Gamma$,  the geodesic distance $d(x,\gamma x')$ is well-defined for all $\gamma\in \Gamma_{x,x'}^c$ in the complement $\Gamma_{x,x'}^c$ of the finite set $ 
\Gamma_{x,x'}:=\{ \gamma\in \Gamma \, |\, -q(x,\gamma x')>1 \}$, 
 and the \emph{Poincaré series} at $(x,x')$ is defined by
\begin{equation}\label{defin_of_mcD_la}
\mc{D}_\la(x,x'):= \sum_{\gamma\in \Gamma_{x,x'}^c} e^{-\la d(x,\gamma x')},
\end{equation}
where $\lambda\in \C$  satisfies ${\rm Re}(\la)>\delta_\Gamma$ with $\delta_\Gamma$ the critical exponent of $\Gamma$, see \eqref{delta_Gamma}. 
The function $\mc{D}_\la(x,x')$ satisfies $\mc{D}_\la(x,\gamma_0x')=\mc{D}_\la(x,x')=\mc{D}_\la(\gamma_0 x,x')$ (see \eqref{D_invariance_by_Gamma}), thus $\Omega_\Gamma\times \Omega_\Gamma\owns (x,x')\to \mc{D}_\la(x,x')\in \C$ descends to a function $\mc{D}_\la:M\times M\to \C$. Notice that due to the dependence of $\Gamma_{x,x'}$ on $x,x'$, the continuity of  $\mc{D}_\la(x,x')$ in $x,x'$ is not clear from the definition \eqref{defin_of_mcD_la}. \color{black} 
Using the resolvent of the flow and microlocal arguments inspired by \cite{Dang-Riviere} (see also \cite{Chaubet}), we prove the following result in Proposition \ref{prop:Poincare_series} and  Theorem \ref{theo:Poincare_series}:
\begin{introthm}\label{introthm:poincare}
The Poincar\'e series $\mc{D}_\la$ extends meromorphically to $\lambda \in \C$ as a map in $C^0(K_1\times K_2)$ for any 
compact sets $K_1\times K_2\subset M\times M$, with the principal part \color{black} 
at each pole $\la_0$ given by the Schwartz kernel of a finite-rank operator. 
The following explicit relation between the Poincaré series and the meromorphic resolvent $R_X(\lambda)$ from Theorem \ref{introthm:resolvent} holds: 
\[ \mc{D}_\la= \frac{1}{2}(\pi_*\otimes \pi_*)R_X(\la+2)+  \frac{1}{2}(\pi_*\otimes \pi_*)R_X(\la-2)- (\pi_*\otimes \pi_*)R_X(\la).
\]
Here we identify $R_X(\la)$ with its distributional kernel in $\D'(T^1M\times T^1M)$ and $\pi_\ast:\D'(T^1M)\to \D'(M)$ is the pushforward of distributions along the bundle projection $\pi:T^1M\to M$, and $\pi_*\otimes \pi_*:\D'(T^1M\times T^1M)\to \D'(M\times M)$ the pushforward 
on both variables. 

One has $h_{\rm top}=\delta_\Gamma$ and $\mc{D}_\la$ has a simple pole at $\la=\delta_{\Gamma}$ with residue 
\[ {\rm Res}_{\delta_\Gamma} \mc{D}_\la(x,x')=c_{\delta_\Gamma} f_{\delta_\Gamma}(x)f_{\delta_\Gamma}(x')\]
and no other pole in ${\rm Re}(\la)>\delta_\Gamma-\eps$ for some $\eps>0$, with $c_{\delta_\Gamma}>0$ some constant and $f_{\delta_\Gamma}$ a smooth positive function given on $\Omega_\Gamma$ by 
\[f_{\delta_\Gamma}(x)=\int_{\Lambda_\Gamma} |q(x,\nu)|^{-\delta_\Gamma}\omega_{\delta_\Gamma}(\nu)\]
with $\omega_{\delta_\Gamma}$ the Patterson-Sullivan measure and $\Lambda_\Gamma\subset \mathbb{S}^1\times \mathbb{S}^1\simeq \pl \AdS_3$ the limit set of $\Gamma$ (recall Definition \ref{def - AdS quasi fuchsian group}) 
\end{introthm}
The relation between the Poincaré series and the resolvent of the flow on $2$-forms is proved in \cite{Dang-Riviere,Chaubet} but 
 the link with the resolvent of the flow acting on functions is rather surprising. A corresponding formula in the Riemannian case is not known, but can be studied using the same method. 
 \color{black}

\subsubsection{Spectrum of the pseudo-Riemannian Laplacian}

 The pseudo-Riemannian Laplacian $\Box_g$ on the quasi-Fuchsian Lorentzian manifold $M=\Gamma\backslash \Omega_\Gamma$ is formally self-adjoint when acting on $C_c^\infty(M)$, but it is not clear whether it is essentially self-adjoint on $L^2(M)$. On $\AdS_3$ however, $\Box_g$ has a self-adjoint extension with spectrum $ \sigma(\Box_g)=[1,\infty) \cup \{1-n^2\,|\,n\in \N^*\}$ (see \cite[Section 5]{Rossmann}) \color{black} and we show in Section \ref{sec:Laplacian} that it has a resolvent 
 \[ (\Box_g+\la(\la+2))^{-1}: L^2(\AdS_3)\to L^2(\AdS_3)\]
inverting $\Box_g+\la(\la+2)$ which is meromorphic in ${\rm Re}(\la)>-1$ with simple poles of infinite multiplicity at $\N$ (corresponding to discrete series of $\SL(2,\R)$), and which extends meromorphically to $\C$ as a map $C_c^\infty(\AdS_3)\to \mc{D}'(\AdS_3)$  with simple poles at $\Z$ and infinite multiplicity. Its integral kernel has the form $F_\la(-q(x,x'))$ for some explicit function $F_\la$, see Proposition \ref{prop:resolventAdS}.
The integral kernel $F_\la(-q(x,x'))$ can be slightly modified to remove the poles at $\la\in \Z$ into another integral kernel $F_\la^h(-q(x,x'))$ solving 
\[(\Box_g+\la(\la+2))F_\la^h(-q(x,x'))=\delta_{x=x'}\] 
in the region $\{(x,x')\in \AdS_3^2 \,|\, -q(x,x')>-1\}$, with $F_\la^h$ holomorphic in $\la\in \C$ and $F_\la^h(\zeta)=C(\la)F_\la(\zeta)$ in the region $\zeta>1$ for some explicit holomorphic function $C(\la)$. The function $F_\la^h$ \color{black} is given by (see Proposition \ref{prop:resolventAdShol})
\begin{equation}\label{def:F^h}
\begin{split}
& F^h _\lambda(\zeta):=\left \{\begin{array}{ll}
\dfrac{i}{4\pi\sqrt{\zeta^2-1}}(\zeta+\sqrt{\zeta^2-1})^{-(\la+1)},& \zeta>1, \\
\dfrac{1}{4\pi\sqrt{1-\zeta^2}}(\zeta+i\sqrt{1-\zeta^2})^{-(\la+1)} ,& \zeta \in (-1,1).
\end{array}\right.
\end{split}
\end{equation} 
For  $x,x'\in \Omega_\Gamma$ one has $-q(x,x')>-1$, thus the operator $R^h_{\Box_g}(\la)$ with integral kernel $F_\la^h(-q(x,x'))$ inverts $(\Box_g+\la(\la+2))$ on $\Omega_\Gamma$.

Motivated by the Riemannian case, we propose as a definition for the resolvent  of $\Box_g$ on $M$ the operator $R_{\Box_g}:C_c^\infty(M)\to \mc{D}'(M)$ with the following integral kernel: For $x,x'\in \Omega_\Gamma$ and ${\rm Re}(\la)>-1$
\begin{equation}\label{defRBox_g}
R_{\Box_g}(\la;x,x'):= \sum_{\gamma\in \Gamma} F_\la^h(-q(x,\gamma x'))
\end{equation}
which descends as an $L^1_{\rm loc}(M\times M)$ integral kernel that solves $(\Box_g+\la(\la+2))R_{\Box_g}(\la)f=f$ for all $f\in C_c^\infty(M)$.
In \eqref{defRBox_g}, we could also replace $F_\la^h(-q(x,\gamma x'))$ by the resolvent kernel $F_\la(-q(x,\gamma x'))$ on the cover $\AdS_3$, the result below would then hold modulo poles of possibly infinite rank at integers. We have not described this case in detail as it seems less natural to us to keep the infinite dimensional integral poles coming from the discrete series of ${\rm SL}(2,\R)$.

The following theorem is the main result concerning the spectral theory of the pseudo-Riemannian Laplacian $\Box_g$ and an associated quantum-classical correspondence.

\begin{introthm}\label{introthm:ext_resolvent}
The operator $R_{\Box_g}(\la)$ defined by \eqref{defRBox_g} admits a meromorphic extension to $\C$ as a continuous operator 
\[R_{\Box_g}(\la): C_c^\infty(M)\to \mc{D}'(M)\]
satisfying $(\Box_g+\la(\la+2))R_{\Box_g}(\la)f=f$ for all $f\in C_c^\infty(M)$, with poles of finite rank, and for all $f_1,f_2\in C_c^\infty(M)$ one has
\begin{equation}\label{relationbetweenresolvents}
\cjg R_{\Box_g}(\la)f_1,f_2\cjd= \frac{1}{2}\cjg R_X(\la)\pi^*f_1,\pi^*f_2\cjd-\frac{1}{2}\cjg R_X(\la+2)\pi^*f_1,\pi^*f_2\cjd+ \cjg H(\la)f_1,f_2\cjd,
\end{equation}
where $R_X(\la)$ is the meromorphic resolvent from Theorem \ref{introthm:resolvent} and $H(\la)$ is a holomorphic family of continuous 
operators $C_c^\infty(M)\to \mc{D}'(M)$.  
In particular, if $\Pi^X_{\la=0}$ and $\Pi_{\la_0}^{\Box_g}$ denote the residues of the resolvents of $X$ and $\Box_g$, for ${\rm Re}(\la_0)>\delta_{\Gamma}-4$, one has $\Pi^{\Box_g}_{\la_0}=\frac{i}{4\pi}\pi_*\Pi^{X}_{\la_0}\pi^*$ and 
\[ \pi_* : {\rm Ran}(\Pi^{X}_{\la_0}) \to  {\rm Ran}(\Pi^{\Box_g}_{\la_0})\]
is surjective. \color{black}

There is $\eps>0$ such that the resolvent $R_{\Box_g}(\la)$ has only the pole $\la=\delta_\Gamma-2$ 
in ${\rm Re}(\la)>\delta_\Gamma-2-\eps>0$, it is a simple pole with residue the rank $1$-operator
\[ \Pi^{\Box_g}_{\delta_\Gamma-2}=\frac{i}{2\pi}c_{\delta_\Gamma}f_{\delta_\Gamma}\otimes f_{\delta_\Gamma}\]
with $c_{\delta_\Gamma}$ and $f_{\delta_\Gamma}$ defined in Theorem \ref{introthm:poincare}.
\end{introthm}

We note that in the Riemannian case, a relation corresponding to \eqref{relationbetweenresolvents} is not known, as far as we know, and it could be possible to use the same type of proofs to investigate the precise relation in that setting too. \color{black}
This result also implies that the quantum resonances are contained in translates of the Ruelle resonance for the spacelike geodesic flow, and the quantum resonant states are necessarily pushforwards of Ruelle resonant states, in a way comparable to the Riemannian case 
\cite{dfg,GHW18a}. We also study the link between the quantum resonant states and the Ruelle resonant states invariant by the unstable derivatives. 
The pushforward $f=\pi_*u$ on  $M=\Gamma\backslash \Omega_\Gamma$ of a Ruelle resonant state (with resonance $\la$) $u\in \mc{D}'(T^1M)$ killed by the unstable derivatives of the geodesic flow satisfies $(\Box_g+\la(\la+2))f=0$ 
and we show in Corollary  \ref{c:resonant_states_correspondence} that there exists $\eps_0>0$ depending on the Lyapunov exponents of the flow such that the resonant states $f$ associated to quantum resonances in ${\rm Re}(\la)>\delta_\Gamma-\eps_0$ are necessarily pushforwards of Ruelle resonant states $u$ killed by unstable derivatives. We are not fully able to prove the converse: we do show that a nontrivial Ruelle resonant state (with resonance $\la$) killed by unstable derivatives provides  a nontrivial $\Gamma$-invariant solution $u\in \mc{D}'(\AdS_3)$ of $(\Box_g+\la(\la+2))u=0$ on $\AdS_3$, but we do not know if the restriction of $u$ to the open set $\Omega_\Gamma$ is identically $0$ or not. Finally, in Section \ref{s:Fuchsian}, we describe completely the quantum resonances and resonant states when the group $\Gamma$ is Fuchsian, i.e., the Lorentzian $3$-manifold becomes 
\begin{equation}\label{intro:Fuchsian} 
M= (-\pi/2,\pi/2) \times \Sigma, \quad g=-d\theta^2+\cos(\theta)^2g_{\Sigma}, 
\end{equation}
with $(\Sigma,g_\Sigma)=\Gamma\backslash \H^2$ a closed oriented hyperbolic surface and $\Gamma\subset \SL(2,\R)$ a co-compact discrete group.

\subsection{Further previous results, comments and open questions}

For the case of pseudo-Riemannian hyperbolic spaces over $\R,\C,\mathbb{H},\mathbb{O}$ which include $\Z_2$-quotients of anti-de Sitter spaces,  \cite{FrahmSpilioti,Roby} defined quantum resonances for the pseudo-Riemannian Laplacian by meromorphic extension of a distributional resolvent. 
In the case of $\AdS_3$, on the globally symmetric level there are many similarities between Riemannian and pseudo-Riemannian Laplacians: the Lorentzian Laplacian on  $\AdS_3=(\SL(2,\R)\times \SL(2,\R))/\mathrm{diag}\simeq \SL(2,\R)$ coincides with the Casimir operator of $\SL(2,\R)$, and as such it has a very well understood and explicitly described spectral theory thanks to the Fourier inversion formula, see Appendix \ref{app:resolvent} and \cite{Andersen,Rossmann,Strichartz89,StrichartzCorrigendum}. In particular it is self-adjoint on $L^2$ and its spectral decomposition reduces to the representation theory of ${\rm SL}(2,\R)$. The spectral theory of pseudo-Riemannian locally homogeneous spaces is initiated in \cite{Kassel_Kobayashi} where it is proved that $\Box_g$ is self-adjoint under certain conditions, and the discrete spectrum is studied; see \cite[Proposition 1.13]{Kassel_Kobayashi} for quotients  $\Gamma\backslash \AdS_3$ by discrete subgroups $\Gamma$ acting properly discontinuously and freely. For quasi-Fuchsian groups $\Gamma$, the action is not properly discontinuous on the whole $\AdS_3$ but only on an open subset $\Omega_\Gamma$. The quotient is a globally hyperbolic non-complete spacetime, and it is not clear how to construct a self-adjoint extension of $\Box_g$, where at least some boundary condition should appear. We are not aware of any previous results on the spectral properties of $\Box_g$, not even for the special case of Fuchsian groups. The problem here is that the ``boundary'' of $M=\Omega_\Gamma/\Gamma$ has a quite singular structure:  the boundary $\partial\Omega_\Gamma$ is made of lightlike geodesics with one endpoint in the limit set $\Lambda_\Gamma$ and the action of $\Gamma$ arbitrarily close to this set is losing more and more its proper discontinuity property. In the Fuchsian case it can be considered as two conical singularities at $\theta=\pm \pi/2$ in the model \eqref{intro:Fuchsian}. 
We have constructed a resolvent for $\Box_g$ by $\Gamma$-periodizing the resolvent of $\Box_g$ on $\AdS_3$ (or rather some subset of $\AdS_3$) and its properties show that it seems to be a reasonable family of operators inverting the Klein-Gordon operator. It would be interesting to understand its finer spectral properties and in particular if it comes from a good self-adjoint problem for $\Box_g$ on $M$.
By classical work of Hadamard \cite{Hadamard}, there are two natural inverses of $(\Box_g+\la(\la+2))$: the advanced and retarded inverses $R^\pm_{\Box_g}(\la)$, which are uniquely defined as maps $C_c^\infty(M)\to C^\infty(M)$ by the property that 
\[ \supp (R^\pm_{\Box_g}(\la)f)\subset J^\pm(\supp f)\]
with $J^\pm(\supp f)$ denoting the causal future/past of $\supp(f)$. These resolvents are holomorphic in $\lambda$ and not related to a spectral problem for $\Box_g$. On the other hand, it is known for example from \cite{Duistermaat_Hormander} that there are parametrices for $(\Box_g+\la(\la+2))$ on $M$, inverting this operator up to smoothing operators, and with particular properties of the wave-front set of these parametrices: these are called the Feynman/anti-Feynman parametrices. Exact Feynman/anti-Feynman inverses, called Feynman propagators, are also obtained in the globally hyperbolic case, see for example \cite{Radzikowski,Islam_Strohmaier,Vasy2} and \cite{Gerard-Wrochna} in the massive case, with the goal of constructing so-called Hadamard states for quantum field theory. It is proved in certain globally hyperbolic settings (static cases or asymptotically Minkowski) in \cite{Derezinski,Vasy0,Taira} that the Feynman propagator can be constructed as a resolvent of the Klein-Gordon operator, which is self-adjoint in these cases. Although we do not discuss it later, it turns out that our inverse $R_{\Box_g}(\la)$ is a Feynman propagator for $\Box_g+\la(\la+2)$ when $\la\in \R$. It is plausible that a Selberg trace formula holds for $\Box_g$ in this setting, relating poles of $R_{\Box_g}(\la)$ and zeros of a Selberg zeta function, this will be investigated in future work. In this direction, a Selberg zeta function is defined in \cite{Policott_Sharp} in the setting of projective Anosov representations in ${\rm SL}(n,\R)$ and is proved to have a meromorphic extension to $\C$. 
In a related direction, a  trace formula is proved in \cite{Strohmaier-Zelditch} for stationary spacetimes, but our case is not stationary and the trace formula obtained in \cite{Strohmaier-Zelditch} is for the timelike Killing vector field rather than $\Box_g$. 

We note that, in the context of the wave equation $\Box_gu=0$, there are notions of resonances (or quasi normal modes) in Lorentzian settings motivated by physics, such as those obtained in  \cite{SaBarretoZworski,Bachelotetal,Dyatlov_QNM}, but these are rather related to resonances of an elliptic operator associated to a spacelike slice or poles of the inverse of a family of elliptic operators coming from the wave equation for $\Box_g$. The resonances we define are rather poles of the resolvent of $\Box_g$ and associated to particular solutions of the Klein-Gordon equations. 

As for the geodesic flow, it is known that the extended space-like geodesic flow $\varphi_t$ on $\M$ is exponentially mixing on its non-wandering set: if $\Gamma$ is irreducible, this is proved in \cite[Thm.\ C]{Delarue_Monclair_Sanders} for Gibbs measures with respect to arbitrary Hölder potentials, and without the irreducibility assumption  in \cite{ChowSarkar2} for the equilibrium measure. 
 
\subsection{Structure of the paper}

In Section \ref{sec:geometry} we introduce all necessary tools from the geometry of the three-dimensional anti-de Sitter space and related concepts such as its isometry group, the conformal boundary and splittings of the tangent bundle of the space-like unit tangent bundle. In Section \ref{sec:Laplacian} we study explicit resolvents of the pseudo-Riemannian Laplacian on the three-dimensional anti-de Sitter space as well as the Poisson transform, for which we prove an injectivity result in Proposition \ref{injectivite_Poisson}. In Section \ref{sec:quasifuchsian} we introduce quasi-Fuchsian groups and study various properties of the obtained quotients. In particular, in Subsection \ref{sec:discdomT1}, we extend the spacelike geodesic flow to a complete flow and study its dynamics in Subsection \ref{sec:nonwanderingset}, the main result being Proposition \ref{prop - spacelike geodesic flow axiom A} establishing the Axiom A property. Utilizing the latter, we establish in Section \ref{sec:resonancesstandalone} the meromorphic continuation of the flow resolvent and define Ruelle resonances and (co-)resonant states. The arguments are formulated in a standalone fashion, only relying on certain abstract properties that are supposed to facilitate independent applications beyond the present paper. In Subsection \ref{sec:applicationofsection} they are shown to be satisfied in the case of the (extended) space-like geodesic flow.  Section \ref{Sec:Poincare} is devoted to the meromorphic continuation of Poincaré series, which are expressed in terms of flow resolvents in Subsection \ref{sec:seriesresolvent}. This is then applied in Section \ref{sec:resolventBox} to define and meromorphically continue resolvents for the pseudo-Riemannian Laplacian on the quasi-Fuchsian manifold $M$.  In Subsection \ref{s:Fuchsian} we study the Fuchsian case as a notable special case in which we obtain quite precise results. Finally, Appendix \ref{app:resolvent} contains an explicit computation of the Schwartz kernel of the resolvent of the pseudo-Riemannian Laplacian on $\AdS_3$ using the Fourier inversion formula for $\SL(2,\R)$.

\subsection*{Acknowledgements} B.~D.\ has received funding from the Deutsche Forschungsgemeinschaft (German Research Foundation, DFG) through the Priority Program (SPP) 2026 ``Geometry at Infinity'' as well as Project-ID 491392403 – TRR 358. 
We thank Alex Strohmaier and Michal Wrochna for providing useful references and Tristan Humbert for useful comments and for remarking 
typos in a preliminary version of the manuscript.  Finally, we thank the anonymous referee for numerous helpful remarks. \color{black}

\subsection*{Statements and Declarations}

The authors have no competing interests to declare. The authors confirm that all data relevant to this work is included in this article.

\section{Geometry of \texorpdfstring{$\AdS_3$}{anti-de Sitter space} and its unit spacelike tangent bundle}\label{sec:geometry}

We begin by introducing the three-dimensional anti-de Sitter space, its isometry group, conformal boundary and space-like unit tangent bundle, and several useful ``models'' (i.e., coordinates/parametrizations).  Each model allows for different simplifications:
\begin{itemize}
\item We first present $\AdS_3$ as a quadric in $\R^4$ in Section \ref{subsec:AdS3andboundary}. It is used to interpret the isometry group as the matrix group $\SO(2,2)$, so its action is understood through linear algebra. It is also a convenient model for the study of geodesics and the geodesic flow.
\item In Section \ref{sec:torus_coord}, we introduce an explicit diffeomorphism from $\AdS_3$ to the solid torus $\mathbb S^1\times\mathbb D$. It gives an explicit relation between the Lorentzian metric on $\AdS_3$ and the hyperbolic metric on the Poincaré disk $\mathbb D$, in a way that naturally extends to the conformal boundaries, making it a good candidate for global considerations.
\item A more Lie-theoretic language is developed in Section \ref{sec:SL2Coord}, where $\AdS_3$ is seen as the Lie group $\SL(2,\R)$ with a bi-invariant Lorentzian metric (induced by the Killing form on the Lie algebra), and the pseudo-Riemannian Laplacian is interpreted as a Casimir operator. The point stabilizers of the various homogeneous spaces associated to $\AdS_3$ have simple descriptions in this model.
\item Finally, we introduce the upper half space coordinates in Section \ref{sec:upperhalfplane}. This is mostly interesting for computations that happen near the conformal boundary of $\AdS_3$, and is useful to analyze the Poisson transform acting on distributions in Section \ref{s:Poisson_transform}, in particular for its injectivity in Proposition \ref{injectivite_Poisson}.
\end{itemize}
\color{black}
 Then we focus on various splittings of the tangent bundle of the space-like unit tangent bundle and pushforwards of distributions. For background and more details, see e.g.\ \cite{mess2007, dgk18,glorieux-monclair} and the references therein. \color{black} 

\subsection{Pseudo-Riemannian geometry of \texorpdfstring{$\AdS_3$}{anti-de Sitter space} and its conformal boundary} \label{subsec:AdS3andboundary}
On $\R^4$ we consider the pseudo-Riemannian metric defined by the non-degenerate quadratic form
\[ q(x):=q(x,x):=-x_1^2-x_2^2+x_3^2+x_4^2\] 
and write 
\begin{equation}\label{def:AdS3}
\AdS_3:=\{x\in\R^4\, |\,  q(x)=-1\}.
\end{equation}

The quadratic form $q$ restricts to the tangent bundle as a Lorentzian metric $g$
\[ \forall x\in \AdS_3, \, v\in T_x\AdS_3=\{ v\in \R^4\,|\, q(x,v)=0\}, \quad 
g_{x}(v,v):= q(v) .\]
Note that $\AdS_3$ inherits an orientation from the standard orientation of $\R^4$ thanks to the decomposition $\R^4=T_x\AdS_3\oplus \R\cdot x$ for any $x\in \AdS_3$. 
 As in any Lorentzian manifold, tangent vectors and geodesics fall into three types: let $x\in \AdS_3$ and $v\in T_x\AdS_3$ (i.e.\ $q(x,v)=0$).
\begin{enumerate}
\item If $v$ is \emph{timelike}, i.e.\ $g_x(v,v)=q(v)<0$, then it is tangent to a closed timelike geodesic on $ \AdS_3$ parametrised by
\[ t\mapsto \cos\left(\sqrt{-q(v)}t\right) x + \sin\left(\sqrt{-q(v)} t\right) \frac{v}{\sqrt{-q(v)}}.\]
\item If $v$ is \emph{lightlike}, i.e.\ $g_x(v,v)=q(v)=0$, the associated geodesic is parametrised by
\[ t\mapsto x+tv.\]
\item If $v$ is \emph{spacelike}, i.e.\ $g_x(v,v)=q(v)>0$, the associated geodesic is parametrised by
\begin{equation}\label{formule_geod_spacelike}
t\mapsto \cosh\left(\sqrt{q(v)}t\right) x + \sinh\left(\sqrt{q(v)} t\right) \frac{v}{\sqrt{q(v)}}.
\end{equation}
\end{enumerate}
The Lorentzian manifold $\AdS_3$ is also \emph{time oriented}, i.e.\ there is a global continuous choice of a connected component of the light cone $\{v\in T\AdS_3\setminus \{0\}\,|\, q(v)\leq 0 \}$, declaring a timelike or lightlike vector $v\in T_x\AdS_3$ to be \emph{future oriented} if $q(v,{\bf T}(x))<0$, where ${\bf T}$ is the vector field on $\AdS_3$ defined by
\[ {\bf T}(x_1,x_2,x_3,x_4)=(x_2,-x_1,x_4,-x_3).\] 
\color{black}
Note that points in $\AdS_3$ are not always connected by geodesics. Given $x,y\in\AdS_3$, there are four possibilities:
\begin{enumerate}
\item If $q(x,y)<-1$, then $x$ and $y$ are joined by a spacelike geodesic.
\item If $q(x,y)=-1$, then $x$ and $y$ are joined by a lightlike geodesic.
\item If $q(x,y)\in (-1,1)$, then $x$ and $y$ are joined by a timelike geodesic.
\item If $q(x,y)\geq 1$, then $x$ and $y$ are not joined by a geodesic.
\end{enumerate}
Spacelike and lightlike geodesics are homeomorphic to $\R$, so when two points $x,y\in\AdS_3$ are joined by a spacelike or lightlike geodesic, we can talk about the \emph{geodesic segment} between $x$ and $y$. On the other hand, timelike geodesics are topological circles, so there is an ambiguity when defining the geodesic segment between two points $x$ and $y$ that are joined by a timelike geodesic, just as in a Riemannian sphere. \color{black}
\begin{lemma}\label{Lem:q_and_distance}
Let $x,y\in \AdS_3$ and assume that $-q(x,y)>1$. Then there is a unique spacelike geodesic segment $\alpha_{x,y}$ with endpoints $x,y$ and 
\[-q(x,y)=\cosh(d(x,y))\]
where $d(x,y)$ denotes the geodesic distance between $x,y$, i.e.\ the length of $\alpha_{x,y}$ with respect to the anti-de Sitter metric $g$.
\end{lemma}
\begin{proof}  The isometry group $G$ of $\AdS_3$ acts transitively on $\AdS_3$ (see Section \ref{sec:isometries}), it thus suffices to consider the case $x={\rm e}=(1,0,0,0)$.
Now, we want to find $v\in T_{\rm e}\AdS_3$  and $t$ such that $q(v)=1$ and $\gamma_t(v)=y$ where $\gamma_t(v)$ is the spacelike geodesic passing through $x={\rm e}$ and tangent to $v$. By \eqref{formule_geod_spacelike}, these must solve ${\rm e}\cosh(t)+v\sinh(t)=y$. Taking the $q$ product of this expression with ${\rm e}$, we get $q({\rm e},y)=-\cosh(t)$ and there is a unique such $t\geq 0$ since $-q({\rm e},y)>1$ by assumption. Now we set 
$v=(\sinh(t))^{-1}(y-{\rm e}\cosh(t))$ and we check that 
\[q(v)=(\sinh(t))^{-2}(-1-\cosh(t)^2-2q({\rm e},y)\cosh(t))=(\sinh(t))^{-2}(-1+\cosh(t)^2)=1\qedhere.\]
\end{proof}

\color{black} We can use the map 
\begin{align*}
 f: \R^+ \times \AdS_3 &\to \R^4\\
  (s,x)&\mapsto sx
\end{align*}
to pull back the flat pseudo-Riemannian metric $q(x,x)$ to $f^*q=-\d s^2+s^2g$. The associated Laplacian $\Box_q=\d^*\d$ acting on functions (with ${\rm d}$ the exterior derivative and ${\rm d}^*$ its formal $L^2$-adjoint) \color{black} has the form 
\bq 
\Box_q=\pl_{x_1}^2+\pl_{x_2}^2-\pl_{x_3}^2-\pl_{x_4}^2=\pl_s^2 +3s^{-1}\pl_s+s^{-2}\Box_{g}.\label{eq:Laplacians}
\eq
We can also view $\AdS_3$ as the quotient 
\[ \AdS_3= \{x\in \R^4 \, |\, q(x)<0 \}/\R^+\]
where $\R^+$ acts by scaling $\lambda\cdot x=\lambda x$, $\lambda\in \R^+$. In this case, a natural compactification is given by 
\[ \bbar{\AdS}_3= \{x\in \R^4\setminus\{0\} \, |\, q(x)\leq 0 \}/\R^+\]
with boundary
\[ \pl \AdS_3= \{x\in \R^4\setminus\{0\} \, |\, q(x)=0 \}/\R^+.\]
We shall denote elements in the quotient spaces above by $[x]$. 
The boundary $\partial\AdS_3$ is naturally endowed with a conformal Lorentzian structure. Indeed, consider the radial projection $\pi:\R^4\setminus\{0\}\to  (\R^4\setminus\{0\})/\R^+\simeq \mathbb S^3$ and let $\Sigma\subset \R^4$ be a surface such that $\pi$ restricts to a diffeomorphism $\pi\vert_\Sigma:\Sigma\to \partial \AdS_3$.  Such surfaces exist, and the reference example throughout this paper will be 
\[\Sigma=\mathbb T^2:=\{x\in\R^4 \,|\, x_1^2+x_2^2=x_3^2+x_4^2=1\}=\mathbb{S}^1\times \mathbb{S}^1\]
\color{black}For any such surface $\Sigma$, the tangent spaces $T_x\Sigma\subset\R^4$ must be of signature $(1,1)$ for any $x\in\Sigma$, so the restriction of $q$ to $T\Sigma$ can be pushed forward by $\pi\vert_\Sigma$ to a Lorentzian metric $g_\Sigma:=\left(\pi\vert_\Sigma\right)_*q\vert_{T\Sigma}$ on $\partial\AdS_3$. Given another such surface $\Sigma'\subset \R^4$, the diffeomorphism $\psi:(\pi\vert_{\Sigma'})^{-1}\circ\pi\vert_\Sigma :\Sigma\to\Sigma'$ can be written as $\psi(\xi)=\lambda(\xi)\xi$ for all $\xi\in\Sigma$, where $\lambda:\Sigma\to \R^+$ is a smooth function. As $\d\psi=\lambda \d\xi+\xi \d\lambda$ and $q(\xi)=q(\xi,\d\xi)=0$ on $\Sigma$, we find that $g_{\Sigma'}=(\pi_*\lambda)^2 g_\Sigma$, so $g_{\Sigma'}$ is conformal to $g_{\Sigma}$, \color{black} and we have defined a Lorentzian conformal class on $\partial \AdS_3$. It is important to note, however, that there is no canonical choice of such a surface, i.e. there is no canonical representative of the conformal class in the sense of invariance under the isometry group of $\AdS_3$. Note moreover that for any smooth Lorentzian metric $g_\pl$ in the conformal class, there is a unique surface $\Sigma\subset\R^4$ such that $g_\pl=g_\Sigma$.  Indeed, write $g_\pl=\lambda^2 g_{\mathbb T^2}$ for some smooth positive function $\lambda:\mathbb T^2\to\R^+$, and set $\Sigma_\lambda=\set{\lambda(\xi)\xi}{\xi\in\mathbb T^2}$, so that $g_{\Sigma_\lambda}=g_\pl$ by the previous computation. \color{black}

We now introduce a useful decomposition of $\AdS_3\times \pl \AdS_3$ according to the causal relation between points and boundary directions:
\begin{equation}\begin{split}\label{defmcO} 
\mc{O}_\pm &:=\set{ (x,[\nu])\in \AdS_3\times \pl \AdS_3 }{\pm q(x,\nu)<0},\\
\mathcal O_0&:=\{ (x,[\nu]) \in \AdS_3\times \pl \AdS_3\,|\, q(x,\nu)=0 \},\\
\quad \mc{O}&:=\mc{O}_+\cup \mc{O}_-,\end{split}
\end{equation}
Note that $\mc{O}_\pm$ are open sets in $\AdS_3\times \pl \AdS_3$ and $\mc{O}_+$ consists of points $(x,[\nu])$ in $\AdS_3\times \pl \AdS_3$ for which  there is a spacelike geodesic through $x$ with endpoint $\nu$\color{black}. Similarly, the points in $\mathcal O_0$ are the initial and terminal points of lightlike geodesics. The closure of $\mc{O}_+$ in $\AdS_3\times \pl \AdS_3$ is 
\[ \bbar{\mc{O}}_+=\mc{O}_+\cup \mathcal O_0=\set{(x,[\nu])\in \AdS_3\times \pl \AdS_3}{q(x,\nu)\leq 0}.\]

There is finally a natural involution $I$ on $\AdS_3$ given by 
\[  I: \AdS_3 \to \AdS_3 , \quad I(x)=-x\]
and acting freely. We denote by $\AdS_3^+=\AdS_3/\{I,-I\}$ the quotient, so that $\AdS_3$ is a double cover of $\AdS_3^+$. The same also holds for the conformal boundary $\pl \AdS_3$ and the conformal boundary 
of $\AdS_3^+$ is $(\pl\AdS_3)/\{I,-I\}$. 
Since $q(x,y)$ satisfies $q(I(x),y)=-q(x,y)$, $q$ does not descend to $\AdS_3^+$, but $|q|$ does.

\subsection{Isometries of \texorpdfstring{$\AdS_3$}{anti-de Sitter space}}\label{sec:isometries}

For background on the material presented here, we refer the reader to \cite{mess2007, dgk18,glorieux-monclair} and the references therein. \color{black} The matrix group 
\[ {\rm O}(2,2)=\set{h\in \GL_4(\R)}{h \begin{pmatrix} 1&&&\\&1&&\\&&-1&\\&&&-1\end{pmatrix} h^T=\begin{pmatrix} 1&&&\\&1&&\\&&-1&\\&&&-1\end{pmatrix} }\]
 acts isometrically and transitively on $\AdS_3$ with stabilizer ${\rm O}(1,2)$ of the point ${\rm e}:=(1,0,0,0)\in \AdS_3$ and therefore $\AdS_3={\rm O}(2,2)/{\rm O}(1,2)$. Its identity component ${\rm SO}(2,2)_\circ$ coincides with the identity component of the isometry group of $\AdS_3$. \color{black}
It is also the group of orientation and time orientation preserving isometries. It  acts transitively on the boundary $\partial\AdS_3$, and this action preserves the conformal structure. Indeed, given a surface   $\Sigma\subset \R^4$  such that the radial projection $\pi$ restricts to a diffeomorphism $\pi\vert_\Sigma:\Sigma\to \partial \AdS_3$ and $h\in {\rm O}(2,2)$, we find that the surface $h^{-1}(\Sigma)$ has the same property and $h^*g_\Sigma=g_{h^{-1}(\Sigma)}$, thus $h^*g_\Sigma$ is conformal to $g_\Sigma$.

It will often be useful to pick a reference metric in this conformal class. In this paper, we choose $g_{\mathbb T^2}$ via the identification $\pl\AdS_3\simeq \mathbb{T}^2$. \color{black}  It will however be important to remember that the metric $g_{\mathbb T^2}$ is not invariant under the isometry group. For $h\in {\rm O}(2,2)$, we will denote by $N_h\in C^\infty(\TT^2)$ the conformal factor, i.e.\ the function such that
\begin{equation}
h^*g_{\mathbb T^2}=N_h^{-2} g_{\mathbb T^2}. \label{eqn conformal factor}
\end{equation}
The formula for $N_h$ is given by 
\[ N_h(\nu)=\sqrt{(h\nu)_1^2+(h\nu)_2^2}.\]
Consider the subgroup 

\[ K=\set{ \begin{pmatrix} \cos\theta_1 & -\sin\theta_1 &&\\ \sin\theta_1&\cos\theta_1&&\\&&\cos\theta_2&-\sin\theta_2\\&&\sin\theta_2&\cos\theta_2\end{pmatrix}}{\theta_1,\theta_2\in\R}\subset {\rm SO}(2,2)_\circ. \]
It is a maximal compact subgroup of ${\rm SO}(2,2)_\circ$  isomorphic to ${\rm SO}(2)\times {\rm SO}(2)$. The action $K\curvearrowright \partial\AdS_3$ is free and transitive, and it preserves the flat Lorentzian metric $g_{\mathbb T^2}$, i.e.\ $N_h\equiv 1$ when $h\in K$.

The actions ${\rm SO}(2,2)_\circ\curvearrowright \AdS_3$ and ${\rm SO}(2,2)_\circ\curvearrowright \partial\AdS_3$ are both transitive, however we will often encounter the diagonal action ${\rm SO}(2,2)_\circ\curvearrowright \AdS_3\times\partial\AdS_3$ which is not. However the actions on the open orbit $\mathcal O_+\subset \AdS_3\times\partial\AdS_3$ and the closed orbit $\mathcal O_0\subset \AdS_3\times\partial\AdS_3$  defined in \eqref{defmcO} are transitive. These facts are proved in Lemma \ref{l:transitive}.

\subsection{The torus coordinates}\label{sec:torus_coord}
There is another model for $\AdS_3$ that is geometrically also appealing and that relates to the hyperbolic space 
$\H^2$. This model has the advantage to exhibit explicitly the conformal compactification of $\AdS_3$ (and its smooth structure) 
and to represent the conformal boundary as $\mathbb{T}^2$, similarly to the ball model for the hyperbolic space $\mathbb{H}^n$. \color{black}
Let $\mathbb D\subset \mathbb R^2$ be the open unit disk, and consider the diffeomorphism\color{black}
\begin{equation}\label{tildepsi} 
\widetilde{\psi}: \AdS_3\to \mathbb{S}^1 \times \mathbb{D} , \quad \widetilde{\psi}(x)=\Big( \frac{x_1}{R},\frac{x_2}{R},\frac{x_3}{1+R},\frac{x_4}{1+R}\Big) 
\end{equation}
where $R:=|(x_1,x_2)|=\sqrt{x_1^2+x_2^2}$.  The metric $g$ becomes 
\begin{equation}\label{metricmodeleTorus}
\widetilde{\psi}_*g= -\frac{(1+r^2)^2}{(1-r^2)^2}{\rm d}\theta^2+\frac{4({\rm d}r^2+r^2{\rm d}\varphi^2)}{(1-r^2)^2}= -\frac{(1+r^2)^2}{(1-r^2)^2}{\rm d}\theta^2+g_{\H^2}
\end{equation}
where we use the coordinates $e^{i\theta} \in \mathbb{S}^1, z=re^{i\varphi}\in \mathbb{D}$ and $g_{\H^2}=4\frac{{\rm d}r^2+r^2{\rm d}\varphi^2}{(1-r^2)^2}$ denotes the hyperbolic metric on $\mathbb{D}$. In these coordinates, the pseudo-Riemannian Laplacian is given by
\begin{equation}\label{Laplacian1}
\Box_{g}= R^{-2}\pl_\theta^2-\frac{1}{4}(1-r^2)^2\pl_r^2-\frac{1}{4}(1-r^2)^2\frac{\pl_r(rR)}{rR}\pl_r-\frac{1}{4}(1-r^2)^2\pl_\varphi^2
\end{equation}
where $R=\frac{1+r^2}{1-r^2}$.
The inverse map of $\til{\psi}$ is (by identifying $\R^4$ with $\C\times \C$)
\[ 
\til{\psi}^{-1}(e^{i\theta},re^{i\varphi})=\Big( e^{i\theta}\frac{1+r^2}{1-r^2}, \frac{2re^{i\varphi}}{1-r^2}\Big)\in \C\times \C.
\]
The map $\widetilde{\psi}$ extends to a homeomorphism $\widetilde\psi:\overline{\AdS_3}\to \mathbb S^1\times\overline{\mathbb D}$ thanks to the formula \eqref{tildepsi}. \color{black} It maps $\partial\AdS_3$ to $\mathbb{S}^1\times \mathbb{S}^1=\mathbb{T}^2$, with a simple formula:
\begin{equation} \label{eqn tildepsi boundary}
\widetilde\psi([\nu])=\left( \frac{1}{\sqrt{\nu_1^2+\nu_2^2}}(\nu_1,\nu_2),\frac{1}{\sqrt{\nu_3^2+\nu_4^2}}(\nu_3,\nu_4) \right). 
\end{equation}
Note that by definition a lift $\nu\in \R^4$ of a point in $\partial\AdS_3$ must satisfy $\nu_1^2+\nu_2^2=\nu_3^2+\nu_4^2$, i.e.\ the two normalising factors in \eqref{eqn tildepsi boundary} are equal. The map $\widetilde{\psi}$ induces a smooth structure on the compactification of $\AdS_3$ by using the smooth structure of $\mathbb{S}^1\times \bbar{\mathbb{D}}$.
 The $\AdS_3$ metric  of \eqref{metricmodeleTorus} is conformal to a non-degenerate Lorentzian metric on $\mathbb{S}^1\times \bbar{\mathbb{D}}$, namely to the metric  
 \[ -(1+r^2)^2\d\theta^2+4\d r^2+4r^2\d\varphi^2.\]
 The  conformal class of the Lorentzian metric 
 \begin{equation}\label{conformal_metric_infty} 
-\d\theta^2+\d\varphi^2
 \end{equation}
on $\mathbb{T}^2$  is pulled back by $\widetilde\psi$ to the aforementioned conformal class on $\partial\AdS_3$. Indeed, \eqref{conformal_metric_infty} is simply the restriction of $g$ to the tangent space $T\mathbb{T}^2$ of $\mathbb{T}^2\subset\R^4$.

It is convenient for calculations of the Laplacian to introduce a new radial coordinate $t$ 
 on the solid torus by setting 
\[ r= \tanh(t/2), \textrm{ with }t\in (0,\infty).\]
In particular, we compute that $R=\cosh(t)$, the metric and volume measure become
\[ g=dt^2+\sinh(t)^2\d\varphi^2-\cosh(t)^2\d\theta^2, \quad {\rm dv}_g=\frac{1}{2}\sinh(2t)\,\d t \d\theta \d\varphi,\] 
and the pseudo-Riemannian Laplacian becomes 
\begin{equation}\label{Boxg_t_coordinate}
\Box_{g}=-\frac{\partial^2}{\partial t^2}-2 \coth(2t)\frac{\partial}{\partial t}+\frac{1}{\cosh(t)^2}\frac{\partial^2}{\partial \theta^2}-\frac{1}{\sinh(t)^2}\frac{\partial^2}{\partial \varphi^2}.
\end{equation}

\subsection{The \texorpdfstring{$\SL(2,\R)$}{SL(2,R)} coordinates}\label{sec:SL2Coord}

The determinant  $\det:M_2(\R)\to\R$ is a quadratic form on the space $M_2(\R)$ of real $2\times 2$ matrices, which allows us to define useful algebraic coordinates on $\AdS_3$.  Consider
 the isomorphism 
\[ \psi:\map{\R^4}{M_2(\R)}{x=(x_1,x_2,x_3,x_4)}{\left(\begin{array}{cc}
x_1+x_3 & x_2+x_4 \\
x_4-x_2 & x_1- x_3
\end{array}\right)}.
\]
Its inverse is 
\begin{equation}\label{identification} 
\psi^{-1}\begin{pmatrix} a&b\\c&d \end{pmatrix} = \left(\frac{a+d}{2},\frac{b-c}{2},\frac{a-d}{2},\frac{b+c}{2} \right). \end{equation}
Note that for all $x,y\in \R^4$ with $q(y)\not=0$  (recall Lemma \ref{Lem:q_and_distance})\color{black}
\begin{equation}\label{detvsq}
\det(\psi(x))= -q(x), \quad \frac{q(y)}{2}{\rm Tr}(\psi(x)\psi(y)^{-1})=-q(x,y).
\end{equation}
This means that $\psi$ restricts to a diffeomorphism 
\[\psi:\AdS_3\to G:= \SL(2,\R)\] 
and the map descends to a map 
$\AdS_3^+\to {\rm PSL}_2(\R)$. Through this correspondence, the ${\rm Id}\in \SL(2,\R)$ element is mapped to the point ${\rm e}=(1,0,0,0)\in \AdS_3$. 
The differential $\d\psi({\rm e})$ at ${\rm e}$ maps 
$T_{\rm e}\AdS_3$ to $\mathfrak{sl}(2,\R)={\rm Lie}(G)$. \color{black} For $v\in T_{\rm e}\AdS_3$ (i.e.\ $v_1=0$), we have
 \[ \d\psi({\rm e})(v)=\left(\begin{array}{cc}
v_3 & v_2+v_4 \\
v_4-v_2 & - v_3
\end{array}\right).\]
The standard basis of $\mathfrak{sl}(2,\R)$  
\begin{equation}\label{basissl2}
X=\left( \begin{array}{cc}
 1/2 & 0\\
 0 & -1/2
\end{array}\right), \quad V=\left( \begin{array}{cc}
 0 & -1/2\\
 1/2 & 0
\end{array}\right), \quad X_\perp=[X,V]=\left( \begin{array}{cc}
 0 & -1/2\\
 -1/2 & 0
\end{array}\right)
\end{equation} 
satisfies $$X=d\psi({\rm e})(0,0,1/2,0),\qquad V=d\psi({\rm e})(0,-1/2,0,0),\qquad X_\perp=d\psi({\rm e})(0,0,0,-1/2).$$

The group $G=\SL(2,\R)$ acts on the left and on the right by multiplication on $M_2(\R)$ (thus on $\SL(2,\R)$\color{black}), which induces an action of $G\times G$ on $\AdS_3$ by
\begin{equation}\label{action_left_right}
 (h_1,h_2)\cdot x = \psi^{-1}\left( h_1\psi(x)h_2^{-1}\right),\quad x\in \AdS_3, (h_1,h_2)\in G\times G.
 \end{equation}
This action is isometric, and the induced morphism $G\times G\to {\rm Isom}(\AdS_3)$ is a local isomorphism onto the connected component ${\rm Isom}(\AdS_3)_\circ$, showing that $G\times G$ is a double cover of ${\rm SO}(2,2)_0$. \color{black}
The isotropy subgroup of the point ${\rm e}\in \AdS_3$ is the diagonal subgroup
$G_d:=\{(h,h)\in G\times G\}$. The $G\times G$-action extends to $\pl \AdS_3$ by the expression
\[ (h_1,h_2)\cdot [\nu] = [\psi^{-1}(h_1\psi(\nu)h_2^{-1})],\]
with the stabilizer of the point $\nu_0:=[1,0,-1,0]\in \pl \AdS_3$ given by 
\bq
(G\times G)_{\nu_0}=\Big\{\left(\begin{pmatrix}
a & 0\\ x & a^{-1}
\end{pmatrix},\begin{pmatrix}
b & y\\ 0 & b^{-1}
\end{pmatrix}\right)\,|\, a,b,x,y\in \R,ab>0\Big\}.\label{eq:parabolic1}
\eq
This gives us a  description of $\pl \AdS_3$ as a homogeneous space: $\pl \AdS_3=(G\times G)/(G\times G)_{\nu_0}$.  \color{black}
 
Let us now give a transitivity result that will be useful for the analysis of the Poisson transform.
\begin{lemma}\label{l:transitive}
1) The diagonal actions of $G\times G$ on the set $\mc{O}_+$ and on the set $\mc{O}_0$ defined in \eqref{defmcO} are transitive\color{black}.\\ 
2) For each $([\nu],[\nu'])\in 
\pl \AdS_3\times \pl \AdS_3$ such that $\nu\not=-\nu'$, there is $(h_1,h_2)\in G\times G$ such that both $(h_1,h_2)\cdot [\nu]$ and $(h_1,h_2)\cdot [\nu']$ lie in the open set $\{[x] \in \pl\AdS_3\,|\, x_2+x_4>0 \}$.
\end{lemma}
\begin{proof}
1) The transitivity of the $G\times G$ action on $\mc{O}_+$ is direct to see: one can bring any $h_0\in G$ to ${\rm Id}$ by left translation by $h_0^{-1}$, and then bring any $[\nu]$ satisfying $-\nu_1=q({\rm e},\nu)<0$ 
to $[(\nu_1,0,\nu_1,0)]=[(1,0,1,0)]  \in \pl \AdS_3$ using the isotropy subgroup $G_d$. This is achieved by diagonalisation of the matrix $\psi(\nu)$ (which has disjoint eigenvalues as ${\rm Tr}(\psi(\nu))=-2q({\rm e},\nu)>0$ and $\det \psi(\nu)=0$). We also bring any  $[\nu]$ satisfying $-\nu_1=q({\rm e},\nu)=0$ to $[(0,1,0,1)] \in \pl \AdS_3$ by the Jordan normal form  of $\psi(\nu)$.\\ 
2) Using the transitivity of the  $G$-action on $\pl \AdS_3$, it suffices to consider the case $[\nu]=[(0,1,0,1)]$.
The transformations in $G\times G$ fixing $\psi(0,1,0,1)$ have the form 
\[h_1=\left(\begin{array}{cc}
a & b \\
0 & a^{-1}
\end{array}\right), \quad h_2=\left(\begin{array}{cc}
a^{-1} & b' \\
0 & a
\end{array}\right), \quad a>0, \,\, b,b'\in \R.\]
A computation then gives that $(h_1,h_2)\cdot \nu'$ has $x_2+x_4$ component (the upper right coefficient of 
$\psi((h_1,h_2)\cdot \nu')$ given by 
\begin{equation}\label{x_2+x_4} 
-ab'(\nu'_1+\nu_3')+\nu'_2+\nu_4'-bb'(\nu'_4-\nu'_2)+ba^{-1}(\nu_1'-\nu_3').
\end{equation}
Taking $b=\nu_1'-\nu_3'$, $b'=-(\nu_1'+\nu_3')$ and $a\gg 1$ large if $\nu_1'+\nu_3'\not=0$ (resp. 
$a^{-1}\gg 1$ large if $\nu_1'-\nu_3'\not=0$), one obtains that \eqref{x_2+x_4} is positive. It remains to consider the case $\nu_1'=\nu_3'=0$. Then, since $\nu'\neq 0$ and $0=q(\nu')=-{\nu'_2}^2+{\nu'_4}^2$, we either have $\nu'_2=-\nu'_4\neq 0$, in which case it suffices to take $b=-1$ and $b'=\mathrm{sign}\, \nu_4'$, or $\nu'_2=\nu'_4>0$ (by our assumption that $\nu\not=-\nu'$), in which case $[\nu']=[\nu]$ and 
 \eqref{x_2+x_4} is positive.
\end{proof}

The Killing form of $G$ is $B(Y,Y'):=4{\rm Tr}(YY')$ if $Y,Y'\in \mathfrak{sl}(2,\R)$. 
If we view $B$ as a left-invariant Lorentzian metric $b$ on $G=\SL(2,\R)$, then for each $h\in G$ and $Y\in T_gG$ 
we have ${\rm Tr}(h^{-1}Y)=0$ and 
\[ b_h(Y,Y)=4{\rm Tr}(h^{-1}Yh^{-1}Y)=-8\det(h^{-1}Y)=-8\det(Y).\] 
We deduce that the isomorphism $\psi:\AdS_3\to {\rm SL}(2,\R)$ \color{black} is an isometry up to a constant scaling factor: $\psi^*b=8g$. The vectors $X,V,X_\perp$ satisfy  $B(X,X)=2$, $B(X_\perp,X_\perp)=2$ and $B(V,V)=-2$, i.e.\ $X,X_\perp$ are spacelike and $V$ timelike. We also have $[V,X_\perp]=X$ and $[X,X_\perp]=V$.
The pseudo-Riemannian Laplacian is equal to minus the Casimir of $\SL(2,\R)$, given by 
\[ \Box_{g}=4(V^2-X^2-X_\perp^2).\]

\subsection{The upper half-space coordinates}\label{sec:upperhalfplane} Let us consider the Lorentzian manifold 
\[  \H^{2,1}:=\set{ (y_1,y_2,s)\in \R^3}{s>0} ,\quad g_{\H^{2,1}}=\frac{\d s^2-\d y_1^2+\d y_2^2}{s^2}\]
Following \cite{Seppi}, we have an isometric embedding $\iota:\H^{2,1}\to\AdS_3$ given by 
\begin{equation}\label{iota_def}
 \iota (y_1,y_2,s)=\Big(\frac{y_1}{s}, \frac{1+m(y)+s^2}{2s},\frac{y_2}{s},\frac{1-m(y)-s^2}{2s}\Big)\in \AdS_3
\end{equation}
 where $m(y)=-y_1^2+y_2^2$ is \color{black} the Minkowski metric on $\R^{2}$. Its image 
is the open set $\{x\in \AdS_3\,|\, x_2+x_4>0\}$ and its closure forms a fundamental domain for the involution $I$. 
This embedding extends to the boundary by
\bq
 \iota : \map{\R^2}{\partial\AdS_3}{(y_1,y_2)}{\Big[y_1,\frac{1+m(y)}{2},y_2,\frac{1-m(y)}{2}\Big].} \label{eq:iotaext1}
 \eq
The upper half-space model will be useful to analyze the Poisson transform.

\subsection{Spacelike geodesic flow and geodesic distance on \texorpdfstring{$\AdS_3$}{anti-de Sitter space}}\label{sec:spacelike_flow}

The tangent bundle $T\AdS_3$ is the submanifold of $\R^8$ given by 
\[ T\AdS_3=\set{(x,v)\in \R^4\times \R^4}{q(x)=-1, q(x,v)=0}. \]
The tangent space at a point $(x,v)\in T\AdS_3$ is given by
\[ T_{(x,v)}T\AdS_3=\set{(\xi_x,\xi_v)\in\R^4\times\R^4}{q(x,\xi_x)=q(x,\xi_v)+q(\xi_x,v)=0}.\]
We shall work with the \emph{unit spacelike tangent bundle}
\begin{equation}\label{T1ads3} 
T^1\AdS_3=\set{(x,v)\in \R^4\times \R^4}{q(x)=-1, q(x,v)=0, q(v)=1}, 
\end{equation}
its tangent space at $(x,v)\in T^1\AdS_3$ is
\[ T_{(x,v)}T^1\AdS_3=\set{(\xi_x,\xi_v)\in\R^4\times\R^4}{q(x,\xi_x)=q(x,\xi_v)+q(\xi_x,v)=q(v,\xi_v)=0}.\]
The action of ${\rm Isom}(\AdS_3)$ on $T^1\AdS_3$ is $h\cdot(x,v)=(hx,hv)$, and it induces an action of $G\times G$ on $T^1\AdS_3$ given explicitly for $h_1,h_2\in G=\SL(2,\R)$ by
\begin{equation}\label{actionGxG}
(h_1,h_2)\cdot(x,v):=((h_1,h_2)\cdot x,(h_1,h_2).v)= (\psi^{-1}( h_1\psi(x)h_2^{-1}),\psi^{-1}( h_1\psi(v)h_2^{-1})).
\end{equation}
The fiber $T^1_x\AdS_3$ inherits the Lorentzian metric $g_x=q|_{T^1_x\AdS_3}$, and if $\pi:T^1\AdS_3\to \AdS_3$ is the projection onto the base, an element $\xi\in \ker d\pi_{(x,v)}$ has the form $\xi=(0,\xi_v)$ with $q(\xi_v,x)=q(\xi_v,v)=0$ and 
\[
g_x(\xi_v,\xi_v)=q(\xi_v,\xi_v).
\]
The \emph{spacelike geodesic flow} on $T^1\AdS_3$ is given by
\[\varphi_t(x,v)=(x\cosh(t)+v\sinh(t),x\sinh(t)+v\cosh(t)),\]
it is generated by the vector field $X$ given\footnote{Not to be confused with the Lie algebra element $X\in {\rm sl}(2,\R)$ defined in \eqref{basissl2} (or its left-invariant vector field); we will see in  \eqref{eq:XXrel} how they are related. } by
\[ X(x,v)=(v,x)\in T_{(x,v)}T^1\AdS_3,\quad (x,v)\in T^1\AdS_3.\]
Analogously to the Riemannian case, this flow can be seen as the restriction to an energy level of a Hamiltonian flow. Indeed, the Liouville symplectic form on $T^*\AdS_3$ can be transferred to $T\AdS_3$ using the metric $g$. It is then given by 
\[ \omega_{(x,v)}(\xi,\xi')= q(\xi_x,\xi'_v)-q(\xi_v,\xi'_x) ,\quad \xi=(\xi_x,\xi_v),\xi'=(\xi'_x,\xi'_v)\in T_{(x,v)}(T\AdS_3).\]
The geodesic vector field $X$ on $T\AdS_3$ is then defined as the Hamilton vector field of $H_q:(x,v)\mapsto \frac{1}{2}q(v)$, thus  
\[ X(x,v)=(v,q(v)x).\]
The geodesic flow $\varphi_t: T\AdS_3\to T\AdS_3$ is the flow of $X$, and a direct computation gives (recall \eqref{formule_geod_spacelike}) \color{black}
\begin{equation}\label{geodflow} 
\left\{\begin{array}{ll}
\varphi_t(x,v)=(x\cosh(t)+v\sinh(t),x\sinh(t)+v\cosh(t)) , & \textrm{if } q(v)=1 \\
 \varphi_t(x,v)=(x\cos(t)+v\sin(t),-x\sin(t)+v\cos(t)) , & \textrm{if } q(v)=-1\\
  \varphi_t(x,v)=(x+tv,v) , & \textrm{if } q(v)=0.
\end{array}\right.
\end{equation}
Thus, the spacelike unit tangent bundle $T^1\AdS_3$ can be identified with the level set $H_q=\frac{1}{2}$ and the geodesic flow $\varphi_t\curvearrowright T^1\AdS_3$ is the restriction of the Hamiltonian flow to $H_q$. As in the Riemannian case, the spacelike unit tangent bundle is a contact manifold when equipped with the tautological one-form $\alpha\in C^\infty(T^1\AdS_3; T^*(T^1\AdS_3))$ given by
\begin{equation}\label{def_alpha} 
\alpha_{(x,v)}(\xi)=q(\xi_x,v),\quad (x,v)\in T^1\AdS_3,~\xi=(\xi_x,\xi_v)\in T_{(x,v)}(T^1\AdS_3).
\end{equation}
It satisfies $\alpha(X)=1$ and the Liouville symplectic form (on $\ker \alpha$) is given by $\omega=d\alpha$. Hence $X$ is the Reeb vector field of $\alpha$. 

\subsection{Splittings of the tangent bundle of \texorpdfstring{$T^1\AdS_3$}{the unit spacelike tangent bundle}}\label{splittings_T}

The tangent bundle $T\left(T^1\AdS_3\right)$ has several subbundles invariant under both the action of ${\rm SO}(2,2)_\circ$ and the differential of the geodesic flow $\varphi_t$. The expression for the latter at a point $(x,v)\in T^1\AdS_3$, applied to $(\xi_x,\xi_v)\in T_{(x,v)}(T^1\AdS_3)$, is
\begin{equation}\label{dvarphi_t}
 {\rm d}\varphi_t(x,v)(\xi_x,\xi_v)=(\cosh(t) \xi_x + \sinh(t) \xi_v, \sinh(t)\xi_x + \cosh(t) \xi_v).
 \end{equation}
The first invariant subbundle is the line bundle $E_0\subset T(T^1\AdS_3)$ spanned by the generator $X$ of the geodesic flow:
\begin{equation} E_0(x,v)=\R\cdot (v,x),\quad (x,v)\in T^1\AdS_3. \label{def:E0}
\end{equation}
Let us define for $(x,v)\in T^1\AdS_3$
\begin{align}
& E_s(x,v):=\set{ (w,-w)\in T_{(x,v)}(T^1\AdS_3)}{ q(x,w)=0=q(v,w)}, \label{def:Es}\\ 
& E_u(x,v):=\set{ (w,w)\in T_{(x,v)}(T^1\AdS_3)}{q(x,w)=0=q(v,w)}.\label{def:Eu}
\end{align}
From \eqref{def_alpha} we get that 
\[\ker \alpha=E_s\oplus E_u, \quad \mc{L}_X\alpha=0.\]
The two distributions $E_s$ and $E_u$ inherit orientations from the identifications between the fibers $E_s(x,v)$, $E_u(x,v)$ with $x^\perp\cap v^\perp$ through the projection maps $(w,\pm w)\mapsto w$. They are tangent to foliations $W_s,W_u$ of $T^1\AdS_3$ whose leaves through $(x,v)\in T^1\AdS_3$ are described by
\begin{align}
W_s(x,v)&=\set{ (x',v')\in T^1\AdS_3}{x'+v'=x+v}, \label{eqn:Ws} \\
W_u(x,v)&=\set{ (x',v')\in T^1\AdS_3}{x'-v'=x-v}. \label{eqn:Wu}
\end{align}
A straightforward computation using \eqref{dvarphi_t} shows that 
\begin{equation}\label{dvarphit} \begin{split}
{\rm d}\varphi_t(x,v)(w,-w)&=e^{-t}(w,-w) ,\\
{\rm d}\varphi_t(x,v)(w,w)&=e^t(w,w).\end{split}
\end{equation}
Thus we have a splitting on $T^1\AdS_3$
\begin{equation}\label{decomposition_Anosov}
 T(T^1\AdS_3)=E_0\oplus E_s\oplus E_u
 \end{equation}
which is invariant under both the differential of the geodesic flow and the isometry group of $\AdS_3$. 
We define the dual decomposition  $T^*(T^1\AdS_3)=E_0^*\oplus E_s^*\oplus E^*_u$ such that
\[ E_0^*(E_s\oplus E_u)=0, \quad E_u^*(E_u\oplus E_0)=0,  \quad E_s^*(E_s\oplus E_0)=0.\]
The choice of notation for $E_s^*$ is motivated by the fact that $(\d\varphi_t^{-1})^\top|_{E_s^*}$ is exponentially contracting in $t\to +\infty$. 
The decomposition \eqref{decomposition_Anosov} can be refined, as the rank $2$ subbundles $E_s$ and $E_u$ each split as the sum of invariant line bundles. Indeed, note that the  differential $\d\pi(x,v)$ at $(x,v)\in T^1\AdS_3$ \color{black} of the projection $\pi:T^1\AdS_3\to \AdS_3$ induces isomorphisms from $E_s(x,v)$ and $E_u(x,v)$ to $v^\perp\cap T_x\AdS_3=v^\perp\cap x^\perp\subset\R^4$ (respectively mapping $(w,-w)$ and $(w,w)$ to $w\in v^\perp\cap x^\perp$). As $v^\perp\cap x^\perp\subset \R^4$ has signature $(1,1)$,  its light cone
\[ \mathcal C_{(x,v)}=\set{w\in v^\perp\cap x^\perp}{q(w)=0}\]
is the union of two lines. In order to label these two lines in an ${\rm SO}(2,2)_\circ$-invariant way, we can use the standard orientation of $\R^4$: write 
\[ \mathcal C_{(x,v)}=\mathcal L_{(x,v)}^R\cup \mathcal L_{(x,v)}^L\]
so that any pair $(w^R,w^L)\in \mathcal L_{(x,v)}^R\times \mathcal L_{(x,v)}^L$ such that $q(w^R,w^L)>0$ satisfies \[\det(x,w^L,w^R,v)>0\,. \]
We can then define line subbundles of $T(T^1\AdS_3)$ by
\begin{align*}
E_s^{L}= E_s \cap \d\pi^{-1}(\mathcal L^L), & \quad E_s^{R}= E_s \cap \d\pi^{-1}(\mathcal L^R),\\
E_u^{L}= E_u \cap \d\pi^{-1}(\mathcal L^L), & \quad E_u^{R}= E_u \cap \d\pi^{-1}(\mathcal L^R),
\end{align*}
yielding an ${\rm SO}(2,2)_\circ$-equivariant decomposition
\begin{equation} T\left(T^1\AdS_3\right) = E_0\oplus E_s^L\oplus E_s^{R}\oplus E_u^{L}\oplus E_u^{R}. \label{eq:decomposition into 5 line bundles}
\end{equation}
For $(x,v)\in T^1\AdS_3$ and $t\in \R$, write $(x_t,v_t)=\varphi_t(x,v)$, and note that the compositions
\[\begin{array}{ccccccc} v^\perp\cap x^\perp & \overset{\left( \d\pi(x,v)\vert_{E_s}\right)^{-1}}{\xrightarrow{\hspace{1.8cm}}} &  E_s(x,v) & \overset{\d\varphi_t(x,v)}{\xrightarrow{\hspace{1cm}}} & E_s(x_t,v_t)  &\overset{\d\pi(x_t,v_t)\vert_{E_s}}{\xrightarrow{\hspace{1.5cm}}} & v_t^\perp\cap x_t^\perp \\
w & \xmapsto{\hspace{1.8cm}} &  (w,-w) & \xmapsto{\hspace{1cm}} & e^{-t}(w,-w)  &\xmapsto{\hspace{1.5cm}} & e^{-t}w
\end{array}
\] 
(note that $v_t^\perp\cap x^\perp_t=v^\perp\cap x^\perp$) \color{black} and
\[\begin{array}{ccccccc} v^\perp\cap x^\perp & \overset{\left( \d\pi(x,v)\vert_{E_u}\right)^{-1}}{\xrightarrow{\hspace{1.8cm}}} &  E_u(x,v) & \overset{\d\varphi_t(x,v)}{\xrightarrow{\hspace{1cm}}} & E_u(x_t,v_t)  &\overset{\d\pi(x_t,v_t)\vert_{E_u}}{\xrightarrow{\hspace{1.5cm}}} & v_t^\perp\cap x_t^\perp \\
w & \xmapsto{\hspace{1.8cm}} &  (w,w) & \xmapsto{\hspace{1cm}} & e^{t}(w,w)  &\xmapsto{\hspace{1.5cm}} & e^{t}w
\end{array}
\]
must both send $\mathcal L^R_{(x,v)}$ (resp. $\mathcal L^L_{(x,v)}$) to $\mathcal L^R_{(x_t,v_t)}$ (resp. $\mathcal L^L_{(x_t,v_t)}$), thus showing that the decomposition \eqref{eq:decomposition into 5 line bundles} is also invariant under $\d\varphi_t$. This invariance can be interpreted in the following way: if a vector field $U$ on $T^1\AdS_3$ is a section of one of the four line bundles $E_s^{L},E_s^{R},E_u^{L},E_u^{R}$, then there is a function $\lambda:T^1\AdS_3\to \R$ such that $[X,U]=\lambda U$.

\begin{lemma}\label{Upm^RL}
There exist non-vanishing sections $U_+^{L}$, $U_+^{R}$, $U_-^{L}$, $U_-^{R}$ of $E_s^{L}$, 
$E_s^{R}$, $E_u^{L}$, $E_u^{R}$, respectively, such that 
\[[X,U_+^{L}]=U_+^{L},\quad  [X,U_+^{R}]=U_+^{R}, \quad [X,U_-^{L}]=-U_-^{L},\quad [X,U_-^{R}]=-U_-^{R}.\]
\end{lemma}
\begin{proof} We treat the construction of $U_+^{L}$; the other cases are similar. Consider the Euclidean norm $\Vert\cdot\Vert$ on $\R^4$. For $(x,v)\in T^1\AdS_3$, there is a unique  vector $w^L_{(x,v)}\in \mc L^L_{(x,v)}$ such that
\begin{enumerate}
\item  $\Vert w^L_{(x,v)}\Vert^2 =1/2$,
\item   $q(w^L_{(x,v)},{\bf T}(x)) <0$, where ${\bf T}(x)=(x_2,-x_1,x_4,-x_3)$.
\end{enumerate}
Let $U_+^{L}$ be the section of $E_s^{L}$  such that $\d\pi  U_+^{L}=w^L$. Let us now fix $(x,v)\in T^1\AdS_3$ and $t\in \R$, and write $\varphi_t(x,v)=(x_t,v_t)$. Note that $v_t^\perp\cap x_t^\perp=v^\perp\cap x^\perp$, thus $\mc L^L_{(x_t,v_t)}=\mc L^L_{(x,v)}$ and $w^L_{(x_t,v_t)}=\varepsilon_tw^L_{(x,v)}$ for some $\varepsilon_t\in\{\pm 1\}$. Continuity yields $w^L_{(x_t,v_t)}=w^L_{(x,v)}$. \color{black} Viewing $\pi:T^1\AdS_3\to \AdS_3\subset \R^4$ as an $\R^4$-valued map, \eqref{dvarphit} can be rephrased as
\[ \d\pi\circ \d\varphi_t(x,v) \vert_{E_s(x,v)}= e^{-t}\d\pi(x,v)\vert_{E_s(x,v)}.\]
It follows that
\begin{align*}
\d\pi\left( U_+^{L}(x_t,v_t)\right) = w^L_{(x_t,v_t)}= w^L_{(x,v)}= \d\pi\left( U_+^{L}(x,v)\right) = e^t \d\pi\left( \d\varphi_t(x,v)U_+^{L}(x,v) \right),
\end{align*}
hence $\varphi_t^*U_+^{L}=e^t U_+^{L}$, which differentiates to $[X,U_+^{L}]=U_+^{L}$.
\end{proof}
Such vector fields cannot be invariant under ${\rm SO}(2,2)_\circ$,  since the stabiliser of $(x,v)$, which is isomorphic to ${\rm SO}(1,1)_\circ\simeq\R$, acts by nonzero multiplication on $\mc L^L$ and $\mc L^R$. \color{black} However the wedge of the dual forms produces an invariant two-form.

\begin{lemma}\label{formes_vol_s/u}
There exist two non-vanishing sections $\omega_s\in \Lambda^2 E_s^*$ and  $\omega_u\in \Lambda^2E_u^*$, invariant under ${\rm SO}(2,2)_\circ$, such that $\mathcal L_{X}\omega_u=-2\omega_u$ and $\mathcal L_{X}\omega_s=2\omega_s$.
\end{lemma}

\begin{proof}
Thanks to the orientation of $E_s$, there is a unique two-form $\omega_u \in \Lambda^2 E_u^*$ 
that is positive on oriented bases and that satisfies
\[ (\omega_u(x,v)((w_1,-w_1),(w_2,-w_2)))^2=q(w_1,w_2)^2-q(w_1)q(w_2)\] 
for all $(x,v)\in T^1\AdS_3$ and $w_1,w_2\in x^\perp\cap v^\perp$. The uniqueness guarantees the ${\rm SO}(2,2)_\circ$-invariance, and \eqref{dvarphit} shows that $\varphi_t^*\omega_u=e^{-2t}\omega_u$. We argue analogously for the construction of a section of $\Lambda^2E_s^*$.
\end{proof}

Let us describe these splittings using the $\SL(2,\R)$ picture. Introduce 
\[A:=\left\{\left(\begin{array}{cc}
e^{t} & 0 \\
0 & e^{-t} 
\end{array}\right)\,\Big|\, t\in \R \right\}\] 
and define $A_d:=G_d\cap (A\times A)$, the diagonal in $A\times A$. Let $v_{\rm e}=(0,0,1,0)\in T_{\rm e}\AdS_3$ and consider the map  
\bq\begin{split}
\Psi: (G\times G)/A_d &\to T^1\AdS_3, \\
  \Psi: [h_1,h_2]&\mapsto  (\psi^{-1}(h_1h_2^{-1}), \psi^{-1}(h_1\psi(v_e)h_2^{-1})).\label{eq:Psi}\end{split}
\eq
\begin{lemma}\label{transitivityT1}
The action of $G\times G$ on $T^1\AdS_3$ in \eqref{actionGxG} is transitive and the isotropy group of $v_{\rm e}$ is $A_d$. The map $\Psi$ is a diffeomorphism.
\end{lemma}
\begin{proof}
Since the $G$-action on the base manifold $\AdS_3$ is transitive, it suffices to show that each $v\in T^1_{\rm e}\AdS_3$ can be represented as $v=(h,h)\cdot v_{\rm e}$ where  $h\in G$. 
Let $\mc{H}:=\{ H\in M_2(\R)\,|\, {\rm Tr}(H)=0, \det(H)=-1\}$ and notice that 
$v\in T_{\rm e}\AdS_3\mapsto \psi(v)\in \mc{H}$ is a diffeomorphism. We have $(h,h)\cdot ({\rm e},v)=({\rm e},\psi^{-1}(h\psi(v)h^{-1}))$, thus it suffices to 
prove that the action of $G$ on $\mc{H}$ by conjugation is transitive and that its isotropy group is $A$. But $\mc{H}$ corresponds to diagonalisable matrices with two distinct eigenvalues $\{-1,1\}$, i.e., matrices conjugate (in $\GL(2,\R)$, hence also in $\SL(2,\R)$) to $h_{\rm e}:=\psi(v_{\rm e})=\mathrm{diag}(1,-1)$. The isotropy group corresponds to the matrices commuting with $h_{\rm e}$, which is precisely $A$. 

Finally, $\Psi$ is a diffeomorphism since for $(x,v)\in T^1\AdS_3$, and  $(\psi(x),\psi(v))$ is such that 
$\psi(v)\psi(x)^{-1}=h_1h_{\rm e}h_1^{-1}$ and $\psi(x)^{-1}\psi(v)=h_2h_{\rm e}h_2^{-1}$ for some $h_1,h_2\in \SL(2,\R)$, using that for $\psi(v)\in T_{\psi(x)}G$ we have ${\rm Tr}(\psi(x)^{-1}\psi(v))={\rm Tr}(\psi(v)\psi(x)^{-1})=0$ and $\det(\psi(x)^{-1}\psi(v))=-1$. Then 
$[(h_1,h_2)]\in G\times G/A_d$ is well-defined and depends smoothly on $(x,v)$.
\end{proof}
Let  $U_\pm:=-(X_\perp\pm V)$ where $X_\perp,V$ are the Lie algebra elements in \eqref{basissl2}, then $[X,U_\pm]=\pm U_\pm$, where $X$ was also introduced in \eqref{basissl2};   $\{X,U_+,U_-\}$ is a basis of $\mathfrak{sl}(2,\R)$. We shall denote $\{X^{(1)},U_+^{(1)},U_-^{(1)},X^{(2)},U_+^{(2)},U_-^{(2)}\}$ the corresponding basis of $\mathfrak{sl}(2,\R)\oplus \mathfrak{sl}(2,\R)$ formed by the basis elements in the first and second copies of $\mathfrak{sl}(2,\R)$, respectively. 
The tangent space of $(G\times G)/A_d$ at the base point $[{\rm Id},{\rm Id}]$ can be identified with the quotient space $(\mathfrak{sl}(2,\R)\oplus \mathfrak{sl}(2,\R))/ \R(X^{(1)}+X^{(2)})$. \color{black} We then observe that 
 \bq
 {\rm d}\Psi_{{\rm Id}}(X^{(1)},-X^{(2)})=X({\rm e},v_{\rm e}).\label{eq:XXrel}
 \eq Similarly, 
\[\begin{gathered}
{\rm d}\Psi_{{\rm Id}}(U_+^{(1)})= \frac{1}{2}(\pl_{x_2}+\pl_{x_4}-\pl_{v_2}-\pl_{v_4}), \quad {\rm d}\Psi_{{\rm Id}}(U_-^{(2)})= \frac{1}{2}(\pl_{x_2}- \pl_{x_4}-\pl_{v_2}+\pl_{v_4}),\\
{\rm d}\Psi_{{\rm Id}}(U_-^{(1)})=\frac{1}{2}( -\pl_{x_2}+\pl_{x_4}-\pl_{v_2}+\pl_{v_4}), \quad {\rm d}\Psi_{{\rm Id}}(U_+^{(2)})= \frac{1}{2}(-\pl_{x_2}-\pl_{x_4}-\pl_{v_2}-\pl_{v_4}),
\end{gathered}\]
so that ${\rm d}\Psi_{{\rm Id}}(\R U_+^{(1)}\oplus \R U_-^{(2)})=E_s({\rm e},v_{\rm e})$, ${\rm d}\Psi_{{\rm Id}}(\R U^{(1)}_-\oplus \R U_+^{(2)})=E_u({\rm e},v_{\rm e})$. By   $G\times G$ \color{black}-equivariance, the same holds at any other point $(x,v)=\Psi(h_1,h_2)$ if we consider the vectors $X^{(i)},U_\pm^{(i)}$ as left-invariant vector fields on   $T((G\times G)/A_d)$\color{black}. 
Note that $q((0,1,0,1),(0,-1,0,1))>0$ and, if $(e_i)_i$ is the canonical basis of $\R^4$, we have  
\[\det(e_1,e_3,-e_2+e_4,e_2+e_4)=2\]
\color{black}so we deduce from the definition of $\mc{L}^{L/R}$ and $U_\pm^{R/L}$ in Lemma \ref{Upm^RL} that
\begin{equation}\label{UpmviaG}
\begin{gathered}
U_+^L = {\rm d}\Psi_{{\rm Id}}(U_+^{(1)}), \quad U_+^R = {\rm d}\Psi_{{\rm Id}}(-U_-^{(2)}),\\
U_-^L= {\rm d}\Psi_{{\rm Id}}(-U_+^{(2)}),\quad U_-^R = {\rm d}\Psi_{{\rm Id}}(-U_-^{(1)}).
\end{gathered}\end{equation}
We continue with a useful remark for later (to deal with Poincar\'e series in Section \ref{Sec:Poincare} after descending the equations proved here to the quotient by a quasi-Fuchsian group). The Liouville $1$-form $\alpha$ on $T^1\AdS_3$ becomes $(\alpha^{(1)}-\alpha^{(2)})/2$ in the $(\mathfrak{sl}(2,\R)\oplus \mathfrak{sl}(2,\R))/\R (X^{(1)}+X^{(2)})$ model, with $\alpha^{(i)}$ the form dual to $X^{(i)}$ in each copy of $\mathfrak{sl}(2,\R)$, then 
\begin{equation}\label{dalphaU}
\begin{gathered}
{\rm d}\alpha(U_+^L,U_-^R)=-\alpha([U_+^L,U_-^R])= \frac{1}{2}\alpha^{(1)}([U_+^{(1)},U_-^{(1)}])=1,\\
{\rm d}\alpha(U_+^R,U_-^L)=-\alpha([U_+^R,U_-^L])=\frac{1}{2}\alpha^{(2)}([U_+^{(1)},U_-^{(1)}])=1,
\end{gathered}
\end{equation}
and ${\rm d}\alpha(U_+^L,U_-^L)={\rm d}\alpha(U_+^R,U_-^R)=0$. We notice in particular that 
\begin{equation}\label{dalpha^2}
{\rm d}\alpha\wedge {\rm d}\alpha(U_+^L,U_-^R,U_+^R,U_-^L)=2.
\end{equation}
Similarly, we compute that  
\[ \omega_u(U_+^R,U_+^L)=\frac{1}{4}, \quad \omega_s(U_-^R,U_-^L)=\frac{1}{4}, \]
which implies that on $\ker \alpha$, we have 
\begin{equation}\label{dalpha^2omegas}
{\rm d}\alpha\wedge {\rm d}\alpha=32 \,\omega_u\wedge \omega_s.
\end{equation}
The projection $\pi : T^1\AdS_3\to \AdS_3$ conjugates, via the maps $\Psi,\psi$, into the map 
$[(h_1,h_2)]\in (G\times G)/A_d\mapsto h_1h_2^{-1}\in G$. The vertical space $\ker d\pi$ is then the tangent space to $G_d/A_d$:  over a point $\Psi([(h,h)])$ projecting to ${\rm e}$, this space is spanned by the two vectors ${\rm d}\Psi_{\rm Id}(U_+^{(1)}+U_+^{(2)})$ and ${\rm d}\Psi_{\rm Id}(U_-^{(1)}+U_-^{(2)})$, and using \eqref{UpmviaG} and $G\times G$\color{black}-equivariance, we see that
\begin{equation}\label{kerdpi} 
\ker {\rm d}\pi = \R (U_+^R+U_-^R) \oplus \R (U_+^L-U_-^L) .
\end{equation}
Moreover, over a point $(x,v)\in T^1\AdS_3$, one easily checks that the volume density $S_{x}$ on $T^1_x\AdS_3$ induced by $g_x$ satisfies  
\begin{equation}\label{normalisationSx} 
|S_x(U_+^R+U_-^R,U_+^L-U_-^L)|=2.
\end{equation}
using the normalisation of Lemma \ref{Upm^RL} for $U_\pm^{R/L}$. This implies that, as densities on $T_x^1\AdS_3$,
\begin{equation}\label{normalisationSxwu} 
|S_x|=8|\omega_u|=8|\omega_s|.
\end{equation}

\subsection{Projections onto the conformal boundary}

There are two  ${\rm SO}(2,2)_\circ$-equivariant projections
\[ B_\pm:\map{T^1\AdS_3}{\partial\AdS_3}{(x,v)}{[x\pm v].}\]
The points $B_\pm(x,v)$ correspond to the two endpoints at $\infty$ (the conformal boundary) of the spacelike 
geodesic passing through $(x,v)$.
\begin{lemma} \label{B_pm}
 The maps $B_\pm: T^1\AdS_3\to \partial\AdS_3$ are smooth trivialisable fibrations. 
 \end{lemma}
 
 \begin{proof}
Note that $B_\pm$ is a smooth surjective submersion, by equivariance under the  action of ${\rm Isom}(\AdS_3)$, whose action is transitive on $\pl \AdS_3$\color{black} . Consider the maximal compact subgroup
 \[ K=\set{ \begin{pmatrix} h_1&\\&h_2\end{pmatrix}}{h_1,h_2\in{\rm SO}(2)}\subset {\rm SO}(2,2)_\circ. \]
Fix a point $\nu_0\in\partial\AdS_3$, and note that the orbit map 
 \[ \theta:\map{K}{\partial\AdS_3}{h}{h\cdot\nu_0} \]
is a diffeomorphism. It follows from the $K$-equivariance of $B_\pm$  that the map
 \[ F_\pm:\map{\partial\AdS_3\times B_\pm^{-1}(\{\nu_0\})}{T^1\AdS_3}{(\nu,(x,v))}{\theta^{-1}(\nu)\cdot (x,v)}\]
 is also a diffeomorphism, whose inverse
 \[F_\pm^{-1}:\map{T^1\AdS_3}{\partial\AdS_3\times B_\pm^{-1}(\{\nu_0\})}{(x,v)}{\left(B_\pm(x,v),\theta^{-1}(B_\pm(x,v))^{-1}\cdot (x,v)\right)}\]
 is a trivialisation of $B_\pm:T^1\AdS_3\to\partial\AdS_3$.
 \end{proof}
 
 These projections are related to the Anosov decomposition of the tangent bundle of $T^1\AdS_3$ by the following relations:
 \[ \ker {\rm d}B_+=E_0\oplus E_s,\quad \ker  {\rm d}B_-=E_0\oplus E_u.\]
Notice that if $(x,v)\in T^1\AdS_3$, we have $(x_1\pm v_1,x_2\pm v_2)\not=0$ (since otherwise $x\pm v=0$) and 
we can define the smooth functions
\begin{equation}\label{defPhi} 
\Phi_\pm :\map{T^1\AdS_3}{(0,\infty)}{(x,v)}{\sqrt{(x_1\pm v_1)^2+(x_2\pm v_2)^2}}.
\end{equation}
These functions are not invariant under the action of the isometry group; instead we find
\begin{equation}\label{equivariance_Phi_pm}
 \Phi_\pm\circ h = (N_h\circ\widetilde{\psi}\color{black}\circ B_\pm)\Phi_\pm,\quad h\in {\rm SO}(2,2)_\circ 
 \end{equation}
where $N_h^{-2}\in C^\infty(\mathbb T^2\color{black})$ is the conformal factor for the Lorentzian metric $g_{\mathbb T^2}$ introduced in \eqref{eqn conformal factor} and  $\widetilde \psi:\pl\AdS_3\to \TT^2$ the diffeomorphism from \eqref{eqn tildepsi boundary}\color{black}. Consider the function
\begin{equation}\label{tilB}
\til{B}_\pm:=\widetilde\psi\circ B_\pm:T^1\AdS_3\to \TT^2=\mathbb S^1\times \mathbb S^1.
\end{equation}
We write 
\bq
 x\pm v=\Phi_\pm(x,v) \til{B}_\pm(x,v).\label{eq:defBpm}
\eq
and we see that 
\[ \Phi_\pm\circ\varphi_t=e^{\pm t}\Phi_\pm, \quad B_\pm\circ\varphi_t=B_\pm,\]
thus, recalling that $X$ is the vector field of $\varphi_t$,
\begin{equation}\label{actionXsurPhi}
 X\Phi_\pm=\pm \Phi_\pm, \quad {\rm d}B_\pm(X)=0.
 \end{equation} 
The maps $\Phi_\pm$ and $B_\pm$ are analogous to those used in the Riemannian case in \cite[Section 3D]{dfg}.
For every $(x,v)\in T^1\AdS_3$, one has
\begin{align}
&q(x,\til{B}_\pm(x,v))=\frac{q(x,x\pm v)}{\Phi_\pm(x,v)}=-\frac{1}{\Phi_\pm(x,v)},\\
& q(\til{B}_+(x,v),\til{B}_-(x,v))=\frac{q(x+v,x- v)}{\Phi_+(x,v)\Phi_-(x,v)}=-\frac{2}{\Phi_+(x,v)\Phi_-(x,v)}.\label{eq:qxpsiBpm}
\end{align}
We note that the functions $\Phi_+$ and $B_+$ (resp. $\Phi_-$ and $B_-$) are constant on leaves of the foliation $W_s$ (resp. $W_u$). It follows that
\[ {\rm d}\Phi_+\vert_{E_s}=0,  \quad {\rm d}\Phi_-\vert_{E_u}=0.\color{black} \]
The leaves through $(x,v)\in T^1\AdS_3$ can also be described as 
\begin{align*} 
W_s(x,v)=\{ y\in T^1\AdS_3 \,|\, B_+(y)=B_+(x,v), \Phi_+(y)=\Phi_+(x,v) \}, \\
W_u(x,v)=\{ y\in T^1\AdS_3 \,|\, B_-(y)=B_-(x,v), \Phi_-(y)=\Phi_-(x,v) \}.
 \end{align*}

\subsection{Fiber-integration in the spacelike unit tangent bundle}\label{sec:fiberint}

Let $x\in \AdS_3$. The spacelike unit tangent space $T^1_x\AdS_3=\{ v\in \R^4 \,|\, q(x,v)=0, q(v,v)=1\}$ is a Lorentzian surface in $(\R^4,q)$, its metric $g_x$ induces an area measure $S_x$. The restriction $B_\pm\vert_{T^1_x\AdS_3}$ is injective, and it induces a diffeomorphism
\begin{equation}\label{Bpm_diffeo}
B_\pm(x,\cdot): T^1_x\AdS_3\to U_x
\end{equation}
where $U_x\subset \partial\AdS_3$ is the open set (recall \eqref{defmcO})
\[
 U_x=\{[\nu]\in\partial\AdS_3\,|\, q(x,\nu)<0\}= \{[\nu]\in\partial\AdS_3\,|\, (x,[\nu])\in\mathcal O_+\}.
\]\color{black}
This map has the inverse
\[[\nu] \mapsto \mp \Big(x+\frac{1}{q(x,\nu)}\nu\Big)\color{black}.\]
Consider the maps 
\[ f_x^\pm:\map{T_x^1\AdS_3}{\mathbb T^2}{v}{\til{B}_\pm(x,v)}.\]

\begin{lemma}\label{S_x}
With  $S_{x}$ the volume density on $T^1_x\AdS_3$ induced by the metric $g_x$, 
the following identity holds: $S_x=\Phi_\pm(x,\cdot)^2(f^\pm_x)^\ast {\rm dv}_{\mathbb T^2}$ where ${\rm dv}_{\mathbb{T}^2}={\rm d}\theta {\rm d}\varphi$ is the measure induced by the flat Lorentzian metric $g_{\mathbb{T}^2}$ on $\mathbb{T}^2$.
\end{lemma}

\begin{proof}
Since $f^\pm_x(v)=\frac{1}{\Phi_\pm(x,v)}(x\pm v)$, its differential is
\[ {\rm d}f^\pm_x = \frac{{\rm d}v}{\Phi_\pm(x,v)} + {\rm d}\left(\frac{1}{\Phi_\pm(x,\cdot)}\right) (x\pm v). \]
As $q(x\pm v)=0$, we find
\[ \left(f^\pm_x\right)^*g_{\mathbb T^2}=\frac{1}{\Phi_\pm(x,\cdot)^2} g_x.\]
The same relation holds for the associated area measures.
\end{proof}

\subsection{Pullback of distributions along \texorpdfstring{$B_\pm$}{B}}\label{sec:Bpm}

From Lemma \ref{B_pm} we know that the two smooth maps $\til{B}_\pm: T^1\AdS_3\to \mathbb{T}^2$ defined in \eqref{tilB} are surjective submersions. Therefore, the pullbacks of distributions
\[
\til{B}_\pm^\ast:  \mc{D}'(\mathbb{T}^2)\to  \mc{D}'(T^1\AdS_3)
\]
are well-defined (see \cite[Thm.\ 6.1.2]{hoermanderI}) and injective. 
\begin{lemma}\label{lem:1stband}The following pullbacks are well-defined and isomorphisms:
\bq \label{eq:spaceR0}
\begin{gathered} 
\til{B}_+^\ast: \mc{D}'(\mathbb{T}^2)\to \{u\in \mc{D}'(T^1\AdS_3)\,|\, X u=0,\; U u=0\;\forall\, U \in \Cinft(T^1\AdS_3; E_s)\},\\
\til{B}_-^\ast: \mc{D}'(\mathbb{T}^2)\to  \{u\in \mc{D}'(T^1\AdS_3)\,|\, X u=0,\; U u=0\;\forall\, U \in \Cinft(T^1\AdS_3;E_u)\}
\end{gathered}
\eq
Here we identify smooth sections of (subbundles of) $T(T^1\AdS_3)$ with differential operators of order $1$. 
\end{lemma}
\begin{proof}
This follows from Lemma \ref{B_pm} and the general fact that given two connected manifolds $M,N$, the  pullback $\pi^*:\mc{D}'(M)\to \mc{D}'(M\times N)$ by the projection $\pi:M\times N\to M$ is an isomorphism onto the subspace
\[ \set{u\in \mc{D}'(M\times N)}{Yu=0\,, \forall\, Y\in C^\infty(M\times N, \{0\}\oplus TN) }. \]
A standard reference implying this fact is \cite[Theorem 3.1.4']{hoermanderI}.\color{black}
\end{proof}
\begin{cor}\label{cor:Qlambda}
For every $\lambda\in \C$, the maps $\omega\mapsto\mathcal Q^\pm_\lambda(\omega):= \Phi_\pm^\lambda \til{B}_\pm^\ast \omega$ define isomorphisms  
\bq\label{eq:spaceR0lambda}
\begin{gathered}
\mathcal Q^+_\lambda: \mc{D}'(\T^2)\to \{u\in \mc{D}'(T^1\AdS_3)\,|\,(X- \lambda) u=0,\; U u=0, \, \forall\, U \in \Cinft(T^1\AdS_3;E_s)\}, \\
\mathcal Q^-_\lambda: \mc{D}'(\T^2)\to \{u\in \mc{D}'(T^1\AdS_3)\,|\,(X+ \lambda) u=0,\; U u=0, \, \forall\, U \in \Cinft(T^1\AdS_3;E_u)\}.
\end{gathered}
\eq
Moreover, for $h\in {\rm SO}(2,2)_\circ$ and 
$u=\mathcal Q^\pm_\lambda(\omega)$, one has $h^*u=u$ if and only if $N_h^{\la}h^*\omega=\omega$.\footnote{Note that in the expressions $N_h^{\la}$ and $\Phi_\pm^\lambda$ the number $\lambda$ is an exponent and not just an upper index.}
\end{cor}
\begin{proof}
The proof follows from Lemma \ref{lem:1stband} and \eqref{actionXsurPhi}. The equivariance with respect to ${\rm SO}(2,2)_\circ$ is a consequence of \eqref{equivariance_Phi_pm}.
\end{proof}
\color{black}
\begin{rem}Similarly to the Riemannian case, where the spaces of distributions analogous to those in Corollary \ref{cor:Qlambda} are identified with (distributional) principal series representations (see e.g.\ \cite{dfg,kuester-weich18}), the representation $\rho^\pm_\lambda$ of $G\times G$ on $\mc{D}'(\mathbb{T}^2)$ obtained by carrying over the pullback $G\times G$-action on the space \eqref{eq:spaceR0lambda} to $\mc{D}'(\mathbb{T}^2)$ using the isomorphism $\Phi_\pm^\lambda \til{B}_\pm^\ast$ is a (distributional) \emph{induced representation}. More precisely, we have $\rho^\pm_\lambda=\mathrm{Ind}_{P^\pm}^G(\chi^\pm_\lambda)$, where $P^-=(G\times G)_{\nu_0}$ as defined in \eqref{eq:parabolic1}, $P^+=\{(g,h)\,|\,(h,g)\in P^-\}$, and $\chi_\lambda^\pm:P^\pm\to \C$ is a character of $P^\pm$. Here the groups $P^\pm\subset G\times G$ are non-minimal parabolic subgroups, whereas in the Riemannian case the principal series representations are induced from \emph{minimal} parabolic subgroups.  \end{rem}
\color{black}

\section{Pseudo-Riemannian Laplacian and Poisson transform}\label{sec:Laplacian}

In this section, we describe the spectral theory and resolvent of the pseudo-Riemannian Laplacian $\Box_g$ on $\AdS_3$, and the properties of the Poisson transform, which associates to a distribution on $\TT^2\simeq \pl \AdS_3$ an eigenfunction of 
$\Box_g$.

\subsection{The Laplacian}
In this section, we consider the pseudo-Riemannian Laplacian $\Box_g$ on the $\AdS_3$ space. As explained in Section \ref{sec:SL2Coord}, this operator is nothing more than the Casimir operator on $\SL(2,\R)$. Its spectral theory has been studied by means of representation theory. It is known (see \cite{Rossmann}) that $\Box_g$ 
is essentially self-adjoint, its $L^2$-spectrum consists of
\[ \sigma(\Box_g)=[1,\infty) \cup \{1-n^2\,|\,n\in \N^*\}, \]
where the half-line $[1,\infty)$ forms the continuous spectrum and each $1-n^2$ with $n\in\N^*$ is an $L^2$-eigenvalue with finite multiplicity. 
In the parametrization $z=-\lambda(\lambda+2)$, we find $z\in \sigma(\Box_g)$ if and only if $\lambda\in \{-1+ir+m\,|\,r\in \R,\,m\in \Z\}$, see Figure \ref{fig:L2spectrum}. 
\color{black}\begin{figure}[h!]
\begin{tikzpicture}
\draw[->] (-3.5,0)--(3.5,0) node[right]{$\Re \lambda$};
\draw[->] (0,-2.5)--(0,2.5) node[right]{$\Im \lambda$};
\draw[-,blue,very thick] (-1,-2.5)--(-1,2.5);
\draw (-0.7,-0.3) node {$-1$};
\draw (0.3,-0.3) node {$0$};
\draw[blue,fill=blue] (-3,0) circle (2pt);
\draw[blue,fill=blue] (-2,0) circle (2pt);
\draw[blue,fill=blue] (0,0) circle (2pt);
\draw[blue,fill=blue] (1,0) circle (2pt);
\draw[blue,fill=blue] (2,0) circle (2pt);
\draw[blue,fill=blue] (3,0) circle (2pt);
\end{tikzpicture}
\caption{Location of the complex numbers $\lambda\in \C$ such that $-\lambda(\lambda+2)\in \sigma(\Box_g)$.\label{fig:L2spectrum}}
\end{figure}
\color{black}
The discrete spectrum is related to the so-called discrete series. A spectral resolution (Plancherel formula) is also known.  We first give a result that may be already known but, as we have not found a reference, we provide a proof  in  Appendix \ref{app:resolvent}.
\begin{prop}\label{prop:resolventAdS}
The resolvent of the pseudo-Riemannian Laplacian 
\[R_{\Box_g}(\la)=(\Box_g+\la(\la+2))^{-1}: L^2(\AdS_3)\to  L^2(\AdS_3),\] 
defined for ${\rm Re}(\la)>-1$ and $\la\notin \N$, can be written in the form 
\[ (R_{\Box_g}(\la)f)(x)= \int_{\AdS_3}R_{\Box_g}(\la;x,y)f(y){\rm dv}_g(y) \]
where the integral kernel $R_{\Box_g}(\la;\cdot,\cdot)$ belongs to $L^1_{\rm loc}(\AdS_3\times \AdS_3)$ and is a convolution kernel by an $L^1_{\rm loc}(\AdS_3)$-function given by $R_{\Box_g}(\la;x,y)= F_\la (-q(x,y))$ with
\begin{align}\label{formuleFlambda}
& F _\lambda(\zeta):=\left \{\begin{array}{ll}
\dfrac{C_0^+(\la)}{\sqrt{\zeta^2-1}}(\zeta+\sqrt{\zeta^2-1})^{-(\la+1)},& \zeta>1, \\
\dfrac{C_1}{i\sqrt{1-\zeta^2}}(\zeta+i\sqrt{1-\zeta^2})^{-(\la+1)} +C_2(\la)G_\la(\zeta),& \zeta \in (-1,1),\\
\dfrac{C_0^-(\la)}{\sqrt{\zeta^2-1}}(-\zeta+\sqrt{\zeta^2-1})^{-(\la+1)},& \zeta<-1, 
\end{array}\right.
\end{align}
with 
\begin{equation}\label{def_Glambda} 
G_\la=\Big(\dfrac{1}{i\sqrt{1-\zeta^2}}(\zeta+i\sqrt{1-\zeta^2})^{-(\la+1)}+\dfrac{1}{-i\sqrt{1-\zeta^2}}(\zeta-i\sqrt{1-\zeta^2})^{-(\la+1)}\Big)
\end{equation}
and the constants $C_0^\pm(\la),C_1,C_2(\la)$ are given by 
\begin{equation}\label{C1C2}
\begin{gathered}
C_1=\frac{i}{4\pi}, \quad C_2(\la)=\frac{e^{-i\pi(\la+1)}}{8\pi \sin(\pi(\la+1))}\\
C_0^-(\la)=-\frac{1}{4\pi \sin(\pi(\la+1))},\quad  C_0^+(\la)=\frac{1}{4\pi \tan(\pi(\la+1))}.\end{gathered}
\end{equation}
The resolvent $R_{\Box_g}(\la):C_c^\infty(\AdS_3)\to \mc{D}'(\AdS_3)$ thus admits a meromorphic extension to $\la\in \C$ with simple poles at $\la\in \Z$. The principal part at the poles has infinite rank.
\end{prop}
Notice that when $-q(x,y)>1$, the resolvent $R_{\Box_g}(\la;x,y)$ can be rewritten as 
\begin{equation}\label{Rla_in_terms_ofd} 
R_{\Box_g}(\la;x,y)=C_0^+(\la)\frac{e^{-(\la+1)d(x,y)}}{\sinh(d(x,y))},
\end{equation}
cf. Lemma \ref{Lem:q_and_distance}. \color{black} Finally, we observe that $F_\la(\zeta)$ has an asymptotic expansion of the form 
\[ F_\la (\zeta)=\zeta^{-\la-2}\sum_{k=0}^\infty a_{2k}(\la)\zeta^{-2k}, \quad a_0(\la)=2^{-\la-1}C_0^+(\la)\]
and using that $F_\la (\zeta)$ is a solution of the ODE \eqref{ODEhyper}, one computes the coefficients $a_{2k}$ and obtains (see the proof of \cite[Lemma 2.1]{Guillope-Zworski})
\[ a_{2k}(\la)=2^{-2k}\frac{\Gamma(-\la+2k)}{\Gamma(-\la+k)\Gamma(k+1)}a_0(\la),\]
which gives in the region $\{\zeta>1\}$
\begin{equation}\label{expansionF_la} 
F_\la (\zeta)=C_0^+(\la) \zeta^{-\la-2}\sum_{k=0}^\infty \frac{2^{-\la-1-2k}\Gamma(-\la+2k)}{\Gamma(-\la+k)\Gamma(k+1)}\zeta^{-2k}.
\end{equation}
We also introduce a holomorphic resolvent in the following proposition.
\begin{prop}\label{prop:resolventAdShol}
Consider the holomorphic family of operators $R^h(\la):C_c^\infty(\AdS_3)\to \mc{D}'(\AdS_3)$ defined for $\la\in \C$ by the integral kernel $R^h(\la;x,y)=F_\la^h(-q(x,y))$ where $F_\la^h: \R \to \C$  is the function 
\begin{align}\label{formuleFlambdabis}
& F^h _\lambda(\zeta):=\left \{\begin{array}{ll}
\dfrac{i}{4\pi\sqrt{\zeta^2-1}}(\zeta+\sqrt{\zeta^2-1})^{-(\la+1)},& \zeta>1, \\
\dfrac{1}{4\pi\sqrt{1-\zeta^2}}(\zeta+i\sqrt{1-\zeta^2})^{-(\la+1)} ,& \zeta \in (-1,1),\\
0,& \zeta<-1.
\end{array}\right.
\end{align}
Then for each $f\in C_c^\infty(\AdS_3)$, the following holds true 
\[ (\Box_{g}+\la(\la+2))R^h(\la)f= f\]
in $\mc{D}'(\Omega_f)$ where $\Omega_f$ is the region 
\[\Omega_{f}:=\{x\in \AdS_3\,|\, \forall y\in {\rm supp}(f), \, -q(x,y)> -1\}.\]
\end{prop}
The proof of Proposition \ref{prop:resolventAdShol} is essentially a byproduct of the proof of Proposition \ref{prop:resolventAdS}, see the end of Appendix \ref{app:resolvent}.  Here, the choice of the function $F^h(\la)$ removes the
poles of the resolvent. The drawback is that $R^h(\la)$ satisfies $(\Box_{g}+\la(\la+2))R^h(\la;x,y)=\delta_{x=y}$ (the Schwartz kernel of the identity) only in the subset $\{(x,y)\in \AdS_3\times \AdS_3\,|\,  -q(x,y)> -1\}$ of $\AdS_3\times \AdS_3$. However, only this region is relevant
when quotienting by a quasi-Fuchsian subgroup $\Gamma$, so this resolvent can be made $\Gamma$-periodic to yield an inverse on the quotient.
\color{black}

\subsection{Poisson transform}\label{s:Poisson_transform} 
There are two types of Poisson transforms in the case of $\AdS_3$ which were studied in \cite{Rossmann} and \cite{Andersen} for so-called $K$-finite functions. In this section we briefly review their definitions and prove 
new results concerning the action of the Poisson transforms on distributions on $\TT^2$ (where $\TT^2$ represents the boundary $\pl \AdS_3$). Recall that, with respect to the identification of $\TT^2$ with $\pl \AdS_3$ via the map $\til{\psi}$, the $G$-invariant dense open set $\OO\subset \AdS_3\times\pl \AdS_3$ defined in \eqref{defmcO} becomes
\[
\OO=\{(x,\nu)\in\AdS_3\times \TT^2\,|\, q(x,\nu)\neq 0 \}.
\]
\begin{definition}\label{def:Poissonkernel}
The \emph{Poisson kernel} is the map
\[ P: \map{\OO}{ (0,\infty)}{(x,\nu)}{|q(x,\nu)|^{-1}.}\]
\end{definition}
Note that for $(x_0,v)\in T^1\AdS_3$ with $(x_0,\til{B}_+(x_0,v))\in \OO$, any point $(x,\nu)\in \OO$, and any $h\in {\rm SO}(2,2)_\circ$, we find
\begin{equation}\label{poisson_h} 
P(x_0,\til{B}_+(x_0,v))=\Phi_+(x,v), \quad P(h(x),h(\nu))=N_h(\nu) P(x,\nu). 
\end{equation}
Note that if $\lambda\in \C$ satisfies $\Re \lambda <0$, then the complex power 
$$
(x,\nu)\mapsto P(x,\nu)^\lambda
$$ 
extends by zero from the dense set $\OO\subset \AdS_3\times \TT^2$ to a continuous map on all of $\AdS_3\times \TT^2$, and this extended map has analogous properties as in \eqref{poisson_h}. It will become clear from the context whether we work with powers of the Poisson kernel on $\AdS_3\times \TT^2$ or only on $\OO$. 
\begin{definition}\label{def:Poisson_transform}
Let $\la\in \C$ with ${\rm Re}(\la)<-2$ and $\sigma\in \{0,1\}$.
For $\omega\in C^\infty(\mathbb{T}^2)$, the Poisson transform $\mc{P}^\sigma_\la \omega$ is the continuous
function on $\AdS_3$ defined by  
\bq\label{eq:Poissontransform}
 (\mc{P}^\sigma_\la \omega)(x):=\int_{\mathbb{T}^2} P(x,\nu)^{2+\lambda}({\rm sign}(q(x,\nu)))^{\sigma}\omega(\nu){\rm dv}_{\mathbb{T}^2}(\nu).
 \eq
\end{definition}  
If $I:x\mapsto -x$ denotes the antipodal involution on $\AdS_3$ and on $\TT^2$, 
we notice that 
\begin{equation}\label{P_lambda_invol}
I^*\mc{P}_\la^0\omega=\mc{P}_\la^0\omega \quad I^*\mc{P}_\la^1\omega=-\mc{P}_\la^1\omega.
\end{equation}

Consider, for a subset $S\subset \mathbb{T}^2$,  the set
\bq
\mc{O}_\pm(S):=\{ x\in \AdS_3\, |\, \pm q(x,\nu)<0 , \forall \nu \in S\}=\{ x\in \AdS_3\, |\, \{x\}\times S\subset \mc{O}_\pm\}\label{eq:OpmS}
\eq
and $\mc{O}(S):=\mc{O}_+(S)\cup \mc{O}_-(S)$.
Note that if $S$ is closed in $\mathbb{T}^2$ then $\mc{O}_\pm(S)$ is open in $\AdS_3$. 
\begin{lemma}\label{Klein-Gordon}
Let $\la\in \C$ with ${\rm Re}(\la)<-4$ and $\sigma\in \{0,1\}$. For each $\omega\in C^\infty(\TT^2)$, one has 
\[ (\Box_{g}+\la(\la+2))\mc{P}_\la^\sigma\omega=0\] 
as $C^0$-functions on $\AdS_3$. The same statement holds for all $\la \in \C$ on the open set $\mc{O}({\rm supp}(\omega))\subset \AdS_3$. If  
$h\in {\rm SO}(2,2)_\circ$, the following holds true:
\[  h^\ast(\mc{P}^\sigma_\la \omega)=\mc{P}^\sigma_\la (N_h^\la h^*\omega).\]
\end{lemma}
\begin{proof} Since for each $\nu\in \TT^2$, $P(\cdot,\nu)^{2+\lambda}({\rm sign}(q(\cdot,\nu)))^{\sigma}$ is $C^2$ when ${\rm Re}(\la)<-4$, uniformly in $\nu$, it suffices to check that $(\Box_g+\la(\la+2))P(x,\nu)^\lambda=0$ in $\mc{O}(\{\nu\})$.
We can use the map 
\[ f: \R^+ \times \AdS_3 \to \R^4 ,\quad (s,x)\mapsto sx\]
which pulls back the metric $q(x,x)$ to $f^*q=-ds^2+s^2g$. The associated Laplacian is  (see \eqref{eq:Laplacians})\color{black}
\[ \Box_q=\pl_{x_1}^2+\pl_{x_2}^2-\pl_{x_3}^2-\pl_{x_4}^2=\pl_s^2 +3s^{-1}\pl_s+s^{-2}\Box_{g}.\]
Let $\widetilde{P}(x,\nu)=|q(x,\nu)|^{-1}$ for $x\in \R^4$ in a neighborhood of $\AdS_3$, so that  $(f^*\widetilde{P}(\cdot,\nu))(s,x)=s^{-1}P(x,\nu)$
and one checks that 
\[ ((\la(\la+2)+\Box_q) \widetilde{P}(x,\nu)^\lambda)|_{\AdS_3}=0=(\la(\la+2)+\Box_g)P(x,\nu)^\lambda. 
\]
For the last statement, by \eqref{eqn conformal factor} we have $h^*{\rm dv}_{\TT^2}=N_h^{-2}{\rm dv}_{\TT^2}$ for each $h\in {\rm SO}(2,2)_\circ$, and combining with \eqref{equivariance_Phi_pm} we deduce that 
\[ \begin{split} 
(\mc{P}^\sigma_\la \omega)(hx)=& \int_{\mathbb{T}^2}P(hx,\nu)^{\la+2}({\rm sign}(q(hx,\nu)))^{\sigma} \omega(\nu) {\rm dv}_{\mathbb{T}^2}(\nu)\\
=& \int_{\mathbb{T}^2}P(x,\nu)^{\la+2} ({\rm sign}(q(x,\nu)))^{\sigma} N_h^{\la} \omega(h\nu) {\rm dv}_{\mathbb{T}^2}(\nu)\qedhere.
\end{split}\]
\end{proof}

We extend the Poisson transform $\mc{P}^\sigma_\la \omega$ meromorphically to $\la\in \C$ and to the case where $\omega\in \mc{D}'(\TT^2)$ is a distribution, and prove its injectivity.

We let $H^\alpha(\mathbb{T}^2)$ be the Sobolev space of order $\alpha\in \R$ on the torus and we denote $\bbar{\H}^{2,1}=\R^2_y\times [0,\infty)_s$ and $\pl \H^{2,1}=\{(y,0)\in \bbar{\H}^{2,1}\}$ the boundary identified with $\R^2$ (recall Section \ref{sec:upperhalfplane})\color{black}.
\begin{prop}\label{injectivite_Poisson}
Let $N\in \N$, $\omega\in H^{-2N}(\mathbb{T}^2)$, and $\sigma\in \{0,1\}$. The following hold:\\ 
1) The pairing
\[\mc{P}^\sigma_\la \omega (x):= \cjg \omega, P(x,\cdot)^{\la+2}({\rm sign}(q(x,\cdot)))^{\sigma} \cjd_{\mathbb{T}^2}\]
is well-defined for ${\rm Re}(\la)<-2N-5$ as a $C^2$-function on $\AdS_3$, 
and extends meromorphically to $\la\in \C$ as a distribution on $\AdS_3$, with simple poles at $-2+\N$ and satisfying 
\[ (\Box_g+\la(\la+2))\mc{P}_\la^\sigma \omega=0.\] 
2) Moreover, for each point $\nu^0\in \pl \AdS_3$, there is a neighborhood $U_{0}\subset \bbar{\AdS}_3$ and a diffeomorphism $h_0:U_{0}\to V_{0}\subset \bbar{\H}^{2,1}$ with $h_0(\nu^0)=0$, 
such that the distribution ${h_0}_*\mc{P}^\sigma_\la \omega\in \mc{D'}( V^\circ_0)$ satisfies the following: 
its restriction to the level sets $s=s_0$ is well-defined for all $s_0\in (0,1)$ and for each 
$\chi\in C_c^\infty(V_0\cap \pl \H^{2,1})$ and $\la\notin -2+\N$, we have for all $L\in \N$ as $s\to 0$
\begin{equation}\label{expansion_Poisson}
\cjg {h_0}_*\mc{P}^\sigma_\la \omega(s,\cdot), \chi \cjd_{\R^2}=s^{-\la}\Big(\sum_{\ell=0}^{L-1} s^{2\ell}C^\sigma_{\ell,\la}(\omega,\chi)\Big)+
s^{\la+2}\Big(\sum_{\ell=0}^{L-1} s^{2\ell}D^\sigma_{\ell,\la}(\omega,\chi)\Big)+\mc{O}(s^{L'})
\end{equation}
with $L':=\min(-{\rm Re}(\la)+2L,2+{\rm Re}(\la)+2L)$
and $C^\sigma_{\ell,\la}(\omega,\chi), D^\sigma_{\ell,\la}(\omega,\chi)$ are complex constants, and there is a positive smooth function 
$L_0$ on $\R^2$ depending only on $h_0$ such that
\begin{equation}\label{Csigma0la}
C^\sigma_{0,\la}(\omega,\chi)=\frac{-2^{\la+2}\pi(\cos(\pi\la)+1-2\sigma)}{(1+\la)\sin(\pi \la)}\cjg {h_0}_*\omega,L_0^{-\la-2} \chi\cjd 
\end{equation}
as meromorphic functions in $\lambda$. \\
3) For $\la\notin \Z$ and $\omega\in \mc{D}'(\TT^2)$,  $\mc{P}^\sigma_\la \omega=0$ if and only if $\omega=0$.\color{black}
\end{prop}
\begin{proof}   1) We start with the meromorphic continuation of the Poisson transform. \color{black} 
First, for $\lambda\in \C$ and $(x,\nu)\in \AdS_3\times \TT^2$, we put $P(x,\nu)^{\la+2}_\pm:=\frac{1}{2}P(x,\nu)^{\la+2}(1\pm{\rm sign}(q(x,\nu)))$ if $q(x,\nu)\neq 0$ and $P(x,\nu)^{\la+2}_\pm:=0$ otherwise. Then the identity
\[
P(x,\nu)^{\la+2}=P(x,\nu)^{\la+2}_++P(x,\nu)^{\la+2}_-
\]
holds whenever the left-hand side is well-defined, so the  study of $\mc{P}_\la^\sigma$ is equivalent to the study of the two pairings
\[ \mc{P}_\la^\pm \omega:=\cjg P(x,\cdot)_\pm^{\la+2},\omega\cjd_{\mathbb{T}^2}.\]
These pairings are well-defined as continuous functions on $\AdS_3$ for ${\rm Re}(\la)<-2N-5$, since
\[ (x,\nu) \mapsto P(x,\nu)^{\la+2}({\rm sign}(q(x,\nu)))^{\sigma} =|q(x,\nu)|^{-\la-2}({\rm sign}(q(x,\nu)))^{\sigma} \in C^{2N+3}(\AdS_3\times \TT^2).\]
Indeed, one can then write 
\[ \cjg P(x,\cdot)_\pm^{\la+2},\omega\cjd_{\mathbb{T}^2}=\cjg(1+\Delta_{\TT^2})^N P(x,\cdot)_\pm^{\la+2},(1+\Delta_{\TT^2})^{-N}\omega \cjd_{\mathbb{T}^2}\]
and $(1+\Delta_{\TT^2})^N P(x,\nu)_\pm^{\la+2}$ is a $C^2$-function of $(x,\nu)$ and $(1+\Delta_{\TT^2})^{-N}\omega\in L^2(\TT^2)$. 

Using Lemma \ref{Klein-Gordon} we see that $(\Box_g+\la(\la+2))\mc{P}^\sigma_\la \omega=0$ for ${\rm Re}(\la)<-2N-5$ and if $\mc{P}^\sigma_\la \omega$ admits a meromorphic extension as a distribution, then by unique continuation the same equation will hold for all $\la$ where it is defined.
Since $\mc{P}_\la^\sigma \omega$ is smooth and well-defined for $\la\in \C$ on the set $\OO(\supp \omega)$ of all points $x\in \AdS_3$ such that 
$q(x,\nu)\not=0$ for all $\nu \in \supp \omega$, it remains to study $\mc{P}_\la^\sigma \omega$ near each $x\in \AdS_3$ such that there exists $\nu\in\supp \omega$ with $q(x,\nu)=0$. Without loss of generality we consider those $x$ in a compact set $K_0\subset \AdS_3$. 
Since  the diagonal action of $G\x G$ on $\mc{O}_0$ is transitive (recall \eqref{defmcO} and Lemma \ref{l:transitive}), for each $(x,\nu)\in K_0\times \supp(\omega)$ 
such that $q(x,\nu)=0$, there is $h=h_{x,\nu}\in G$ such that 
 $h(x)\in \iota(\H^{2,1})$ and $h(\nu)\in \til{\psi}(\iota(\R^2))$ where $\iota:\H^{2,1}\to \AdS_3$ is the map described in \eqref{iota_def}, extended to the boundary $\R^2$ of the half-space $\H^{2,1}$. There is also a neighborhood 
 $U_x\subset \AdS_3$ of $x$ and $U_\nu\subset \TT^2$ of $\nu$ such that $h(U_x)\subset \iota(\H^{2,1})$ and $h(U_\nu)\subset \til{\psi}(\iota(\R^2))$.
By compactness, we can thus extract a neighborhood $W$ of $K_0(\omega):=\{ (x,\nu)\in K_0\times \supp(\omega)\,|\, q(x,\nu)=0\}$ in $\AdS_3\times \TT^2$ covered by a finite collection $(U_{x^j}\times U_{\nu^j})_{j=1}^N$ and a finite set of elements $h_j:=h_{x^j,\nu^j}\in G$
such that $h_j(U_{x^j})\subset  \iota(\H^{2,1})$ and $h_j(U_{\nu^j})\subset \iota(\R^2)$. 
Let $\chi_j\in C_c^\infty(U_{x^j}\times U_{\nu^j})$  such that $\chi=\sum_j \chi_j$ is equal to $1$ on $K_0(\omega)$. We will now show that each of the functions
\[
F^\pm_j(x):= \cjg \chi_j(x,\cdot) P(x,\cdot)_\pm^{\la+2},\omega\cjd_{\mathbb{T}^2}
\]
defined for ${\rm Re}(\la)<-2N-3$ admits a meromorphic extension to $\la\in \C$ as a distribution on $U_{x_j}$. 
We use the half-space model $\H^{2,1}$ of $\AdS_3$ described in \eqref{iota_def} to cover the region $\{x\in \AdS_3\,|\, x_2+x_4>0\}$ which contains both $U_{x^j}$ and $U_{\nu^j}$.
The embedding $\iota$ extends to the conformal boundary by
\[ \iota : \R^2\to \TT^2 , \quad \iota(y_1,y_2)=\frac{1}{\sqrt{y_1^2+(1+m(y))^2/4}}\Big(y_1,\frac{1+m(y)}{2},y_2,\frac{1-m(y)}{2}\Big), \]
where $m=-dy_1^2+dy_2^2$ is the Minkowski metric on $\R^{2}$. 
Here we are committing a mild abuse of notation because this is actually the composition of the map  $\iota:\R^2\to \pl \AdS_3$ considered in \eqref{eq:iotaext1} with the inverse of $\til{\psi}:\TT^2\to \pl \AdS_3$, which we omit in the notation.

One checks that $(\iota(y,s),\iota(y'))\in \OO$ iff $s^2+m(y-y')\neq 0$, and then
\begin{equation}\label{PdansH21}
 P(\iota(y,s),\iota(y'))=\sqrt{{y'_1}^2+\frac{(1+m(y'))^2}{4}}\frac{2s}{|s^2+m(y-y')|}.
 \end{equation}
The factor $\sqrt{(y'_1)^2+\frac{(1+m(y'))^2}{4}}>0$ is smooth and will not play an essential role, we will study the Poisson kernel
\[ P_{\H^{2,1}}(y,s,y'):=\frac{2s}{|s^2+m(y-y')|},\qquad (y,s,y')\in \H^{2,1}\times\R^2,\quad s^2+m(y-y')\neq 0.\]
Similarly to the original Poisson kernel, any complex power $P_{\H^{2,1}}(y,s,y')^\mu$ with $\Re \mu<0$ extends by zero to a continuous function on all of $\H^{2,1}\times\R^2$. Moreover, writing for $\lambda \in \C$
\[
P_{\H^{2,1}}(y,s,y')^{\la+2}_\pm:=\frac{1}{2} P_{\H^{2,1}}(y,s,y')^{\la+2}(1\pm \mathrm{sign}(s^2+m(y-y')))
\]
if $s^2+m(y-y')\neq 0$ and $P_{\H^{2,1}}(y,s,y')^{\la+2}_\pm:=0$ otherwise, we have
\bq
P_{\H^{2,1}}(y,s,y')^{\la+2}=P_{\H^{2,1}}(y,s,y')^{\la+2}_+-P_{\H^{2,1}}(y,s,y')^{\la+2}_-\label{eq:sumPs}
\eq
whenever the left-hand side is defined and it suffices to study pairings with  $P_{\H^{2,1}}(y,s,y')^{\la+2}_\pm$. 

Let $V_{j}:=\iota(U_{\nu^j})\subset \R^2$ and $W_j:=\iota(U_{x^j})\subset \H^{2,1}$.
To prove that $F_j$ extends meromorphically as a distribution, it suffices to show that for any $\omega'\in H^{-2N}(V_j)$ with compact support in $V_j$ and any $f\in C_c^\infty(W_j)$ the two pairings
\[   \cjg f\otimes \omega', P_{\H^{2,1}}(\cdot,\cdot)_\pm^{\la+2}\cjd\]
extend meromorphically from ${\rm Re}(\la)<-2N-3$ to $\la\in \C$.
It is convenient to consider the $1$-parameter family of distributions  on $\R^2$ for $s\geq 0$,
\[ G^\pm_{s,\la}(Y):=(s^2+m(Y))_\pm^{-\la},\]
where $(s^2+m(Y))_\pm:=\frac{1}{2} (1\pm \mathrm{sign}(s^2+m(Y)))$. 
By \cite[Chapter III, 2.9]{GelfandShilov}, this is a holomorphic family in $\la\in \C$ of distributions 
on $\R^2$ for $s>0$, while for $s=0$ it is meromorphic with first order poles at $\la=k\in\N$ with residues 
given by derivatives of order $2(k-1)$ of $\delta_0$. We can consider the convolution $G^\pm_{s,\la}*\omega'$ with a compactly supported distribution $\omega'\in \R^2$, this produces a family depending continuously on $s\in \R^+$ of distributions in $\mc{D}'(\R^2)$, down to $s=0$ if $\la\notin \N$, and thus also a distribution in the variable $(y,s)$ (in $s>0$). 
In particular  one has (with the convolution in the $Y\in \R^2$ variable)
\[ \cjg f\otimes \omega', P_{\H^{2,1}}(\cdot,\cdot)_\pm^{\la+2}\cjd_{\H^{2,1}}=  \cjg G^\pm_{\cdot,\la+2}*\omega',2s^{\la+2}f\cjd_{\H^{2,1}}\]
which is meromorphic with possible first order poles at $\la\in -2+\N$, proving the claim that the $F_j$ extend meromorphically to $\C$ with values in $\mc{D}'(\AdS_3)$.

 2) We next prove the asymptotic behaviour of $\mc{P}_\la^\pm\omega$ near a given point $\nu^0\in \pl \AdS_3$. \color{black}
For convenience of notation, we identify $\pl \AdS_3$ with $\TT^2$ by the map \eqref{eqn tildepsi boundary}. First, we claim that if the asymptotic expansion \eqref{expansion_Poisson} holds under the assumption that $\omega=0$ near $-\nu^0$, then it also holds for all $\omega\in \mc{D}'(\TT^2)$. Indeed, if $\omega$ is supported near the point $-\nu^0$ we have 
\[ \cjg {h_0}_*\mc{P}^\sigma_\la \omega(s,\cdot), \chi \cjd_{\R^2}=(-1)^\sigma  \cjg {h_0}_*\mc{P}^\sigma_\la I^*\omega(s,\cdot), \chi \cjd_{\R^2}\]
by using  $I^*\mc{P}^\sigma_\la \omega=(-1)^\sigma \mc{P}^\sigma_\la \omega=\mc{P}^\sigma_\la I^*\omega$. 
Since $I^*\omega$ is supported near $\nu^0$, we can then use the asymptotic expansion \eqref{expansion_Poisson} for distributions supported near $\nu^0$. 
Let us therefore assume now that there is an open neighborhood  $W_-$ of $-\nu^0$ such that 
$\supp(\omega)\cap \bbar{W}_-=\emptyset$.
By using Lemma \ref{l:transitive}, the map $\iota$ of \eqref{iota_def} and the compactness of $\TT^2$, there is a small neighborhood $U_{\nu^0}$ of $\nu^0$ in $\bbar{\AdS}_3$ and a finite covering $\cup_{j=0}^J W_j$ of $\TT^2\setminus \bbar{W}_-$ by open sets such that for each $j$, there is a diffeomorphism $h_j: U_j\subset \AdS_3\to \H^{2,1}$ which extends smoothly to $\pl \AdS_3\cap \bbar{U}_j$, such that $U_{\nu^0}\subset \bbar{U}_j$ and $W_j\subset \bbar{U}_j\cap \pl\AdS_3$, and $h_j(\nu^0)=0\in \R^2=\pl \H^{2,1}$. Without loss of generality, we also assume that $\nu^0\in W_0$ and $\nu^0\notin W_j$ for $j>0$.
Moreover, $h_j$ is a composition of an isometry of $\AdS_3$ with the inverse of the map $\iota$ of \eqref{iota_def}, it is thus an isometry and by \eqref{poisson_h} and  \eqref{PdansH21} 
\[ P(h_j^{-1}(y,s),h_{j}^{-1}(y'))=L_j(y')P_{\H^{2,1}}(y,s,y') \]
for some positive smooth function $L_j$. Let us decompose $\omega=
\sum_{j=0}^J\omega_j$ with ${\rm supp}(\omega_j)\subset W_j$ and $\omega=\omega_0$ near $\nu^0$. We first study $\mc{P}_\la^\pm \omega_j$  in $U_{\nu^0}$ for $j\geq 0$.
For $\chi \in C_c^\infty(\R^2)$ and a diffeomorphism $\tau:[0,1]\times \R^2\to [0,1]\times \R^2$ of the form $\tau(s,y)=(s a(s,y),b(s,y))$, consider
\[ H^\pm_{j,\tau}(s,\la+2):=\cjg \mc{P}^\pm_\la \omega_j(h_j^{-1}\circ \tau(s,\cdot)), \chi\cjd_{\R^2}.\]
We start with the case where $\omega_j\in C_c^\infty(W_j)$ and will see that the analysis extends to $\omega_j$ a distribution. Let us write $b^{-1}_s$ for the inverse of $y\mapsto b(s,y)$ viewed as a diffeomorphism on $\R^2$ and 
$a_s(y):=a(s,b_s^{-1}(y))$.
We write for $s>0$ small
\[  \begin{split}
H^\pm_{j,\tau}(s,\la)=& 2s^{\la}\int_{\R^2 }\int_{\R^2}\frac{\omega_j(h_j^{-1}(y')) L_j^{-\la}(y')\rho_j(y')}{ (s^2a(s,y)^2+m(b(s,y)-y'))_\pm^{\la}}dy'a(s,y)^{\la}\chi(y)dy\\ =& 2s^{\la}\cjg G^\pm_{sa_s(\cdot),\la}* ((\omega_j\circ h_j^{-1})\rho_jL_j^{-\la}),(\chi\circ b_s^{-1})a_s^\la |\det db_s^{-1}| \cjd_{\R^2}
\end{split}\]
where $\rho_j>0$ is a smooth function defined by ${h_j}_*{\rm dv}_{\TT^2}=\rho_j(y)dy$ and we used a change of variable $y\mapsto b_s^{-1}(y)$ in the second line. The Fourier transforms in $Y$ of $G^\pm_{s,\la}(Y)$ are given by \cite[Chapter III, 2.9]{GelfandShilov}
\[\begin{split} 
\hat{G}^+_{s,\la}(\xi)=& -\Gamma(1-\la)(2s)^{1-\la}\Big( 
e^{i\pi \la}\frac{K_{1-\la}(sm(\xi-i0)^{1/2})}{m(\xi-i0)^{\frac{1}{2}(1-\la)}}+e^{-i\pi \la}\frac{K_{1-\la}(sm(\xi+i0)^{1/2})}{m(\xi+i0)^{\frac{1}{2}(1-\la)}} \Big)\\
\hat{G}^-_{s,\la}(\xi)=& \Gamma(1-\la)(2s)^{1-\la}\Big( 
\frac{K_{1-\la}(sm(\xi-i0)^{1/2})}{m(\xi-i0)^{\frac{1}{2}(1-\la)}}+\frac{K_{1-\la}(sm(\xi+i0)^{1/2})}{m(\xi+i0)^{\frac{1}{2}(1-\la)}} \Big)
\end{split}\]
where $K_\nu(z)$ is the modified Bessel function. We can use the asymptotic expansion of $K_\nu(z)$ at $z=0$ \cite{abramovitz} to deduce that for $f\in C_c^\infty(\R^2)$ the following asymptotic holds as $s\to 0^+$:
\[ \begin{split} 
& (2s)^{\la}\cjg \hat{G}^+_{s,\la},\hat{f}\cjd  \\
&=   \frac{\pi\Gamma(1-\la)}{\sin(\la\pi)}
\sum_{\ell=0}^N \frac{s^{\la+2\ell}2^{1-\la-2\ell}}{\ell!\Gamma(\la+\ell)} \cjg e^{i\pi\la}m(\cdot-i0)^{\la-1+\ell}+e^{-i\pi\la}m(\cdot+i0)^{\la-1+\ell},\hat{f}\cjd\\
&\quad  +\frac{2\pi \Gamma(1-\la)}{\tan(\la\pi)}\sum_{\ell=0}^N \frac{s^{2-\la+2\ell}2^{\la-1-2\ell}}{\ell!\Gamma(2-\la+\ell)} \cjg m(\cdot)^{\ell},\hat{f}\cjd+ \mc{O}(s^{{\rm Re}(\la)+2N+2})+\mc{O}(s^{-{\rm Re}(\la)+2N+4})
\end{split}\]
and 
\[ \begin{split} 
& (2s)^{\la}\cjg \hat{G}^-_{s,\la},\hat{f}\cjd  \\
&=   -\frac{\pi\Gamma(1-\la)}{\sin(\la\pi)}
\sum_{\ell=0}^N \frac{s^{\la+2\ell}2^{1-\la-2\ell}}{\ell!\Gamma(\la+\ell)} \cjg m(\cdot-i0)^{\la-1+\ell}+m(\cdot+i0)^{\la-1+\ell},\hat{f}\cjd\\
&\quad  +\frac{2\pi \Gamma(1-\la)}{\sin(\la\pi)}\sum_{\ell=0}^N \frac{s^{2-\la+2\ell}2^{\la-1-2\ell}}{\ell!\Gamma(2-\la+\ell)} \cjg m(\cdot)^{\ell},\hat{f}\cjd + \mc{O}(s^{{\rm Re}(\la)+2N+2})+\mc{O}(s^{-{\rm Re}(\la)+2N+4})
\end{split}\]
Let us write $\omega'_j:=(\omega_j\circ h_j^{-1})\rho_jL_j^{-\la}$ and $\chi'_s=(\chi\circ b_s^{-1})a_s^\la |\det db_s^{-1}|$, then we deduce that for $s>0$ small (with $\Box_{m}=\pl_{x_1}^2-\pl_{x_2}^2$)
\[\begin{split}
&H^+_{j,\tau}(s,\la)=\sum_{\ell=0}^N s^{2-\la+2\ell}c^+_{\ell,\la}\cjg \Box_{m}^\ell \omega_j',a_s^{2-\la+2\ell}\chi'_s \cjd_{\R^2}\\
&\quad  +  \sum_{\ell=0}^N s^{\la+2\ell}d^+_{\ell,\la}\cjg e^{-i\pi(\frac{1}{2}-\la)} \mc{F}(m(\cdot-i0)^{\la-1+\ell}-e^{i\pi(\frac{1}{2}-\la)}\mc{F}(m(\cdot+i0)^{\la-1+\ell} )*\omega_j',a_s^{2-\la+2\ell}\chi'_s \cjd_{\R^2}\\
& \quad +\mc{O}(s^{{\rm Re}(\la)+2N+2})+\mc{O}(s^{-{\rm Re}(\la)+2N+4})
\end{split}
\]
for some explicit holomorphic coefficients $c^+_{\ell,\la},d^+_{\ell,\la}$ and $c^+_{0,\la}=\frac{2^{\la}\pi}{(1-\la)\tan(\pi \la)}$. We can then  Taylor expand the smooth functions  $a_s$ and $\chi'_s$ at $s=0$ to obtain the asymptotic expansion of $H^+_{j,\tau}(s,\la)$ as $s\to 0$: only powers $s^{2-\la+\ell}$ and $s^{\la+\ell}$ with $\ell\in \N$ appear.
Notice that this expression extends continuously to the case where $\omega_j$ is a compactly supported distribution. 
Now for $\chi\in C_c^\infty(W_0)$ such that $\supp(\chi)\cap \supp(\omega_j)=\emptyset$ for $j\not=0$, 
\[ \cjg \mc{P}^+_\la \omega_j\circ h^{-1}_0(s,\cdot), \chi \cjd_{\R^2}= H^+_{j,\tau}(s,\la+2)\]
where $\tau=h_j\circ h_0^{-1}$. We observe now that the pairing $\cjg \Box_{m}^\ell \omega_j',a_s^{-\la+2\ell}\chi'_s \cjd_{\R^2}=0$ for small enough $s>0$, since $(\chi\circ b_s^{-1})(\omega_j \circ b_0^{-1})=0$ for small enough $s>0$. This means that $H^+_{j,\tau}(s,\la+2)$ has an expansion with only the powers $s^{\la+2+\ell}$ for $\ell\in \N$. For $j=0$ we apply the same reasoning, but now $\supp(\chi)\cap \supp(\omega_j)\not=\emptyset$ and $\tau(s,y)=(s,y)$: we obtain an asymptotic expansion as claimed in \eqref{expansion_Poisson} and the coefficient of $s^{-\la}$ is given by (recall ${h_0}_*{\rm dv}_{\TT^2}=\rho_0(y)dy$)
\[ c^+_{0,\la+2} \cjg (\omega_0\circ h_0^{-1})\rho_0L_0^{-\la-2},\chi\cjd s^{-\la}=c^+_{0,\la+2} \cjg  {h_0}_*\omega_0, L_0^{-\la-2}\chi\cjd s^{-\la}.\]
We have thus proved the asymptotic expansion. 
The same argument applies to $\cjg \mc{P}^-_\la \omega_j\circ h^{-1}_0(s,\cdot), \chi \cjd_{\R^2}$ and the coefficient of $s^{-\la}$ is as before but with $c^+_{0,\la+2}$ replaced by 
$c^-_{0,\la+2}=-\frac{2^{\la+2}\pi}{(\la+1)\sin(\pi \la)}$. Finally, with \eqref{eq:sumPs} we obtain the desired result for $\mc{P}_\la^\sigma\omega$.

3) If $\mc{P}_\la^\sigma\omega=0$ and $\la\notin \Z$, for each $\nu^0$ and small neighborhood $U_0\subset \bbar{\AdS}_3$ of $\nu^0$,
 one has by \eqref{expansion_Poisson} that $C^\sigma_{0,\la}(\omega,\chi)=0$ for each small $\chi\in  C_c^\infty(V_0\cap \pl \H^{2,1})$. By \eqref{Csigma0la}, this implies that $\omega=0$ near $\nu^0$. Thus $\omega=0$.
\color{black}
\end{proof}

Let us describe the Poisson transform $\mc{P}^\sigma_\la \omega$ in $\mc{O}_+(\supp\omega)$ in terms of the pushforward of a distribution on $T^1\AdS_3$. Notice that $\mc{P}^0_\la \omega=(-1)^\sigma\mc{P}^\sigma_\la \omega$ in  $\mc{O}_+(\supp\omega)$. 
We begin with the relation to the map $\mathcal Q^+_\lambda$ from Corollary \ref{cor:Qlambda}. For this purpose, let us introduce the following terminology and notation: For an open set $U\subset \AdS_3$ we denote by $\pi_\ast:\mc{D}'_\mathrm{f.c.}(T^1U)\to \mc{D}'(U)$ the pushforward of distributions with fiberwise compact support along the restriction of the bundle projection $\pi:T^1\AdS_3\to \AdS_3$ to $T^1U \subset T^1\AdS_3$. Here we say that  $u\in \mc{D}'(T^1U)$ has fiberwise compact support if ${\rm supp}(u)\cap T_x^1\AdS_3$ is compact for each $x\in U$. This ensures that $\pi$ is proper on ${\rm supp}(u)$, so that the pushforward $\pi_\ast u$ is well-defined. Recall that in Section \ref{sec:fiberint} we have chosen a family of measures $S_x$ on the fibers of $T^1\AdS_3$, which gives us an embedding of the space $C^{\infty}_\mathrm{f.c.}(T^1U)$ of smooth functions with fiberwise compact support into $\mc{D}'_\mathrm{f.c.}(T^1U)$  as a dense subspace. On this subspace, $\pi_\ast$ acts by integration over the fibers with respect to the measures $S_x$. 
\begin{lemma}\label{lem:pushfoward}
Let $\omega\in \mc{D}'(\mathbb{T}^2)$. Then the distribution $\mathcal Q^\pm_\lambda \omega|_{T^1\mc{O}_+({\rm supp}(\omega))}\in \D'(T^1\mc{O}_+({\rm supp}(\omega)))$ has 
fiberwise compact support and one has\color{black}
\bq
(\mc{P}^\sigma_\la \omega)|_{\mc{O}_+({\rm supp}(\omega))} =(-1)^\sigma \pi_\ast (\mathcal Q^\pm_\lambda \omega|_{T^1\mc{O}_+({\rm supp}(\omega))}).\label{eq:poissonpushfwd}
\eq
In particular, the distribution $\pi_\ast (\mathcal Q^+_\lambda \omega|_{T^1\mc{O}_+({\rm supp}(\omega))})$ on $\mc{O}_+({\rm supp}(\omega))$ is a smooth function.
\end{lemma}
\begin{proof} We only deal with $\mc{Q}_\la^+(\omega)$, the other case works exactly the same way.
Since $\mathcal Q^+_\lambda \omega=\Phi_+^\lambda \til{B}_+^\ast \omega$ differs from $\til{B}_+^\ast\omega$ only by multiplication with  $\Phi_+^\lambda$ (see Section \ref{sec:Bpm}), it suffices for the first claim to show that $\supp(\til{B}_+^\ast\omega)\cap T^1_x\AdS_3$ is compact for all $x\in \mc{O}_+({\rm supp}(\omega))$. For such $x$, we know from \eqref{Bpm_diffeo} that $\til{B}_+(x,\cdot):T^1_x\AdS_3\to \mathbb{T}^2$ is a diffeomorphism onto an open set containing $\supp (\omega)$ (recall \eqref{eq:OpmS}), and therefore $\supp(\til{B}_+^\ast\omega)\cap T^1_x\AdS_3 \subset B_+(x,\cdot)^{-1}(\supp (\omega))$ is compact. This shows that the pushforward on the right-hand side of \eqref{eq:poissonpushfwd} is well-defined.  \color{black} Using  the density of $\Cinft(\mathbb{T}^2)$ in $\mc{D}'(\mathbb{T}^2)$, it suffices to prove \eqref{eq:poissonpushfwd} in the case where $\omega\in \Cinft(\mathbb{T}^2)$. Then $\mathcal Q^+_\lambda \omega|_{T^1\mc{O}_+({\rm supp}(\omega))}$ is a smooth function on $T^1\mc{O}_+({\rm supp}(\omega))$ with fiberwise compact support, and $\pi_\ast$ acts on this function by integration over the fibers of $T^1\mc{O}_+({\rm supp}(\omega))$ with respect to the measures  $S_x$ from Section \ref{sec:fiberint}. Recalling the Definitions \ref{def:Poissonkernel} and \eqref{eq:Poissontransform} of the Poisson kernel and the Poisson transform, and using the transformation formula \eqref{S_x} for $S_x$, we then compute for $x\in 
\mc{O}_+({\rm supp}(\omega))$:
\begin{align*}
\pi_\ast (\mathcal Q^+_\lambda \omega|_{T^1\mc{O}_+({\rm supp}(\omega))})&=\int_{T^1_x\AdS_3}(\mathcal Q^+_\lambda \omega)(x,v)\d S_x(v)\\
&=\int_{T^1_x\AdS_3}\Phi_+^\lambda(x,v) (\til{B}_+(x,\cdot)^\ast\omega)(x,v)\d S_x(v)\\
&=\int_{\mathbb T^2}\underbrace{\Phi_+^{\lambda+2}(x,v^+_{x,\nu})}_{=P(x,\nu)^{\lambda+2}} \omega(\nu)\d S(\nu)=\mc{P}^0_\la \omega(x).
\end{align*}
where we wrote $v^+_{x,\nu}$ for the vector such that $\til{B}_+(x,v^+_{x,\nu})=\nu$. The fact that 
$\mc{P}^0_\la \omega(x)$ is smooth is a direct consequence of the fact that $P(x,\nu)>0$ is smooth in $x\in  \mc{O}_+({\rm supp}(\omega))$ for each $\nu\in {\rm supp}(\omega)$. 
This completes the proof.
\end{proof}

\section{Quasifuchsian spacetimes}\label{sec:quasifuchsian}

We now turn to spacetimes that are locally modeled on $\AdS_3$, more precisely obtained as quotients of open subsets $U\subset\AdS_3$ by discrete subgroups $\Gamma\subset G\times G$, where $G=\SL(2,\R)$. Here and below, we identify $G\times G$ with ${\rm Isom}_\circ(\AdS_3)$.
Such quotients $$M=\Gamma\backslash U$$ inherit a Lorentzian metric, denoted by $g$, and we can still consider their spacelike unit tangent bundle $$T^1M=\Gamma\backslash T^1U.$$
 We will see that for the examples that we are interested in, called \emph{quasi-Fuchsian}, the geodesic flow on $T^1M$ is not complete 
 but there exists a larger manifold $\mc M$ containing $T^1M$ on which the geodesic vector field $X$ has a complete flow.

\subsection{Acausal circles in \texorpdfstring{$\partial \AdS_3$}{the boundary of AdS} and their invisible domains}

Consider the isometry $\widetilde\psi:\AdS_3\to \mathbb S^1\times \mathbb D$ from \eqref{tildepsi} and denote again by $\widetilde \psi$ its extension to the boundary $\partial \AdS_3\to \TT^2= \mathbb{S}^1\times \mathbb{S}^1$ from \eqref{eqn tildepsi boundary}. As before, we will occasionally use this map to identify $\partial \AdS_3= \TT^2$ in proofs to simplify the notation. 

A map  $F:\mathbb{S}^1\to \mathbb{S}^1$ will be called \emph{distance-decreasing} if  $\vert F(z)-F(z')\vert < \vert z-z'\vert$ for any distinct $z,z'\in \mathbb{S}^1$, where $|\cdot|$ is the usual absolute value on $\mathbb{S}^1\subset \C$. 

\begin{definition}\label{def - acausal circle}~
A subset $\Lambda\subset \partial \AdS_3$ is called:
\begin{itemize} 
\item An \emph{acausal circle} if $\Lambda=\{\widetilde\psi^{-1}(F(z),z)\,|\, z\in\mathbb{S}^1\}$ for some distance-decreasing $F:\mathbb S^1\to\mathbb{S}^1$.
\item  A \emph{round circle} if there is an isometry $h\in {\rm Isom}(\AdS_3)$ and some $w_0\in \mathbb S^1$ such that 
\[\Lambda=h\cdot\{ \widetilde\psi^{-1}(w_0,z)\,|\, z\in\mathbb{S}^1\}.\]
\end{itemize}
\end{definition}

Alternatively, round circles are boundaries of totally geodesic copies of $\mathbb H^2$ in $\AdS_3$. Note that a round circle is an acausal one (by taking as $F$ the constant map with value $w_0$), and \color{black} two distinct points $[\nu_1],[\nu_2]$ on an acausal circle $\Lambda$ must satisfy $q(\nu_1,\nu_2)< 0$. Indeed, considering $(w,z),(w',z')\in\mathbb S^1\times \mathbb S^1$, we can use the fact that $|w-w'|^2=2(1-w\cdot w')$ and $|z-z'|^2=2(1-z\cdot z')$ (where $w\cdot w'$ and $z\cdot z'$ represent the usual dot product in $\R^2$) and \eqref{eqn tildepsi boundary} to obtain
\begin{align*}
q\left( \widetilde\psi^{-1}(w,z),\widetilde\psi^{-1}(w',z')\right)& =  -w\cdot w' + z\cdot z'\\
&= \frac{1}{2}\Big( |w-w'|^2-|z-z'|^2 \Big)\,.
\end{align*}

\color{black}

\begin{definition} \label{def - convex hull subset boundary}
Let $\Lambda=\{\widetilde\psi^{-1}(F(z),z)\,|\,z\in\mathbb{S}^1\}$ be an acausal circle in $\partial \AdS_3$, and consider $\widetilde \Lambda=\set{ (F(x,y),x,y)}{ x^2+y^2=1}\subset \R^4$. The \emph{convex hull} $C(\Lambda)\subset \AdS_3$ is  the smallest geodesically convex subset of $\AdS_3$ (i.e.\ given $x,y\in C(\Lambda)$ there is exactly one geodesic segment joining $x$ to $y$ lying in $C(\Lambda)$) containing  all spacelike geodesics with endpoints in $\Lambda$. 

Alternatively, we can describe $C(\Lambda)$ as
\[ C(\Lambda) = \left\lbrace \frac{x}{\sqrt{|q(x)|}} \,\Big|\, x\in {\rm Conv}(\til{\Lambda} ), q(x)<0 \right\rbrace, \]
where ${\rm Conv}(\til{\Lambda})\subset \R^4$ is the usual convex hull.  \color{black} 
\end{definition}

If $\Lambda$ is a round circle, then $C(\Lambda)$ is a totally geodesic copy of $\H^2$, otherwise $C(\Lambda)$ has non-empty interior in $\AdS_3$ \cite[Lemma 3.13]{barbot-merigot}.

\begin{definition} \label{def - invisible domain}
The \emph{invisible domain} of a subset $\Lambda\subset \partial \AdS_3$ is the set $\mathcal O_+(\Lambda)\subset \AdS_3$ defined by \eqref{eq:OpmS}, i.e., $\mathcal O_+(\Lambda)=\{ x\in \AdS_3 \,|\,  {q}(x,\widetilde\psi(\nu))<0, \forall\, \nu \in \Lambda \}$.
\end{definition}

Note that the set $\mathcal O_+(\Lambda)$ is convex, and in view of \eqref{eq:qxpsiBpm} it contains all spacelike geodesics with endpoints in $\Lambda$ so it contains the convex hull $C(\Lambda)$. It also means that two points $x,y\in\mathcal O_+(\Lambda)$ are always joined by a geodesic, hence ${q}(x,y)<1$.

\begin{lemma} \label{lem - q<-1 locally uniformly on the discontinuity  domain}
Let $\Lambda\subset \partial\AdS_3$ be an acausal circle and $K\subset \mathcal O_+(\Lambda)$ be a compact subset. There is $\delta<1$ such that ${q}(x,y)\leq \delta$ for all  $x\in K$ and $y\in \mathcal O_+(\Lambda)$.
\end{lemma}

\begin{proof}
Suppose for contradiction that this is not the case. Then there are $x\in K$ and $y\in \partial\mathcal O_+(\Lambda)$ with $q(x,y)=1$. Since $x\in \mathcal O_+(\Lambda)$, this forces $y\notin \Lambda=\partial\mathcal O_+(\Lambda)\cap \partial \AdS_3$ \cite[Corollary 3.7]{barbot-merigot}, hence $y\in \AdS_3$. But $\overline{\mathcal O_+(\Lambda)}\cap \AdS_3$ is also convex, so $x$ and $y$ are related by a geodesic. The equality ${q}(x,y)=1$ forces this geodesic to be timelike, and $y=-x$. It follows that ${q}(y,\nu)>0$ for all $\nu\in \Lambda$, but as $y\in\partial \mathcal O_+(\Lambda)$ we should also have  ${q}(y,\nu)\leq 0$ for all $\nu\in \Lambda$, which is a contradiction.
\end{proof}

\subsection{Quasi-Fuchsian subgroups}\label{QFsubgroups}

\begin{definition} \label{def - AdS quasi fuchsian group}
A subgroup $\Gamma\subset {\rm Isom}(\AdS_3)_\circ$ is called \emph{quasi-Fuchsian} if it preserves an acausal  circle $\Lambda_\Gamma\subset \partial \AdS_3$ (called the \emph{limit set} of $\Gamma$) and the action on $C(\Lambda_\Gamma)$ is free, properly discontinuous and co-compact.

It is called \emph{Fuchsian} if $\Lambda_\Gamma$ is a round circle.
\end{definition}
If  $\Gamma\subset {\rm Isom}(\AdS_3)_\circ$ is quasi-Fuchsian, then the acausal circle $\Lambda_\Gamma\subset \partial \AdS_3$ preserved by $\Gamma$ is unique \cite[Lemma 2.11]{dgk18}.

\begin{prop}[Mess \cite{mess2007}] \label{thm - Mess}
Consider the isomorphism ${\rm Isom}(\AdS_3)_\circ\simeq G\times G$. A subgroup $\Gamma \subset G\times G$ is quasi-Fuchsian if and only if there is an orientable closed surface $S$ of genus $\geq 2$ equipped with two hyperbolic metrics  with associated holonomy representations \color{black} $\rho_1,\rho_2:\pi_1(S)\to G$ (with respect to the same arbitrary base point) \color{black} such that \[\Gamma=\set{(\rho_1(\gamma),\rho_2(\gamma))}{\gamma\in \pi_1(S) }.\]
Furthermore, $\Gamma$ is Fuchsian if and only if $\rho_1,\rho_2$ are conjugate.
\end{prop}

\begin{rem}\label{rem:defrem} The definition of quasi-Fuchsian in terms of the action on the convex hull of an acausal circle is quite different from Mess' approach, and more in the spirit of \cite{barbot-merigot} and \cite{dgk18}. The equivalence can be easily shown by using the fact that if the limit set is not a round circle, then the interior of its convex hull is  globally hyperbolic. 
\end{rem}

Quasi-Fuchsian subgroups  $\Gamma \subset G\times G$ are part of a vastly more general theory of discrete subgroups of higher rank semisimple Lie groups called \emph{Anosov subgroups}, which are studied in the broader field of \emph{Anosov representations} (see e.g.\ \cite{labourie, barbot-merigot,guichardwienhard2012,dgk18,dgk24}). The hierarchy is as follows: Quasi-Fuchsian subgroups belong to the so-called \emph{$\mathbb{H}^{p,q}$-convex co-compact subgroups} (for $p=2$, $q=1$), see \cite{dgk18}. The $\mathbb{H}^{p,q}$-convex co-compact subgroups belong to the \emph{projective Anosov subgroups}, see \cite{dgk24}, which are particular Anosov subgroups of $\mathrm{SL}(n,\R)$ for $n\in \N$. Anosov subgroups are always Gromov-hyperbolic and possess so-called boundary maps from the Gromov boundary $\partial_\infty\Gamma$ \color{black} into appropriate flag manifolds. The images of these boundary maps are the limit sets in this more general context. In our setting of a quasi-Fuchsian subgroup  $\Gamma \subset G\times G$, the abstract limit set  given by the image of the Anosov boundary map into $\mathbb{RP}^3$ coincides with the limit set $\Lambda_\Gamma$ introduced above (see e.g.\ the much more general \cite[Thm.~1.15]{dgk24}).  In particular, the recent dynamical results on projective Anosov subgroups from \cite{Delarue_Monclair_Sanders} can be applied to quasi-Fuchsian groups. Rather than directly applying these still somewhat abstract results to our setting, we will reprove the relevant ones here for a more self-contained and explicit presentation.

\subsection{Discontinuity domains in \texorpdfstring{$\AdS_3$}{AdS}}

\begin{prop}[Mess \cite{mess2007}] \label{prop - properly discontinuous action on invisible domain}\label{prop:Mess}
Let $\Gamma\subset {\rm Isom}(\AdS_3)$ be a quasi-Fuchsian group. The action of $\Gamma$ on the invisible domain $\mathcal O_+(\Lambda_\Gamma)$ is free and properly discontinuous.
\end{prop}
We write
\[
M:=\gam \mc{O}_+(\Lambda_\Gamma)
\]
for the quasi-Fuchsian manifold associated to $\Gamma$. 

The proper discontinuity of the $\Gamma$-action on $\mathcal O_+(\Lambda_\Gamma)$  roughly means that orbits of points in $\mc{O}_+(\Lambda_\Gamma)$ cannot accumulate in $\mc{O}_+(\Lambda_\Gamma)$. This is a consequence of the following stronger result.

\begin{prop}[{\cite[Lem.~2.11]{dgk18}}]
Let $\Gamma\subset {\rm Isom}(\AdS_3)$ be a quasi-Fuchsian group, and let $K\subset \mc{O}_+(\Lambda_\Gamma)$ be a non-empty compact subset. The set of accumulation points in $\AdS_3\cup\partial\AdS_3$ of $\Gamma\cdot K$ is equal to $\Lambda_\Gamma$.
\end{prop}

We will need a stronger result, controlling how elements in the orbit of some compact set $K_1\subset\mathcal O_+(\Lambda_\Gamma)$ become spacelike separated from a fixed compact set $K_2\subset \mathcal O_+(\Lambda_\Gamma)$.\color{black}

\begin{lemma} \label{lem - q<-1 up to moving by an element of the group}
Let $\Gamma\subset{\rm Isom}(\AdS_3)$ be a quasi-Fuchsian subgroup and $c>0$. For each pair of compact subsets $K_1,K_2\subset \mathcal O_+(\Lambda_\Gamma)$, there exists a finite subset $\Gamma_{K_1,K_2,c}\subset \Gamma$ such that ${q}(\gamma x,y)<-c$ for all $x\in K_1$, $y\in K_2$ and $\gamma\in \Gamma\setminus \Gamma_{K_1,K_2,c}$.
\end{lemma}
\begin{proof}
Arguing by contradiction, consider an infinite sequence $\{\gamma_k\}\subset \Gamma$ and $x_k\in K_1$, $y_k\in K_2$  such that $q(\gamma_kx_k,y_k)\geq -c$ for all $k\in\N$. We use the convergence property of the action of $\Gamma$ on $\mathbb{RP}^3$: up to a subsequence, there are $\lambda_+,\lambda_-\in \Lambda_\Gamma$ such that $\gamma_k\ell\to \lambda_+$ uniformly on compact subsets of $\set{\ell\in \mathbb{RP}^3}{{q}(\ell,\lambda_-)\neq 0}\supset\mathcal O_+(\Lambda_\Gamma)$ \cite[Prop. 6.6]{barbot-merigot}\color{black}.  In particular, $\gamma_k[x_k]\to \lambda_+$ in $\mathbb{RP}^3$. Up to another subsequence, it means that there exists $v\in\R^4\setminus\{0\}$ with $[v]=\lambda_+$, and a sequence $t_k>0$ such that $t_k\gamma_kx_k\to v$. Now $-t_k^2=t_k^2q(\gamma_kx_k)=q(t_k\gamma_kx_k) \to q(v)=0$, so $t_k\to 0$, and  assuming that $y_k$ converges to some $y\in K_2$ we find
\[ q(v,y)=\lim_{k\to+\infty}q(t_k\gamma_kx_k,y_k)=\lim_{k\to+\infty} t_k\underbrace{q(\gamma_kx_k,y_k)}_{\geq -c} \geq -c\lim_{k\to+\infty} t_k=0\,.\]
But $q(v,y)\geq 0$ is a contradiction with $[v]\in \Lambda_\Gamma$ and $y\in \mathcal O_+(\Lambda_\Gamma)$.
\end{proof}

\subsection{Discontinuity domains in \texorpdfstring{$T^1\AdS_3$}{the spacelike unit tangent bundle}}\label{sec:discdomT1}

The proper discontinuity of the action of a  quasi-Fuchsian group $\Gamma\subset{\rm Isom}(\AdS_3)$ on the open domain $\mathcal O_+(\Lambda_\Gamma)\subset \AdS_3$ immediately implies proper  discontinuity for the action on $T^1\mathcal O_+(\Lambda_\Gamma)\subset T^1\AdS_3$. There is, however, a larger discontinuity domain in $T^1\AdS_3$ (as mentioned above, this was already observed in \cite{Delarue_Monclair_Sanders} in the more general context of projective Anosov subgroups, and very recently this observation was even generalized to arbitrary Anosov subgroups in \cite{Delarue_Monclair_SandersII}).

\begin{prop} \label{prop - discontinuity domain unit spacelike tangent bundle}
Let $\Gamma\subset {\rm Isom( AdS_3)}$ be a quasi-Fuchsian group, and consider the open set 
\[ \widetilde{\mc{M}}=\set{ (x,v)\in T^1\AdS_3 }{ \exists t\,\in \R, \pi(\varphi_t(x,v))\in \mathcal{O}_+(\Lambda_\Gamma) } = \bigcup_{t\in\R}\varphi_t\left(T^1\mathcal{O}_+(\Lambda_\Gamma)\right).\]
 The action of $\Gamma$ on $\widetilde{\mc{M}}$ is free and properly discontinuous.
\end{prop}
\begin{proof}
 
Let $(x_k,v_k)$ be a sequence in $\widetilde{\mc{M}}$ converging to $(x,v)\in \widetilde{\mc{M}}$, and consider a sequence $\gamma_k\to \infty$ in $\Gamma$. Let $t\in\R$  be such that $(y,w):=\varphi_t(x,v)\in T^1\mathcal{O}_+(\Lambda_\Gamma)$. For $k$ large enough, we also have $(y_k,w_k):=\varphi_t (x_k,v_k)\in T^1\mathcal{O}_+(\Lambda_\Gamma)$. The proper discontinuity of the action on $\mathcal{O}_+(\Lambda_\Gamma)$ (Prop.\ \ref{prop:Mess}) guarantees that up to a subsequence,  $\gamma_k y_k\to \nu \in \Lambda_\Gamma$.  
It follows that the sequence $\gamma_k(y_k,w_k)$ cannot converge to a point in $\widetilde{\mc{M}}$ and neither can the sequence $\gamma_k(x_k,v_k)$ because $\varphi_t$ commutes with the $\Gamma$-action. So the action of $\Gamma$ on $\widetilde{\mc{M}}$ is properly discontinuous. The group $\Gamma$ being torsion free, any properly discontinuous action of $\Gamma$ is free. \color{black}
\end{proof}

We denote the quotient manifold 
\[ \mc{M}:=\Gamma\backslash \widetilde{\mc{M}}.\]
Note that the geodesic flow $\varphi_t:\mc{M} \to \mc{M}$ is complete, while its restriction to $\Gamma\backslash T^1\mathcal O_+(\Lambda_\Gamma)$ is never complete: the only spacelike geodesics entirely contained in $\mathcal O_+(\Lambda_\Gamma)$ are those with both endpoints in $\Lambda_\Gamma$.\color{black}

\subsection{The non-wandering set of the geodesic flow}\label{sec:nonwanderingset}
Let $\Gamma\subset {\rm Isom}(\AdS_3)_\circ$ be a quasi-Fuchsian group. Consider its limit set $\Lambda_\Gamma\subset \partial \AdS_3$, and the discontinuity domains $\mathcal{O}_+(\Lambda_\Gamma)\subset \AdS_3$ and $\widetilde{\mc{M}}\subset T^1\AdS_3$.

For any $h\in {\rm Isom}(\AdS_3)_\circ=\mathrm{SO}(2,2)_\circ$, the logarithms of the moduli of the (complex) eigenvalues of $h$ can be ordered as 
\begin{equation}\label{lambda_i(h)}
\lambda_1(h)\geq \lambda_2(h)\geq -\lambda_2(h)\geq -\lambda_1(h)\,.
\end{equation}
According to Mess' Theorem (Proposition \ref{thm - Mess}), a nontrivial element  $\gamma\in \Gamma\subset \SO(2,2)_\circ$ corresponds, via   the identification ${\rm Isom}(\AdS_3)_\circ\simeq G\times G$, to an element $(\rho_1(\gamma),\rho_2(\gamma))\in G\times G$, where $\rho_i(\gamma)\in G=\SL(2,\R)$ is hyperbolic (because it belongs to the holonomy group of a hyperbolic metric on a closed surface). It is therefore conjugate to 
 \[ \left( \begin{pmatrix} e^{\ell_1} & \\ & e^{-\ell_1}\end{pmatrix} ,\begin{pmatrix} e^{\ell_2} & \\ & e^{-\ell_2}\end{pmatrix} \right)\,,\]
 for some $\ell_1,\ell_2>0$. This element corresponds through  the isomorphism $G\times G\simeq {\rm Isom}(\AdS_3)_\circ$ to the matrix
 \[ \begin{pmatrix} e^{\ell_1+\ell_2} &&& \\ & e^{\ell_1-\ell_2} && \\ && e^{\ell_2-\ell_1} & \\ &&& e^{-\ell_1-\ell_2} \end{pmatrix}\,.\]
It follows that the eigenvalues of  $\gamma\in\Gamma$ are real, positive, and $\lambda_1(\gamma)> \lambda_2(\gamma)$. The unique eigendirection in $\R^4$ associated to the eigenvalue $ e^{\lambda_1(\gamma)}$ projects to an attracting fixed point $\gamma^+\in \Lambda_\Gamma$, while the eigendirection for $e^{-\lambda_1(\gamma)}$ yields a  repelling  fixed point $\gamma^-\in\Lambda_\Gamma$. 

The following two lemmas are key technical ingredients in the proof of the Axiom A property of the flow in Proposition \ref{prop - spacelike geodesic flow axiom A}.\color{black}

\begin{lemma} \label{lem - fixed points on the boundary}
Let $\gamma\in\Gamma\setminus\{{\rm id}\}$ and let $\nu\in\partial \AdS_3$ be a fixed point  of $\gamma$ distinct from $\gamma^+$ and $\gamma^-$. Then $q(\nu,\gamma^+)=q(\nu,\gamma^-)=0$.
\end{lemma}

\begin{proof}Recall that $\partial \AdS_3=\{q= 0\}/\R_{>0}$ is a space of half-lines. \color{black}
Let $z^\pm$ be vectors representing $\gamma^\pm$, and $\ell^\pm$ be the associated eigenvalues (note that $\ell^+\ell^-=1$). Consider also 
$z$ representing $\nu$ and $\ell$ its eigenvalue. Now $q(z,z^\pm)=q(\gamma z,\gamma z^\pm)= \ell \ell^\pm q(z,z^\pm)$. Since $\nu$ is distinct from $\gamma^+$ and $\gamma^-$, it follows  that $\ell\neq \ell^+$ and $\ell\neq \ell^-$, thus $\ell \ell^+\neq 1$ and $\ell \ell^-\neq 1$, and the result follows.
\end{proof}

\begin{lemma} \label{lem - lambda 1 vs lambda 2}
There is a constant $c>0$ such that for all $\gamma\in \Gamma$, $\lambda_1(\gamma)-\lambda_2(\gamma)\geq c \lambda_1(\gamma)$.
\end{lemma}
\begin{proof}
Consider a finite generating set $S\subset \Gamma$, the associated word length $|\cdot|$ on $\Gamma$ and the stable word length $|\gamma|_\infty=\lim_{n\to+\infty}\frac{1}{n}|\gamma^n|$. Consider the singular values  $\mu_i(h)=\frac{1}{2}\log \lambda_i(h^Th)$ for $h\in \mathrm{SO}(2,2)$  (in decreasing order), \color{black} so that 
\[ \lambda_i(h)=\lim_{n\to+\infty}\frac{\mu_i(h^n)}{n}\,,~\forall i=1,2\,.\] 
For $i=1$, this equality comes from the fact that $\mu_1(h)$ is the logarithm of the operator norm of $h$ (see \cite[Proposition 4.4.1]{DSmatrices}). The formula for $i=2$ reduces to that for $i=1$ by considering the representation of $\SO(2,2)\subset\GL(4,\R)$ in $\GL(\Lambda^2\R^4)$. Since $\mu_1(h)$ is the logarithm of the operator norm of $h$, it is subadditive. Setting $\alpha=\max\set{\mu_1(\sigma)}{\sigma\in S}>0$, we find $\mu_1(\gamma)\leq \alpha|\gamma|$ for all $\gamma\in \Gamma$, thus
\[ \lambda_1(\gamma)=\lim_{n\to +\infty}\frac{\mu_1(\gamma^n)}{n} \leq \alpha|\gamma|_\infty\,.\]
Now from the characterisation of Anosov subgroups (of which quasi-Fuchsian groups are special cases, see the paragraph after Remark \ref{rem:defrem})  in terms of singular values \cite{KLP,BPS} there is $\beta>0$ such that $\mu_1(\gamma)-\mu_2(\gamma)\geq \beta |\gamma| - 1/\beta$ for all $\gamma\in\Gamma$, thus
\[ \lambda_1(\gamma)-\lambda_2(\gamma)=\lim_{n\to+\infty}\frac{\mu_1(\gamma^n)-\mu_2(\gamma^n)}{n}\geq \lim_{n\to +\infty}\frac{\beta|\gamma^n|}{n}-\frac{1}{n\beta}=\beta|\gamma|_\infty\,.\]
 The lemma is proved with  $c=\beta/\alpha$.
\end{proof}

Consider the set of points $(x,v)\in T^1\AdS_3$ so that the geodesic passing through $(x,v)$ has  one or  both of its endpoints in the limit set $\Lambda_\Gamma$:
\bq\label{eq:Kpmdef}
\widetilde{\mc{K}}_\pm :=\set{(x,v)\in T^1\AdS_3}{B_\pm(x,v)\in \Lambda_\Gamma}, \quad 
\widetilde{\mc{K}} :=\widetilde{\mc{K}}_+\cap \widetilde{\mc{K}}_-.
\eq
In view of \eqref{eq:qxpsiBpm}, we have 
\bq
\widetilde{\mc{K}}_\pm \subset T^1\mathcal{O}_+(\Lambda_\Gamma),\label{eq:inclKpm}
\eq
while for $\widetilde{\mc{K}}$ we get the chain of inclusions of $\Gamma$-invariant subsets
 \[  
 \widetilde{\mc{K}}\subset T^1C(\Lambda_\Gamma)\subset T^1\mathcal{O}_+(\Lambda_\Gamma) \subset \widetilde{\mc{M}}.
 \] 
 Note that amongst all these sets, only $\widetilde{\mc{K}}$, $\widetilde{\mc{K}}_\pm$  and $\widetilde{\mc{M}}$ are invariant under the spacelike geodesic flow. Define 
\bq
\mc{K}:= \Gamma\backslash \widetilde{\mc{K}}.\label{eq:Kdef}
\eq
The following lemma is a consequence of \cite[Lemma 3.9, Thm.~2]{Delarue_Monclair_Sanders}, but we reprove it here for a more self-contained presentation. It establishes the first part of the Axiom A property, the compactness property being a crucial assumption for our microlocal applications (see Assumption \ref{ass:A0}).\color{black}
\begin{lemma} \label{lem - non-wandering set for the geodesic flow}
The set $\mc{K}$  defined in \eqref{eq:Kdef}, \eqref{eq:Kpmdef}\color{black} is the non-wandering set  of the geodesic flow on $\mc{M}$, and it is equal to the closure of the set of periodic points. Moreover, it is compact.
\end{lemma}
\begin{proof}
Let $\mathcal{NW}$ be the non-wandering set  of the geodesic flow on $\mc{M}$, and let us start by showing that $\mathcal{NW}\subset \mc{K}$. For this purpose, consider  $p\in \mathcal{NW}$. This means that there exist sequences $p_k\to p$ and $t_k\to +\infty$ such that $\varphi_{t_k}(p_k)\to p$. Consider a lift $(x,v)\in \widetilde{\mc{M}}$ of $p$, and lifts $(x_k,v_k)\in \widetilde{\mc{M}}$ of $p_k$ such that $(x_k,v_k)\to (x,v)$. Since $\varphi_{t_k}(p_k)\to p$, there is a sequence $\gamma_k$ in $\Gamma$ such that $\gamma_k \varphi_{t_k} (x_k,v_k)\to (x,v)$. Note that up to composing with $\varphi_t$ for some given $t$, we can assume that $x\in \mathcal{O}_+(\Lambda_\Gamma)$ (this is fine because $\mc{K}$ is invariant under the geodesic flow).

 Write $(y_k,w_k)=\varphi_{t_k} (x_k,v_k)$. Since $t_k\to +\infty$, we have $y_k\to \nu^+\in \partial \AdS_3\cap \overline{\mathcal{O}_+(\Lambda_\Gamma)}=\Lambda_\Gamma$. 
 By continuity and $\varphi_t$-invariance of $B_+$, it follows that  $B_+(x,v)=\nu^+$. We can apply the same argument with $p'_k=\varphi_{t_k}p_k\to p$ and $\varphi_{-t_k}p'_k\to p$ to deduce that $B_-(x,v)\in \Lambda_\Gamma$, hence $(x,v)\in \widetilde{\mc{K}}$, and $p\in \mc{K}$.  This achieves the proof of $\mathcal{NW}\subset \mc{K}$.

 Let $\P\subset \mc{M}$ be the set of periodic points for the geodesic flow, and $\widetilde{\P}\subset \widetilde{\mc{M}}$ its lift. For any $\gamma\in\Gamma\setminus\{{\rm id}\}$, the spacelike geodesic with endpoints $\gamma^-$ and $\gamma^+$ lies in $\widetilde\P$, and in $\widetilde{\mc{K}}$. Since the set of axes of $\Gamma$, i.e.\ pairs $(\gamma^+,\gamma^-)$ for $\gamma\in \Gamma\setminus\{{\rm id}\}$, is dense in $\Lambda_\Gamma^{(2)}$ (this is a general fact for Gromov hyperbolic groups \cite[Cor. 8.2.G]{Gr}), we find that $\widetilde{\mc{K}}$ is included in the closure of $\widetilde\P$, therefore $\mc{K}\subset \overline{\P}$. Now the inclusion $\overline{\P}\subset \mathcal{NW}$ holds for any continuous flow, so we finally get $\mc{K}=\overline{\P}= \mathcal{NW}$.
 
 The action of $\Gamma$ on $C(\Lambda_\Gamma)$ being co-compact, we can consider a compact set $L\subset C(\Lambda_\Gamma)$ such that $\Gamma\cdot L=C(\Lambda_\Gamma)$. The set $\widetilde L:=\{(x,v)\in \widetilde{\mc{K}}\,|\, x\in L\}$ satisfies $\Gamma\cdot \widetilde L=\widetilde{\mc{K}}$ (because $(x,v)\in \widetilde{\mc{K}}$ implies that $x\in C(\Lambda_\Gamma)$). Consider now the two diffeomorphisms
 \[ \Upsilon_\pm:\map{T^1\AdS_3}{\mathcal O_\pm\subset \AdS_3\times\partial\AdS_3}{(x,v)}{(x,B_\pm(x,v)).}\]
Recalling the definition of $\widetilde{\mc{K}}$, we see that $\widetilde L=\Upsilon_+^{-1}(L\times\Lambda_\Gamma)\cap \Upsilon_-^{-1}(L\times\Lambda_\Gamma)$. Thus $\widetilde L$ is compact as the intersection of two compact sets. Since $\Gamma\cdot \widetilde L=\widetilde{\mc{K}}$, this means that the action of $\Gamma$ on $\widetilde{\mc{K}}$ is co-compact, i.e. that $\mc{K}$ is compact. 
\end{proof}

Further convenient properties of the set $\mathcal K$ are established in \cite[Thm.\ 2]{Delarue_Monclair_Sanders}. The following proposition is a special case of \cite[Thm.\ A/Lemma 3.13]{Delarue_Monclair_Sanders}, again we give a direct proof in our context.  Recall that a fixed-point free smooth flow is said to satisfy Smale's \emph{Axiom A} \cite{smale67} if its non-wandering set is compact, hyperbolic, and equal to the closure of the set of periodic points.

\begin{prop} \label{prop - spacelike geodesic flow axiom A}
Let $\Gamma\subset {\rm Isom}(\AdS_3)$ be a quasi-Fuchsian group. The spacelike geodesic flow of $\mc{M}=\Gamma\backslash \widetilde{\mc{M}}$ satisfies the Axiom A.
\end{prop}
\begin{proof} Thanks to Lemma \ref{lem - non-wandering set for the geodesic flow} it remains to show that the non-wandering set ${\mc{K}}$ is a hyperbolic set for the geodesic flow.

 Consider the splitting $T\left( T^1\AdS_3\right)=\R X\oplus E^s_R\oplus E^s_L\oplus E^u_R\oplus E^u_L$ from \eqref{eq:decomposition into 5 line bundles}. Their ${\rm SO}(2,2)_\circ$-invariance means that they also define a splitting of the tangent bundle of $\Gamma\backslash \widetilde{\mc{M}}$. Now consider a Riemannian metric $g_{\mc{M}}$  on $\Gamma\backslash \widetilde{\mc{M}}$ that makes this decomposition  orthogonal. Its lift to $\widetilde{\mc{M}}$, still denoted $g_{\mc{M}}$, is  $\Gamma$-invariant and also makes the decomposition \eqref{eq:decomposition into 5 line bundles} orthogonal.  We will use the notation $\|\cdot\|_{g_{\mc{M}}}$ for this norm.

Consider a compact set ${\mc{F}}\subset \widetilde{\mc{K}}$ such that $\Gamma\cdot{\mc{F}}=\widetilde{\mc{K}}\supset \widetilde{\mathcal{P}}$ (Lemma \ref{lem - non-wandering set for the geodesic flow}). Let $p\in \P$ be a periodic point of the geodesic flow on $\Gamma\backslash \widetilde{\mc{M}}$, and $w\in E_s(p)$. Lift $p$ to $\rho=(x,v)\in \mc{F}\subset \widetilde{\mc{M}}$ and $w$ to $\widetilde w\in E_s(\rho)$.  
Recall that $T^1\AdS_3$   is  a submanifold of $\R^4\times \R^4$ \eqref{T1ads3}, and $E_s(\rho)\subset T_\rho T^1\AdS_3$  is a vector subspace of $\R^4\times \R^4$ \eqref{def:Es}. We know from \eqref{dvarphit} that  $\d\varphi_t(\rho)\widetilde w=e^{-t}\widetilde{w} \in E_s(\varphi_t(\rho))=E_s(\rho)$, thus  for $t>0$
\begin{equation} \|\d\varphi_t(\rho)\widetilde w\|_{g_{\mc{M}}(\varphi_t(\rho))} =e^{-t} \|\widetilde w\|_{g_{\mc{M}}(\varphi_t(\rho))}.  \label{lemeq hyperbolicity 1}
\end{equation}

Let $T>0$ be the period of the orbit of  $p$, and $\gamma\in\Gamma$ be such that $\varphi_{T}(\rho)=\gamma\cdot\rho$. By Lemma \ref{lem - fixed points on the boundary}, we have  $T=\lambda_1(\gamma)$. 
The compact set $\mc{F}$ contains a lift of the whole orbit $\{\varphi_s(p)\,|\,s\in [0,T]\}$ of $p\in \P$.  Let $k\in \N$ and $t'\in (0,T)$ such that  $\varphi_{t'}(\rho)\in \mc{F}$ and $\varphi_{lT+t'}(\rho)\not\in \mc{F}$ for all integers $l>k$. Denote $t:=kT+t'$ and
 let  $\rho'=\varphi_{t'}(\rho)\in \mc{F}$ so that  $\varphi_t(\rho)=\gamma^k\cdot\rho'$. \color{black} 
 The $\Gamma$-invariance of the norm $\|\cdot\|_{g_{\mc{M}}}$ yields 
\begin{align}
\begin{split}\|\widetilde w\|_{g_{\mc{M}}(\varphi_t(\rho))}&= \|\gamma^{-k}\widetilde w\|_{g_{\mc{M}}(\gamma^{-k}\varphi_t(\rho))}\\
&=\|\gamma^{-k}\widetilde w\|_{g_{\mc{M}}(\rho')}.
\end{split} \label{lemeq hyperbolicity 2}
\end{align}
 Writing $\widetilde w'= \d\varphi_{t'}(\rho)\widetilde w=e^{-t'}\widetilde w\in E_s(\rho')$, \eqref{lemeq hyperbolicity 1} and \eqref{lemeq hyperbolicity 2} yield
\begin{equation}
\|\d\varphi_t(\rho)\widetilde w\|_{g_{\mc{M}}(\varphi_t(\rho))} = e^{t'-t}\|\gamma^{-k}\widetilde w'\|_{g_{\mc{M}}(\rho')}. \label{lemeq hyperbolicity 3}
\end{equation} 
Let us now decompose $\widetilde w'= w_L+ w_R\in E_s^L(\rho')\oplus E_s^R(\rho')$ according to \eqref{eq:decomposition into 5 line bundles}, then write $w_L=(\xi_L,-\xi_L)$ and $w_R=(\xi_R,-\xi_R)$ where $\xi_L,\xi_R\in x^\perp\cap v^\perp\subset \R^4$ \eqref{def:Es}.  Note that $\gamma  w_L\in E_s^L(\gamma\rho')=E_s^L(\varphi_T(\rho'))$, thus $\gamma\xi_L \in {\mc L}^L_{\varphi_T(\rho')}={\mc L}^L_{\rho'}$, so $\xi_L$ is an eigenvector of $\gamma$, and so is $\xi_R$. Since these vectors are orthogonal to each other and to the eigenvectors $x \pm v$ of $\gamma$ for the eigenvalues $e^{\pm \lambda_1(\gamma)}$, the corresponding eigenvalues are $e^{\pm \lambda_2(\gamma)}$, thus
\[ 
\gamma^{-k}\widetilde w' = e^{\pm  k\lambda_2(\gamma)}  w^L + e^{\mp k\lambda_2(\gamma)}  w^R.
\]
Since the norm $\|\cdot\|_{g_{\mc{M}}}$ makes the decomposition $E_s(\rho')=E_s^L(\rho')\oplus E_s^R(\rho')$ orthogonal, we find 
\begin{equation}\label{gamma^-kwbis}
\| \gamma^{-k}\widetilde w'\|_{g_{\mc{M}}(\rho')} \leq e^{k\lambda_2(\gamma)}\|\widetilde w'\|_{g_{\mc{M}}(\rho')}.
\end{equation}
Using \eqref{lemeq hyperbolicity 3}, \eqref{gamma^-kwbis}, and the fact that $t-t'=kT=k\lambda_1(\gamma)$, we find
\begin{align*} 
\|\d\varphi_t(\rho)\widetilde w\|_{g_{\mc{M}}(\varphi_t(\rho))}& \leq e^{t'-t+k\lambda_2(\gamma)}\|\widetilde w'\|_{g_{\mc{M}}(\rho')}\\
&\leq e^{k(\lambda_2(\gamma)-\lambda_1(\gamma))}e^{-t'} \|\widetilde w\|_{g_{\mc{M}}(\rho')}
\end{align*}
By Lemma \ref{lem - lambda 1 vs lambda 2} there is a constant $c>0$ such that $\lambda_1(\gamma)-\lambda_2(\gamma)\geq c \lambda_1(\gamma)$, thus
\begin{align*} \|\d\varphi_t(\rho)\widetilde w\|_{g_{\mc{M}}(\varphi_t(\rho))} &\leq  e^{-kc\lambda_1(\gamma)} e^{-t'}\|\widetilde w\|_{g_{\mc{M}}(\rho)} \\
&\leq  e^{-ct} e^{(c-1)t'}\|\widetilde w\|_{g_{\mc{M}}(\rho)}.
\end{align*}
The geodesic flow acts properly on $T^1\AdS_3$, so there is a real number $A_{\mc{F}}>0$ depending only on the compact set $\mc{F}$ such that $\vert t'\vert \leq A_{\mc{F}}$ (because $\rho\in \mc{F}$ and $\rho'=\varphi_{t'}(\rho)\in \mc{F}$).

 There is also a real number $B_{\mc{F}}>0$ depending only on $\mc{F}$ such that 
 $\|\widetilde w\|_{g_{\mc{M}}(\rho')}\leq B_{\mc{F}} \|\widetilde w\|_{g_{\mc{M}}(\rho)}$. Indeed, take $B_{\mc{F}}$ to be the maximum of the function $\|\pi_{E_s(q')}(z)\|_{g_{\mc{M}}(q')}$ for $q,q'\in \mc{F}$, $z\in E_s(q)$ with $\|z\|_{g_{\mc{M}}(q)}=1$, where $\pi_{E_s(q')}:\R^8\to E_s(q')$ is the orthogonal projection for the pseudo-Riemannian metric. By letting $C=B_{\mc{F}}e^{(c+1)A_{\mc{F}}}$ we find
\[ \|\d\varphi_t(\rho)w\|_{g_{\mc{M}}(\varphi_t(p))} \leq C e^{-ct}\|w\|_{g_{\mc{M}}(p)}. \]
This was shown for $w\in E_s(p)$ and $p\in \P$. It is also true for $p\in \mc{K}$ because of the density of $\P$ (Lemma \ref{lem - non-wandering set for the geodesic flow}). The same techniques show that the flow contracts $E_u$ in the past, so $\mc{K}$ is a hyperbolic set for the geodesic flow. 
\end{proof}

For the following rather technical lemma, we do not give a proof and instead cite the following result:

\begin{lemma}[{\cite[Lemma 3.10]{Delarue_Monclair_Sanders}}]\label{lem:3895021859205}
Consider sequences $p_k\in\M$ and $t_k\to +\infty$ and points $p,q\in \M$ such that $p_k\to p\in \M$ and $\varphi_{t_k}(p_k)\to q$. Let $(x,v),(y,w)\in \widetilde \M$ be lifts of $p,q$, respectively. Then $B_+(x,v)\in \Lambda_\Gamma$ and $B_-(y,w)\in \Lambda_\Gamma$. 
\end{lemma}
This lemma has the following direct consequence, which will be used later in Section \ref{sec:resonancesstandalone}:
\begin{cor}\label{cor:ass1} 
For every pair of compact sets $\mathcal C_1,\mathcal C_2\subset \mc{M}$  such that $\{t\geq 0\,|\,\varphi_t(\mathcal C_1)\cap \mathcal C_2\neq \emptyset\}$ is an unbounded set, we have $\mathcal C_1\cap {\mathcal K}_+\neq \emptyset$ and $\mathcal C_2\cap {\mathcal K}_-\neq \emptyset$, where ${\mathcal K}_\pm:=\gam\widetilde{\mc{K}}_\pm$.
\color{black}
\end{cor}
\begin{proof}
If $\{t\geq 0\,|\,\varphi_t(\mathcal C_1)\cap \mathcal C_2\neq \emptyset\}$ is unbounded, there are sequences $t_k\to +\infty$, $p_k\in \mathcal C_1$ such that $q_k:=\varphi_{t_k}(p_k)\in \mathcal C_2$ for all $k$. Up to a subsequence, $p_k\to p\in \mathcal C_1$ and $q_k\to q\in \mathcal C_2$. Then $p\in  {\mathcal K}_+$ and $q\in  {\mathcal K}_-$ by Lemma \ref{lem:3895021859205}.
\end{proof}

Let $\Gamma\subset G\times G$ be a quasi-Fuchsian group, and denote by $X$ the generator of the geodesic flow on $\mc{M}=\Gamma\backslash \widetilde{\mc{M}}$. Recall that we use the notation $E_s,E_u$ and $E_s^{R},E_s^{L},E_u^{R},E_u^{L}$ also for the bundles on $\M$ induced by the bundles denoted by the same symbols on $\widetilde\M$.  In the next three lemmas, we estbalish the existence of smooth sections of the line subbundles $E_s^{R/L}$ and $E_u^{R/L}$ and that they satisfy interesting properties with respect to non-degenerate volume forms $\omega_s$ and $\omega_u$ on the stable and unstable leaves. This will play an important role in the proof of Theorem \ref{theo:Poincare_series} where we relate the Poincar\'e series to the resolvent of the geodesic vector field acting on functions (as opposed to $2$-forms). \color{black}

\begin{lemma}\label{splitting_stable/instable}
There exist smooth vector fields $U^{R}_+,U^{L}_+,U_-^{R},U_-^L$ on $\mc{M}$, which are sections of $E_s^{R},E_s^{L},E_u^{R},E_u^{L}$, respectively, and positive functions $\lambda_+^{R},\lambda_+^{L},\lambda_-^{R},\lambda_-^{L}\in C^\infty(\mc{M})$, such that on $\mc K$ we have
\begin{align*}
& \left[ X,U_+^{R}\right]=\lambda_+^{R} U^{R}_{+},\,\,\, \qquad \,
\left[ X,U^{L}_{+}\right]=\lambda^{L}_+ U^{L}_{+},\\
& \left[ X,U^{R}_{-}\right]=-\lambda^{R}_- U^{R}_{-},\qquad 
\left[X,U^{L}_{-}\right]=-\lambda^{L}_- U^{L}_{-}.
\end{align*}
\end{lemma}

\begin{proof} 
Consider a smooth Riemannian metric $g_{\mc{M}}$  on $\mc{M}$. By Proposition \ref{prop - spacelike geodesic flow axiom A}  there are constants $\nu,C>0$ \color{black} with
\begin{align}
&\forall z \in \mc K, \forall w_s\in E_s(z), \forall t\geq 0, \quad \left\Vert \d\varphi_t(z)w_s\right\Vert_{g_{\mc{M}}(\varphi_t(z))} \leq  Ce^{-\nu t} \Vert w_s\Vert_{g_{\mc{M}}(z)} , \label{eqn adapted metric stable}\\
&\forall z\in \mc K, \forall w_u\in E_u(z), \forall t\leq 0, \quad \left\Vert \d\varphi_t(z) w_u\right\Vert_{g_{\mc{M}}(\varphi_t(z))} \leq  Ce^{-\nu |t|} \Vert w_u\Vert_{g_{\mc{M}}(z)}. \label{eqn adapted metric unstable}
\end{align}
For $z\in \Gamma\backslash \widetilde{\mc{M}}$ let $U^{R}_{+}(z)\in E_s^R(z)$ (resp. $U^{L}_{+}(z)\in E_s^{L}(z)$, $U^R_-(z)\in E_u^R(z)$ and $U^L_-(z)\in E_u^L(z)$) be the unique future directed vector of norm $1$ for $g_{\mc{M}}$. The invariance of the bundles $E_s^{R},E_s^{L},E_u^{R},E_u^{L}$ by the geodesic flow implies the existence of the functions $\lambda_+^{R}$, $\lambda_+^{L}$, $\lambda_-^{R}$, $\lambda_-^{L}\in C^\infty(\mc{M})$, where the smoothness is a consequence of the smoothness of $E_s^{R/L}(z)$ and $E_u^{R/L}(z)$ with respect to $z$. To see that these functions can be assumed to be positive, we note that it \color{black} is possible to choose the metric $g_{\mc{M}}$ to be adapted, i.e.\ so that the constant $C$ in \eqref{eqn adapted metric stable} and \eqref{eqn adapted metric unstable} is equal to $1$, but the drawback is that $g_{\mc{M}}$ would only be H\"older continuous. In that case one has $\la_{\pm}^{R/L}>0$ on $\mc K$ but this function is not smooth, so it suffices to take a smooth approximation of an adapted metric and the positivity $\la_{\pm}^{R/L}>0$ remains true while the regularity is $C^\infty$.
\end{proof}

The existence of non-vanishing $2$-forms in $\Lambda^2E_s^*$ and $\Lambda^2E_u^*$ proved in Lemma \ref{formes_vol_s/u} implies:  
\begin{lemma}\label{formes_vol_s/u_Gamma}
There exist non-vanishing $2$-forms $\omega_s\in\Lambda^2
E_s^*$ and $\omega_u\in\Lambda^2E_u^*$  such that $\mathcal L_{X}\omega_s=2\omega_s$ and $\mathcal L_{X}\omega_u=-2\omega_u$, where $E_{s/u}^\ast$ are the dual stable/unstable bundles of $\M$ defined explicitly in the proof of Proposition \ref{prop - spacelike geodesic flow axiom A}\color{black}. 
\end{lemma}
\begin{proof}
Any ${\rm SO}(2,2)_\circ$-invariant differential form on $T^1\AdS_3$ descends to a differential form on $\Gamma\backslash \widetilde{\mc{M}}$, so the result follows from Lemma \ref{formes_vol_s/u}.
\end{proof}
\begin{lemma}\label{lem:URLpm}
By relaxing the requirement of positivity of the smooth functions $\la_\pm^{R/L}$ from Lemma \ref{splitting_stable/instable} to positivity up to a coboundary, i.e.\ $\la_\pm^{R/L}-X(u_{\pm}^{R/L})>0$ for some smooth functions $u_\pm^{R/L}$, we can choose the vector fields $U_\pm^{R/L}$ on $\M$ and functions $\la_\pm^{R/L}\in \Cinft(\M)$ such that, in addition to the properties 
\begin{align*}
& \left[ X,U_+^{R}\right]=\lambda_+^{R} U^{R}_{+},\,\,\, \qquad \,
\left[ X,U^{L}_{+}\right]=\lambda^{L}_+ U^{L}_{+},\\
& \left[ X,U^{R}_{-}\right]=-\lambda^{R}_- U^{R}_{-},\qquad 
\left[X,U^{L}_{-}\right]=-\lambda^{L}_- U^{L}_{-}
\end{align*}\color{black}
mentioned in Lemma \ref{splitting_stable/instable}, we have
\begin{equation}\label{sumlapm} 
\la_+^R+\la_+^L=2, \quad  \la_-^R+\la_-^L=2,
\end{equation}
\begin{equation}\label{sumlapm2} 
\la_{+}^L=\la_-^R, \quad \la_{-}^L=\la_+^R,
\end{equation}
and \begin{equation}\label{wuws_on_basis}
\omega_u(U_+^R,U_+^L)=\omega_s(U_-^R,U_-^L)=\frac{1}{4}, \quad d\alpha(U^{L}_+,U_-^R)=d\alpha(U_+^R,U_-^L)=1.
 \end{equation}
\end{lemma}
\begin{proof} 
We start with vector fields $U_\pm^{R/L}$ as in Lemma \ref{splitting_stable/instable}  which we multiply by smooth positive functions $f_\pm^{R/L}$: this defines $\widetilde{U}_\pm^{R/L}=f_\pm^{R/L}U_\pm^{R/L}$.
We get 
\begin{equation}\label{4_equations}
\begin{gathered}
d\alpha(\widetilde{U}^{L}_+,\widetilde{U}_-^R)=f_+^Rf_-^R d\alpha(U^{L}_+,U_-^R), \quad d\alpha(\widetilde{U}_+^R,\widetilde{U}_-^L)=f_+^Lf_-^L d\alpha(U_+^R,U_-^L),\\
\omega_u(\widetilde{U}_+^R,\widetilde{U}_+^L)=f_+^Rf_+^L \omega_u(U_+^R,U_+^L),\quad \omega_s(\widetilde{U}_-^R,\widetilde{U}_-^L)=f_-^Rf_-^L \omega_s(U_-^R,U_-^L).
\end{gathered}
\end{equation}
Using \eqref{dalphaU} and since $U_\pm^{R/L}$ are given in any small open set as positive functions times 
the vector fields in Lemma \ref{Upm^RL}, we know that $d\alpha(U^{L}_+,U_-^R)$, $d\alpha(U_+^R,U_-^L)$,
$\omega_s(U_-^R,U_-^L)$ and $ \omega_u(U_+^R,U_+^L)$ are smooth positive functions on $\mc{M}$. Taking the log of the 4 equations in \eqref{4_equations}, we obtain a linear system in the variables $\log f_\pm^{R/L}$ of rank 3 where the first 3 equations are linearly independent. This means that we can choose $f_\pm^{R/L}$ in a such a way that 
\[d\alpha(\widetilde{U}^{L}_+,\widetilde{U}_-^R)=d\alpha(\widetilde{U}_+^R,\widetilde{U}_-^L)=1, \quad \omega_u(\widetilde{U}_+^R,\widetilde{U}_+^L)=1/4.\]
 This implies, using $\mc{L}_Xd\alpha=0$, that 
\[ \begin{split}
0=X(d\alpha(\widetilde{U}^{L}_+,\widetilde{U}_-^R))=& (\mc{L}_Xd\alpha)(\widetilde{U}_+^R,\widetilde{U}_+^L)+d\alpha(\mc{L}_X\widetilde{U}_+^L,\widetilde{U}_-^R)+d\alpha(\widetilde{U}_+^L,\mc{L}_X\widetilde{U}_-^R)\\
=&\widetilde{\la}_{+}^{L}-\widetilde{\la}_{-}^{R}
\end{split}\]
and similarly $\widetilde{\la}_{+}^{R}-\widetilde{\la}_{-}^{L}=0$.
In the same way, 
this implies that (using Lemma \ref{formes_vol_s/u_Gamma})
\[ \begin{split}
0=X(\omega_u(\widetilde{U}_+^R,\widetilde{U}_+^L))&= (\mc{L}_X\omega_u)(\widetilde{U}_+^R,\widetilde{U}_+^L)+\omega_u(\mc{L}_X\widetilde{U}_+^R,\widetilde{U}_+^L)+\omega_u(\widetilde{U}_+^R,\mc{L}_X\widetilde{U}_+^L)\\
&= -2\omega_u(\widetilde{U}_+^R,\widetilde{U}_+^L)+(\widetilde{\la}_{+}^{R}+\widetilde{\la}_{+}^{L})\omega_u(\widetilde{U}_+^R,\widetilde{U}_+^L)
\end{split}\]
which gives that $\widetilde{\la}_+^R+\widetilde{\la}_+^L=2$, and therefore  $\widetilde{\la}_-^R+\widetilde{\la}_-^L=2$.  
Now, we also have 
\[ \begin{split}
X(\omega_s(\widetilde{U}_-^R,\widetilde{U}_-^L))&= (\mc{L}_X\omega_s)(\widetilde{U}_-^R,\widetilde{U}_-^L)+\omega_s(\mc{L}_X\widetilde{U}_-^R,\widetilde{U}_-^L)+\omega_s(\widetilde{U}_-^R,\mc{L}_X\widetilde{U}_-^L)\\
&= 2\omega_s(\widetilde{U}_-^R,\widetilde{U}_-^L)+(\widetilde{\la}_{-}^{R}+\widetilde{\la}_{-}^{L})\omega_s(\widetilde{U}_-^R,\widetilde{U}_-^L)=0
\end{split}\]
which shows that on $\mathcal{K}$ we have $\omega_s(\widetilde{U}_-^R,\widetilde{U}_-^L)=c$ for some $c\in \R$ since the $\varphi_t$-action on $\mathcal{K}$ is topologically transitive \cite[Thm.~2]{Delarue_Monclair_Sanders}. Since the commutator identities from Lemma \ref{splitting_stable/instable} only have to hold on $\mathcal K$, we are free to arrange that in fact $\omega_s(\widetilde{U}_-^R,\widetilde{U}_-^L)=c$ on all of $\M$. By using \eqref{dalpha^2omegas} and \eqref{dalpha^2}, it follows that $c=1/4$.
Notice finally that on $\mathcal K$ we have by Lemma \ref{splitting_stable/instable}
\[ \begin{split}
\pm \widetilde{\la}_\pm^{R/L}\widetilde{U}_\pm^{R/L}&=[X,\widetilde{U}_\pm^{R/L}]=f^{R/L}_\pm [X,U_\pm^{R/L}]+X(f^{R/L}_\pm)U_\pm^{R/L}\\
& =\pm \la_{\pm}^{R/L}\widetilde{U}_\pm^{R/L}+X(\log f^{R/L}_\pm)\widetilde{U}_\pm^{R/L}
\end{split}\]
thus $\widetilde{\la}_\pm^{R/L}=\la^{R/L}_{\pm}+X(\log f^{R/L}_\pm)$, and again we can make this hold on all of $\M$ since outside of $\mathcal K$ there are no constraints. This proves the positivity of $\widetilde{\la}_\pm^{R/L}$ up to coboundary.
\end{proof}
 We remark that the vector fields $U_{\pm}^{R/L}$ are not uniquely defined, but \color{black}
we now fix them and the functions $\lambda_\pm^{R/L}$ as in Lemma \ref{lem:URLpm}. For a fixed convex neighborhood $\mc{U}$ of $\mc{K}$ and a fixed Riemannian metric $g_{\mc{U}}$ on $\mc{U}$, for all $\eps>0$ there is $C>0$ such that for all $z\in \mc{U}$ and $t\geq 0$ one has
\begin{equation}\label{expansionU} 
\begin{gathered}
\varphi_t(z)\in \mc{U}\implies \|\d\varphi_t(z) U_+^{R/L}(z)\|_{g_{\mc{U}}}
\leq Ce^{-(\nu_+^{R/L}(z)-\eps)t}\|U_+^{R/L}(\varphi_t(z))\|_{g_{\mc{U}}}, \\
\varphi_{-t}(z)\in \mc{U}\implies \|\d\varphi_{-t}(z) U_-^{R/L}(z)\|_{g_{\mc{U}}}\leq Ce^{-(\nu_-^{R/L}(z)-\eps)|t|}\|U_-^{R/L}(\varphi_{-t}(z))\|_{g_{\mc{U}}}
\end{gathered}
\end{equation}
with 
\[ \nu_+^{R/L}(z)=\liminf_{t\to \infty}\frac{1}{t}\int_0^t \la_+^{R/L}(\varphi_s(z))ds>0, \quad  \nu_-^{R/L}(z)=\liminf_{t\to \infty}\frac{1}{t}\int_{-t}^0 \la_-^{R/L}(\varphi_s(z))ds>0,\]
which satisfy $\nu_{+}^L=\nu_-^R$, $\nu_{-}^L=\nu_+^R$ and $\nu_+^R+\nu_+^L=\nu_{-}^R+\nu_-^L=2$ by \eqref{sumlapm} and \eqref{sumlapm2}. Here the positivity of $\nu_\pm^{R/L}$ holds thanks to the positivity of $\la_\pm^{R/L}$ up to coboundary.

\subsection{Convex domains with smooth boundary} 
Dynamical notions of convexity will play an important role in the analysis of the spacelike geodesic flow.

\begin{definition}\label{def dynamical convexity} ~\begin{itemize}
\item A subset $\mc C\subset \mc M$ is called \emph{dynamically convex} if for any $z\in \mc C$,  $\{t\in\R\,|\, \varphi_t(z)\in\mc C\}$ is an interval.
\item An open set $\mc U\subset\mc M$  whose closure $\overline{\mc U}$ is a smooth manifold with boundary is said to have \emph{dynamically strictly convex boundary} if, locally, there is a boundary defining function $\rho:\overline{\mc U}\to\R$  (i.e., $\rho\in \Cinft(\overline{\mathcal U})$ with $\rho>0$ on $\mathcal U$, $\rho=0$ on $\partial \mathcal U$, and $d\rho\neq 0$ on $\partial \mathcal U$) such that $X^2 \rho(z) <0$ for all $z\in \partial \mathcal U$ such that $d\rho(z)(X(z))=0$.
\end{itemize}
\end{definition}

Recall that the manifold $\mc M=\Gamma\backslash \widetilde{\mc{M}}$ contains $T^1M$ where $M=\Gamma\backslash \mc O_+(\Lambda_\Gamma)$.

\begin{definition}\label{def geodesic convexity}~
\begin{itemize}
\item A subset $C\subset \AdS_3$ is called \emph{convex} (resp.\ \emph{strictly convex}) if there is a convex (resp.\ strictly convex) set $C'\subset\{x\in\R^4\,|\,q(x)<0\}$ that has a one-to-one projection onto $C$.
\item  A subset $C\subset M$ is called \emph{convex} (resp.\ \emph{strictly convex}) if $C=\Gamma\backslash C'$ where $C'\subset\mc O(\Lambda_\Gamma)$ is $\Gamma$-invariant and convex (resp.\ strictly convex).
\end{itemize}
\end{definition}

\begin{prop}[{\cite[Lem. 6.4]{dgk18}}] There is a strictly convex compact subset $C_\mathrm{DGK}\subset M$  with smooth boundary such that $\Gamma\backslash C(\Lambda_\Gamma)\subset C_\mathrm{DGK}$
\end{prop}

\begin{rem} The statement of \cite[Lem. 6.4]{dgk18} only mentions a boundary with  $C^1$ regularity, but it concerns (with the notation of that paper) the domain $\Omega\subset\mc O_+(\Lambda_\Gamma)\cup\Lambda_\Gamma\subset  \AdS_3\cup \partial\AdS_3$ projected onto $C_\mathrm{DGK}\subset M$ when quotienting by $\Gamma$. Their construction is only $C^1$ because it includes $\Lambda_\Gamma$, but $\partial \Omega\setminus\Lambda_\Gamma$, hence $\partial C_\mathrm{DGK}$, is smooth.
\end{rem}

\begin{cor}\label{coro dynamically convex from DGK}
The subset $\overline{\mc U}=T^1C_\mathrm{DGK}\subset\mc M$ is a smooth manifold with dynamically strictly convex boundary such that both $\mc U$ and $\overline{\mc U}$ are dynamically convex. Moreover, $\mc K\subset\mc U$.
\end{cor}
\begin{proof}
Consider a boundary defining function $f: C_{\mathrm{DGK}}\to\R$.   Then  $\rho(x,v):=f(x)$ is a boundary defining function for $\mc U$ with the desired properties: if $z\in \partial \mathcal U$, then for $(x,v)\in T^1M$ near $z=(x_z,v_z)$ we can decompose $X(x,v)=X_\mathrm{hor}(x,v)+X_\mathrm{vert}(x,v)$ where $X_\mathrm{hor}$ and $X_\mathrm{vert}$ denote the horizontal and vertical components of $X$. In particular   if $\d\rho(z)(X(z))=0$, then $d\pi X_\mathrm{hor}(x_z,v_z)\in T_{x_z}\partial C_{\mathrm{DGK}}\setminus\{0\}$ ($\pi$ the projection onto the base) and consequently $X^2\rho(z)=X_\mathrm{hor}^2(\pi^*f)(x_z,v_z)<0$ by the strict convexity of $C_{\mathrm{DGK}}$.  \color{black} The inclusion $\mc K\subset\mc U$ comes from the inclusion of the convex core $\Gamma\backslash C(\Lambda_\Gamma)$ in $C_\mathrm{DGK}$.
\end{proof}

\subsection{Pullback of distributions \texorpdfstring{on $\widetilde{\mc{M}}$}{}}

Recall from  Corollary \ref{cor:Qlambda} that we have for every $\lambda\in \C$ a pair of isomorphisms 
\bqn
\mathcal Q^\pm_\lambda: \mc{D}'(\mathbb{T}^2)\to \{u\in \mc{D}'(T^1\AdS_3)\,|\,(X\mp \lambda) u=0,\; U_\pm u=0\;, \forall  U_\pm \in \Cinft(T^1\AdS_3;E_{\varepsilon_\pm})\},
\eqn
where $\varepsilon_+:=s$, $\varepsilon_-:=u$ and $\mathcal Q^\pm_\lambda(\omega)=\Phi_\pm^\lambda \til{B}_\pm^\ast(\omega)$ for all $\omega\in \mc{D}'(\mathbb{T}^2)$.  Furthermore, recall that the sets $\widetilde {\mathcal K}_\pm\subset T^1\AdS_3$ contain by definition exactly those $(x,v)\in T^1\AdS_3$ such that $B_\pm(x,v)\in \Lambda_\Gamma$, and by \eqref{eq:inclKpm} the sets $\widetilde {\mathcal K}_\pm$ are contained in the open set $\widetilde{\mc{M}}\subset T^1\AdS_3$. Define the open subsets
\[
\widetilde{\mc{M}}^{\pm\infty}:=B_\pm(\widetilde{\mc{M}})\subset \partial \AdS_3
\]
which contain $\Lambda_\Gamma$. Then, because distributions on  $\widetilde{\mc{M}}^{\pm\infty}$ that are supported in $\Lambda_\Gamma$ extend by zero to $B_\pm(T^1\AdS_3)=\partial \AdS_3$ and distributions on 
$\widetilde{\mc{M}}$ that are  supported in $\widetilde {\mathcal K}_\pm$ also extend by zero to $T^1\AdS_3$, it follows that the isomorphism $\mathcal Q^\pm_\lambda$ restricts to an isomorphism
\begin{multline}
\mathcal Q^\pm_\lambda:\{\omega\in\mc{D}'(\TT^2)\,|\, \supp \omega \subset \Lambda_\Gamma\}\longrightarrow\\
\{u\in \mc{D}'(T^1\AdS_3)\,| \,\supp u\subset \widetilde {\mathcal K}_\pm,\;(X\mp \lambda) u=0,\; U_\pm u=0,\; \forall U_\pm \in C^\infty(T^1\AdS_3;E_{\varepsilon_\pm})\}. \label{eq:IsoU_Gamma}
\end{multline}
We deduce from this the following:
\begin{prop}\label{resonance-Poisson}
Let $\Gamma\subset{\rm SO}(2,2)_\circ$ be a quasi-Fuchsian group, and denote by $X$ the generator of the geodesic flow on $\mc{M}=\Gamma\backslash \widetilde{\mc{M}}$. Let $M=\Gamma\backslash \mc{O}_+(\Lambda_\Gamma)$ be the quasi-Fuchsian Lorentzian manifold associated to $\Gamma$, and $\pi: T^1M\to M$ the projection onto the base.
For each $u\in \mc{D}'(\mc{M})$ nontrivial with $\supp(u)\subset \mc{K}_\pm$ and satisfying 
\[(X\mp \lambda) u=0, \quad U^{R}_\pm u=U^{L}_\pm u=0,\]
the distribution $f:=\pi_*(u|_{T^1M}) \in C^\infty(M)$ is a solution of the Klein-Gordon equation 
\[ (\Box_g+\la(\la+2))f=0 \textrm{ on }M\]
such that its lift $\tilde{f}$ to $\mc{O}_+(\Lambda_\Gamma)$ admits a nontrivial extension $\tilde{f}\in \mc{D}'(\AdS_3)$  such that 
\[\forall \gamma\in \Gamma, \,  \gamma^* \tilde{f}=\tilde{f} \textrm{ on }\AdS_3, \quad (\Box_g+\la(\la+2))\tilde{f}=0.\]
\end{prop}
\begin{proof} We first lift $u$ to $\widetilde{\mc{M}}$ as a $\Gamma$-invariant distribution $\tilde{u}$. Then by \eqref{eq:IsoU_Gamma}, there is a nontrivial  $\omega\in \mc{D}'(\TT^2)$ with $\supp(\omega)\subset \Lambda_\Gamma$ such that $\tilde{u}=\mc{Q}_\la^\pm \omega$. By Corollary \ref{cor:Qlambda} we also see that $\omega$ is $\Gamma$-equivariant in the sense that $\gamma^*\omega=N_\gamma^{-\la}\omega$ for all $\gamma\in \Gamma$.
Consider $\tilde{f}=\mc{P}_\la^\sigma \omega\in \mc{D}'(\AdS_3)$: by Proposition \ref{injectivite_Poisson} this is a $\Gamma$-invariant distribution on $\AdS_3$ that solves $(\Box_g+\la(\la+2))\tilde{f}=0$ on $\AdS_3$ and which is nontrivial since $\omega\not\equiv 0$. 
By Lemma \ref{lem:pushfoward},  $\tilde{f}=\pi_*(\tilde{u}|_{T^1\mc{O}_+(\Lambda_\Gamma)})$ on $\mc{O}_+(\Lambda_\Gamma)$ and it is a smooth function that descends to $M$ and solves $(\Box_g+\la(\la+2))f=0$ on $M$.
\end{proof}

\section{Flow resolvent and Ruelle resonances}\label{sec:resonancesstandalone}

 In this  standalone section, we prove that under reasonable conditions, the resolvent of the generator $X$ of an Axiom A flow on a non-compact manifold admits a meromorphic extension to the complex plane, and that more generally this holds for first order differential operators on bundles with scalar principal symbol given by the principal symbol of $X$. This includes the action of the Lie derivative $\mc{L}_X$ on symmetric tensors or differential forms for example, and 
the setting considered here includes the spacelike geodesic flow of quasi-Fuchsian anti-de Sitter spaces. The main result of this section is a consequence of the work \cite{DG16}, the main difference is that now we now work in a general non-compact setting  rather than in the asymptotically hyperbolic Riemannian manifold setting (or another type of non-compact settings) considered in \cite[Section 5]{DG16}. That is, we generalize some results from \cite{DG16} to a setting into which our space-like geodesic flow fits. \color{black} The connection with the concrete setting of the rest of this paper is made at the end in Section \ref{sec:applicationofsection}.

We begin by introducing a setup similar to that of \cite{DG16}. Let $\mc{M}$ be a smooth manifold without boundary and $X$ a nowhere-vanishing smooth vector field on $\mc{M}$ generating a complete flow \begin{align*}
\varphi:\R\times \mc{M}&\to \mc{M}\\
(t,x)&\mapsto \varphi(t,x)=:\varphi_t(x).
\end{align*}
As is common, we will usually abuse notation and write $\varphi_t$ instead of $\varphi$.

Let $\E\to \mc{M}$ be a smooth complex vector bundle and $\Xbf:\Cinft(\mc{M};\E)\to \Cinft(\mc{M};\E)$ a first order differential operator lifting $X$ in the sense that
\bq
\Xbf(fs)=(Xf) s + f \Xbf s\quad \forall\, s\in \Cinft(\mc{M};\E),f\in \Cinft(\mc{M}).\label{eq:liftingX}
\eq
Notice that the case ${\bf X}=X+V$ for some smooth potential $V$ and $\mc{E}=\mc{M}\times \R$ the trivial bundle fits into this setup.
Fix a smooth measure $\mu$ on $\mc{M}$ corresponding in each chart to a positive function times the Lebesgue measure \color{black} and a Hermitian metric $h$ on $\E$, neither of which is assumed to be invariant under $\varphi_t$. This identifies the dual bundle $\E^\ast$ with $\E$ and defines the Banach spaces $L^p(\mc{M};\mathcal E)$ for $p\in [1,\infty]$ as well as the formal $L^2$-adjoint $\mathbf X^\ast:\Cinft(\mc{M};\E)\to \Cinft(\mc{M};\E)$ of $\Xbf$, which also satisfies \eqref{eq:liftingX}, with $X$ replaced by $-X$. Furthermore, these choices allow us to identify the space $\mc{D}'(\mc{M};\E)$ of $\E$-valued distributions with the space of $\E$-valued distributional densities and fix an embedding $\CT(\mc{M};\E)\hookrightarrow  \mc{D}'(\mc{M};\E)$ by the $L^2$-inner product. The family of transfer operators
\bq
e^{-t\mathbf X}:\CT(\mc{M};\E)\to \CT(\mc{M};\E),\qquad t\in \R\label{eq:transferop}
\eq
extends to operators $\mc{D}'(\mc{M};\E)\to \mc{D}'(\mc{M};\E)$ and satisfies
\begin{align}
e^{-t'\mathbf X}\circ e^{-t\mathbf X}&=e^{-(t+t')\mathbf X}&&\forall\; t,t'\in \R,\label{eq:compprop}\\
\mathbf X u &= -\frac{d}{dt}\Big|_{t=0}e^{-t\mathbf X}u&&\forall\; u\in \CT(\mc{M};\E),\label{eq:charactetizationTO}\\
e^{-t\mathbf X}(fu)&=(f\circ \varphi_{-t})e^{-t\mathbf X}u&&\forall\, u\in \CT(\mc{M};\E), f\in \CT(\mc{M}).\label{eq:suppproperty}
\end{align}
There is a positive function $C_h\in \Cinft(\R\times\M)$ depending on the Hermitian metric $h$ such that for all $u\in C_c^\infty(\M;\mc{E})$ we have
\bq
|e^{-t{\bf X}}u(x)|_h \leq C_h(t,\varphi_{-t}(x)) |u(\varphi_{-t}(x))|_h\qquad \forall\; (t,x)\in \R\times\M. \label{eq:Ch}
\eq
Similarly, for the measure $\mu$ we get a positive function $C_\mu\in \Cinft(\R\times\M)$  such that
\bq\label{eq:Cmu}
(\varphi_t)_\ast \mu = C_\mu(t,\cdot)\mu\qquad \forall\;t\in \R. 
\eq
In order to obtain global meromorphic extensions of resolvents on all of $\M$, it is helpful to introduce the following two boundedness assumptions:
 \begin{enumerate}
 [label=\textbf{(B\arabic*)}]\setcounter{enumi}{0}
\item \emph{For all $t\in [0,1]$ one has $C_h(t,\cdot)\in L^\infty(\M)$.}\label{ass:B1}
\item \emph{For some $p\in [1,\infty)$ one has $C_h(t,\cdot)^pC_\mu(t,\cdot)\in L^\infty(\M)$ for all $t\in [0,1]$.}\label{ass:B2}
\end{enumerate}
Suppose that Assumption \ref{ass:B1} holds. We then obtain, for all $t\in [0,1]$,  $u\in \CT(\M;\E)$:
\[
\norm{e^{-t{\bf X}}u}_{L^\infty(\M;\E)}\leq  \|C_h(t,\cdot)\|_{L^\infty(\M)} \norm{u}_{L^\infty(\M;\E)},
\]
in particular $e^{-t{\bf X}}$ extends to a bounded operator $L^\infty(\M;\E)\to L^\infty(\M;\E)$.  Similarly, if  Assumption \ref{ass:B2} holds for some $p\in [1,\infty)$, then we find for all $u\in \CT(\M;\E)$:
\[
\sup_{t\in [0,1]}\norm{e^{-t{\bf X}}u}_{L^p(\M;\E)}\leq  \sup_{t\in [0,1]}\|C_h(t,\cdot)^pC_\mu(t,\cdot)\|_{L^\infty(\M)}^{1/p} \norm{u}_{L^p(\M;\E)},
\]
showing in particular that $e^{-t{\bf X}}$ extends to a continuous operator $L^p(\M;\E)\to L^p(\M;\E)$. In both cases, let us define for $p\in [1,\infty]$
\bq\begin{split}
c_p=\max\Big(1,\max_{t\in [0,1]}\|e^{-t{\bf X}}\|_{\mc{L}(L^p(\M,\mc{E}))}\Big), \quad C_p:=\log(c_p), \label{eq:Cp}\end{split}
\eq
which are finite constants by continuity in $t$, \color{black} and observe that by the group property \eqref{eq:compprop} $e^{-t{\bf X}}$ extends to a continuous operator on $L^p(\M;\E)$ for all $t\in \R$ and satisfies for $t\geq 0$
\[ \|e^{-t{\bf X}}\|_{\mc{L}(L^p(\M,\mc{E}))} \leq c_p^{t+1}=c_pe^{t C_p}.\] This implies that for each $\lambda\in \C$ with $\Re \lambda>C_p$, we can define a continuous operator
\bq\label{eq:Rlambdaintegral}
R_{{\bf X}}(\lambda): L^p(\M;\E)\to L^p(\M;\E),\quad 
R_{{\bf X}}(\lambda)
u:=\int_0^\infty e^{-t(\mathbf X+\lambda)} u\d t
\eq
with norm
\[ \|R_{{\bf X}}(\lambda)
\|_{\mc{L}(L^p(\M;\E))}\leq \frac{c_p}{{\rm Re}(\la)-C_p}.\]
If the assumption  \ref{ass:B1} or  \ref{ass:B2} holds, the operator $R_{{\bf X}}(\lambda)$ is thus continuous as a map
\bq
R_{{\bf X}}(\lambda)
: \CT(\mc{M};\E)\to \mc{D}'(\M;\E)\label{eq:globalresolvent}
\eq
and defines a holomorphic operator family $R_{{\bf X}}(\lambda)
$ on $\{\Re \lambda>\min_{p\in [1,\infty]} C_p\}\subset \C$. 

If one does not assume that  \ref{ass:B1} or  \ref{ass:B2} hold, we can still restrict to any compact set: For all relatively compact open sets $\mc{C}\subset \M$,  the operator 
$$R_{\bf X}^{\mc{C}}(\la):L^\infty(\mc{C};\mc{E})\to L^\infty(\mc{C};\mc{E})$$
defined by $R_{\bf X}^{\mc{C}}(\la)f=\int_0^\infty (e^{-t(\mathbf X+\lambda)} f)|_{\mc{C}}\d t$ is bounded for ${\rm Re}(\la)>\sup_{t\in [0,1]}
\|C_h(t,\cdot)\|_{L^\infty(\mc{C})}$, so the operator $R^{\mc{C}}_{{\bf X}}(\lambda)$ is also continuous as a map
\bq
R^{\mc{C}}_{{\bf X}}(\lambda)
: \CT(\mc{C};\E)\to \mc{D}'(\mc{C};\E)\label{eq:globalresolventC}
\eq
and produces a holomorphic operator family $R^{\mc{C}}_{{\bf X}}(\lambda)
$ on ${\rm Re}(\la)>\sup_{t\in [0,1]}
\|C_h(t,\cdot)\|_{L^\infty(\mc{C})}$.

Note that $R_{{\bf X}}(\lambda)
$ is a resolvent of $\Xbf$ in the sense that
\bq
(\Xbf + \lambda)(R_{{\bf X}}(\lambda)
u)=R_{{\bf X}}(\lambda)
((\Xbf + \lambda)u)=u\qquad \forall\,u\in\CT(\mc{M};\E),\label{eq:resolventprop}
\eq
where  $(\Xbf + \lambda)$ is considered as an operator $\mc{D}'(\mc{M};\E)\to \mc{D}'(\mc{M};\E)$ and $\CT(\mc{M};\E)\to\CT(\mc{M};\E)$ on the left and on the right, respectively. It has the property that if $u\in C_c^\infty(\mc{M};\E)$, then  $\supp (R_{\bf X}(\la)u)$ is contained in the forward flowout $\cup_{t\geq 0}\varphi_t(\supp(u))$ of the support of $u$.

The restricted resolvents $R^{\mc{C}}_{{\bf X}}(\lambda)$ have analogous properties inside the relatively compact open set $\mc{C}$, except that typically the transported support $\varphi_t(\supp(u))$ of a section $u\in \CT(\mathcal C;\E)$ will no longer be contained in $\mc{C}$ when $t$ is large.

Our goal is to continue the operator families  $R_{{\bf X}}(\lambda)
$ and  $R^{\mc{C}}_{{\bf X}}(\lambda)
$  meromorphically to $\C$. To this end, we will need some technical preparations and some assumptions.

\subsection{Dynamical preliminaries and assumptions}

For any subset $S\subset \mc{M}$ we define the \emph{backward/forward trapped sets of $\varphi$ in $S$} by 
\[
{\mathcal K}_\pm(S):= \{x\in S\,|\,\varphi_{\pm t}(x)\in S\;\forall\,t\in[0,\infty)\},
\]
as well as the \emph{trapped set of $\varphi$ in $S$} 
\[
\mathcal K(S):={\mathcal K}_+(S)\cap {\mathcal K}_-(S).
\]
The \emph{global backward/forward trapped sets} and the \emph{global trapped set} of $\varphi$ are
\bq
{\mathcal K}_\pm:=\bigcup_{\substack{\mathcal C\subset \mc{M} \\\text{compact}}}{\mathcal K}_\pm(\mathcal C),\qquad \mathcal K:={\mathcal K}_+\cap {\mathcal K}_-.\label{eq:defKpmK}
\eq
The \emph{non-wandering set} of $\varphi$ is defined as
\[
\mathcal{NW}:=\{x\in \mc{M}\,|\,\forall \,T>0,\, \mathcal{V}\subset \mc{M}\text{ open},\,x\in \mathcal{V},\exists \,t\geq T,\; \varphi_t(\mathcal{V})\cap \mathcal{V}\neq \emptyset\}.
\]
It is easy to see that if $S$ is closed in $\mc{M}$, then the sets ${\mathcal K}_\pm(S)$ are closed in $\M$, too. Thus each of the sets ${\mathcal K}_\pm(\mathcal C)$ featured in \eqref{eq:defKpmK} is compact. The global forward/backward trapped sets ${\mathcal K}_\pm$ are $\varphi_t$-invariant, which need not be the case for sets of the form ${\mathcal K}_\pm(S)$. When $\mc{M}$ is non-compact, the relation between the sets $\mathcal K$ and $\mathcal{NW}$ may be subtle. $\mathcal{NW}$ is always closed  in $\M$ but it is possible that ${\mathcal K}_+$, ${\mathcal K}_-$, and $\mathcal K$ are all non-closed and $\mathcal K\subsetneq \mathcal{NW}$ holds. 
On the other hand, when $\M$ is compact, one trivially has $\mathcal{NW}\subset \mathcal K=\M$ and usually $\mathcal{NW}\subsetneq \mc{M}$ holds, for example in the Axiom A case, where $\mathcal{NW}$ is the closure of the union of all periodic orbits.

The following two Assumptions ensure that the global (forward/backward) trapped sets ${\mathcal K}_+$, ${\mathcal K}_-$, and $\mathcal K$ have good properties:
\begin{enumerate}
[label=\textbf{(A\arabic*})]\setcounter{enumi}{-1}
\item \emph{$\mathcal K$ is compact.}
 \label{ass:A0}
\item \emph{For every pair of compact sets $\mathcal C_1,\mathcal C_2\subset \mc{M}$  such that $\{t\geq 0\,|\,\varphi_t(\mathcal C_1)\cap \mathcal C_2\neq \emptyset\}$ is an unbounded set, we have $\mathcal C_1\cap {\mathcal K}_+\neq \emptyset$ and $\mathcal C_2\cap {\mathcal K}_-\neq \emptyset$.}
\label{ass:A1}
\end{enumerate}
Assumptions \ref{ass:A0} and \ref{ass:A1} together imply that the sets ${\mathcal K}_\pm$ are closed in $\M$ and that $\mathcal{NW}\subset \mathcal K$ (so that $\mathcal{NW}$ is compact, too), see Lemma \ref{lem:KNW}.  

For several of our results we will make the following convenient additional assumptions, which make the results of \cite{DG16} immediately applicable: 
\begin{enumerate}
[label=\textbf{(A\arabic*})]\setcounter{enumi}{1}
\item\emph{There is a compact smooth submanifold with boundary $\overline{\mathcal U}\subset \mc{M}$ of the same dimension as $\M$ such that
\begin{enumerate}
\item $\mathcal K\subset\overline{\mathcal U}$;
\item Both $\overline{\mathcal{U}}$ and its manifold interior $\mathcal U$ are dynamically convex in the sense  of Definition \ref{def dynamical convexity}, that is:\color{black}
\bq
\begin{split}
\forall\, T\geq 0, x\in \mathcal U:\varphi_T(x)\in \mathcal U&\implies \varphi_t(x)\in \mathcal U,\;\forall\, t\in [0,T],\\
\forall\, T\geq 0, x\in \overline{\mathcal U}:\varphi_T(x)\in \overline{\mathcal U}&\implies \varphi_t(x)\in \overline{\mathcal U},\;\forall\, t\in [0,T];\label{eq:Uconvex}
\end{split}
\eq 
\item The boundary $\partial \mathcal U$ is strictly convex (Definition \ref{def:strictlyconvex}).
\end{enumerate}  }\label{ass:A2}
\item \emph{The flow $\varphi_t$ is uniformly hyperbolic on $\mathcal K$, i.e., there is a splitting
\bq
T\mc{M}|_{\mathcal K}=E_0\oplus E_s\oplus E_u\label{eq:splittingTN}
\eq
 into continuous  $\d\varphi_t$-invariant  subbundles such that $E_0=\R X$ and for some $C,c>0$ one has for all $x\in \mathcal K$}
\[
\norm{\d\varphi_t(x)v}_{\varphi_t(x)}\leq Ce^{-c|t|}\norm{v}_x\quad \text{if}\;\begin{cases} t\geq 0,\; v\in E_s|_x\qquad \text{or}\\
t\leq 0,\; v\in E_u|_x,\end{cases}
\]
where $\norm{\cdot}$ is a continuous bundle norm on $T\mc{M}|_{\mathcal K}$, for example coming from some Riemannian metric on $\mc{M}$ (unrelated to the Hermitian metric $h$ on $\E$).\label{ass:A3}
\end{enumerate} 
In order to give other useful characterizations of the Ruelle resonances, we need to prepare some more notation. Let us denote by 
\begin{equation}\label{def_Phi_t} 
\Phi_t(x,\xi)=(\d\varphi_t^{-1}(x))^\top \xi
\end{equation}
the symplectic lift of the flow $\varphi_t$ to $T^*\mc{M}$.
Given a splitting as in \eqref{eq:splittingTN}, we define the $\Phi_t$-invariant dual bundles $E^\ast_s,E^\ast_u\subset T^\ast \mc{M}|_{\mathcal K}$ by
\bq\label{eq:Epmstar}
E_s^\ast(E_0\oplus E_s)=0,\qquad E_u^\ast(E_0\oplus E_u)=0.
\eq
These bundles extend to ${\mathcal K}_+$ and ${\mathcal K}_-$, respectively.
 
\begin{lemma}[{C.f.~\cite[Lemma~1.10]{DG16}}]\label{lem:Gammaextension}Suppose that Assumptions \emph{\ref{ass:A0}--\ref{ass:A3}} hold. Then there exist unique continuous $\d\varphi_t$-invariant vector subbundles $E_{s}^\ast\subset T^\ast \mc{M}|_{{\mathcal K}_+}$ and $E_{u}^\ast\subset T^\ast \mc{M}|_{{\mathcal K}_-}$ extending respectively the bundles $E_s^\ast$ and  $E_u^\ast$ defined initially on $\mc{K}$, such that
\begin{enumerate}
\item $\xi(X|_x)=0$ for all  $x\in {\mathcal K}_+$ and $\xi\in E_{s}^\ast|_x$, and for all $x\in {\mathcal K}_-$ and $\xi\in E_{u}^\ast|_x$;
\item For each compact subset $\mathcal C\subset \mc{M}$ there exist $C,c>0$ such that for all $x\in {\mathcal K}_\pm(\mathcal C)$ one has
\begin{align}
&\forall x\in \mc{K}_+(\mc{C}), \xi\in E_{s}^\ast|_x, \,  \norm{ \Phi_t(x,\xi)}_{\varphi_t(x)}\leq Ce^{-ct}\norm{\xi}_x\quad \text{if } t\geq 0,\; \\
&\forall x\in \mc{K}_-(\mc{C}), \xi\in E_{u}^\ast|_x, \,  \norm{ \Phi_t(x,\xi)}_{\varphi_t(x)}\leq Ce^{-c|t|}\norm{\xi}_x\quad \text{if } t\leq 0.
\end{align}
\end{enumerate}
\end{lemma}
\begin{proof}
In the proof of \cite[Lemma~1.10]{DG16} the compactness of the ambient manifold $\mc{M}$ considered there is not needed; it is only the compactness of $\mathcal K$ that is required (in order to be able to introduce a Riemannian metric on a neighborhood of $\mathcal K$ such that the injectivity radius of this metric is uniformly bounded from below). Since also the attraction property of the trapped set \cite[Lemma~1.3]{DG16} generalizes in our setting to Corollary \ref{cor:attract} below\footnote{Corollary \ref{cor:attract} neither uses any of the statements of Lemma \ref{lem:Gammaextension} nor any results relying on Lemma \ref{lem:Gammaextension}, so there is no  circular argument.}, the arguments of the proof of \cite[Lemma~1.10]{DG16} remain valid and establish the existence of continuous $\d\varphi_t$-invariant vector subbundles $E_s^\ast$ and $E_u^*$  satisfying (1)--(2). Finally, the uniqueness, which is not discussed in \cite[Lemma~1.10]{DG16}, follows from (1) and (2) since for each $x\in {\mathcal K}_\pm$ a dual vector $\xi\in T^\ast_x \mc{M}$ cannot be both exponentially contracted and expanded under the action of $\Phi_t$ by $\d\varphi_t$ as $t\to +\infty$, and similarly for $t\to -\infty$.
\end{proof}
\begin{rem}If $\varphi_t$ happens to be globally uniformly hyperbolic, i.e., there is a splitting
\bq
T\mc{M}=E_0\oplus E_s\oplus E_u\label{eq:splittingTNglobal}
\eq
 into continuous  $\d\varphi_t$-invariant  subbundles such that $E_0=\R X$ and for some $C,c>0$ one has for all $x\in \mc{M}$
\[
\norm{\d\varphi_t(x)v}_{\varphi_t(x)}\leq Ce^{-c|t|}\norm{v}_x\quad \text{if}\;\begin{cases} t\geq 0,\; v\in E_s|_x\\
t\leq 0,\; v\in E_u|_x\end{cases}
\]
with respect to some Riemannian norm $\norm{\cdot}$ on $\mc{M}$, then Assumption \ref{ass:A3} is satisfied by restricting these bundles on $\mc{K}$. In particular the splitting $E_0^*\oplus E_s^*\oplus E_u^*$ of $T^*\mc{M}$  dual to \eqref{eq:splittingTNglobal} (with $E_s^*(E_s\oplus E_0)=0$, etc) are such that $E_s^*$ and $E_u^*$ restrict respectively on $\mc{K}_+$ and $\mc{K}_-$ to the vector bundles obtained in Lemma \ref{lem:Gammaextension}, as follows from the uniqueness statement made there.
\end{rem}

\subsection{Main results}

We can now formulate the first main result of this section.

\begin{definition}\label{def:ResXj}
Assume \ref{ass:A0}--\ref{ass:A3} and 
consider the subbundles $E_{\pm}^\ast\subset T^\ast \mc{M}|_{{\mathcal K}_\pm}$ in Lemma \ref{lem:Gammaextension}. Then, for each $j\in \N$ and $\lambda_0\in \C$, the space of \emph{generalized Ruelle resonant states} of $\mathbf X$ of order $j$ at $\lambda_0$ is
\bq
\mathrm{Res}_{\mathbf X}^{j}(\lambda_0):=\{u\in \D'(\mc{M};\E)\,|\, \supp u\subset {\mathcal K}_-,\mathrm{WF}(u)\subset E_u^\ast, (\mathbf X + \lambda_0)^j u=0 \}\label{eq:defResXj}
\eq
where ${\rm WF}(u)$ denotes the wave-front set of $u$.
We call $\mathrm{Res}_{\mathbf X}^{1}(\lambda_0)$
the space  of \emph{Ruelle resonant states} of $\mathbf X$ at $\lambda_0$.

If $\mc{C}\subset \M$ is a relatively compact open set, then for each $j\in \N$ and $\lambda_0\in \C$, the space of \emph{generalized Ruelle resonant states} of $\mathbf X$ \emph{on $\mc{C}$} of order $j$ at $\lambda_0$ is
\bq
\mathrm{Res}_{\mathbf X,\mc{C}}^{j}(\lambda_0):=\{u\in \D'(\mc{C};\E)\,|\, \supp u\subset {\mathcal K}_-(\mc{C}),\mathrm{WF}(u)\subset E_u^\ast|_{\mc{C}}, (\mathbf X + \lambda_0)^j u=0 \}\label{eq:defResXjC}
\eq
Again, we call $\mathrm{Res}_{\mathbf X,\mc{C}}^{1}(\lambda_0)$
the space  of \emph{Ruelle resonant states} of $\mathbf X$ \emph{on $\mc{C}$} at $\lambda_0$.
\end{definition}

Let us denote by $\mc{D}_c'(\mc{M};\mc{E}
)$  the space of $\mc{E}$-valued distributions with compact support in $\mc{M}$.
\begin{theo}\label{thm:resmain1}
Suppose that Assumptions \emph{\ref{ass:A0}--\ref{ass:A3}} hold and that \ref{ass:B1} or \ref{ass:B2} hold. Then the resolvent $R_{{\bf X}}(\lambda)
:C_c^\infty(\mc{M};\E)\to  \mc{D}'(\mc{M};\E)$ from \eqref{eq:globalresolvent} continues meromorphically to $\lambda \in \C$.
For all $\la_0\in \C$, the following properties hold true: 
\begin{enumerate}
\item For each compactly supported distribution $u\in \mc{D}'_c(\mc{M};\E)$ with ${\rm WF}(u)\cap E_s^*=\emptyset$,   $R_{{\bf X}}(\lambda)u$ is a well-defined meromorphic family of distributions in $\mc{D}'(\mc{M})$ with wave-front set
\begin{equation}\label{WFsetRu} 
{\rm WF}(R_{{\bf X}}(\lambda)u)\subset E_u^*\cup \{ \Phi_t (x,\xi)\,|\, (x,\xi)\in {\rm WF}(u), t\geq 0 \}.
\end{equation} 
\item There is an integer $J(\lambda_0)\in \N_0$ such that the spaces $\mathrm{Res}_{\mathbf X}^{j}(\lambda_0)$ are non-zero and finite-dimensional for all $j\leq J(\lambda_0)$ and $\mathrm{Res}_{\mathbf X}^{j}(\lambda_0)=\mathrm{Res}_{\mathbf X}^{J(\lambda_0)}(\lambda_0)$ \color{black} for all $j> J(\lambda_0)$. 
\item The resolvent \eqref{eq:globalresolvent} can be written as
\bq\label{eq:claimThm2}
R_{{\bf X}}(\lambda)
 = R_H(\lambda) + \sum_{j=1}^{J(\lambda_0)} \frac{(-1)^{j-1}(\mathbf X + \lambda_0)^{j-1} \Pi_{\lambda_0}}{(\lambda-\lambda_0)^j},
\eq
where $R_H(\lambda):\CT(\mc{M};\E)\to \D'(\mc{M};\E)$ is holomorphic near $\lambda_0$ and 
\[\Pi_{\lambda_0}: \CT(\mc{M};\E)\to \D'(\mc{M};\E)\] 
is a finite rank operator whose image is given by $\mathrm{Res}_{\mathbf X}^{J(\lambda_0)}(\lambda_0)$. In particular, the range of the residue of  $R_{{\bf X}}(\lambda)
$ at $\lambda_0$ is equal to $\mathrm{Res}_{\mathbf X}^{J(\lambda_0)}(\lambda_0)$.
\item Let $K_{\Pi_{\lambda_0}}\in \D'(\mc{M}\times \mc{M})$ be the Schwartz kernel of the operator $\Pi_{\lambda_0}$.
 Then 
\bq\label{eq:wavefrontpropX}
\supp K_{\Pi_{\lambda_0}}\subset {\mathcal K}_-\times {\mathcal K}_+,\qquad \mathrm{WF}'(\Pi_{\lambda_0})\subset E_u^\ast\times E_s^\ast. 
\eq
where $\mathrm{WF}'(\Pi_{\lambda_0}):=\{(x,\xi,y,-\eta)\,|\, (x,\xi,y,\eta)\in \mathrm{WF}(K_{\Pi_{\lambda_0}})\}\subset T^\ast(\mc{M}\times \mc{M})$.
\item The operator $\Pi_{\lambda_0}$ commutes with $\mathbf X$ in the sense that $\mathbf X\circ\Pi_{\lambda_0}=\Pi_{\lambda_0}\circ\mathbf X$, where $\mathbf X$ acts on $\mc{D}'(\mc{M};\E)$  and on $\CT(\mc{M};\E)$ on the left- and right-hand side, respectively\color{black}.
\item The operator $\Pi_{\lambda_0}$ is a projection in the following sense: the support and wavefront set properties \eqref{eq:wavefrontpropX} allow one to extend the operator $\Pi_{\lambda_0}$ from $\CT(\mc{M};\E)$ to $\mathrm{Ran}(\Pi_{\lambda_0})=\mathrm{Res}_{\mathbf X}^{J(\lambda_0)}(\lambda_0)$, so that it can be composed with itself, and then it satisfies
\bq
\Pi_{\lambda_0}\circ \Pi_{\lambda_0} = \Pi_{\lambda_0}.\label{eq:projectionX}
\eq 
\end{enumerate}
If Assumptions \ref{ass:B1} and \ref{ass:B2} are removed, the same results hold for the resolvents $R_{{\bf X}}^{\mc{C}}(\lambda)
$ for any relatively compact open set $\mc{C}\subset \mc{M}$ such that $\mc{U}\subset \mc{C}$, where $\mc{U}$ is the set in Assumption  \ref{ass:A2}, by replacing $\mc{M}$ by $\mc{C}$.
\end{theo}
The proof of Theorem \ref{thm:resmain1} is given in Section \ref{sec:proof1}. 
\begin{definition}The poles of the meromorphic operator family $R_{{\bf X}}(\lambda)
$ (resp.\ $R_{{\bf X}}^{\mc{C}}(\lambda)
$) of Theorem \ref{thm:resmain1} are called the \emph{Ruelle resonances} of $\mathbf X$ (resp.\ of $\mathbf{X}$ on $\mc{C}$).
\end{definition}

\subsection{Open systems embedded in $\mc{M}$}\label{sec:analysisopensys}

Here we recall the setting and the main results of \cite{DG16}, which will be used in the proof of Theorem \ref{thm:resmain1}\color{black}. Let $\overline{\mathcal U}\subset \mc{M}$ be a compact smooth manifold with  boundary of the same dimension as $\M$. Denote by $\mathcal U$ its interior and by $\partial \mathcal U$ its boundary. The following definition of strict convexity is a restatement of Definition \ref{def dynamical convexity} in the present more abstract setting.
\begin{definition}\label{def:strictlyconvex}
The boundary $\partial \mathcal U$ is \emph{strictly convex} if there is a boundary defining function $\rho:\overline{\mathcal U}\to \R$ (i.e., $\rho\in \Cinft(\overline{\mathcal U})$ with $\rho>0$ on $\mathcal U$, $\rho=0$ on $\partial \mathcal U$, and $d\rho\neq 0$ on $\partial \mathcal U$) such that $X^2 \rho(x) <0$ for all $x\in \partial \mathcal U$ with $X \rho(x) =0$. 
\end{definition}
If $\overline{\mathcal U}$ contains $\mathcal K$ and $\partial \mathcal U$ is strictly convex (which is in particular the case if Assumption \ref{ass:A2} holds)\color{black}, then $\mathcal K\subset \mathcal U$ (see  \cite[Lemma 1.2]{DG16}) and it immediately follows that
\bq
\mathcal K(\mathcal U)=\mathcal K(\overline{\mathcal U})=\mathcal K.\label{eq:KofU}
\eq
Moreover, if Assumptions \ref{ass:A0} and \ref{ass:A1} hold, then the convexity properties \eqref{eq:Uconvex} and Corollary \ref{cor:attract} imply
\bq
\mathcal K_\pm(\mathcal U)=\mathcal K_\pm \cap \mathcal U,\qquad \mathcal K_\pm(\overline{\mathcal U})=\mathcal K_\pm \cap \overline{\mathcal U}.\label{eq:KKpmofU}
\eq
From now on, suppose that $\overline{\mathcal U}$ has all the properties from Assumption \ref{ass:A2} and that also Assumptions  \ref{ass:A0}, \ref{ass:A1} and \ref{ass:A3} hold. Let $ 1_{\mathcal U}:\mc{M}\to \R$ be the characteristic function of $\mathcal U$.  The restricted resolvent
\[
R_{\bf X}^{\, \mathcal U}(\lambda)=1_{\mathcal U}R_{{\bf X}}(\lambda)
1_{\mathcal U}:\CT(\mathcal U;\E)\to \mc{D}'(\mathcal U;\E)
\]
does not depend on the dynamics of $\varphi_t$ outside $\mathcal U$  thanks to \eqref{eq:Uconvex}. Let us now summarize the relevant results on $R_{\bf X}^{\, \mathcal U}(\lambda)$ from \cite{DG16}, where for \eqref{eq:RanPi} we implicitly use \eqref{eq:KKpmofU}:
\begin{prop}[{\cite[Thm.~1, Thm.~2, Lemma 1.10]{DG16}}]\label{prop:extendedDGres}The operator family $R_{\bf X}^{\, \mathcal U}(\lambda)$ continues meromorphically to $\C$. Moreover:
\begin{enumerate}
\item Each pole $\lambda_0\in \C$ of $R_{\bf X}^{\, \mathcal U}(\lambda)$ is of finite order $J(\lambda_0)\in \N$ with finite rank residue, in the following sense: we can write
\[
R_{\bf X}^{\, \mathcal U}(\lambda) = R^{\,\mathcal U}_H(\lambda) + \sum_{j=1}^{J(\lambda_0)} \frac{(-1)^{j-1}(\mathbf X + \lambda_0)^{j-1} \Pi^{\mathcal U}_{\lambda_0}}{(\lambda-\lambda_0)^j},
\] 
where the operator family $R^{\, \mathcal U}_H(\lambda):\CT(\mathcal U;\E)\to \D'(\mathcal U;\E)$ is holomorphic near $\lambda_0$ and $\Pi^{\mathcal U}_{\lambda_0}:\CT(\mathcal U;\E)\to \D'(\mathcal U;\E)$ is a finite rank operator. Moreover, if $u\in \mc{D}_c'(\mc{U})$ satisfies ${\rm WF}(u)\cap E_s^*=\emptyset$, then 
$R_{\bf X}^{\, \mathcal U}(\lambda)u$ is a well-defined meromorphic family of distributions on $\mc{U}$ and 
\begin{equation}\label{WFRU}
{\rm WF}(R_{\bf X}^{\, \mathcal U}(\lambda)u)\subset E_u^*\cup \{\Phi_t(x,\xi)\,|\, (x,\xi)\in {\rm WF}(u), t\geq 0\}.
\end{equation}
\item The image of the finite rank operator $\Pi^{\mathcal U}_{\lambda_0}$ is given by 
\bq
\mathrm{Ran}(\Pi^{\mathcal U}_{\lambda_0})=\mathrm{Res}_{\mathbf X,\mathcal U}^{J(\lambda_0)}(\lambda_0),\label{eq:RanPi}
\eq 
where the space on the right-hand side was defined in \eqref{eq:defResXjC}.
\item Let $K_{\Pi^{\mathcal U}_{\lambda_0}}\in \D'(\mathcal U\times \mathcal U)$ be the Schwartz kernel of  $\Pi^{\mathcal U}_{\lambda_0}$ and let $$\mathrm{WF}'(\Pi^{\mathcal U}_{\lambda_0}):=\{(x,\xi,y,-\eta)\,|\, (x,\xi,y,\eta)\in \mathrm{WF}(K_{\Pi^{\mathcal U}_{\lambda_0}})\}\subset T^\ast(\mathcal U\times \mathcal U)=T^\ast\mathcal U\oplus T^\ast\mathcal U.$$ Then one has 
\bq
\supp K_{\Pi^{\mathcal U}_{\lambda_0}}\subset {\mathcal K}_-(\overline{\mathcal U})\times {\mathcal K}_+(\overline{\mathcal U}),\qquad \mathrm{WF}'(\Pi^{\mathcal U}_{\lambda_0})\subset E_u^\ast\times E_s^\ast. \label{eq:wavefrontpropY}
\eq
\item The operator $\Pi^{\mathcal U}_{\lambda_0}$ commutes with $\mathbf X$ in the sense that $\mathbf X\circ\Pi^{\mathcal U}_{\lambda_0}=\Pi^{\mathcal U}_{\lambda_0}\circ\mathbf X$, where $\mathbf X$ acts on $\mc{D}'(\mathcal U;\E)$ in the former case and on $\CT(\mathcal U;\E)$ in the latter case.
\item The operator $\Pi^{\mathcal U}_{\lambda_0}$ is a projection in the following sense: the support and wavefront set properties \eqref{eq:wavefrontpropY} allow for extending the operator $\Pi^{\mathcal U}_{\lambda_0}$ from $\CT(\mathcal U;\E)$ to $\mathrm{Ran}(\Pi^{\mathcal U}_{\lambda_0})$, so that it can be composed with itself, and then it satisfies
\bq \label{eq:projection}
\Pi^{\mathcal U}_{\lambda_0}\circ \Pi^{\mathcal U}_{\lambda_0} = \Pi^{\mathcal U}_{\lambda_0}.
\eq 
\end{enumerate} 
\end{prop}
Proposition \ref{prop:extendedDGres} implies that a complex number $\lambda_0\in \C$ is a Ruelle resonance of $\mathbf X$ on $\mathcal U$ iff $\mathrm{Res}_{\mathbf X,\mathcal U}^{1}(\lambda_0)\neq \{0\}$.

\subsection{Global dynamical preliminaries}\label{sec:globaldyn}

Here we collect a few elementary observations concerning the global dynamics of the flow $\varphi_t$ on $\mc{M}$. 
\begin{lemma}[{C.f.~\cite[Lemma~1.3]{DG16}, \cite[Eq.~(2.2)]{Guillarmou-Mazzucchelli-Tzou}}]\label{lem:attract}Suppose that Assumption  \emph{\ref{ass:A0}} holds. Then,  for every compact subset $\mathcal C\subset \mc{M}$ and every neighborhood $\mathcal V\subset \mc{M}$ of $\mathcal K$, there is a $T\geq 0$ such that
\bqn
\varphi_{\pm t}({\mathcal K}_\pm(\mathcal C))\subset \mathcal V\qquad \forall \,t\in[T,\infty). 
\eqn
\end{lemma}
\begin{proof} First, let us assume that $\mathcal K\subset \mathcal C$. Then $\mathcal K =\mathcal K(\mathcal C)$ and the proof of \cite[Lemma 1.3]{DG16} works verbatim, with $\mathcal C$ playing the role of $\overline{\mathcal U}$ in their notation. Now, if $\mathcal K\not\subset \mathcal C$, for each $\mc{V}$ we know   that $\varphi_{\pm t}({\mathcal K}_\pm(\mathcal C\cup \mathcal K))\subset \mc{V}$ if $t>T$ for some $T\geq 0$,  and since ${\mathcal K}_\pm(\mathcal C)\subset{\mathcal K}_\pm(\mathcal C\cup \mathcal K)$ we get the result.  
\end{proof}
\begin{lemma}\label{lem:Gammaintext}Suppose that Assumptions  \emph{\ref{ass:A0}}  and  \emph{\ref{ass:A1}} hold and let $\mathcal C_0\subset \mc{M}$ be a compact set. Then there is a compact set $\mathcal C\subset \mc{M}$ such that
\[
\mathcal C_0\cap {\mathcal K}_\pm\subset {\mathcal K}_\pm(\mathcal C).
\] 
\end{lemma}
\begin{proof}If we obtain the statement with different sets $\mathcal C_+,\mathcal C_-$ in the two cases ``$+$'' and ``$-$'', we can simply put $\mathcal C:=\mathcal C_+\cup\mathcal C_-$ to get what is claimed. Without loss of generality, we consider only the ``$+$'' case, the other one is analogous.  Suppose that the statement does not hold.  Then there is  a sequence of points $x_n\in \mathcal C_0\cap {\mathcal K}_+$, and a sequence $t_n\to +\infty$ such that $\varphi_{t_n}(x_n)\to \infty$ in $\mc{M}$ as $n\to \infty$ (where the latter means that $\varphi_{t_n}(x_n)$ leaves any compact set). Since $x_n\in \mc{K}_+$,  there exists a compact set $C_n$ containing $x_n$ such that $x_n\in \mc{K}_+(C_n)$.
Let $U\subset \mc{M}$ be a precompact neighborhood of $\mathcal C_0\cup \mathcal K$. Then $\partial U\subset \mc{M}$ is a compact set disjoint from $\mathcal K$.  For each $n\in \N$ we can apply Lemma \ref{lem:attract} to get a $T_n\geq 0$ such that $\varphi_{t}(x_n)\in U$ for all $t \geq T_n$. Note that, since $\mathcal K_+$ is $\varphi_t$-invariant, we have $\varphi_{t}(x_n)\in \mathcal K_+$ for all $t\in \R$, in particular we get $\varphi_{t}(x_n)\in U\cap \mathcal K_+$ for all $t \geq T_n$.  On the other hand, because $\varphi_{t_n}(x_n)\to \infty$ in $\mc{M}$, there is an $N$ such that $\varphi_{t_n}(x_n)\in \mc{M}\setminus \overline U$ for all $n\geq N$. Thus, for each $n\geq N$ we have $T_n>t_n$ and can define $t^\mathrm{out}_n:=\min\{t\geq 0 \,| \varphi_t(x_n)\in \partial U\cap\mathcal K_+\}\in (0,t_n)$, the parameter at which $\varphi_t(x_n)$ exits $\overline U$ for the first time, and  $t^\mathrm{in}_n:=\max\{t\geq 0 \,| \varphi_t(x_n)\in \partial U\cap\mathcal K_+\}\in (t_n,T_n)$,  the time of the last re-entry into $\overline{U}$. We now consider $t_n':=t^\mathrm{in}_n-t^\mathrm{out}_n>0$, where $n\geq N$. Suppose that $t_n'$ is bounded as $n\to \infty$, say $t_n'\leq T$. Then $\varphi_{t_n}(x_n)\in \cup_{t\in [0,T]}\varphi_t(\overline U)$ for all $n\geq N$, contradicting the fact that  $\varphi_{t_n}(x_n)\to \infty$ in $\mc{M}$.  Thus the sequence $t_n'$, and consequently also the set $\{t\geq 0\,|\,\varphi_t(\partial U\cap \mathcal K_+)\cap (\partial U\cap \mathcal K_+)\not=\emptyset\}$, is unbounded. Assumption \ref{ass:A1} applied with $\mc{C}_1=\mc{C}_2=\partial U\cap \mathcal K_+$ now implies that $\emptyset\neq\partial U\cap \mathcal K_-\cap \mathcal{K}_+=\partial U\cap \mathcal K$, a contradiction. 
\end{proof}
\begin{cor}\label{cor:attract}
Suppose that Assumptions  \emph{\ref{ass:A0}}  and  \emph{\ref{ass:A1}} hold. Then, for every compact subset $\mathcal C_0\subset \mc{M}$ and every neighborhood $\mathcal V\subset \mc{M}$ of $\mathcal K$, there is a $T\geq 0$ such that
\bqn
\varphi_{\pm t}(\mathcal C_0\cap{\mathcal K}_\pm)\subset \mathcal V\qquad \forall \,t\in[T,\infty).
\eqn
\end{cor}
\begin{proof}
We first apply Lemma \ref{lem:Gammaintext} and then Lemma \ref{lem:attract}, this gives the result.
\end{proof}
\begin{cor}\label{cor:smallGammaC} Suppose that Assumptions  \emph{\ref{ass:A0}}  and  \emph{\ref{ass:A1}} hold. Let $\mathcal V_0\subset \mc{M}$ be a neighborhood of $\mathcal K$. Then there is an open set $\mathcal V\subset \mathcal V_0$ containing $\mathcal K$ such that
\[
\mathcal V\cap {\mathcal K}_\pm\subset {\mathcal K}_\pm(\mathcal V_0).
\] 
\end{cor}
\begin{proof}Let $U\subset \mc{M}$ be an arbitrary precompact open set containing $\mathcal K$. Then Corollary \ref{cor:attract} gives us a $T\geq 0$ such that $\varphi_{\pm t}(\overline{U}\cap {\mathcal K}_\pm)\subset \mathcal V_0$ for all $t\geq T$. Now $\mathcal V:=\varphi_{\pm T}(U)$ works: since ${\mathcal K}_\pm$ is $\varphi_t$-invariant, we have for all $t\geq 0$
\[
\varphi_{\pm t}(\mathcal V\cap {\mathcal K}_\pm)=\varphi_{\pm t}(\varphi_{\pm T}(U)\cap {\mathcal K}_\pm)=\varphi_{\pm (T+t)}(U)\cap {\mathcal K}_\pm=\varphi_{\pm (T+t)}(U\cap {\mathcal K}_\pm)\subset \mathcal V_0.\qedhere
\]
\end{proof}

\begin{lemma}\label{lem:KNW} If Assumptions \emph{\ref{ass:A0}} and \emph{\ref{ass:A1}} hold, then the global forward/backward trapped sets ${\mathcal K}_+$ and ${\mathcal K}_-$ are closed in $\M$ and one has $\mathcal{NW}\subset \mathcal K$. 
\end{lemma}
\begin{proof}
We begin with the first claim.  Without loss of generality, we focus on ${\mathcal K}_+$. Let $x_n\to x$ be a convergent sequence in $\M$ with $x_n\in  {\mathcal K}_+$. Let $U\subset \M$ be a compact neighborhood of $x$. Then by Lemma \ref{lem:Gammaintext} there is a compact set $\mathcal C\subset \M$ such that $U\cap {\mathcal K}_+\subset {\mathcal K}_+(\mathcal C)$. Now, except possibly for finitely many indices, we have $x_n\in U\cap {\mathcal K}_+$, and because ${\mathcal K}_+(\mathcal C)$ is compact, a subsequence $x_{n_k}$ converges to some $x_0\in {\mathcal K}_+(\mathcal C)\subset {\mathcal K}_+$. However, since the whole sequence already converges to $x$ in $\M$, we must have $x_0=x$, proving that $x\in {\mathcal K}_+$. 

For the second claim, let $x\in \mathcal{NW}$ and let $\mathcal C\subset \mc{M}$ be a compact neighborhood of $x$. Then $\{t\geq 0\,|\,\varphi_t(\mathcal C)\cap \mathcal C\neq \emptyset\}$ is an unbounded set, so  Assumption \ref{ass:A1} implies that $\mathcal C\cap {\mathcal K}_+\neq \emptyset$ and $\mathcal C\cap {\mathcal K}_-\neq \emptyset$. Since $\mathcal C$ was an arbitrary compact neighborhood of $x$ and $\mc{M}$ is a locally compact Hausdorff space, this implies that $x$ lies in the closures of both ${\mathcal K}_+$ and ${\mathcal K}_-$ in $\mc{M}$. These closures agree with ${\mathcal K}_+$ and ${\mathcal K}_-$ by the first claim, hence $x\in \mathcal K$.
\end{proof}

\subsection{Proof of Theorem \ref{thm:resmain1}}\label{sec:proof1}We are now ready to prove Theorem \ref{thm:resmain1}. The proof is inspired  by \cite[proof of Prop.~3.4]{Guillarmou-Mazzucchelli-Tzou} and \cite[proof of Lemma 1.9]{DG16}.  The strategy is as follows. 
We choose cutoff functions $\chi_\pm$ supported, and equal to $1$, near $\mc{K}_\pm$. First, we write 
$(\chi_-+(1-\chi_-))R_{\bf X}(\la)(\chi_++(1-\chi_+))$ and prove that, by splitting the sum into $4$ terms, all terms admit holomorphic extensions as operators mapping $C_c^\infty(\mc{M};\mc{E})\to \mc{D}'(\mc{M};\mc{E})$, except for the term $\chi_-R_{{\bf X}}(\lambda)\chi_+$. 
We then analyze this term: if $\chi_{\mc{U}}$ is supported in $\mc{U}$ and equal to $1$ on $\mc{K}$, we may use \cite{DG16} to deduce that 
$\chi_{\mc{U}}\chi_-R_{{\bf X}}(\lambda)\chi_+\chi_{\mc{U}}$ is meromorphic. We are left with  terms of the form $(1-\chi_{\mc{U}})\chi_-R_{{\bf X}}(\lambda)\chi_+$ and $\chi_-R_{{\bf X}}(\lambda)\chi_+(1-\chi_{\mc{U}})$ which can be reduced again to the setting of \cite{DG16} by using 
that small compact sets supported near points in $\mc{K}^\pm$ can be transported to $\mc{U}$ in finite time $T$,
and combining this with the fact  that the propagator $e^{T{\bf X}}$ commutes with $R_{{\bf X}}(\la)$.
\color{black}
\begin{proof}[Proof of Theorem \ref{thm:resmain1}] 
Let $\overline{\mathcal U}\subset \M$ be a strictly convex neighborhood of $\mc{K}$ as in assumption \ref{ass:A2}.
Let  $\chi_\pm \in C^\infty(\mc{M},[0,1])$ be cutoff functions equal to $1$ in a small neighborhood of 
$\mc{K}_\pm$ but with ${\rm supp}(\chi_\pm)\cap \mc{K}_\mp\subset \mc{U}$. First, we claim that for any such $\chi_\pm$
\begin{equation}\label{Rholo} 
(1-\chi_-)R_{{\bf X}}(\lambda)
 \textrm{ and } R_{{\bf X}}(\lambda)
(1-\chi_+) \quad \textrm{ are holomorphic in }\la\in \C
\end{equation}
as continuous maps $\mc{D}_c'(\mc{M};\mc{E})\to \mc{D}'(\mc{M};\mc{E})$. Moreover, for $u\in \mc{D}'_c(\mc{M};\mc{E})$
\begin{equation}\label{WF1Prop}
\begin{gathered} 
{\rm WF}((1-\chi_-)R_{{\bf X}}(\lambda)u)\subset \{ \Phi_t (x,\xi)\in T^*\mc{M}\,|\, (x,\xi)\in {\rm WF}(u), t\geq 0 \}\\
{\rm WF}(R_{{\bf X}}(\lambda)(1-\chi_+)u)\subset \{ \Phi_t (x,\xi)\in T^*\mc{M}\,|\, (x,\xi)\in {\rm WF}(u), t\geq 0 \}.
\end{gathered} 
\end{equation} 
Indeed, for each 
$u_1\in \mc{D}_c'(\mc{M};\mc{E}
)$ and $u_2\in C_c^\infty(\mc{M};\mc{E}
)$, we have by assumption \ref{ass:A1} that there is $T$ depending on $\supp(\chi_-)$, $\supp(u_1)$ and  $\supp(u_2)$ such that for all $t>T$, 
\[ \varphi_t(\supp(u_1))\cap \supp(1-\chi_-)\cap \supp u_2=\emptyset.\]
This implies that 
\[ \cjg (1-\chi_-)R_{{\bf X}}(\lambda)
u_1,u_2\cjd= \int_{\mc{M}} \int_0^T e^{t\la}\cjg e^{-t{\bf X}}(u_1(x)), u_2(x)\cjd_{\mc{E}}(1-\chi_-(x)){\rm dv}_g(x)dt\]
admits a holomorphic extension to $\la\in \C$, and the extension is continuous in $u_1,u_2$.
The same argument works  for $R_{{\bf X}}(\lambda)(1-\chi_+)$. The wavefront set property \eqref{WF1Prop} is a direct consequence of the transformation law of wave-front set under pull-back by diffeomorphisms.
Moreover, since $\mc{K}_-$ is invariant by the flow $\varphi_t$ and $E_u^*$ is invariant by the symplectic lift $(\d\varphi_t^{-1})^\top$ on $T^*\mc{M}$,  it follows that 
\begin{equation}\label{support_prop}
\begin{gathered}
 {\rm supp}(u)\subset \mc{K}_-\Longrightarrow \supp(R_{{\bf X}}(\lambda)
(1-\chi_+)u)\subset \mc{K}_-\\
{\rm WF}(u)\subset E_u^* \Longrightarrow {\rm WF}(R_{{\bf X}}(\lambda)
(1-\chi_+)u)\subset E_u^*.
 \end{gathered}
\end{equation}
As we have shown \eqref{Rholo}, to prove the meromorphic extension of 
\begin{equation}\label{Rwithcutoff}
R_{{\bf X}}(\lambda)
=  \chi_-R_{{\bf X}}(\lambda)
\chi_+ + (1-\chi_-)R_{{\bf X}}(\lambda)
+\chi_-R_{{\bf X}}(\lambda)
(1-\chi_+),
\end{equation} 
it suffices to deal with $\chi_-R_{{\bf X}}(\lambda)
\chi_+$. We will now assume that the cutoffs $\chi_\pm \in C^\infty(\mc{M},[0,1])$ have further properties in addition to being equal to $1$ in a small neighborhood of $\mc{K}_\pm$ and satisfying ${\rm supp}(\chi_\pm)\cap \mc{K}_\mp\subset \mc{U}$. The functions $\chi_\pm$ were already chosen above, so logically we don't choose them here but replace them by better ones. Let $I_\pm \subset \pl \mc{U}$ be a small neighborhood of $\mc{K}_\pm \cap \pl \mc{U}$ 
such that $\bbar{I}_\pm \cap \mc{K}_\mp=\emptyset$; notice that thanks to the convexity \eqref{eq:Uconvex} of $\overline{\mathcal U}$ and the strict convexity of $\partial \mathcal U$ the vector field $X$ is pointing outside $\mc{U}$ in $\bbar{I}_-$ and inside $\mc{U}$ in $\bbar{I}_+$ if $\bbar{I}_\pm$ are chosen close enough to $\mc{K}_\pm$. 
We then choose $\chi_\pm$ such that $X\chi_\pm=0$ in a neighborhood of $\mc{M}\setminus \mc{U}$ and 
\[\supp \chi_\pm\cap (\mc{M}\setminus \mc{U})\subset \cup_{t\geq 0}\varphi_{\mp t}(I_\pm).\]
This can be done by first choosing $\chi_\pm|_{I_\pm}\in C_c^\infty(I_\pm)$ to be equal to $1$ near $\mc{K}_\pm \cap I_{\pm}$ and 
extending it by $\chi_\pm(\varphi_{\mp t}(x))=\chi_\pm(x)$ for all $t\geq 0$ and $x\in I_\pm$, where for $t>0$ we have $\varphi_{\mp t}(x)\in \M\setminus \overline {\mathcal U}$ by the convexity of $\overline{\mathcal U}$ and the fact that $X$ points outwards (resp.\ inwards) on $I_-$ (resp.\ $I_+$).

Let $\tilde{\chi}_{\mc{U}},\chi_{\mc{U}}\in C_c^\infty(\mc{U})$ be equal to $1$ in a neighborhood $\mc{U}_0\subset \mc{U}$ of $\mc{K}$ and 
$\tilde{\chi}_{\mc{U}}=1$ on  $\supp(\chi_{\mc{U}})\cup \supp(X\chi_\pm)$. 
For each $u\in C_c^\infty(\mc{M})$ we can write for ${\rm Re}(\la)>C_p$
\[ \begin{split}
\chi_-R_{{\bf X}}(\lambda)
(\chi_+u)=& \chi_-\tilde{\chi}_{\mc{U}}R_{{\bf X}}(\lambda)
(\chi_{\mc{U}}\chi_+u) +\chi_-(1-\tilde{\chi}_{\mc{U}})R_{{\bf X}}(\lambda)
(\chi_\mc{U}\chi_+u) \\
& +\chi_-R_{{\bf X}}(\lambda)
((1-\chi_{\mc{U}})\chi_+u)
\\
=:& R_1(\la)u+R_2(\la)u+R_3(\la)u.
\end{split}\]
By  Proposition \ref{prop:extendedDGres}, $R_1(\la):C_c^\infty(\mc{M};\mc{E})\to \mc{D}'(\mc{M};\mc{E}
)$ admits a meromorphic extension as 
continuous operator, whose principal part at each pole $\la_0$ is a finite rank operator with Schwartz kernel supported in ${\mathcal K}_-\times {\mathcal K}_+$ and wave-front set (in the sense of \eqref{eq:wavefrontpropX}) contained in $E_u^\ast\times E_s^\ast$.
A distribution $w_1(\la_0)$  in the range of the polar part at $\la_0$ is supported in $\mc{K}_-$ and has wavefront set ${\rm WF}(w_1(\la_0))\subset E_u^*$. Moreover, $R_1(\la)u$ is well defined and extends meromorphically provided that $u\in \mc{D}_c'(\mc{M};\mc{E})$ has wave-front set ${\rm WF}(u)\cap E_s^*=\emptyset$, and by \eqref{WFRU}
\[{\rm WF}(R_1(\la)u)\subset E_u^*\cup \{ \Phi_t (x,\xi)\in T^*\mc{M}\,|\, (x,\xi)\in {\rm WF}(u), t\geq 0 \}.\]
We now analyze $R_2(\la)$. First, observe that  $(1-\tilde{\chi}_{\mc{U}})X(\chi_-)=0$. Second, we also have that $\chi_-X(\tilde{\chi}_{\mc{U}})$ is supported in $\mc{U}$ and disjoint from $\mc{K}_+$. Let $\chi_1:=\chi_-(1-\tilde{\chi}_{\mc{U}})$ and write for ${\rm Re}(\la)>C_p$
\[ ({\bf X}+\la)\chi_1R_{{\bf X}}(\lambda)
\chi_{\mc{U}}=(X\chi_1)R_{{\bf X}}(\lambda)
\chi_{\mc{U}}=-\chi_-X(\tilde{\chi}_{\mc{U}})R_{{\bf X}}(\lambda)\chi_{\mc{U}}\]
thus 
\begin{equation}\label{split_term_in_2}
\begin{split}
\chi_1R_{{\bf X}}(\lambda)\chi_{\mc{U}} =&- R_{{\bf X}}(\lambda)
\chi_-(X\tilde{\chi}_{\mc{U}})1_{\mc{U}}R_{\bf X}^{\, \mc{U}}(\la)\chi_{\mc{U}}.
\end{split}
\end{equation}
Since $\chi_-X(\tilde{\chi}_{\mc{U}})$ has compact support in $\mathcal U$ not intersecting $\mc{K}_+$, $R_{{\bf X}}(\lambda)
\chi_-X\tilde{\chi}_{\mc{U}}$ admits a holomorphic extension to $\C$ as a continuous operator $\mc{D}'_c(\mc{M};\mc{E})\to \mc{D}'(\mc{M};\mc{E})$ by \eqref{Rholo}, thus 
\begin{equation}\label{doubleR}
R_{{\bf X}}(\lambda)
\chi_-(X\tilde{\chi}_{\mc{U}})1_{\mc{U}}R_{\bf X}^{\, \mc{U}}(\la)\chi_{\mc{U}} 
\end{equation} 
admits a meromorphic extension to $\la\in \C$ as a continuous operator $C_c^\infty(\mc{M};\mc{E})\to \mc{D}'(\mc{M};\mc{E})$ whose polar part at each pole is a finite rank operator. It also acts on the set of $u\in \mc{D}'_c(\mc{M};\mc{E})$ with ${\rm WF}(u)\cap E_s^*=\emptyset$ since $1_{\mc{U}}R_{\bf X}^{\, \mc{U}}(\la)\chi_{\mc{U}}$ does, and 
\[ {\rm WF}(R_2(\la)u) \subset E_u^*\cup \{ \Phi_t (x,\xi)\in T^*\mc{M}\,|\, (x,\xi)\in {\rm WF}(u), t\geq 0 \}\]
by \eqref{WFRU} and  \eqref{WF1Prop}.
Moreover, using \eqref{support_prop}, the polar part of \eqref{doubleR} at a pole $\la_0$ has Schwartz kernel with support  in $\mc{K}_-\times \mc{K}_+$ and wavefront set contained in $E_u^*\times E_s^*$. 

Let us now deal with $R_3(\la)$. For $\chi_3\in C_c^\infty(\mc{M})$, it suffices to prove that 
$\chi_-R_{{\bf X}}(\lambda)(1-\chi_{\mc{U}})\chi_+\chi_3$ has a meromorphic extension as continuous operators 
$C_c^\infty(\mc{M};\mc{E})\to \mc{D}'(\mc{M};\mc{E})$.
One has $\supp((1-\chi_{\mc{U}})\chi_+\chi_3)\cap \mc{K}_-=\emptyset$ and by Corollary \ref{cor:attract} 
there is $T>0$ such that
\[\varphi_t(\supp((1-\chi_{\mc{U}})\chi_+\chi_3)\cap \mc{K}_+)\subset \mc{U}_0 \quad  \forall t\geq T.\]
This implies that if $\tilde{\chi}_+\in C_c^\infty(\mc{M})$ is supported close enough to $\supp((1-\chi_{\mc{U}})\chi_+\chi_3)\cap \mc{K}_+$ 
and $\tilde{\chi}_+=\chi_+$ near that set, there is $T'>0$ such that 
\[\varphi_{T'}(\supp((1-\chi_{\mc{U}})\tilde{\chi}_+\chi_3)\cap \mc{K}_+)\subset \mc{U}_0.\]
Since $\tilde{\chi}_+-\chi_+=0$ near $\mc{K}_+$, the argument above shows  that $R_{{\bf X}}(\lambda)
(1-\chi_{\mc{U}})(\chi_+-\tilde{\chi}_+)\chi_3$ is holomorphic in $\la\in \C$ as continuous operator $C_c^\infty(\mc{M};\mc{E})\to \mc{D}'(\mc{M};\mc{E}
)$, thus to prove the extension of $R_3(\la)$, it suffices to replace $\chi_+$ by $\tilde{\chi}_+$ in the definition of $R_3(\la)$, which we do now.
Using that $e^{T{\bf X}}R_{{\bf X}}(\lambda)
e^{-T{\bf X}}=R_{{\bf X}}(\lambda)
$ we get as operators  for ${\rm Re}(\la)>C_p$ \color{black} 
\begin{equation}\label{R2}
\begin{split}
R_3(\la)\chi_3=& \chi_-e^{T{\bf X}}R_{{\bf X}}(\lambda)
e^{-T{\bf X}}(1-\chi_{\mc{U}})\tilde{\chi}_+\chi_3 \\
=&e^{T{\bf X}} (\chi_-\circ \varphi_{-T})\chi_{\mc{U}}R_{\bf X}^{\, \mc{U}}(\la)1_{\mc{U}}(((1-\chi_{\mc{U}})\tilde{\chi}_+\chi_3)\circ \varphi_{-T})
e^{-T{\bf X}}\\
& + e^{T{\bf X}} (\chi_-\circ \varphi_{-T})(1-\chi_{\mc{U}})R_{{\bf X}}(\lambda)
1_{\mc{U}}(((1-\chi_{\mc{U}})\tilde{\chi}_+\chi_3)\circ \varphi_{-T})e^{-T{\bf X}}.
\end{split}\end{equation}
The term involving $R_{\bf X}^{\, \mc{U}}(\la)$ is meromorphic as continuous operator $C_c^\infty(\mc{M};\mc{E})\to \mc{D}'(\mc{M};\mc{E}
)$ (by Proposition \ref{prop:extendedDGres}), the polar part at a pole $\la_0$ is a finite rank operator, whose Schwartz kernel is
supported in $\mc{K}_-\times \mc{K}_+$ and wavefront set contained in $E_u^*\times E_s^*$.
Up to compositions by $e^{\pm T{\bf X}}$ in the right and left hand side, the second term in \eqref{R2} has the same form as $R_2(\la)$ and can be dealt with the same argument as for $R_2(\la)$. Thus $R_3(\la)\chi_3$ is also meromorphic with polar part at a pole $\la_0$ a finite rank operator, whose Schwartz kernel has 
support contained in $\mc{K}_-\times \mc{K}_+$ and wavefront set contained in $E_u^*\times E_s^*$ (since $e^{\pm T{\bf X}}$ preserves distributions with wave front sets in $E_\pm^*$). Moreover, as above, $R_3(\la)\chi_3u$ also makes sense as a distribution if $u\in \mc{D}'_c(\mc{M};\mc{E})$ satisfies 
${\rm WF}(u)\cap E_s^*=\emptyset$, and  by \eqref{WFRU} and  \eqref{WF1Prop}
\[ {\rm WF}(R_3(\la)u) \subset E_u^*\cup \{ \Phi_t (x,\xi)\in T^*\mc{M}\,|\, (x,\xi)\in {\rm WF}(u), t\geq 0 \}.\]
 Recalling \eqref{Rwithcutoff}, this finishes the proof of the meromorphic extension of $R_{{\bf X}}(\lambda)$. \color{black} 

Let $\la_0$ be a pole of $R_{{\bf X}}(\lambda)
$, which is thus also a pole of $R_{\bf X}^{\, \mc{U}}(\la)$, we write 
\[ R_{{\bf X}}(\lambda)
=\sum_{j=1}^{J(\la_0)} \frac{\Pi_{\la_0}^{j}}{(\la-\la_0)^j}+R_H(\la)  
\]
where $\Pi_{\la_0}^{j}:C_c^\infty(\mc{M};\mc{E}
)\to \mc{D}'(\mc{M};\mc{E}
)$ are continuous and $R_H(\la):C_c^\infty(\mc{M};\mc{E}
)\to \mc{D}'(\mc{M};\mc{E}
)$ is holomorphic near $\la_0$. First, notice that the support and wave-front set properties \eqref{eq:wavefrontpropX} of the Schwartz kernel of $\Pi_{\la_0}^{j}$  follow from the discussion above; this shows item (3) in Theorem \ref{thm:resmain1}. 
We have seen that for each $\chi\in C_c^\infty(\mc{M})$,  the operators $\Pi_{\la_0}^{j}\chi$ have finite rank as maps 
$C_c^\infty(\mc{M};\mc{E})\to \mc{D}'(\mc{M};\mc{E})$. We shall now prove that $\Pi^{j}_{\la_0}$ have finite rank as operators mapping continuously $C_c^\infty(\mc{M};\mc{E}
)$ to $\mc{D}'(\mc{M};\mc{E}
)$.
Expanding the identity $({\bf X}+\la)R_{{\bf X}}(\lambda)
={\rm Id}$ at $\la_0$ gives that  $\Pi_{\la_0}^{j}=(-1)^j({\bf X}+\la_0)^{j-1}\Pi_{\la_0}$ with $\Pi_{\la_0}=-\Pi^1_{\la_0}$; this shows \eqref{eq:claimThm2}.
If $u$ belongs to the range of $\Pi_{\la_0}$, then we have seen that $({\bf X}+\la_0)^{J(\la_0)}u=0$, $\supp(u)\subset \mc{K}_-$ and ${\rm WF}(u)\subset E_u^*$, we thus have
\begin{equation}\label{inclusion_range}
{\rm Ran}(\Pi_{\la_0})\subset  {\rm Res}_{\bf X}^{J(\la_0)}(\la_0).
\end{equation} 
We will show below that this inclusion is an equality. By Proposition \ref{prop:extendedDGres} item (2), for $u\in  {\rm Res}_{\bf X}^{J(\la_0)}(\la_0)$,  the restriction $u|_{\mc{U}}\in {\rm Res}_{{\bf X},\mc{U}}^{J(\la_0)}(\la_0)$ is in the range of the residue $\Pi_{\la_0}^{\mc{U}}$ of $R_{\bf X}^{\, \mc{U}}(\la)$ at $\la_0$.
Let us then prove that 
\begin{equation}\label{restriction_iso}
 u\in {\rm Ran}(\Pi_{\la_0})\mapsto u|_{\mc{U}}\in  {\rm Ran}(\Pi^{\mc{U}}_{\la_0})
 \end{equation}
is an isomorphism. First, we show that 
\begin{equation}\label{restriction_iso2} 
u\in {\rm Res}_{\bf X}^{J(\la_0)}(\la_0)\mapsto u|_{\mc{U}}\in {\rm Ran}(\Pi^{\mc{U}}_{\la_0})={\rm Res}_{{\bf X}}^{J(\la_0)}(\la_0)
\end{equation}
is injective, which will in turn show that \eqref{restriction_iso} is injective.
If $u|_{\mc{U}}=0$, then $u$ vanishes in a neighborhood of $\mc{K}$. Since we already know that $u$ vanishes outside $\mc{K}_-$, it suffices to consider $\chi\in C_c^\infty(\mc{M};\mc{E}
)$ supported in a small neighborhood $B(y)$ of a point 
$y\in \mc{K}_-\cap \mc{U}^c$ and by Corollary \ref{cor:attract} there is $T>0$ such that $\varphi_{-T}(B(y))\subset \mc{U}$, which implies that $\supp(e^{T{\bf X}^*}\chi)\subset \mc{U}$. Since $({\bf X}+\la_0)^{J(\la_0)}u=0$,  we have for each $t\in \R$
\[e^{t({\bf X}+\la_0)}u=\sum_{j=0}^{J(\la_0)-1}\frac{t^j({\bf X}+\la_0)^ju}{j!} \]
and thus 
\[ \cjg u,\chi\cjd = \cjg e^{-T{\bf X}}u, e^{T{\bf X}^*}\chi\cjd=e^{T\la_0}\sum_{j=0}^{J(\la_0)-1}\frac{(-T)^j}{j!}\cjg ({\bf X}+\la_0)^ju,e^{T{\bf X}^*}\chi\cjd=0.\]
This proves that \eqref{restriction_iso2} is injective, and hence so is \eqref{restriction_iso}. The map \eqref{restriction_iso} is also surjective since for $u$ in the range of $\Pi_{\la_0}^{\mc{U}}$
there is $\chi\in C_c^\infty(\mc{U};\mc{E}
)$ such that $u=\Pi_{\la_0}^{\mc{U}}\chi$, and by using \eqref{Rwithcutoff} we have 
$(R_{{\bf X}}(\lambda)
\chi)|_{\mc{U}}=R_{\bf X}^{\, \mc{U}}(\la)\chi+ h(\la)$ where $h(\la)$ is holomorphic near $\la_0$: the residue of $(R_{{\bf X}}(\lambda)
\chi)|_{\mc{U}}$ at $\la_0$  is thus equal to $u$. We have thus proven that $\Pi_{\la_0}^{j}:C_c^\infty(\mc{M};\mc{E})\to \mc{D}'(\mc{M};\mc{E})$ are finite rank operators and that the range of $\Pi_{\la_0}$ is isomorphic to the range of $\Pi_{\la_0}^{\mc{U}}$ by restriction to $\mc{U}$.
In view of \eqref{inclusion_range}, to prove that the residue of  $R_{{\bf X}}(\lambda)
$ at $\lambda_0$ is equal to $\mathrm{Res}_{\mathbf X}^{J(\lambda_0)}(\lambda_0)$, it suffices to prove that for $u\in \mc{D}'(\mc{M};\mc{E}
)$ such that $({\bf X}+\la_0)^{J(\la_0)}u=0$ with $\supp(u)\subset \mc{K}_-$ and ${\rm WF}(u)\subset E_u^*$, then $u\in {\rm Ran}(\Pi_{\la_0})$. First, we take the restriction $u|_{\mc{U}}$, which by Proposition \ref{prop:extendedDGres} (item 2)) and the relation $\mathcal K_-(\mathcal U)=\mathcal K_- \cap \mathcal U$ observed in \eqref{eq:KKpmofU} is in the range of $\Pi_{\la_0}^{\mc{U}}$. By the discussion above, it can thus be written as $u|_{\mc{U}}=({\rm Res}_{\la=\la_0}(R_{{\bf X}}(\lambda)
\chi))|_{\mc{U}}=(\Pi_{\la_0}\chi)|_{\mc{U}}$ for some $\chi\in C_c^\infty(\mc{U})$.
Since \eqref{restriction_iso2} is injective, which implies that $u=\Pi_{\la_0}\chi$. We have proved items (1) and (2) of Theorem \ref{thm:resmain1}. 

Item (4) follows directly from the fact that ${\bf X}R_{\bf X}(\la)=R_{\bf X}(\la){\bf X}$.

We finally prove item (5). First, observe that $\Pi_{\la_0}\circ \Pi_{\la_0}$ is well defined since $\mc{K}_+\cap \mc{K}_-$ is compact, and 
$E_u^*\cap E^*_s$ is contained in the zero section of $T^*\mc{M}$ (thus H\"ormander's criterion for multiplication of distributions is satisfied).
Let $u\in C_c^\infty(\mc{M};\mc{E})$ be supported in $\mc{U}$. First,  
since  $\supp(\Pi_{\la_0}u)\subset \mc{K}_-$ and $\Pi_{\la_0}f=0$ for all $f\in C_c^\infty(\mc{M};\mc{E})$ 
such that $\supp(f)\cap \mc{K}_+=\emptyset$, as $\mc{K}_+\cap \mc{K}_-\subset \mc{U}$ 
we deduce that $\Pi_{\la_0} (\Pi_{\la_0}u)=\Pi_{\la_0} (1_{\mc{U}}\Pi_{\la_0}u)=\Pi_{\la_0}(1_{\mc{U}}\Pi_{\la_0}^{\mc{U}}u)$. Thus, using \eqref{eq:projection}, we obtain
\[(\Pi_{\la_0} \Pi_{\la_0}u)|_{\mc{U}}=\Pi_{\la_0}^{\mc{U}}(\Pi_{\la_0}^{\mc{U}}u)=\Pi_{\la_0}^{\mc{U}}u=(\Pi_{\la_0}u)|_{\mc{U}}.\]
Since \eqref{restriction_iso} is an isomorphism, this implies that $\Pi_{\la_0} (\Pi_{\la_0}u)=\Pi_{\la_0}u$. The same argument applies for $u\in C_c^\infty(\mc{M};\mc{E})$ with $\supp(u)\cap \mc{K}_+=\emptyset$. It remains to consider the case where $\supp(u)\cap \mc{K}_+=\supp(u)\cap \mc{K}_+\cap \mc{U}^c\not=\emptyset$. By Corollary \ref{cor:attract}, there is $T>0$ such that $\varphi_{T}(\supp(u)\cap \mc{K}_+)\subset \mc{U}_0$. But since $e^{T{\bf X}}\Pi_{\la_0}e^{-T{\bf X}}=\Pi_{\la_0}$, we have for $u_T:=\chi_{\mc{U}}e^{-T{\bf X}}u$ 
\[\Pi_{\la_0}u=e^{T{\bf X}}\Pi_{\la_0}u_T, \]
\[\Pi_{\la_0}(\Pi_{\la_0}u)=\Pi_{\la_0}(e^{T{\bf X}}\Pi_{\la_0}u_T)=e^{T{\bf X}}\Pi_{\la_0}(\Pi_{\la_0}u_T)=e^{T{\bf X}}(\Pi_{\la_0}u_T)=\Pi_{\la_0}u,\]
where we used that $\supp(u_T)\subset \mc{U}$ and the result just proved above in that case. This shows item (5).

Finally, given a relatively compact open set $\mc{C}\subset \M$ containing $\mc{U}$, the proof can be applied verbatim to $R^{\mc{C}}_{{\bf X}}(\lambda)
$ by replacing $\mc{M}$ with $\mc{C}$ and $\mc{K}_\pm$ with $\mc{K}_\pm(\mc{C})$.
\end{proof}

\subsection{Resolvent of $X$ on Lorentzian quasi-Fuchsian spaces}\label{sec:applicationofsection}

We start with a proof that the global resolvent of the geodesic vector field acting on functions is meromorphic.
\begin{lemma}
Let $\Gamma\subset {\rm Isom}(\AdS_3)$ be a quasi-Fuchsian group acting on $\mc{O}_+(\Lambda_\Gamma)\subset \AdS_3$ and let $\mc{M}=\Gamma\backslash \widetilde{\mc{M}}$ with $\widetilde{\mc{M}}\subset T^1\AdS_3$ the set defined in Proposition \ref{prop - discontinuity domain unit spacelike tangent bundle}. Then the spacelike geodesic flow $\varphi_t$ and vector field $X$ on $\mc{M}$ satisfy the Assumptions  \ref{ass:A0}, \ref{ass:A1}, \ref{ass:A2} and \ref{ass:A3}. For ${\bf X}=X$ acting on smooth functions on $\mc{M}$, the operator  ${\bf X}$ satisfies also the Assumptions  \ref{ass:B1} and \ref{ass:B2}.
Thus the results of Theorem  \ref{thm:resmain1} apply.
\end{lemma} 
\begin{proof} Assumption \ref{ass:A0} is proved in Lemma \ref{lem - non-wandering set for the geodesic flow}, Assumption \ref{ass:A1} is proved in Corollary \ref{cor:ass1},  Assumption \ref{ass:A2} is proved in Corollary \ref{coro dynamically convex from DGK} and Assumption \ref{ass:A3} is proved in Proposition \ref{prop - spacelike geodesic flow axiom A}. The $1$-form $\alpha$ of Section \ref{sec:spacelike_flow} is such that $\mu:=\alpha\wedge d\alpha^{2}$ is an invariant volume form on $T^1\AdS_3$ which descends to $\widetilde{\mc{M}}/\Gamma$, thus $C_\mu=1$ in \eqref{eq:Cmu}. This shows that \ref{ass:B1} and \ref{ass:B2} hold by taking the canonical metric on the trivial bundle $\mc{M}\times \C$.
\end{proof}

We now describe the resonant states of the geodesic vector field acting on functions. As in the general setting above, a pole $\la_0$ of $R_X(\la)$ is called a \emph{Ruelle resonance}. The associated \emph{Ruelle resonant states} are the elements $u\in {\rm Ran}(\Pi^X_{\la_0})$ satisfying $(X+\la_0)u=0$, where
\[ \Pi^X_{\la_0}:={\rm Res}_{\la_0}R_X(\la)\]
denotes the associated spectral projection.

\begin{lemma}\label{uniqueres}
Let $h_{\rm top}>0$ be the topological entropy of $\varphi_1$ on $\mc{K}$. There is $\eps>0$ such that the resolvent $R_X(\la)$ has a unique pole  in the region ${\rm Re}(\la)>h_{\rm top}-2-\eps$, namely at $\la_0=h_{\rm top}-2$. It is a simple pole and the residue 
$\Pi^X_{\la_0}$ has rank $1$.
\end{lemma}
\begin{proof}
Let us consider the zeta function $Z_X(\la)$ defined by its logarithmic derivative
\begin{equation}\label{defZ'/Z}
\frac{Z'_{X}(\la)}{Z_{X}(\la)}=\sum_{\gamma \in \mc{P}}\sum_{m=1}^\infty \frac{\ell_\gamma e^{-m\la \ell_\gamma }}{|\det(1-P_{\gamma}^m)|}
\end{equation}
with $\mc{P}$ the set of primitive periodic orbits of the flow $\varphi_t$ of $X$ and $\ell_\gamma$ the period, $P_\gamma$ the linearized flow $\d\varphi_{\ell_\gamma}$ along $\gamma$. The unstable Jacobian 
$J^u(x)=-\pl_t \log \det(\d\varphi_t(x)|_{E_u(x)})$ of the flow is constant and given by $J^u=-2$ by \eqref{sumlapm} (or Lemma \ref{formes_vol_s/u_Gamma}).
We then obtain
\[ \begin{split}
|\det(1-P_{\gamma}^m)|=& |\det(\d\varphi_{m\ell_\gamma}|_{E_u})|\cdot |\det(1-\d\varphi_{m\ell_\gamma}|_{E_s})|\cdot |\det(1-\d\varphi_{m\ell_\gamma}^{-1}|_{E_u})|\\
 =& e^{2m\ell_\gamma}(1+\mc{O}(e^{-\nu m\ell_\gamma}))
\end{split}\] 
for some $\nu>0$. By \cite[Theorem E]{Delarue_Monclair_Sanders}, the counting function $N_{\rm pr}(T):=\#\{ \gamma \in \mc{P}\,|\, \ell_\gamma \leq T\}$ satisfies the asymptotic 
\[N_{\rm pr}(T)=\int_2^{e^{h_{\rm top}T}}\frac{du}{\log(u)}(1+\mc{O}(e^{(h_{\rm top}-\eps)T}))\] 
for some $\eps>0$. This implies that the counting function
$N(T):=\#\{ \gamma \textrm{ periodic orbit } \,|\, \ell_\gamma \leq T\}$ for the lengths of periodic orbits satisfies 
\begin{equation}\label{NTnon-primitive}
N(T)=N_{\rm pr}(T)+\mc{O}(e^{h_{\rm top}T/2}).
\end{equation}
We deduce from \eqref{NTnon-primitive} that there is $\eps>0$ such that for ${\rm Re}(\la)>h_{\rm top}-2$
\[\begin{split}
\frac{Z'_{X}(\la)}{Z_{X}(\la)} =& \int_1^\infty te^{-(\la+2)t}(1+\mc{O}(e^{-\eps t}))dN_{\rm pr}(t) +A(\la) \\
 =& \frac{e^{-(\la+2-h_{\rm top})}}{(\la+2-h_{\rm top})}  + B(\la)
\end{split}\]
where $A(\la), B(\la)$ are real analytic for $\la \in (h_{\rm top}-2-\eps,1)$. In particular, it follows that 
$\frac{Z'_{X}(\la)}{Z_{X}(\la)}$ is meromorphic in  $\la \in (h_{\rm top}-2-\eps,1)$ 
with only a simple pole at $\la=h_{\rm top}-2$ and residue equal to $1$. By \cite[Theorem 4]{DG16}, the poles of $R_X(\la)$ are exactly those of $\frac{Z'_{X}(\la)}{Z_{X}(\la)}$ and the residue at a pole $\la_0$ of this function is the rank of $\Pi^X_{\la_0}$. This completes the proof.
\end{proof}

\begin{definition}
We say that $u\in \mc{D}'(\mc{M})$ is a resonant state in the first band of resonances if 
$\la_0\in \C$ is a pole of $R_X(\la)$, $u\in {\rm Ran}(\Pi^X_{\la_0})$ and $U_-^Ru=U_-^Lu=0$.
\end{definition}
Using the arguments of \cite{dfg,GHW18a}, we can characterize the resonant states in the first band in terms of equivariant distributions supported on $\Lambda_\Gamma$.

\begin{lemma}\label{first_band}
A resonant state $u$ with resonance $\la_0$ belongs to the first band if and only if there exists a distribution $\omega\in \mc{D}'(\TT^2)$ supported in $\Lambda_\Gamma$ and $\Gamma$-equivariant in the sense that $\gamma^*\omega=N_\gamma^{-\la_0}\omega$ for all $\gamma\in \Gamma$, such that the lift $\tilde{u}$ of $u$ to the cover $\widetilde{\mc{M}}$ is given by 
\begin{equation}\label{tilderesonance}
 \tilde{u}=\mc{Q}^-_{\lambda_0}(\omega)=\Phi_-^{\la_0} \widetilde{B}_-^*\omega.
 \end{equation}
\end{lemma} 
\begin{proof} A resonant state $u$ has lift $\tilde{u}$ that is supported in the lift $\widetilde{\mc{K}}_-$ of the backward trapped set $\mc{K}_-$ to $\widetilde{\mc{M}}$ by the projection $\widetilde{\mc{M}}\to \mc{M}$, and this set is exactly $\widetilde{B}_-^{-1}(\Lambda_\Gamma)$. By \eqref{eq:IsoU_Gamma}, 
it has the form $\tilde{u}=\mc{Q}^-_{\lambda_0}(\omega)$ for some $\omega\in \mc{D}'(\TT^2)$ supported in $\Lambda_\Gamma$ if $U_-^Ru=U_-^Lu=0$. Moreover, $\omega$ is $\Gamma$-equivariant by Corollary \ref{cor:Qlambda}. Conversely, consider $\omega$ a $\Gamma$-equivariant distribution supported in $\Lambda_\Gamma$, then $\tilde{u}=\mc{Q}^-_{\lambda_0}(\omega)$ is supported in $\widetilde{\mc{K}}_-$, $\Gamma$-invariant, thus descends to a distribution $u$ on $\mc{M}$, supported in $\mc{K}_-$, 
and by Corollary \ref{cor:Qlambda} it satisfies 
\[(X+\la_0)u=0, \quad U_-^Ru=U_-^Lu=0.\]
By elliptic regularity and because $X,U_-^R,U_-^L$ are smooth vector fields, we have that  ${\rm WF}(u)\subset E_u^*$, and by Theorem \ref{thm:resmain1} it means that $u$ is a resonant state.
\end{proof}

In what follows, we shall prove that a subset of Ruelle resonances for $X$ are indeed killed by $U_-^R$ and $U_-^L$, i.e. they belong the first band.
For $\la_-^R,\la_-^L$ the two smooth functions defined in Lemma \ref{splitting_stable/instable} and such that $[X,U_-^{R/L}]=-\la_-^{R/L}$, we define for $z\in \mc{K}$
\[ \nu_-^{R}(z)=\liminf_{t\to +\infty}\frac{1}{t}\int_{-t}^0 \la_-^R(\varphi_s(z))ds>0, \quad  \nu_-^{L}(z)=\liminf_{t\to +\infty}\frac{1}{t}\int_{-t}^0 \la_-^{L}(\varphi_s(z))ds>0\]
and let 
\begin{equation}\label{eps_-R/L}
\eps^R_-:=\inf_{z\in \mc{K}} \nu_-^R(z)>0, \quad \eps^L_-:=\inf_{z\in \mc{K}} \nu_-^L(z)>0.
\end{equation} 
Let us now consider the Ruelle spectrum of the operator $X+\la_-^R$.  

\begin{lemma}\label{noresonance}
The resolvent $R_{X+\la_-^R}(\la)$ has no pole in the region ${\rm Re}(\la)>h_{\rm top}-2-\eps_-^R$ 
 and $R_{X+\la_-^L}(\la)$ has no pole in the region ${\rm Re}(\la)>h_{\rm top}-2-\eps_-^L$. 
\end{lemma}
\begin{proof} To analyse the poles of $R_{X+\la_-^R}(\la)$, we consider the dynamical zeta function $Z_{X+\la_-^R}(\la)$ as in \cite[Theorem 4]{DG16} defined by its logarithmic derivative:  
\begin{equation}\label{logdetZ} 
\frac{Z'_{X+\la_-^R}(\la)}{Z_{X+\la_-^R}(\la)}=\sum_{\gamma \in \mc{P}}\sum_{m=1}^\infty \frac{\ell_\gamma e^{-m(\la \ell_\gamma +\int_\gamma \la_-^R)}}{|\det(1-P_{\gamma}^m)|}
\end{equation}
with the same notations as in \eqref{defZ'/Z}. For a continuous potential $V\in C^0(\mc{K})$,  
denote by ${\rm P}(V)$ the topological pressure of $V$ with respect to the flow $\varphi_1$. 
The topological pressure can be expressed in terms of periodic orbits by the expression
\[ {\rm P}(V)=\lim_{T\to \infty}\frac{1}{T}\log \sum_{\gamma \in \mc{P},\ell_\gamma\leq T}e^{\int_{\gamma}V} \]
thus the series \eqref{logdetZ} converges in  the half-plane
\[ {\rm Re}(\la)>{\rm P}(J^u-\la_-^R)=\sup_{\mu} \Big(h_\mu -\int_{\mc{K}}(2+\la_-^R)d\mu\Big)\]
where $\mu$ are probability measures on $\mc{K}$ invariant by the flow $\varphi_1$. The right hand side is bounded above by $h_{\rm top}-2-\eps_-^R$, thus $Z_{X+\la_-^R}(\la)$ is analytic in ${\rm Re}(\la)>h_{\rm top}-2-\eps_-^R$. Since, by \cite[Theorem 4]{DG16}, the poles of $R_{X+\la_-^R}(\la)$ are given exactly by the poles of \eqref{logdetZ}, we deduce that $R_{X+\la_-^R}(\la)$ has no pole in ${\rm Re}(\la)>h_{\rm top}-2-\eps_-^R$. The same argument works with $R_{X+\la_-^L}(\la)$.
\end{proof}

\begin{cor}\label{noruelleres}
Let $u\in \mc{D}'(\mc{M})$ be a Ruelle resonant state for $X$ with Ruelle resonance $\la_0$, then 
\[ {\rm Re}(\la_0)>h_{\rm top}-2-\eps_-^R \Longrightarrow U_-^Ru=0, \quad  {\rm Re}(\la_0)>h_{\rm top}-2-\eps_-^L \Longrightarrow U_-^Lu=0\]
and such resonant states can be written on the cover $\widetilde{\mc{M}}$ by \eqref{tilderesonance} for some equivariant distributions $\omega\in \mc{D}'(\TT^2)$ supported on $\Lambda_\Gamma$.
In particular, the Ruelle resonant state associated to the leading resonance $\la_0=h_{\rm top}-2$ is killed by $U_-^L$ and $U_-^R$. 
\end{cor}
\begin{proof} Consider $U_-^Ru$ if $(X+\la_0)u=0$ with $u$ a resonant state. Clearly $U_-^Ru$ has the same support and wave-front set properties as $u$ since $U_-^R$ is smooth. Now we have, thanks to $[X,U_-^R]=-\la_-^RU_-^R$ 
\[ 0=U_-^R(X+\la_0)u=(X+\la_0+\la_-^R)U_-^Ru\]
and thus $U_-^Ru$ is a resonant state with resonance $\la_0$ for the operator $X+\la_-^R$. We can then use Lemma \ref{noresonance} to conclude, and the existence of an equivariant distribution $\omega$ supported on the limit set $\Lambda_\Gamma$ follows from Lemma \ref{first_band}.
\end{proof}

\section{Poincar\'e series}\label{Sec:Poincare}
In this section, we introduce the Poincar\'e series for a quasi-Fuchsian group $\Gamma\subset {\rm Isom}(\AdS_3)$ and prove that it extends meromorphically to the whole complex plane by using the resolvent of the spacelike geodesic flow.  

\subsection{Poincar\'e series}\label{sec:poincare_series}
First, for each pair $(x,y)\in \mc{O}_+(\Lambda_\Gamma)\times \mc{O}_+(\Lambda_\Gamma)$ we define 
$\Gamma_{x,y}\subset \Gamma$ by its complement: 
\begin{equation}\label{Gammaxy^c}
\Gamma_{x,y}^c:=\{ \gamma\in \Gamma \, |\, -q(x,\gamma y)>1 \}.
\end{equation} 
By Lemma \ref{lem - q<-1 up to moving by an element of the group}, the set $\Gamma_{x,y}$ is finite.
As in \cite{glorieux-monclair}, we define the exponent of the group $\Gamma$ by 
\begin{equation}\label{delta_Gamma} 
\delta_\Gamma:= \limsup_{T\to \infty}\frac{1}{T}\log\#\{ \gamma \in \Gamma_{{\rm e},{\rm e}}\,|\, d({\rm e},\gamma\cdot{\rm e})\leq T \}.
\end{equation}
It is shown in \cite{glorieux-monclair} that $\delta_\Gamma\leq 1$ with equality if and only if $\Gamma$ is Fuchsian. Moreover, one can replace $\Gamma_{{\rm e},{\rm e}}$ in \eqref{delta_Gamma} by $ \Gamma_{x,y}$ for any pair $(x,y)\in \mc{O}_+(\Lambda_\Gamma)\times \mc{O}_+(\Lambda_\Gamma)$. 

Associated to $\delta_\Gamma$, it is proved in \cite[Section 4]{glorieux-monclair} that there exists a measure $\omega_{\delta_\Gamma}$, called \emph{Patterson-Sullivan measure}, supported on $\Lambda_\Gamma$ and which satisfies the equivariance 
$\gamma^*\omega_{\delta_\Gamma}=N_\gamma^{2-\delta_\Gamma}\omega_{\delta_\Gamma}$ for all $\gamma\in \Gamma$. 
 It is uniquely characterized by these properties.\color{black}

\begin{definition}
Let $\Gamma\subset {\rm Isom}(\AdS_3)$ be a quasi-Fuchsian subgroup.  For each $x,y\in  \mc{O}_+(\Lambda_\Gamma)^2$, the Poincar\'e series of $\Gamma$ is given for ${\rm Re}(\la)>\delta_\Gamma$ by
\begin{equation}\label{defNla} 
\mc{D}_\la(x,y):= \sum_{\gamma\in \Gamma_{x,y}^c} e^{-\la d(x,\gamma y)}.
\end{equation}
\end{definition}
The series converges  for ${\rm Re}(\la)>\delta_\Gamma$ uniformly on compact sets $K_1\times K_2\subset \mc{O}_+(\Lambda_\Gamma)\times \mc{O}_+(\Lambda_\Gamma)$ by Lemma \ref{lem - q<-1 up to moving by an element of the group}. 
Moreover, 
one has for each $\gamma_0\in \Gamma$
\begin{equation}\label{D_invariance_by_Gamma} 
\mc{D}_\la(x,\gamma_0 y)= \sum_{\gamma\in \Gamma, -q(x,\gamma \gamma_0y)>1} e^{-\la d(x,\gamma\gamma_0 y)}= \sum_{\gamma'\in \Gamma_{x,y}^c} e^{-\la d(x,\gamma' y)}
\end{equation}
and similarly $ \mc{D}_\la(\gamma_0x,y)= \mc{D}_\la(x,y)$, which implies that $\mc{D}_\la$ descends to a bounded function on 
$M^2=(\Gamma\backslash \mc{O}_+(\Lambda_\Gamma))^2$, that we still denote by $\mc{D}_\la$.\\

We will write the Poincar\'e series in terms of the resolvent of the geodesic flow on differential forms, following the argument of Dang-Rivi\`ere \cite{Dang-Riviere}.  To this end we need to consider $\mc{L}_{X}: C^\infty(\mc{M};\Lambda T^*\mc{M} )\to C^\infty(\mc{M};\Lambda T^*\mc{M})$, the Lie derivative in the direction $X$, and we set 
\[ {\bf X}_k:= \mc{L}_{X}|_{\Lambda^k T^*\mc{M}}.\]
We will also denote by $[Y]\in \mc{D}'(\mc{M};\Lambda^3 T^*\mc{M} )$  the $2$-dimensional current of integration on $Y$ if $Y\subset \mc{M}$ is an oriented  $2$-dimensional \color{black} submanifold, viewed as a distribution acting on the space $C_c^\infty(\mc{M};\Lambda^2T^*\mc{M})$ of compactly supported $2$-forms on $\mc{M}$. We recall that it is considered as a distributional $3$-form since it can be approximated by smooth $3$-forms $(Y_n)_{n\in \N}$  of the form $Y_n=f_n\eta$, where $\eta\in C^\infty(\M;\Lambda^3T^*\mc{M})$ is such that  $\eta|_Y\wedge \mathrm{vol}_{Y}=\mathrm{vol}_{\M}|_Y$, 
where $\mathrm{vol}_{Y}$ and $\mathrm{vol}_{\mc{M}}$ are the chosen volume forms on $Y$ and $\M$, respectively, and $f_n\to \delta_{Y}$ in $\D'(\M)$ as $n\to \infty$\color{black}: for $u\in C_c^\infty(\mc{M};\Lambda^2T^*\mc{M})$ 
\[ \cjg [Y],u\cjd =\int_{Y}\iota_Y^*u=\lim_{n\to \infty} \int Y_n\wedge u\]
where $\iota_Y:Y\to M$ is the inclusion map. Here we use that $\M$ is oriented. On $T_x^1M$, we have the volume density $S_x$ induced by the Lorentzian metric $g_x$, and we choose the orientation so that $\omega_u>0$ on $T^1_xM$. In particular, by  \eqref{normalisationSxwu}, 
we have for $f\in C_c^\infty(T^1_xM)$
\begin{equation}\label{dSxomega}
 \int_{T^1_xM} f dS_x= 8\int_{T^1_xM}f\omega_u.
 \end{equation}

We first prove the following result on the Poincar\'e series:
\begin{prop}\label{prop:Poincare_series}
The Poincar\'e series $\mc{D}_\la(x,y)$ extends meromorphically to $\lambda \in \C$ as a map in $C^0(K_1\times K_2)$ for any 
compact sets $K_1\times K_2\subset M\times M$, with polar part 
at any pole $\la_0$ being the Schwartz kernel of a finite rank operator. Moreover,  for $x\not=y$ there is $\chi \in C_c^\infty(\mc{M})$ such that
\[ \mc{D}_\la(x,y)= -\cjg 1, \chi [T_y^1M]\wedge R_{{\bf X}_2}(\la) \iota_{X}[T_x^1M]\chi\cjd, \]
while for $x=y$, 
\[\mc{D}_\la(x,x)= -\cjg 1, \chi [\varphi_{-\eps}(T_y^1M)]\wedge R_{{\bf X}_2}(\la) \iota_{X}[\varphi_\eps(T_x^1M)]\chi \cjd \] 
for any small $\eps>0$.
\end{prop}
\begin{proof}
First, we claim that for ${\rm Re}(\la)>1$, we can rewrite the Poincar\'e series in the form 
\bq
\mc{D}_\la(x,y)=\sum_{t>0, \varphi_t(T^1_xM)\cap T^1_yM\not=\emptyset}e^{-\la t}. \label{eq:claimps1}
\eq
Indeed, for $x,y\in M$, consider  lifts $\tilde{x},\tilde{y}\in \mc{O}_+(\Lambda_\Gamma)$ of these points; the set of spacelike geodesics with endpoints $x,y$ in $M$ is in one-to-one correspondence with the set of spacelike geodesic segments $\alpha_{\tilde{x},\gamma \tilde{y}}$ in $\mc{O}_+(\Lambda_\Gamma)$ with endpoints $\tilde{x},\gamma \tilde{y}$ for $\gamma\in \Gamma$ and the length of such a geodesic segment in $M$ is equal to the distance $d(\tilde{x},\gamma \tilde{y})$ in the cover $\mc{O}_+(\Lambda_\Gamma)$ (recall Lemma \ref{Lem:q_and_distance}). Thus, the claimed formula \eqref{eq:claimps1} holds. Next, consider the set 
\[V_{x,y}:=\{ v\in T^1_xM\,|\, \exists t>0, \varphi_t(x,v)\in T_yM\},\] 
which lifts to the set $V_{\tilde{x},\Gamma \tilde{y}}:=\bigcup_{\gamma\in \Gamma}V_{\tilde{x},\gamma \tilde{y}}$, where $V_{\tilde{x},\gamma \tilde{y}}$ consists of all \color{black} vectors $\tilde{v}_\gamma \in T^1_{\tilde{x}}\mc{O}_+(\Lambda_\Gamma)$  tangent to the geodesics $\alpha_{\tilde{x},\gamma\tilde{y}}$. 
The set $V_{\tilde{x},\Gamma \tilde{y}}$ is relatively compact in 
$T^1_{\tilde{x}}\mc{O}_+(\Lambda_\Gamma)$ since the accumulation points of the orbit $\{\gamma \tilde{y}\,|\, \gamma \in \Gamma\}$ in the closure $\bbar{\AdS}_3$ lie in $\Lambda_\Gamma$: \color{black} This means that the accumulation points of $V_{\tilde{x},\Gamma \tilde{y}}$ in $T^1_{\tilde{x}}\mc{O}_+(\Lambda_\Gamma)$ are contained in 
\[\{\tilde{v}\in T^1_{\tilde{x}}\mc{O}_+(\Lambda_\Gamma)\,|\, B_+(\tilde{x},\tilde{v})\in \Lambda_\Gamma\}\] 
which is a compact set (because $B_+(\tilde{x},\cdot)$ is a diffeomorphism onto an open subset of $\partial \AdS_3$ containing $\Lambda_\Gamma$, recall \eqref{Bpm_diffeo}). Therefore $V_{x,y}$ is a relatively compact set in $T_x^1M$. Similarly $V_{y,x}\subset T_y^1M$ is relatively compact.

We shall now use the method of Dang-Rivi\`ere \cite[Section 4]{Dang-Riviere}, and more particularly the version used by Chaubet \cite[Section 3.2]{Chaubet}. Let $\chi\in C_c^\infty(\mc{M})$ be equal to $1$ on a large enough compact set containing  $V_{x,y}\cup V_{y,x}$.
To prove the desired formula for $\mc{D}_\la(x,y)$ in ${\rm Re}(\la)>1$, it suffices to consider
$[T^1_xM]$ and $[T^1_yM]$ the $2$-dimensional currents of integration over $T^1_xM$ and $T^1_yM$ respectively (viewed as elements in the dual of the space  $C_c^\infty(T^1M;\Lambda^2(T^1M))$  of compactly supported
$2$ forms on $T^1M$), and prove the following property.
Choose  $\eps\geq 0$ small such that $\varphi_{-\eps}(T^1_yM)\cap \varphi_{\eps}(T_x^1M)=\emptyset$ (if $x\not=y$, we can take $\eps=0$), then it suffices to check that for each $z\in M$
\begin{equation}\label{trans1} 
N^*(T_z^1M)\cap E_s^* =\{0\} =N^*(T_z^1M)\cap E_u^*
\end{equation}
(where $N^*(T_z^1M)$ denotes the conormal bundle of $T_z^1M$ in $T^\ast(TM)$) and (recall \eqref{def_Phi_t})
\begin{equation}\label{trans2} 
 \{ \Phi_t(x,v,\xi_x,\xi_v)\in T^*\mc{M}\,|\, t\geq 0, \, \xi=(\xi_x,\xi_v)\in N^*(T_x^1M), \xi(X)=0 \}\cap N^*(T_y^1M)=\{0\}.
 \end{equation}
The condition \eqref{trans1} insures that ${\rm WF}(\iota_X[T_xM]\chi)\cap E_s^*=\emptyset$ so that $R_{{\bf X}_2}(\la)\iota_X[T_xM]\chi$ is well-defined, and \eqref{trans1} with \eqref{trans2} insure that the distributional pairing of $R_{{\bf X}_2}(\la)\iota_X[T_xM]\chi$ with $[T_yM]\chi$ makes sense  by using \eqref{WFsetRu}.

Let us prove \eqref{trans1} and \eqref{trans2}.
We can lift the flow to the cover $T^1\mc{O}_+(\Lambda_\Gamma)$ and it suffices to check the property there; we can thus assume that $x,y\in \mc{O}_+(\Lambda_\Gamma)$. Let us prove the first property in \eqref{trans1}, the second one being similar. It suffices to show that $E_u(x,v)+T_{(x,v)}(T_x^1M)+\R X(x,v)=T_{(x,v)}(T^1\mc{O}_+(\Lambda_\Gamma))$ (here we use that the lift of $E_s^*$ at a point of $\mc{K}_+$ is the annihilator of $E_s\oplus \R X$ in $T^1\mc{O}_+(\Lambda_\Gamma)$ at the lifted point).
This directly follows from (recall \eqref{def:Es})
\begin{align*}
& T_{(x,v)}(T^1_xM)= \{(0,\zeta_v)\in \R^4\times \R^4\,|\, q(\zeta_v,x)=q(\zeta_v,v)=0\}, \\
& E_s(x,v)= \{ (\zeta_v,-\zeta_v) \in \R^4\times \R^4\,|\, q(\zeta_v,x)=q(\zeta_v,v)=0\},\\
& \R X(x,v)=\{(\la v,\la x)\in \R^4\times \R^4\,|\, \la\in\R\}
\end{align*}
which are in direct sums. To prove \eqref{trans2}, it suffices to show that if $\varphi_t(T^1_x\mc{O}_+(\Lambda_\Gamma))\cap T^1_y\mc{O}_+(\Lambda_\Gamma)\not=\emptyset$ for some $t>0$, then the submanifolds $T^1_y\mc{O}_+(\Lambda_\Gamma)$ and $\{ \varphi_s(x,v) \,|\, s\in [0,t], v\in T^1_x\mc{O}_+(\Lambda_\Gamma)\}$ intersect transversely.
We have $\dim T^1_x\mc{O}_+(\Lambda_\Gamma)+\dim T^1_y\mc{O}_+(\Lambda_\Gamma)+1=\dim T^1\mc{O}_+(\Lambda_\Gamma)=5$ and  the geodesic vector field 
$X$ is transversal to $T^1_x\mc{O}_+(\Lambda_\Gamma)$ and to $T^1_y\mc{O}_+(\Lambda_\Gamma)$.  It suffices to check that 
\[ \d\varphi_t(T_{(x,v)}(T^1_x\mc{O}_+(\Lambda_\Gamma)))\cap T_{\varphi_t(x,v)}(T^1_y\mc{O}_+(\Lambda_\Gamma)) =0. \]
If $(0,\xi_v)\in T_{(x,v)}(T^1_x\mc{O}_+(\Lambda_\Gamma))$ we get
\[ \d\varphi_t(x,v)(0,\xi_v)=(\sinh(t)\xi_v, \cosh(t)\xi_v)\]
and it belongs to $T_{\varphi_t(x,v)}(T^1_y\mc{O}_+(\Lambda_\Gamma))$ only if $\sinh(t)\xi_v=0$, which implies that $\xi_v=0$ if $t>0$.

 By applying the same argument as in Proposition 3.4 of \cite{Chaubet} (see also Lemma 2.1 and Proposition 3.8 of \cite{Dang-Riviere})  and using the relative compactness of $V_{x,y}$, we obtain that 
 for ${\rm Re}(\la)>1$ and $\eps>0$ small enough,
\[\begin{split}
\mc{D}_\la(x,y)= &\sum_{t>2\eps, T_yM\cap \varphi_t(T_xM)\not=\emptyset}e^{-\la t}\\
=& -\cjg 1, \chi [\varphi_{-\eps}(T_y^1M)]\wedge R_{{\bf X}_2}(\la)\iota_{X}[\varphi_\eps(T_x^1M)]\chi\cjd 
\end{split}\]
Here the role of the function $\chi$ is simply for $u=\iota_{X}[\varphi_\eps(T_x^1M)]\chi$ to be a distribution with compact support  with ${\rm WF}(u)\cap E_s^*=\emptyset$ so that $R_{{\bf X}_2}(\la)u$ is well-defined as a meromorphic family (in $\la\in \C$) of distributions by Theorem \ref{thm:resmain1}.  
Now, due to the fact that the tangent space of $\varphi_\eps(T_x^1M)$ and of $\varphi_{-\eps}(T_y^1M)$ are uniformly transverse to the unstable and stable directions for $\eps\geq 0$ small and $x,y$ in a small ball, 
the currents $[\varphi_\eps(T_x^1M)]$ and $[\varphi_{-\eps}(T_y^1M)]$ are continuous functions of $x$ and $y$ with values in  spaces of distributions with wave-front sets contained in small conic neighborhoods of $T^*(T^1M)$ where the argument above can be done locally uniformly with respect to $x,y$.
This directly implies that $\mc{D}_\la(x,y)$ is continuous with respect to $x,y$ by the continuity of the operation of multiplication of two distributions with disjoint wave-front sets \cite[Thm.~8.2.13]{hoermanderI}. 
\end{proof}

\subsection{Poincar\'e series in terms of \texorpdfstring{$R_X(\la)$}{the flow resolvent}} \label{sec:seriesresolvent}
We next show that the Poincar\'e series can be expressed purely in terms of the resolvent $R_X(\la)$ of the flow 
$X$ acting on functions. For a distribution $u\in \mc{D}_{\rm fc}'(T^1M\times T^1M)$ whose support is compact in the fibers, we define the pushforward  of $u$  as the distribution on $M\times M$ 
\[ \cjg (\pi_*\otimes \pi_*)u,\chi \cjd:=  \cjg u, (\pi^*\otimes \pi^*)\chi\cjd,\qquad \chi\in \CT(M\times M),\color{black}\]
where $((\pi^*\otimes \pi^*)\chi)(x,v,x',v'):=\chi(x,x')$.
\begin{theo} \label{theo:Poincare_series}
The Poincar\'e series is given by the following expression  for all $\lambda\in \C$\color{black}:
\bq
 \mc{D}_\la= \frac{1}{2}(\pi_*\otimes \pi_*)R_X(\la+2)+ \frac{1}{2}(\pi_*\otimes \pi_*)R_X(\la-2)- (\pi_*\otimes \pi_*)R_X(\la),\label{eq:pushfwdformula}
\eq
where we identify the resolvents on the right-hand side with their distributional kernels. In particular, the sum of distributions on the right-hand side is a continuous function.

Furthermore, we have $h_{\rm top}=\delta_\Gamma$ and there exists $\eps>0$ such that the only pole of the Poincar\'e series in ${\rm Re}(\la)>\delta_\Gamma-\eps$ is the simple pole $\la=\delta_\Gamma$. The unique Ruelle resonant state of $X$, up to normalisation, at $\la=\delta_\Gamma-2$ has the form 
\begin{equation}\label{defudelta} 
u_{\delta_\Gamma}= \Phi_-^{\delta_\Gamma-2}B_-^*\omega_{\delta_\Gamma}
\end{equation}
where $\omega_{\delta_\Gamma}\in \mc{D}'(\TT^2)$ is the Patterson-Sullivan measure.
The residue of $\mc{D}_\la(x,y)$ at $\la=\delta_\Gamma$ has the form 
\[ {\rm Res}_{\la=\delta_\Gamma} \mc{D}_\la(x,y)=c_{\delta_\Gamma}f_{\delta_\Gamma}(x)f_{\delta_\Gamma}(y)\]
where $c_{\delta_\Gamma}>0$ and $f_{\delta_\Gamma}=\pi_*u_{\delta_\Gamma}$ is a smooth function  solving $(\Box_g+\delta_\Gamma(\delta_\Gamma-2))f_{\delta_\Gamma}=0$ in $M$ and whose lift to $\O_+(\Lambda_\Gamma)$ extends as a distribution $\tilde{f}_{\delta_\Gamma}$ to $\AdS_3$ and can be written as a Poisson transform of the form 
\[ \tilde{f}_{\delta_\Gamma}=\mc{P}^0_{\delta_\Gamma-2}(\omega_{\delta_\Gamma}).\]
with $\mc{P}^0_{\delta_\Gamma-2}(\omega_{\delta_\Gamma})>0$ in $\O_+(\Lambda_\Gamma)$, i.e. $f_{\delta_\Gamma}>0$ on $M$.
\end{theo}
\begin{proof}
We begin by studying the resolvent $R_{{\bf X}_2}(\la)$ more carefully. Recall from Lemma \ref{splitting_stable/instable} and using the Liouville form \eqref{def_alpha} that we have a global splitting 
\[T^*\mc{M}= \R\alpha \oplus E_u^*\oplus E_s^* , \quad \ker \iota_X=E_u^*\oplus E_s^* \]
and $E_u^*,E_s^*$ admit subsplittings 
\[ E_u^*=(E_u^R)^*\oplus (E_u^L)^*, \quad E_s^*=(E_s^R)^*\oplus (E_s^L)^*\] 
by taking $(E_u^R)^*,(E_u^L)^*$ the smooth line subbundles of $E_u^*$ annihilating respectively $E_s^{L}$ and $E_s^R$, and similarly for $(E_s^{R})^*, (E_s^L)^*$.
The bundle  $\Lambda_0^2 T^*\mc{M}$ of $2$-forms killed by $\iota_X$ splits smoothly into 
\[  \Lambda_0^2 T^*\mc{M}= (\Lambda^2 E_u^*) \oplus( \Lambda^2E_s^*) \oplus  (E_u^*\wedge E_s^*)= \R \omega_u \oplus \R \omega_s\oplus (E_u^*\wedge E_s^*) \]
where we have used Lemma \ref{formes_vol_s/u_Gamma}. The bundle $E_u^*\wedge E_s^*$ has rank $4$ and can be decomposed further using $(E_s^R)^*$, $(E_s^L)^*$, $(E_u^R)^*$, $(E_u^L)^*$. Furthermore, with Lemma \ref{formes_vol_s/u_Gamma}, \color{black} we see that the Lie derivative ${\bf X}_2=\mc{L}_X$ preserves this splitting and acts by 
\[ {\bf X}_2 (f_u\omega_u+f_s\omega_s+w)= ((X-2)f_u) \omega_u +  ((X+2)f_s) \omega_s+ {\bf X}_2w\] 
if $f_u,f_s\in C_c^\infty(\mc{M})$ and $w\in C_c^\infty(\mc{M};E_u^*\wedge E_s^*)$. The resolvent $R_{\bf X}(\la)=({\bf X}+\la)^{-1}$ then also splits in the same way for $\la\in \C$
\[ R_{{\bf X}_2}(\la)(f_u\omega_u+f_s\omega_s+w)=(R_{X}(\la-2)f_u)\omega_u+(R_{X}(\la+2)f_s)\omega_s+R_{{\bf X}_2}(\la)w\]
where $R_X(\la)$ is the resolvent of the geodesic vector field $X$ acting on functions. 
Let us analyse further the resolvent $R_{{\bf X}_2}(\la)$ on $E_u^*\wedge E_s^*$. Let $\beta_+^R,\beta_+^L,\beta_-^R,\beta_-^L$ be smooth $1$-forms belonging respectively to $(E_u^R)^*,(E_u^L)^*, (E_s^R)^*, (E_s^L)^*$
and satisfying  
\[ \beta_-^R(U_+^R)=\beta_-^L(U_+^L)=\beta_+^R(U_-^R)=\beta_+^L(U_-^L)=1\] 
 where $U_\pm^{R/L}$ are the vector fields of Lemma \ref{lem:URLpm}.  In  particular $\omega_u=\frac{1}{4}\beta_-^R\wedge \beta_-^L$ and $\omega_s=\frac{1}{4}\beta_+^R\wedge \beta_+^L$.
 We observe that for each $f\in C_c^\infty(\mc{M})$
 \[\begin{gathered}
  {\bf X}_2(f\beta_-^R\wedge \beta_+^R)= ((X+\la_+^R-\la_-^R)f)\beta_-^R\wedge \beta_+^R,  \quad 
  {\bf X}_2(f\beta_-^L\wedge \beta_+^L)= ((X+\la_+^L-\la_-^L)f)\beta_-^L\wedge \beta_+^L,\\
  {\bf X}_2(f\beta_-^R\wedge \beta_+^L)= (Xf)\beta_-^R\wedge \beta_+^L, \quad 
  {\bf X}_2(f\beta_-^L\wedge \beta_+^R)= (Xf)\beta_-^L\wedge \beta_+^R,
 \end{gathered}\]
 where we have used \eqref{sumlapm2} in the second line. We thus have for 
 $f\in C_c^\infty(\mc{M})$
\[\begin{gathered}
 R_{{\bf X}_2}(\la)(f\beta_-^R\wedge \beta_+^R)= (R_{X+\la_+^R-\la_-^R}(\la)f)\beta_-^R\wedge \beta_+^R, 
 \quad R_{{\bf X}_2}(\la)(f\beta_-^R\wedge \beta_+^L)= (R_{X}(\la)f)\beta_-^R\wedge \beta_+^L,\\
R_{{\bf X}_2}(\la)(f\beta_-^L\wedge \beta_+^L)= (R_{X+\la_+^L-\la_-^L}(\la)f)\beta_-^L\wedge \beta_+^L,\quad 
R_{{\bf X}_2}(\la)(f\beta_-^L\wedge \beta_+^R)= (R_{X}(\la)f)\beta_-^L\wedge \beta_+^R.
 \end{gathered}\]
Now, in Proposition \ref{prop:Poincare_series}, we see that $\mc{D}_\la(x,y)$ is expressed in terms of integration of 
$F_{\la}:=R_{{\bf X}_2}(\la) \iota_{X}[T_x^1M]$ in the fiber $T^1_yM$. Let us decompose as above 
\bq \iota_{X}[T_x^1M]=f_{-} \omega_u+f_{+}\omega_s+f^R_{-+}\beta_-^R\wedge \beta_+^R+f_{-+}^L\beta_-^L\wedge \beta_+^L+f_{-+}^{RL}\beta_-^R\wedge \beta_+^L+f_{-+}^{LR}\beta_-^L\wedge \beta_+^R\label{eq:decomp932023}
\eq
\[ F_{\la}= f_{-}(\la) \omega_u+f_{+}(\la)\omega_s+f^R_{-+}(\la)\beta_-^R\wedge \beta_+^R+f_{-+}^L(\la)\beta_-^L\wedge \beta_+^L+f_{-+}^{RL}(\la)\beta_-^R\wedge \beta_+^L+f_{-+}^{LR}(\la)\beta_-^L\wedge \beta_+^R\]
where $f_\pm, f_{-+}^{R/L}, f_{-+}^{RL},f_{-+}^{LR}$ are distributions and  
\[ \begin{gathered}
f_\pm(\la)=R_X(\la\pm 2)f_\pm, \quad f^{RL}_{-+}(\la)=R_{X}(\la)f^{RL}_{-+}, \quad  f^{LR}_{-+}(\la)=R_{X}(\la)f^{LR}_{-+},\\
f^R_{-+}(\la)=R_{X+\la_-^{L}-\la_{-}^R}(\la)f^R_{-+}, \quad f^L_{-+}(\la)=R_{X+\la_-^{R}-\la_{-}^L}(\la)f^L_{-+}.
\end{gathered}\]
It turns out that two components in this decomposition integrate to $0$ on $T^1_yM$:
by \eqref{kerdpi} and \eqref{wuws_on_basis} we observe that 
\begin{equation}\label{2form_on_vert_fiber}
\begin{gathered}
(\beta_-^R\wedge \beta_+^R)|_{\ker d\pi}=(\beta_+^L\wedge \beta_-^L)|_{\ker d\pi}=0\\
-\omega_s(U_+^R+U_-^R,U_+^L-U_-^L)=\omega_u(U_+^R+U_-^R,U_+^L-U_-^L)=1/4\\
(\beta_-^R\wedge \beta_+^L)(U_+^R+U_-^R,U_+^L-U_-^L)=(\beta_-^L\wedge \beta_+^R)(U_+^R+U_-^R,U_+^L-U_-^L)=-1
\end{gathered}\end{equation}
In particular $\omega_s=-\omega_u$ when restricted to $T^1_xM$. 
For $\chi \in C_c^\infty(\mc{M};\Lambda^3T^*\mc{M})$
\[  \int_{T_xM} \iota_X\chi  =\cjg  \iota_X[T_xM], \chi\cjd. \]
Taking $\chi=\chi ' \alpha\wedge \omega_s$ for some $\chi'\in C_c^\infty(\mc{M})$, we get using \eqref{2form_on_vert_fiber}, \eqref{eq:decomp932023} and \eqref{dSxomega}
\[ \cjg \delta_{T^1_xM},\chi'\cjd:=\int_{T_xM}\chi' dS_x =-8\int_{T_xM}\chi ' \omega_s= -8\cjg f_- \omega_u, \chi'\alpha\wedge \omega_s\cjd=-\frac{1}{4}\cjg f_-,\chi'\cjd \]
by choosing the measure on $\mc{M}$ given by the density $|\alpha\wedge d\alpha^2|=32|\alpha\wedge \omega_s\wedge \omega_u|$ to define the distributional pairing on the right hand side, and with $\alpha\wedge d\alpha^2>0$ with respect to the orientation of $T^1M$. Indeed, since $\omega_s=\frac{1}{4}\beta_+^R\wedge \beta_+^L$ and all terms in  \eqref{eq:decomp932023}  except $f_-\omega_u$ contain either $\beta_+^R$ or $\beta_+^L$, only $f_-\omega_u$ survives when pairing $\omega_s$ with $\iota_{X}[T_x^1M]$. \color{black}

This shows $f_-=-4\delta_{T^1_xM}$. Similarly, one obtains (with \eqref{2form_on_vert_fiber},  \eqref{eq:decomp932023} and \eqref{dSxomega})
\[ f_+=4\delta_{T^1_xM}, \quad f_{-+}^{L}=f_{-+}^R=0, \quad f_{-+}^{LR}=-\delta_{T^1_xM}, \quad  f_{-+}^{RL}=-\delta_{T^1_xM}.\]
As a consequence, we obtain 
\[ \begin{split}
 \cjg 1,[T^1_yM]\wedge (R_{{\bf X}_2}(\la) \iota_{X}[T_x^1M]\cjd=&
-4\int_{T^1_yM}(R_X(\la- 2)\delta_{T^1_xM})\omega_u +4\int_{T^1_yM}(R_X(\la+ 2)\delta_{T^1_xM})\omega_s\\
& -\int_{T^1_yM} (R_X(\la)\delta_{T^1_xM})\beta_-^R\wedge \beta_+^L- \int_{T^1_yM} (R_X(\la)\delta_{T^1_xM})\beta_-^L\wedge \beta_+^R\\
 =& -\frac{1}{2}\cjg \delta_{T_y^1M},R_X(\la- 2)\delta_{T^1_xM}\cjd  -\frac{1}{2}\cjg \delta_{T_y^1M},R_X(\la+2)\delta_{T^1_xM}\cjd\\
&  + \cjg \delta_{T_y^1M},R_X(\la)\delta_{T^1_xM}\cjd.
\end{split}\]
This term can be rewritten in terms of pushforward $\pi_*\otimes \pi_*$ and this ends the proof of \eqref{eq:pushfwdformula}. 

We can then use that $R_X(\la-2)$ has for only pole $\la_0=h_{\rm top}$ in ${\rm Re}(\la)>h_{\rm top}-\eps$ by Lemma \ref{uniqueres}, and since $\delta_\Gamma$ is the exponent of convergence of the series we get $\delta_\Gamma=h_{\rm top}$. The computation of the residue of $\mc{D}_\la(x,y)$ follows from the computation of the residue of $((\pi_*\otimes \pi_*)R_X(\la-2))(x,y)$. Now, by Lemma \ref{first_band}, $u_{\delta_\Gamma}$ defined by \eqref{defudelta} is a resonant state for $X$ in the first band, with resonance $\delta_\Gamma-2$, and by 
Lemma \ref{Klein-Gordon} and Proposition \ref{resonance-Poisson} we obtain that $\tilde{f}_{\delta_\Gamma}:=\mc{P}^0_{\delta_\Gamma-2}(\omega_{\delta_\Gamma})$ is a $\Gamma$-periodic eigenfunction of $\Box_g$, that descends to $M$ when restricting it on $\O_+(\Lambda_\Gamma)$. The residue of $R_X(\la)$ at $\la_0=\delta_\Gamma-2$ can be written as 
\[\Pi^X_{\la_0}=u_{\delta_\Gamma}\otimes v_{\delta_\Gamma}\]
for some co-resonant state $v_{\delta_\Gamma}\in \mc{D}'(\mc{M})$ with support in $\mc{K}_+$ and ${\rm WF}(v_{\delta_\Gamma})\subset E_s^*$. Since $X^*=-X$ for the $L^2$-pairing, for $\la>0$ one has $R_X(\la)^*=R_{-X}(\la)$, and applying the same reasoning as above for $-X$ intead of $X$ (the flow in backward time), one obtains that the Ruelle resonance with the leading eigenvalue for $-X$ is $\la_0=\delta_\Gamma-2$ and $v_{\delta_\Gamma}$ is the associated resonant state, which has the form $v_{\delta_\Gamma}= c_{\delta_\Gamma} \Phi_+^{\delta_\Gamma-2}B_+^*\omega_{\delta_\Gamma}$ for some $c_{\delta_\Gamma}\not=0$. Its pushforward is given by $c_{\delta_\Gamma}\tilde{f}_{\delta_\Gamma}$ which descends to $M$ (after restricting to $\O_+(\Lambda_\Gamma)$). Since $\mc{D}_\la(x,y)>0$ for $\la>\delta_{\Gamma}$, we get that $c_{\delta_\Gamma}>0$. Finally, it suffices to bound below $P(x,\nu)$ by $c_0(x):=\min_{\nu\in \Lambda_\Gamma} P(x,\nu)$ to deduce that $f_{\delta_\Gamma}(x)>c_0(x)\cjg \omega_{\delta_\Gamma},1\cjd >0$.
\end{proof}

\section{Resolvent of the pseudo-Riemannian Laplacian on quasi-Fuchsian anti-de Sitter spacetimes}\label{sec:resolventBox}

In this section, we shall prove that there is a meromorphic resolvent associated to $\Box_g$ on $M=\Gamma\backslash \mc{O}_+(\Lambda_\Gamma)$. This is achieved by meromorphically extending the $\Gamma$-periodized function $F_\la^h(-q(x,x'))$ defining the resolvent kernel of $\Box_g$ on $\mc{O}_+(\Lambda_\Gamma)$.
\begin{theo}\label{th:ext_resolvent}
Let $\Gamma\subset {\rm Isom}(\AdS_3)$ be a quasi-Fuchsian subgroup and $F_\la^h$ the function defined in Proposition \ref{prop:resolventAdShol}. 
Then for ${\rm Re}(\la)>1$, the integral kernel 
\[ R_{\Box_g}(\la;x,x'):= \sum_{\gamma \in \Gamma} F_\la^h(-q(x,\gamma x'))\] 
defined for $x,x'\in \mc{O}_+(\Lambda_\Gamma)$ descends to an $L^1_{\rm loc}(M\times M)$ 
function and extends meromorphically to $\la\in \C$. The meromorphic extension defines a continuous operator 
\[R_{\Box_g}(\la): C_c^\infty(M)\to \mc{D}'(M)\]
satisfying $(\Box_g+\la(\la+2))R_{\Box_g}(\la)f=f$ for all $f\in C_c^\infty(M)$ and has poles of finite rank. Moreover, one has for $f,f'\in C_c^\infty(M)$
\begin{equation}\label{Rbox_RX} 
\cjg R_{\Box_g}(\la)f,f'\cjd=\frac{i}{4\pi}\Big(\cjg R_X(\la)\pi^*f,\pi^*f'\cjd-\cjg R_X(\la+2)\pi^*f,\pi^*f'\cjd\Big)+ \cjg H(\la)f,f'\cjd
\end{equation}
where $H(\la)$ is the holomorphic family of operators with integral kernels 
\[H(\la;x,x')=\sum_{\gamma\in \Gamma_{x,x'}}F_\la^h(-q(x,\gamma x'))\]
featuring the finite set $\Gamma_{x,x'}\subset \Gamma$ whose complement was defined in \eqref{Gammaxy^c}. \color{black}
\end{theo}
\begin{proof}  We have by Proposition \ref{prop:resolventAdShol} that for ${\rm Re}(\la)>0$
\[\begin{split}
\sum_{\gamma\in \Gamma_{x,x'}^c}F^h_\la (-q(x,\gamma  x')) =&  \frac{i}{4\pi} \sum_{\gamma\in  \Gamma_{x,x'}^c}\frac{e^{-(\la+1)d( x, \gamma x')}}{\sinh(d( x, \gamma x'))}= 
\frac{i}{2\pi} \sum_{k=1}^\infty \sum_{\gamma\in \Gamma_{x,x'}^c}e^{-(\la+2k)d( x, \gamma x')}\\
=& \frac{i}{2\pi}\sum_{k=1}^\infty \mc{D}_{\la+2k}(x,x')
\end{split}\]
where $\Gamma_{x,x'}^c$ is defined by \eqref{Gammaxy^c}. We can thus write 
\[R_{\Box_g}(\la; x, x')= \frac{i}{2\pi}\sum_{k=1}^\infty \mc{D}_{\la+2k}(x,x')+\sum_{\gamma\in \Gamma_{x,x'}}F_\la^h(-q(x,\gamma x'))\]
where the last sum is finite. Here we notice that $\mc{D}_{\la+2k}(x,y)$ converges in ${\rm Re}(\la)>-2k$ and the series in $k$ is uniformly convergent for $\la$ in a compact set not intersecting a pole of $\mc{D}_{\la+2k}$ for some $k$. Indeed, for ${\rm Re}(\la)>-N$, we have for $2k\geq N+3$
\[ |\mc{D}_{\la+2k}(x,y)| \leq \mc{D}_{-N+2k}(x,y)\leq e^{ (-2k+N+3) \min_{\gamma\in \Gamma_{x,y}^c}d(x,\gamma y)}\mc{D}_{3}(x,y)\]
and $\min_{\gamma\in \Gamma_{x,y}^c}d(x,\gamma y)>0$ by definition of $\Gamma_{x,y}^c$ (recall Lemma \ref{Lem:q_and_distance}).

By Lemma \ref{lem - q<-1 up to moving by an element of the group}, for each compact sets $K_1,K_2\subset \mc{O}_+(\Lambda_\Gamma)$, we have that $\cup_{(x,x')\in K_1\times K_2}\Gamma_{x,x'}\subset \Gamma_{K_1,K_2}$ which is a finite set. This implies that 
\[ (x,x')\in K_1\times K_2 \mapsto \sum_{\gamma\in \Gamma_{x,x'}}F_\la^h (-q(x,\gamma x'))  \in L^1(K_1\times K_2)\]
and thus $R_{\Box_g}(\la;\cdot,\cdot)\in L^1_{\rm loc}(\mc{O}_+(\Lambda_{\Gamma})^2)$ by Proposition \ref{prop:Poincare_series}.
Now, $\sum_{\gamma\in  \Gamma}F_\la^h(-q(\gamma x, x'))$ is $\Gamma$ invariant as a function of $(x,x')\in \mc{O}_+(\Lambda_\Gamma)^2$ and since $-q(x,\gamma x')>-1$ for such $(x,x')$ by Lemma \ref{lem - q<-1 locally uniformly on the discontinuity  domain} and $F_\la^h(-q(x,x'))$ satisfies in the distribution sense
\[ (\Box_{g}+\la(2+\la))F_{\la}^h(-q(x,x))=\delta_{x=x'}\]
in $\mc{O}_+(\Lambda_\Gamma)^2$. Therefore it descends to $M\times M$ as an $L^1_{\rm loc}$ function satisfying $(\Box_g+\la(\la+2))R_{\Box_g}(\la;\cdot,x')=\delta_{x'}$ on $M$. When viewed as an operator it thus satisfies $(\Box_g+\la(\la+2))R_{\Box_g}(\la)f=f$ for all $f\in C_c^\infty(M)$. The poles of $R_{\Box_g}(\la)$ are contained in the union of the  poles of $\mc{D}_{\la+2k}$, which is a discrete set since $\mc{D}_{\la+2k}$ is holomorphic in ${\rm Re}(\la)>-2k+2$.
\end{proof}

We shall call \emph{quantum resonance} a pole $\la_0$ of $R_{\Box_g}(\la)$ and \emph{quantum resonant state} associated to $\la_0$ a distribution $u\in \mc{D}'(M)$ in the range of 
\[\Pi_{\la_0}^{\Box_g}:={\rm Res}_{\la=\la_0}R_{\Box_g}(\la)\]
satisfying $(\Box_g+\la_0(2+\la_0))u=0$.

We next relate the quantum resonances and quantum resonant states of $\Box_g$ and 
the Ruelle resonances and resonant states for the spacelike geodesic vector field $X$.
\begin{cor}\label{c:resonant_states_correspondence}
There exists $\eps>0$ such that the resolvent $R_{\Box_g}(\la)$ has for only pole $\la=\delta_\Gamma-2$ 
in ${\rm Re}(\la)>\delta_\Gamma-2-\eps$, it is a simple pole with residue the rank $1$-operator
\[ \Pi^{\Box_g}_{\delta_\Gamma-2}=\frac{i}{2\pi}c_{\delta_\Gamma}f_{\delta_\Gamma}\otimes f_{\delta_\Gamma}.\]
with $c_{\delta_\Gamma}$ and $f_{\delta_\Gamma}$ defined in Theorem \ref{theo:Poincare_series}. 
The poles of $R_{\Box_g}(\la)$ are contained in 
\[ \cup_{k\in \N} ({\rm Res}(X)-2k),\]
and for ${\rm Re}(\la_0)>\delta_{\Gamma}-4$, one has $\Pi^{\Box_g}_{\la_0}=\frac{i}{4\pi}\pi_*\Pi^{X}_{\la_0}\pi^*$ and 
\[ \pi_* : {\rm Ran}(\Pi^{X}_{\la_0}) \to  {\rm Ran}(\Pi^{\Box_g}_{\la_0})\]
is surjective. \color{black}
The quantum resonant states for $\Box_g$ at a quantum resonance $\la_0$ are smooth functions in $M$ of the form $\pi_*u$ for $u$ a Ruelle resonant state. Moreover, if $\la_0$ is a quantum resonance with ${\rm Re}(\la_0)>\delta_{\Gamma}-2-\min(\eps_-^R,\eps_-^L)$, with $\eps_-^R,\eps_-^L$ defined in \eqref{eps_-R/L}, then there is a non zero Ruelle resonant state $u$ with Ruelle resonance $\la_0$ such that $U_-^Ru=U_-^Lu=0$, $\pi_*u$ is a non-zero quantum resonant state with quantum resonance $\la_0$, and the $\Gamma$-invariant lift of $f:=\pi_*u$ to $\mc{O}_+(\Lambda_\Gamma)$ extends to $\AdS_3$ as a distribution $\tilde{f}\in \mc{D}'(\AdS_3)$ of the form 
\[\tilde{f}=\mc{P}^0_{\la_0}\omega\]
for some $\Gamma$-equivariant $\omega\in \mc{D}'(\TT^2)$. 
\end{cor}
\begin{proof} The first two statements are a direct consequence of \eqref{Rbox_RX} and Theorem \ref{theo:Poincare_series}. 

To prove that the pushforward $\pi_*u$ is smooth if $u$ is a Ruelle resonant state, we can use the wave-front set property of $u$, namely that ${\rm WF}(u)\subset E_u^*$, and the properties of pushforward of wave-front sets \cite[Proposition 11.3.3.]{FrJo}: one has 
\[{\rm WF}(\pi_*u)\subset \{(x,\xi)\in T^*M \setminus \{0\}\,|\,  (x,v, d\pi^\top_{x,v} \xi) \in E_u^* \}\]
but $\xi\circ d\pi_{x,v}\in E_u^*$ implies that $0=\xi(d\pi_{x,v}X)=\xi(d\pi_{x,v}U_{+}^{R})=\xi(d\pi_{x,v}U_{+}^{L})$ and this implies that $\xi=0$, as $(v,d\pi_{x,v}U_{+}^{R},d\pi_{x,v}U_{+}^{L})$ is a basis of $T_xM$ by section \ref{splittings_T}. This implies that  $\pi_*u$ has no wave-front set in $M$, and is thus smooth there. 

For the last statement of the Corollary, \eqref{Rbox_RX} implies that a quantum resonance $\la_0$ with ${\rm Re}(\la)>\delta_\Gamma-2-\min(\eps_-^R,\eps_-^L)$ is a Ruelle resonance and that there is a resonant state $u$ with $\pi_*u=f$ being a non-zero quantum resonant state. Then Corollary \ref{noruelleres} says that $U_-^Ru=U_-^Lu=0$ and 
Lemma \ref{first_band} finally shows that its $\Gamma$-invariant lift can be written as a Poisson transform, which extends to $\AdS_3$. 
\end{proof}

\textbf{Remark:} The fact that the lift of $f$ to $\mc{O}_+(\Lambda_\Gamma)$ extends to $\AdS_3$ as a global solution of the Klein-Gordon equation is not a priori an obvious statement: even though it is smooth in $\mc{O}_+(\Lambda_\Gamma)$, its behaviour at $\pl \mc{O}_+(\Lambda_\Gamma)$ looks complicated.

\section{The Fuchsian case}\label{s:Fuchsian}
In this section, we compute the resolvent of $\Box_g$ and its pole for special cases, the \emph{Fuchsian} anti-de Sitter spaces. In that case, it is possible to express explicitely the resonances and resonant stated of $\Box_g$ in terms of eigenvalues of the Laplacian on the underlying hyperbolic surface. Let  
\[\H^2:=\{(x_2,x_3,x_4)\in \R^3\,|\, -x_2^2+x_3^2+x_4^2=-1, x_2>0\}\] 
be the hyperbolic plane, 
that we can also view as the 
submanifold $\H^2=\{x\in \AdS_3\,|\, x_1=0,x_2>0\}$ of $\AdS_3$. The quadratic form $q(x)$ on $\AdS_3$ restricts 
to $T\H^2$ and gives a Riemannian metric of curvature $-1$, which is  the hyperbolic metric denoted by $g_{\H^2}$. The hyperbolic 
plane $\H^2$ is thus equipped with the restriction of the Lorentzian metric $q$ of $\AdS_3$ and it is a totally geodesic submanifold.
Consider the map 
\begin{equation}\label{Theta} 
\Theta: (t,x_2,x_3,x_4)\in (-\frac{\pi}{2},\frac{\pi}{2})\times \H^2 \mapsto (\sin(t),x_2\cos(t),x_3\cos(t),x_4\cos(t))\in \AdS_3
\end{equation}
which is a diffeomorphism with range $Q:=\{(x_1,x_2, x_3,x_4)\in \AdS_3\,|\, |x_1|<1\}$.
The anti-de Sitter metric $\Theta^*g$ on $Q$  is given by
\[ \Theta^*g=-dt^2+ \cos(t)^2 g_{\H^2} ,\quad \textrm{ on } W:=(-\pi/2,\pi/2)\times \H^2.\]
For notational convenience we shall write $g$ instead of $\Theta^*g$.
In these coordinates, the operator $\Box_{g}$ can be written as 
\[ \Box_g=\pl_t^2 +\frac{1}{\cos(t)^2}\pl_t(\cos(t)^2)\pl_t+\frac{1}{\cos(t)^2}\Delta_{g_{\H^2}}\]
where $\Delta_{g_{\H^2}}$ is the Laplacian on $\H^2$. The volume form is 
\begin{equation}\label{dvg} 
{\rm dv}_g=\cos(t)^2dt d{\rm v}_{g_{\H^2}}.
\end{equation}
Consider a co-compact Fuchsian subgroup $\Gamma\subset G=\mathrm{SO}(2,1)\simeq \SL(2,\R)$ and, when viewed as elements in $\SL(2,\R)$, we let it act diagonally on 
$\SL(2,\R)\simeq \AdS_3$ by conjugation $\gamma\cdot g:=\gamma g \gamma^{-1}$. The limit set is the unit circle $\pl \H^2\subset 
\{x_1=0, x_2=1, x_3^2+x_4^2=1\}$ viewed in the conformal boundary $\pl \AdS_3$ of $\AdS_3$.
The domain of discontinuity of $\Gamma$ acting on $\AdS_3$ is $\mc{O}_+(\Lambda_\Gamma)=Q$, which is precisely 
the range of the map $\Theta$. Indeed,  for each $\nu=(0,1,\nu_3,\nu_4)\in \Lambda_\Gamma$ we have 
\[ q(x,\nu)=-x_2+x_3\nu_3+x_4\nu_4\leq -x_2+\sqrt{x_3^2+x_4^2}=-x_2+\sqrt{-1+x_1^2+x_2^2}\]
and the right hand side is attained by taking $(\nu_3,\nu_4)=(x_3,x_4)/\sqrt{x_3^2+x_4^2}$. Thus $q(x,\nu)<0$ for all $\nu\in  \Lambda_\Gamma$ if and only if $\sqrt{-1+x_1^2+x_2^2}<x_2$, which is equivalent to $|x_1|<1$. The action of $\gamma\in \Gamma$ on $\mc{O}_+(\Lambda_\Gamma)$ is by $\gamma \cdot\Theta(t,\bar{x})=\Theta(t,\gamma \cdot\bar{x})$ if $\bar{x}=(x_2,x_3,x_4)\in \H^2$, and the action $\gamma \cdot\bar{x}$ is the natural action of $\gamma$ viewed as a hyperbolic isometry.
We compute for $t,t'\in (-\pi/2,\pi/2)$ and $\bar{x},\bar{x}'\in \H^2$, 
\begin{equation}\label{qAdS-qH} 
-q(\Theta(t,\bar{x}),\Theta(t',\bar{x}'))=\sin(t)\sin(t')-\cos(t)\cos(t')q_{\H^2}(\bar{x},\bar{x}')
\end{equation}
where $-q_{\H^2}(\bar{x},\bar{x}')=-q(0,\bar{x},0,\bar{x}')=\cosh(d_{\H^2}(\bar{x},\bar{x}'))$. 

\begin{definition}We call \emph{modified Poincar\'e series} the series, convergent if ${\rm Re}(\la)>1$,  given for $x,x'\in \mc{O}(\Lambda_\Gamma)$ by 
\[\sum_{\gamma \in \Gamma_{x,x'}^c} |q(x,\gamma x')|^{-\la}.\]
\end{definition}
\begin{lemma}
The modified Poincar\'e series satisfies the following: there is $C>0$ such 
that for all $x=\Theta(t,\bar{x}),x'=\Theta(t',\bar{x}')\in Q= \mc{O}_+(\Lambda_\Gamma)$ with $\bar{x},\bar{x}'$ in a fixed fundamental domain of $\Gamma$ acting on $\H^2$, and for all $\la>1$, 
\begin{equation}\label{bound_poincare_modified}
\sum_{\gamma \in \Gamma_{x,x'}^c} |q(x,\gamma x')|^{-\la}\leq \frac{C(2/3)^{\la}}{(\la-1)\cos(t)\cos(t')}
\end{equation}
\end{lemma}
\begin{proof} For $x=\Theta(t,\bar{x}),x'=\Theta(t',\bar{x}')\in Q$, without loss of generality we can replace $\Gamma_{x,x'}^c$ defined in \eqref{Gammaxy^c} by  $\Gamma_{x,x'}^c:=\{\gamma \in \Gamma\,|\, -q(x,\gamma x')>4\}$. For $\la>1$, we have by \eqref{qAdS-qH}
\[\sum_{\gamma \in \Gamma_{x,x'}^c} |q(x,\gamma x')|^{-\la}\leq 
\sum_{\gamma \in \Gamma_{x,x'}^c} (\frac{|q(x,\gamma x')|}{2}+2)^{-\la}\leq 
\sum_{\gamma \in \Gamma_{x,x'}^c} \Big(|q_{\H^2}(\bar{x},\gamma \bar{x}')|\cos(t)\cos(t')+\frac{3}{2}\Big)^{-\la}. \]
We can use that $|q_{\H^2}(\bar{x},\gamma \bar{x}')|=\cosh(d_{\H^2}(\bar{x},\gamma \bar{x}'))$ and the classical counting estimate for co-compact Fuchsian groups  acting on $\H^2$ (see \cite{Borthwick}): for all $T>1$ large 
\begin{equation}\label{counting}
N_{\bar{x},\bar{x}'}(T):=\#\{\gamma\in \Gamma\,|\, d_{\H^2}(\bar{x},\gamma \bar{x}')\leq T\}=\mc{O}(e^{T})
\end{equation} 
uniformly for $\bar{x},\bar{x}'$ on compact sets on $\H^2$. We then get that for each $\la_0>0$, there is $C>0$ such that for all $\la\in (1,\la_0)$, $t,t'\in (-\pi/2,\pi/2)$ we obtain (for some uniform $C_0>0$)
\[\begin{split}
\sum_{\gamma \in \Gamma_{x,x'}^c} |q(x,\gamma x')|^{-\la}\leq &\int_0^\infty  \Big(\cosh(T)\cos(t)\cos(t')+\frac{3}{2}\Big)^{-\la}dN(T)\\
& \leq C_0 \int_0^\infty  \Big(e^T\cos(t)\cos(t')+\frac{3}{2}\Big)^{-\la}e^{T}dT
\end{split}\]
which implies \eqref{bound_poincare_modified} for $\bar{x},\bar{x}'$ in a compact set of $\H^2$.
\end{proof}

Conjugating by the map $\Psi: u\in L^2(\mc{O}_+(\Lambda_\Gamma),{\rm dv}_g)\mapsto \cos(t)u \in L^2(\mc{O}_+(\Lambda_\Gamma),dtd{\rm v}_{g_{\H^2}})$, we see that $\Box_g$ is unitarily equivalent to the operator $P:=\Psi\circ \Box_g\circ \Psi^{-1}$ which is equal to
\begin{equation}\label{BoxFuchsian} 
P= \pl_t^2+1 + \frac{1}{\cos(t)^2}\Delta_{g_{\H^2}}.
\end{equation}

We recall a quite standard result about the Poincar\'e series of $\Gamma\backslash \H^2$ if $\Gamma\subset \SL(2,\R)$ is a co-compact subgroup, the proof can be essentially found in \cite[Theorem 7.3]{Guillarmou-Mazzeo} in a much more general setting.
\begin{lemma}\label{l:poincare_series_H^2}
Let $\Gamma\subset \SL(2,\R)$ be a co-compact subgroup and $\Sigma=\Gamma\backslash \H^2$ be the associated 
compact hyperbolic surface. The modified Poincar\'e series, defined for $x,y\in \H^2$ and ${\rm Re}(\la)>1$ 
\[ \mc{C}^\Sigma_\la(x,y):=\frac{1}{\Gamma(\la+1/2)}\sum_{\gamma \in \Gamma\setminus \{{\rm Id}\}} (-q_{\H^2}(x,\gamma y))^{-\la}\] 
is convergent and admits a meromorphic extension to $\la\in \C$, where, except at $\la\in 1/2-2\N^*$, 
its poles are first order and located at 
\begin{equation}\label{poles_poincare} 
\bigcup_{k\in \N}\bigcup_{s \in \mc{S}_{\Delta_{\Sigma}}}(-2k+s)) \setminus -\N^* 
\end{equation} 
where $\mc{S}_{\Delta_\Sigma}:=\{ s\in \C\,|\,  s(1-s)\in {\rm Sp}(\Delta_\Sigma)\}$, and with residue 
\[ {\rm Res}_{\la=-2k+s}\mc{C}^\Sigma_\la(x,y)=c_{s,k} \Pi_{s(1-s)}(x,y)\]
where $\Pi_{z}(x,y)$ is the integral kernel of the spectral projector of $\Delta_\Sigma$ on $\ker \Delta_\Sigma-z$ and $c_{s,k}\not=0$ some explicit constant. At $\la=1/2-2k$ with $k\in \N$, there is a 
pole if and only if $1/4\in {\rm Sp}(\Delta_\Sigma)$, it  has order $2$ and polar part of the form, for some constant $c_{1/2,2k}\not=0$
\[ \frac{c_{1/2,2k}\Pi_{1/4}(x,y)}{(\la-1/2+2k)^2}.\]
\end{lemma}
\begin{proof}
Let $R_{\Delta_{\H^2}}(\la;x,y)$ be the  integral kernel of the resolvent of the Laplacian on the hyperbolic plane $\H^2$; by 
 \cite[Lemma 2.1]{Guillope-Zworski} it is meromorphic with first order poles at the non-positive integers 
 $-\N$, and the residues are integral kernels of finite rank operators acting on $C_c^\infty(\H^2)$.
Since the Laplacian $\Delta_{\Sigma}$ is self-adjoint on $L^2(\Sigma)$ with discrete spectrum ${\rm Sp}(\Delta_\Sigma)\subset \R^+$, by the spectral theorem 
the resolvent $R_{\Delta_{\Sigma}}(\la)=(\Delta_{\Sigma}-\la(1-\la))^{-1}$ has a Schwartz kernel which 
is meromorphic in $\la\in \C$, with first order poles at those $\la$ such that  $\la(1-\la)\in {\rm Sp}(\Delta_\Sigma)$, except possibly at $\la=1/2$ where the pole is of order $2$ (if $1/4\in {\rm Sp}(\Delta_\Sigma)$), and except at $\la=1/2$, the residues at all these poles are the integral kernels of the spectral projectors  of $\Delta_\Sigma$. 
Moreover, the integral kernel of $R_{\Delta_\Sigma}$ can be lifted to $\H^2$ and is explicitly given for ${\rm Re}(\la)>1$ and $x\not=y$ in a fundamental domain of $\Gamma$ in $\H^2$  by 
(see \cite[Lemma 2.1]{Guillope-Zworski})
\[
\begin{split}
R_{\Delta_\Sigma}(\la;x,y)-R_{\Delta_{\H^2}}(\la;x,y)&= \frac{\pi^{-1/2}}{2^{\la+1}}\sum_{\gamma\in \Gamma\setminus \{{\rm Id}\}}
\sum_{j=0}^\infty \frac{2^{-2j}\Gamma(\la+2j)}{\Gamma(\la+\frac{1}{2}+j)\Gamma(j+1)}\cosh(d_{\H^2}(x,\gamma y))^{-\la-2j}\\
&= \pi^{-1/2}\sum_{j=0}^{\infty}  \frac{2^{-\la-2j-1}\Gamma(\la+2j)}{\Gamma(\la+\frac{1}{2}+j)\Gamma(j+1)}\Gamma(\la+\tfrac{1}{2}+2j) \mc{C}^\Sigma_{\la+2j}(x,y).
\end{split}
\]
The left hand side is meromorphic in $\la\in \C$, with poles at those $\la$ such that  $\la(1-\la)\in {\rm Sp}(\Delta_\Sigma)$ 
and at $\la\in -\N$, the poles are first order with finite rank operators (except at $\la=1/2$ where it is order $2$) 
and the residues are integral kernels of spectral projectors of $\Delta_\Sigma$ if $\la\not\in -\N$. Moreover by elliptic theory,
the left hand side is smooth as a function of $x,y$ (see for example the proof of \cite[Theorem 7.3]{Guillarmou-Mazzeo}). This shows that the right hand side has the same properties. But since $\mc{C}^\Sigma_{\la+2N}(x,y)$ is holomorphic in $\{{\rm Re}(\la)>1-2N\}$ for each $N>0$ (the $\Gamma$-series defining $\mc{C}^\Sigma_{\la+2N}(x,y)$  is convergent), we obtain 
that 
\[\sum_{j=0}^{N-1}  \frac{2^{-\la-2j-1}\Gamma(\la+2j)}{\Gamma(\la+\frac{1}{2}+j)\Gamma(j+1)}\Gamma(\la+\tfrac{1}{2}+2j) \mc{C}^\Sigma_{\la+2j}(x,y)\]
is meromorphic with the same properties for the poles and residues as the integral kernel $R_{\Delta_\Sigma}(\la;x,y)-R_{\Delta_{\H^2}}(\la;x,y)$ in the region $\{{\rm Re}(\la)>1-2N\}$. Using an induction argument, this gives the desired result.
\end{proof}
 
Using this result, we can describe the poles of the resolvent of $\Box_g$ on the Lorentzian manifold 
$M=\Gamma\backslash \mc{O}_+(\Lambda_\Gamma)$.
\begin{cor}
Consider a Fuchsian quotient $M=\Gamma\backslash \mc{O}_+(\Lambda_\Gamma)$ and $\Sigma=\Gamma\backslash \H^2$. 
In the region $\C\setminus -\N$, the poles of the meromorphic extension of $R_{\Box_g}(\la)$ of $\Box_g$ on $M$  are located at 
\[ \la \in \bigcup_{m \in \N}\bigcup_{s(1-s) \in {\rm Sp}(\Delta_{\Sigma})}(-2-m+s).  \]
They are first order poles except when $s(1-s)=1/4$ where they are of order at most $3$, and when $s\not=1/2$, the  residues at $\la=-m-2+s$ have the form
\[\Pi_{s(1-s)}(\bar{x},\bar{x}')\sum_{\substack{(\ell,j)\in \N\times \N^* \\ 2j+\ell=m+2}}
(\cos(t)\cos(t'))^{-s+2(j-1)}P_{\ell,j}(\sin(t)\sin(t'))\] 
with $P_{\ell,j}$ a polynomial of degree $\ell$, and the elements in the range of the residues are linear combinations of functions 
\[\begin{gathered}
f_s^{(n)}(\bar{x}) \cos(t)^{-s}
{\rm F}_1(-\tfrac{m}{2},1+\tfrac{m}{2}-s,\tfrac{3}{2}-s;\cos^2(t)), \quad m \textrm{ even } \\
f_s^{(n)}(\bar{x}) \cos(t)^{-s}
\sin(t){\rm F}_1(-\tfrac{m-1}{2}, 2+\tfrac{m-1}{2}-s,\tfrac{3}{2}-s;\cos^2(t)), \quad m \textrm{ odd }
\end{gathered} \]
where $f_s^{(n)}$ are eigenfunctions of $\Delta_{\Sigma}$ with eigenvalues $s(1-s)$, $n\leq \dim \ker(\Delta_{\Sigma}-(s(1-s))$ and ${\rm F}_1(a,b,c;z)$ is the standard hypergeometric function. The resonant states are thus in $L^2(M)$.
\end{cor} 
\begin{proof} Let $\mc{F}$ be a fundamental domain for the action of $\Gamma$ on 
$\H^2$. For $x=\Theta(t,\bar{x}),x'=\Theta(t',\bar{x}')\in Q= \mc{O}_+(\Lambda_\Gamma)$ with $\bar{x},\bar{x}'$ in $\mc{F}$ and $\Gamma^c_{x,x'}$ defined in \eqref{Gammaxy^c} (and $\Gamma_{x,x'}$ its complementary set in $\Gamma$), for ${\rm Re}(\la)>1$
\begin{equation}\label{split_Rbox}
\begin{split}
R_{\, \Box_g}(\la;x,x')= &\sum_{\gamma\in \Gamma_{x,x'}}F_\la^h(-q(x,\gamma x'))\\
& +
C_0^+(\la)\sum_{k=0}^\infty \frac{2^{-\la-1-2k}\Gamma(-\la+2k)}{\Gamma(-\la+k)\Gamma(k+1)}\sum_{\gamma\in \Gamma_{x,x'}^c} |q(x,\gamma x')|^{-\la-2-2k}.\end{split}
\end{equation}
Now we consider the set $\Gamma_{x,x'}$: we have $\gamma\in \Gamma_{x,x'}$ iff 
\[ \cosh(d_{\H^2}(\bar{x},\gamma \bar{x}'))=-q_{\H^2}(\bar{x},\gamma\bar{x}')<\frac{1-\sin(t)\sin(t')}{ \cos(t)\cos(t')} \]
thus by \eqref{counting}, we obtain that there is $C>0$ uniform such that for all $x=\Theta(t,\bar{x}),x=\Theta(t',\bar{x}')$ as above
\begin{equation}\label{countingGammaxx'}
|\Gamma_{x,x'}|\leq C\frac{1-\sin(t)\sin(t')}{ \cos(t)\cos(t')}.
\end{equation}
We denote by $\mc{F}_\eps:=\{ \Theta(t,\bar{x})\in Q\,|\, \bar{x}\in \mc{F}, |t|<\pi/2-\eps \}$ if $\eps\geq 0$.
By Proposition \ref{prop:resolventAdS}, the first term in the sum \eqref{split_Rbox} extends as a meromorphic family with value in $L^1(\mc{F}_\eps\times \mc{F}_{\eps'})$ for each $\eps,\eps'>0$, with poles contained in $\Z$,  and using \eqref{countingGammaxx'}, for each $\la$ in a compact set of  $\C\setminus \Z$,
  there is $C$ depending on this set such that for all $\eps>0,\eps'>0$
\[  \int_{\mc{F}_\eps}\int_{\mc{F}_{\eps'}}\sum_{\gamma\in \Gamma_{x,x'}}|F_{\la}^h(-q(x,\gamma x'))| {\rm dv}_g(x){\rm dv}_g(x')\leq C(\eps\eps')^{-1}\]  
In  the second line of \eqref{split_Rbox}, we notice that there is $C>0$ such that for all $k\in \N$ and ${\rm Re}(\la)>1$
\[ \Big|\frac{2^{-\la-1-2k}\Gamma(-\la+2k)}{\Gamma(-\la+k)\Gamma(k+1)}\Big|\leq C\]
thus using \eqref{bound_poincare_modified} we get that for $\la$ in any compact set $K$ of $\{\la\in \C\setminus \N \, | \, {\rm Re}(\la)>1\}$
\[ \|R_{\Box_g}(\la;\cdot,\cdot)\|_{L^1(\mc{F}_\eps\times \mc{F}_{\eps'})}\leq C(\eps\eps')^{-1}\]
for some constant $C>0$ depending only on $K$. 

To describe the meromorphic extension of $R_{\Box_g}(\la;\cdot,\cdot)$, we need to consider the second line of \eqref{split_Rbox}, and we use  \eqref{qAdS-qH}:
\[\begin{split}
 |q(x,\gamma x')|^{-\la-2-2k}&=\frac{|q_{\H^2}(\bar{x},\gamma \bar{x}')|^{-\la-2-2k}}{(\cos(t)\cos(t'))^{\la+2+2k}}\Big(1-\frac{\tan(t)\tan(t')}{q_{\H^2}(\bar{x},\gamma \bar{x}')}\Big)^{-\la-2-2k}\\
&= \frac{|q_{\H^2}(x,\gamma x')|^{-\la-2-2k}}{(\cos(t)\cos(t'))^{\la+2+2k}}\sum_{j=0}^\infty \frac{\Gamma(-\la-1-2k)}{\Gamma(-\la-1-2k-j)j!}
\Big(\frac{\tan(t)\tan(t')}{|q_{\H^2}(x,\gamma x')|}\Big)^j.
 \end{split}\]
We deduce that there are functions $B_\ell(\la,u)$ that are polynomials of degree $\ell$ in $u$ and are holomorphic in $\la\in \C$ such that  
\begin{align*}
&R_{\Box_g}(\la;x,x')\\
& = \sum_{\gamma\in \Gamma_\eps}F_{\la}^h(-q(x,\gamma x'))+C_0^+(\la)\sum_{\ell=0}^\infty \frac{B_\ell(\la,\sin(t)\sin(t'))}{(\cos(t)\cos(t'))^{\la+2+\ell}}\sum_{\gamma \in \Gamma_\eps^c}
(-q_{\H^2}(\bar{x},\gamma \bar{x}))^{-\la-2-\ell}
\end{align*}
By Lemma \ref{l:poincare_series_H^2}, for each $\eps>0,\eps'>0$ the series $\sum_{\gamma \in \Gamma_\eps^c}
(-q_{\H^2}(\bar{x},\gamma \bar{x}))^{-\la-2-\ell}$ is meromorphic in $\la$ as an $L^1(\mc{F}_\eps\times \mc{F}_{\eps'})$ function, 
with first order poles at the points $\la=-\ell-2j+s$ where $j\in \N^*$ and $s(1-s)\in {\rm Sp}(\Delta_\Sigma)\setminus \{1/4\}$
and poles at $\la\in-3/2-\N$ that are of order at most $3$. 
The residue at pole $\la=-2j-\ell+s$ with $s(1-s)\in {\rm Sp}(\Delta_\Sigma)\setminus \{1/4\}$ and $j\in\N^*$ has the form 
\[{\rm Res}_{\la=-2j-\ell+s}\sum_{\gamma \in \Gamma_\eps^c}
(-q_{\H^2}(\bar{x},\gamma \bar{x}))^{-\la-2-\ell} = c_{s,j-1} \Pi_{s(1-s)}(x,y)\]
which implies that, if $\la\notin -\N$ and $m\in \N$
\[ {\rm Res}_{\la=-m-2-s} R_{\, \Box_g}(\la;x,x')= \Pi_{s(1-s)}(\bar{x},\bar{x}')\sum_{\substack{(\ell,j)\in \N \times \N^* \\ 2j+\ell=m+2}}
(\cos(t)\cos(t'))^{-s+2(j-1)}P_{\ell,j}(\sin(t)\sin(t'))\]
with $P_{\ell,j}(Z)$ a polynomial of degree $\ell$. Now, we also see, from the expression \eqref{BoxFuchsian}, that an element in the range of the residue at $\la=-m-2+s$ must be of the form $u(t)f_s(\bar{x})/\cos(t)$ with $(\Delta_{\Sigma}-s(1-s))f_s=0$ and $u(t)$ solves the ODE
\[ \pl_t^2 u(t)+(\la+1)^2u(t)+\frac{s(1-s)}{\cos^2(t)}u(t)=0\]
This equation reduces to a hypergeometric equation, and the two independent odd/even solutions on $t\in (-\pi/2,\pi/2)$ are 
\[ \begin{split}
u_1(t)=& \cos(t)^{s}\sin(t)
{\rm F}_1(\tfrac{s-\la}{2},\tfrac{s+\la+2}{2},\tfrac{3}{2};\sin^2(t)), \\
 u_2(t)=& \cos(t)^{s}{\rm F}_1(\tfrac{-\la+s-1}{2},\tfrac{\la+s+1}{2},\tfrac{1}{2};\sin^2(t))\end{split}\]
where ${F}_1$ is the standard hypergeometric function. They have asymptotics at $t=\pm \pi/2$ given by (see \cite[Chap. 15]{abramovitz}) 
\[\begin{split}
 u_1(t)=& \frac{\Gamma(\tfrac{3}{2})\Gamma(\tfrac{1}{2}-s)}{\Gamma(\tfrac{3-s+\la}{2})\Gamma(\tfrac{1-s-\la}{2})}
 \cos(t)^s\sin(t){\rm F}_1(\tfrac{s-\la}{2},\tfrac{s+\la+2}{2},s+\tfrac{1}{2};\cos^2(t))\\
& +  \frac{\Gamma(\tfrac{3}{2})\Gamma(s-\tfrac{1}{2})}{\Gamma(\tfrac{s-\la}{2})\Gamma(\tfrac{s+\la+2}{2})}\cos(t)^{1-s}\sin(t){\rm F}_1(\tfrac{3-s+\la}{2},\tfrac{1-s-\la}{2},\tfrac{3}{2}-s;\cos^2(t)),
\end{split}\]
\[\begin{split}
 u_2(t)=& \frac{\Gamma(\tfrac{1}{2})\Gamma(\tfrac{1}{2}-s)}{\Gamma(\tfrac{2+\la-s}{2})\Gamma(\tfrac{-\la-s}{2})}
 \cos(t)^{s}{\rm F}_1(\tfrac{-\la+s-1}{2},\tfrac{\la+s+1}{2},s+\tfrac{1}{2};\cos^2(t))\\
& +  \frac{\Gamma(\tfrac{1}{2})\Gamma(s-\tfrac{1}{2})}{\Gamma(\tfrac{-\la+s-1}{2})\Gamma(\tfrac{\la+s+1}{2})}\cos(t)^{1-s}
{\rm F}_1(\tfrac{2+\la-s}{2},\tfrac{-\la-s}{2},\tfrac{3}{2}-s;\cos^2(t))
\end{split}\]

We see that when $\la=s-2m-2$ with $m\in \N$, the function $u_2(t)$ (that is even in $t$) becomes
\[ u_2(t)= \frac{\Gamma(\tfrac{1}{2})\Gamma(s-\tfrac{1}{2})}{\Gamma(m+\tfrac{1}{2})\Gamma(s-m-\tfrac{1}{2})}\cos(t)^{1-s}
{\rm F}_1(-m,1+m-s,\tfrac{3}{2}-s;\cos^2(t)).\]
The asymptotic expansion of $u_2(t)f_s(\bar{x})/\cos(t)$ at $t=\pm \pi/2$ contains only terms of the form $\cos(t)^{-s+2\ell}$ for $\ell\in \N$, which matches the asymptotic expansion of elements in the range of the residue of  
$R_{\, \Box_g}(\la)$ at $\la=-2m-2-s$.  The function $u_1(t)$ has nontrivial terms of the form $\cos(t)^{s}$ and $\cos(t)^{1-s}$ in its asymptotic expansion at $t=\pm \pi/2$, and thus $u_1(t)f_s(\bar{x})/\cos(t)$ does not appear as a resonance state. If now $\la=s-2m-3$, then we see that 
\[u_1(t)=\frac{\Gamma(\tfrac{3}{2})\Gamma(s-\tfrac{1}{2})}{\Gamma(m+\tfrac{3}{2})\Gamma(s-m-\tfrac{1}{2})}\cos(t)^{1-s}\sin(t){\rm F}_1(-m, 2+m-s,\tfrac{3}{2}-s;\cos^2(t))\]
has the correct asymptotic expansion at $t=\pm \pi/2$ to be in the range of the residue of $R_{\Box_g}(\la)$ at $\la=s-2m-3$, while $u_2(t)$ is not.  
 This means that elements in the range of the residue at $\la=-2m-2+s$ (resp.   $\la=-2m-3+s$)
 must be linear combinations of functions of the form 
 $f_s^{(n)}u_2(t)/\cos(t)$ (resp. $f_s^{(n)}u_1(t)/\cos(t)$.
\end{proof}

\appendix

\section{Resolvent of \texorpdfstring{$\Box_{g}$}{the Laplacian} on $\AdS_3$}\label{app:resolvent}
In this Appendix, we prove Propositions \ref{prop:resolventAdS} and \ref{prop:resolventAdShol}.
We start by analysing the resolvent by decomposing the functions on $L^2(\AdS_3)$ into $K$-types, i.e.\ 
in Fourier modes on $\TT^2$ when we view $\AdS_3$ as $\R^+\times \TT^2$. This analysis is mostly contained 
in the works of Rossmann \cite{Rossmann} and Andersen \cite{Andersen}. 
For this, we use the coordinates $(t,\theta,\varphi)\in \R^+\times \mathbb{S}^1\times \mathbb{S}^1$ introduced in Section \ref{sec:torus_coord}.
The Laplacian is then given by the expression \eqref{Boxg_t_coordinate} and acts on the Hilbert space 
\[ L^2(\R^+\times \mathbb{S}^1\times \mathbb{S}^1, \tfrac{1}{2}\sinh(2t)dtd\theta \d\varphi).\]
We see from \eqref{Boxg_t_coordinate} that $\Box_g$ preserves the Fourier modes $e_{k,\ell}(\varphi,\theta):=e^{i\ell \varphi+ik\theta}$ for each $(k,\ell)\in \Z^2$, and the resolvent $(\Box_g+\la(\la+2))^{-1}$ does as well for ${\rm Re}(\la)>-1$ with $\la\notin \N$.  Let us describe the resolvent restricted to functions $f(t)e_{k,\ell}(\varphi,\theta)$.
First, by conjugating $\Box_g$ by $\sqrt{\sinh(2t)}$, we see that $\Box_g$ is unitarily equivalent to 
\[ \sqrt{\sinh(2t)}\Box_g \frac{1}{\sqrt{\sinh(2t)}}= -\pl_t^2-\frac{1}{\sinh(t)^2}\pl_{\varphi}^2+\frac{1}{\cosh(t)^2}\pl_{\theta}^2
+\frac{1}{4\cosh(t)^2}-\frac{1}{4\sinh(t)^2}+1\] 
acting on  $\mc{H}:=L^2(\R^+\times \mathbb{S}^1\times \mathbb{S}^1, dtd\theta \d\varphi)$.
Using Fourier series in $(\theta,\varphi)$, the Hilbert space 
$\mc{H}$ 
decomposes as a direct sum 
\[ \mc{H}=\bigoplus_{k,\ell=\in \Z}\mc{H}_{k,\ell}, \quad \mc{H}_{k,\ell}= \{f(t)e_{k,\ell}\,|\, f\in L^2(\R^+,dt)\}\]
and the conjugated pseudo Riemannian Laplacian is unitary equivalent to the direct sum of operators 
\[ P_{k,\ell }:=-\pl_t^2+\frac{1/4-k^2}{\cosh(t)^2}+\frac{\ell^2-1/4}{\sinh(t)^2}+1\]
acting on $\mc{H}_{k,\ell}$.
These are Schr\"odinger operators with P\"oschl-Teller potentials, which are analyzed in detail in \cite[Appendix]{Guillope-Zworski_JFA}. Using this reference, we deduce that the unique solutions of 
$(P_{k,\ell}+\la(\la+2))f=0$ which satisfy the Dirichlet condition at $u=0$ are necessarily constant multiples of 
\[ E_{\la,k,\ell}(t)=\sinh(t)^{\frac{1}{2}+|\ell|}\cosh(t)^{\frac{1}{2}+|k|}{\rm F}_1(\tfrac{|\ell|+|k|-\la}{2},\tfrac{|\ell|+|k|+\la+2}{2},|\ell|+1,-\sinh(t)^2)\] 
where $_{2}{\rm F}_1(a,b,c;z)$ is the classical hypergeometric function. The solution that decays for ${\rm Re}(\la)>-1$ as $t\to \infty$ is unique up to a multiplicative constant and given by 
\[ G_{\la,k,\ell}(t)= \sinh(t)^{-\frac{3}{2}-\la+|k|}\cosh(t)^{\frac{1}{2}-|k|}{\rm F}_1\Big(-\tfrac{(|\ell|+|k|-\la-2)}{2},\tfrac{|\ell|-|k|+\la+2}{2},\la+2,\frac{-1}{\sinh(t)^2}\Big).\]
The asymptotic expansion at $t\to \infty$ (which corresponds to $\pl \AdS^+_3$) is given by 
\[ \begin{split}
E_{\la,k,\ell}(t)\sim & \frac{\Gamma(|\ell|+1)\Gamma(\la+1)}{\Gamma(\tfrac{|\ell|+|k|+\la+2}{2})\Gamma(\tfrac{|\ell|-|k|+\la+2}{2})}e^{(\la+1)t}(1+\sum_{n\geq 1}c_n(\la) e^{-2nt})+\\
& +\frac{\Gamma(|\ell|+1)\Gamma(-\la-1)}{\Gamma(\tfrac{|\ell|+|k|-\la}{2})\Gamma(\tfrac{|\ell|-|k|-\la}{2})}e^{-(\la+1)t}(1+\sum_{n\geq 1}d_n(\la) e^{-2nt})
\end{split}\]
for some constants $c_n(\la),d_n(\la)$. By Sturm-Liouville theory, the resolvent 
\[R_{k,\ell}(\la):=(P_{k,\ell}+(\la+1)^2)^{-1}\]
of $P_{k,\ell}$ is bounded on $\mc{H}_{k,\ell}=L^2(\R^+,dt)$ for ${\rm Re}(\la)>-1$, and its integral kernel is
\[ R_{k,\ell}(\la; t,t')=W_{k,\ell}(\la)\Big(\textbf{1}_{t>t'}G_{\la,k,\ell}(t)E_{\la,k,\ell}(t')+\textbf{1}_{t<t'}G_{\la,k,\ell}(t')E_{\la,k,\ell}(t)\Big)\]
with the inverse of the Wronskian of $G_{\la,k,\ell}$ and $E_{\la,k,\ell}$  given by 
\[W_{k,\ell}(\la):=\frac{\Gamma(\tfrac{|\ell|+|k|+\la+2}{2})\Gamma(\tfrac{|\ell|-|k|+\la+2}{2})}{2^{\la+2}\Gamma(|\ell|+1)\Gamma(\la+2)}.\]
It is a bounded operator on $L^2(\R^+)$ if ${\rm Re}(\la)>-1$, and by the fact that $\Box_g$ is self-adjoint, we also have in this region: for all $f\in L^2(\R^+\times \TT^2,\sinh(2t)dtd\theta \d\varphi)$ written 
as $f=\sum_{k,\ell}f_{k,\ell}e_{k,\ell}$ with $f_{k,\ell}\in L^2(\R^+,\sinh(2t)dt)$
\begin{equation}\label{R_Ktype}
(R_{\Box_g}(\la)\frac{f}{{\sqrt{\sinh(2\cdot)}}})(t,\varphi,\theta)=\frac{1}{\sqrt{\sinh(2t)}}\sum_{k,\ell}e_{k,\ell}(\varphi,\theta)(R_{k,\ell}(\la) f_{k,\ell})(t)
\end{equation}
and (by spectral theory of self-adjoint operators) the $L^2\to L^2$ norm of $R_{k,\ell}(\la)$ is uniformly bounded with respect to $k,\ell$ for $\la$ in compact sets of $\{ \la\in \C\setminus \N\,|, {\rm Re}(\la)>-1\}$.
For each $k,\ell$, the resolvent $R_{k,\ell}(\la)$ extends meromorphically as map $L^2_{\rm comp}(\R^+)\to L^2_{\rm loc}(\R^+)$ for each $k,\ell\in \Z$. Since the poles of $G_{\la,k,\ell}$ are simple and contained in $-2+\N$, $G_{\la,k,\ell}/\Gamma(\la+2)$ is analytic, and since $E_{\la,k,\ell}$ is also analytic in $\la\in \C$, we deduce that the poles of $R_{k,\ell}(\la)$ are at most of order $2$ and contained in 
\begin{equation}\label{Zkell} 
\mc{Z}_{k,\ell}=-2-|\ell|+|k|-2\N.
\end{equation}
When $\la_0\in \mc{Z}_{k,\ell}$, the Wronskian of $G_{\la_0,k,\ell}$ and $E_{\la_0,k,\ell}$ is $0$ thus $G_{\la_0,k,\ell}(t)=c_{\la_0,k,\ell}E_{\la_0,k,\ell}(t)$ for some constant $c_{\la_0,k,\ell}$.
 When $W_{k,\ell}(\la)$ has a pole of order $2$ at $\la_0\in \mc{Z}_{k,\ell}$, then we also have that 
 $\pl_\la G_{\la_0,k,\ell}(t)=d_{\la_0,k,\ell}\pl_\la E_{\la_0,k,\ell}(t)$ for some constant $d_{\la_0,k,\ell}\in \C$. 
 This implies that the polar part of $R_{k,\ell}(\la)$ is an operator of finite rank with range given by  $\C E_{\la_0,k,\ell}$  if $\la_0$ is a simple pole and by  $\C E_{\la_0,k,\ell}\oplus \C \pl_\la E_{\la_0,k,\ell}$. 
To prove that the sum over $k,\ell$ of $R_{k,\ell}(\la)$ for ${\rm Re}(\la)\leq -1$ makes sense as a continuous 
operator mapping $C_c^\infty(\AdS_3)$ to $\mc{D}'(\AdS_3)$, one would need to estimate some operator 
norms of $R_{k,\ell}(\la)$ as a function of $k,\ell$, which is quite tedious. As we also want a more tractable expression for the resolvent integral kernel, we shall proceed differently and will use the Plancherel formula for $\SL(2,\R)$.

\begin{lemma}\label{lem:rz1}If $\lambda\in \C$ is such that $-\lambda(\lambda+2)$ lies outside the spectrum of $\Box$ and $\Re \lambda>0$, then the resolvent kernel $R(\lambda;\cdot,\cdot)$ satisfies
\[
R(\lambda;x,y)=C_0\cdot\begin{dcases}f_{-\lambda(\lambda+2)}(-q(x,y)),& \text{if } |q(x,y)|<1,\\
f_{-\lambda(\lambda+2)}^+(-q(x,y)), &\text{if }  -q(x,y)>1,\\
f_{-\lambda(\lambda+2)}^-(-q(x,y)),&\text{if }  -q(x,y)<-1,\end{dcases}
\]
for some universal constant $C_0>0$ and the functions $f_z,f^\pm_z$ defined respectively in $|\zeta|<1$ and $\pm \zeta>1$ are as follows:
\begin{align*}
f_z(\zeta)&=\frac{1}{2 i\sqrt{1 - \zeta^2}}\sum_{n=1}^\infty \frac{n\big((\zeta +i \sqrt{1 - \zeta^2})^{-n} -  (\zeta -i \sqrt{1 - \zeta^2})^{-n}\big)}{1-n^2-z},\\
f_z^+(\zeta)&=\frac{1}{\sqrt{\zeta^2-1}}\bigg(\sum_{n=1}^\infty\frac{n\big(|\zeta|+\sqrt{\zeta^2-1}\big)^{-n}}{1-n^2-z} +\int_0^\infty \frac{{s}\cos({s} t)}{({s}^2+1-z)\tanh(\pi {s})}\,d{s}\bigg),\\
f_z^-(\zeta)&=\frac{1}{\sqrt{\zeta^2-1}}\bigg(\sum_{n=1}^\infty\frac{n(-1)^{n+1}\big(|\zeta|+\sqrt{\zeta^2-1}\big)^{-n}}{1-n^2-z} -\int_0^\infty  \frac{{s}\cos({s} t)}{({s}^2+1-z)\sinh(\pi {s})}\,d{s}\bigg),
\end{align*}
where for $\zeta \in \R$ with $|\zeta|>1$ we write $t:=\log(|\zeta|+ \sqrt{\zeta^2-1})$. 
\end{lemma}
\begin{proof}[Proof of Lemma \ref{lem:rz1}]
Consider the diffeomorphism 
\bq
 \psi:\AdS_3\to \SL(2,\R)=:G,  \quad \psi(x)=\left(\begin{array}{cc}
x_1+x_3 & x_2+x_4 \\
x_4-x_2 & x_1- x_3
\end{array}\right).\label{eq:psidiff}
\eq
By Section \ref{sec:SL2Coord}, we have $\psi^\ast (-\mathcal{C})=\Box_g$, where $\mathcal C:\CT(G)\to \CT(G)$ is the Casimir operator of $G=\SL(2,\R)$. Thus it suffices to study the spectral theory of the latter. The pushforward, $dh:=\psi_\ast{\rm dv}_{g}$ is a Haar measure on $G$, which differs by an overall multiplicative constant $C_0>0$ from the Haar measure used in \cite[Sec.~X.2, Eq.~(10.7)]{knapp2}. We now use the Fourier-inversion formula \cite[II.7, (2.25)]{knapp2}: for all $f\in \CT(G)$,  $g\in G$ one has
\begin{multline}
 C_0f(g)=4\sum_{n=1}^\infty n\,(\mathrm{Tr}(\mathscr D^+_{n+1}+\mathscr D^-_{n+1})\ast f)(g)\\ + \int_{\R} (\mathrm{Tr}(\mathscr{P}^{+,i{s}})\ast f)(g)\, {s} \tanh\Big(\frac{\pi{s}}{2}\Big)\,d{s}+ \int_{\R} (\mathrm{Tr}(\mathscr{P}^{-,i{s}})\ast f)(g)\, {s} \coth\Big(\frac{\pi{s}}{2}\Big)\,d{s},\label{eq:Fourierinv}
\end{multline}
where the convolution $\ast$ on $G$ is defined by 
\bq
({\rm F}_1\ast f_2)(g):= \int_G {\rm F}_1(h)f_2(gh)\d h=\int_G {\rm F}_1(g^{-1}h)f_2(h)\d h,\qquad {\rm F}_1,f_2\in L^1_\mathrm{loc}(G),\label{eq:conv}
\eq
and the functions $\mathrm{Tr}(\mathscr D^+_{n+1}+\mathscr D^-_{n+1}),\mathrm{Tr}(\mathscr{P}^{\pm,i{s}})\in L^1_\mathrm{loc}(G)$ are explicitly known \cite[Sec.~X.2]{knapp2}: For $t,\theta\in \R$, write $$
a_t:=\begin{pmatrix}e^t &\0 \\ 0 & e^{-t}\end{pmatrix},\qquad k_\theta:=\begin{pmatrix}\cos \theta &\sin  \theta \\ -\sin \theta & \cos \theta\end{pmatrix}\quad \in G.
$$
Every $g\in G$ with $|\mathrm{tr}(g)|<2$ (resp.\ $|\mathrm{tr}(g)|>2$)  is conjugate to exactly one element of the form $k_\theta$ with $\theta\in (0,\pi)$ (resp.\ $\sgn(\mathrm{tr}(g))a_t$ with $t\in (0,\infty)$), and the set of elements $g\in G$ with $|\mathrm{tr}(g)|=2$ has Haar measure zero. One then has the explicit formulas
\begin{align}\begin{split}
\mathrm{Tr}(\mathscr D^+_{n+1}+\mathscr D^-_{n+1})(h)&=\begin{dcases}(\pm 1)^{n+1}\frac{e^{-nt}}{|\sinh(t)|},&\text{ if }h\sim \pm a_t,\;t>0,\\
-\frac{\sin(n\theta)}{\sin(\theta)},&\text{ if }h\sim k_\theta
\end{dcases}\label{eq:TrD}\end{split}
\end{align}
by \cite[Sec.~X.2, Cor.~10.13]{knapp2}, while by \cite[Sec.~X.2, Prop.~10.12]{knapp2} we have
\begin{align}\begin{split}
\mathrm{Tr}(\mathscr{P}^{+,i{s}})(h)&=\begin{dcases}\frac{\cos({s} t)}{|\sinh(t)|},&\text{ if }h\sim \pm a_t,\\
0,&\text{ if }h\sim k_\theta,
\end{dcases}\\
\mathrm{Tr}(\mathscr{P}^{-,i{s}})(h)&=\begin{dcases}\pm \frac{\cos({s} t)}{|\sinh(t)|},&\text{ if }h\sim \pm a_t,\\
0,&\text{ if }h\sim k_\theta.
\end{dcases}\label{eq:TrP}\end{split}
\end{align}
Let $f\in \Cinft(\R)$ be in the principal series representation $\mathscr P^{\pm,w}$ with parameter $w\in \C$ in the notation of \cite[p.~38]{knapp2}. That is, an element $g=\begin{pmatrix}
a & b\\
c & d
\end{pmatrix}\in \SL(2,\R)$ acts on $f$ according to
\[
(g\cdot f)(x)=\begin{dcases}|-bx+d|^{-1-w}f\Big(\frac{ax-c}{-bx+d}\Big),& \text{if }f\in \mathscr P^{+,w},\\
\sgn(-bx+d)|-bx+d|^{-1-w}f\Big(\frac{ax-c}{-bx+d}\Big),& \text{if }f\in \mathscr P^{-,w}.
\end{dcases}
\]
Then one can verify by direct computation that the action of the Casimir operator $\mathcal C$ on $f$ is given by
\bq
\mathcal{C} f=(w^2-1)f.\label{eq:Casimireigenvalue1}
\eq
In particular, since $\mathscr D^+_{n+1}\oplus\mathscr D^-_{n+1}\subset \mathscr P^{+,n}$ if $n$ is odd and $\mathscr D^+_{n+1}\oplus\mathscr D^-_{n+1}\subset \mathscr P^{-,n}$ if $n$ is even \cite[II,(2.19)]{knapp2}, \eqref{eq:Casimireigenvalue1} implies that
\bq
\mathcal Cf = (n^2-1)f\qquad \forall\; f\in \mathscr D^+_{n+1}\oplus\mathscr D^-_{n+1},\; n\in \N.\label{eq:CasimirEigenvalueDpm}
\eq
In view of \eqref{eq:Casimireigenvalue1} and \eqref{eq:CasimirEigenvalueDpm}, the Fourier inversion Formula \eqref{eq:Fourierinv} diagonalizes the resolvent
\[
\tilde R(z):=(-\mathcal C-z)^{-1}:\CT(G)\to \mc{D}'(G)
\]
according to
\begin{multline*}
C_0(\tilde R(z)f)(g)=4\sum_{n=1}^\infty \frac{n\,(\mathrm{Tr}(\mathscr D^+_{n+1}+\mathscr D^-_{n+1})\ast f)(g)}{1-n^2-z}\\ + \int_{\R} \frac{(\mathrm{Tr}(\mathscr{P}^{+,i{s}})\ast f)(g)}{{s}^2+1-z}\, {s} \tanh\Big(\frac{\pi{s}}{2}\Big)\,d{s}+ \int_{\R} \frac{(\mathrm{Tr}(\mathscr{P}^{-,i{s}})\ast f)(g)}{{s}^2+1-z}\, {s} \coth\Big(\frac{\pi{s}}{2}\Big)\,d{s}.
\end{multline*}
From this and \eqref{eq:conv} we read off that the Schwartz kernel $\tilde R(z;\cdot,\cdot)\in \D'(G\times G)$ of $\tilde R(z)$ reads (with the embedding $L^1_\mathrm{loc}(G\times G)\to \D'(G\times G)$ given by integration with respect to the measure $dh \times dh)$
\begin{multline}
C_0\tilde R(z;g,h)=4\sum_{n=1}^\infty \frac{n\,\mathrm{Tr}(\mathscr D^+_{n+1}+\mathscr D^-_{n+1})(g^{-1}h)}{1-n^2-z}\\ + \int_{\R} \frac{\mathrm{Tr}(\mathscr{P}^{+,i{s}})(g^{-1}h)}{{s}^2+1-z}\, {s} \tanh\Big(\frac{\pi{s}}{2}\Big)\,d{s}+ \int_{\R} \frac{\mathrm{Tr}(\mathscr{P}^{-,i{s}})(g^{-1}h)}{{s}^2+1-z}\, {s} \coth\Big(\frac{\pi{s}}{2}\Big)\,d{s}.
\label{eq:firsttilderz}
\end{multline}
To get the Schwartz kernel $´R(\lambda;\cdot,\cdot)\in \D'(\AdS_3\times \AdS_3)$, we compose $\tilde R(z;\cdot,\cdot)$ with $\psi\times \psi$ and put $z=-\lambda(\lambda+2)$. For $x,y\in \AdS_3$, write $g_x:=\psi(x),g_y:=\psi(y)\in G$. Then a direct computation reveals that the product matrix $g_x^{-1}g_y\in G$ satisfies
\bq
\mathrm{Tr}(g_x^{-1}g_y)=-2q(x,y)\label{eq:trq}
\eq
and
\begin{align}
g_x^{-1}g_y \sim \pm a_t,\;t>0 &\iff \begin{cases}|q(x,y)|>1,\; \mp q(x,y)>0,\\
t=\log\big(|q(x,y)|+\sqrt{q(x,y)^2-1}\big),\end{cases}\label{eq:gxyt}\\
g_x^{-1}g_y \sim k_\theta &\iff |q(x,y)|< 1,\; \cos\theta=-q(x,y),\; \sin \theta =\sqrt{1-q(x,y)^2}.\label{eq:gxytheta}
\end{align}
Now the infinite sum
\[
\sum_{n=1}^\infty \frac{n\,(\mathrm{Tr}(\mathscr D^+_{n+1}+\mathscr D^-_{n+1})(g^{-1}h)}{1-n^2-z}
\]
featured in \eqref{eq:firsttilderz} can be evaluated: Using \eqref{eq:TrD} and  \eqref{eq:gxytheta}
, one finds in the case $|q(x,y)|< 1$ the expression
\begin{align*}
\sum_{n=1}^\infty& \frac{n\,\mathrm{Tr}(\mathscr D^+_{n+1}+\mathscr D^-_{n+1})(g_x^{-1}g_y)}{1-n^2-z}=-\frac{1}{\sin(\theta)}\sum_{n=1}^\infty \frac{n\sin(n\theta)}{1-n^2-z}.
\end{align*}
in terms of the 
angle $\theta=\arccos(q(x,y))$. Now we use that for $\zeta\in (-1,1)$ one has
\begin{align*}
\frac{\sin(n\arccos(\zeta))}{\sin(\arccos(\zeta))}&=-\frac{1}{2} \frac{(\zeta +i \sqrt{1 - \zeta^2})^{-n} -  (\zeta -i \sqrt{1 - \zeta^2})^{-n}}{i\sqrt{1 - \zeta^2}},
\end{align*}
which is the well-known Chebychev polynomial $U_{n-1}(\zeta)$ of the second kind.

In the case $|q(x,y)|> 1$, $\mp q(x,y)>0$ we compute with \eqref{eq:TrD} and  \eqref{eq:gxyt}:
%
\begin{align*}
\sum_{n=1}^\infty \frac{n\,\mathrm{Tr}(\mathscr D^+_{n+1}+\mathscr D^-_{n+1})(g_x^{-1}g_y)}{1-n^2-z}
&=\frac{1}{\sqrt{q(x,y)^2-1}}\sum_{n=1}^\infty\frac{n(\pm 1)^{n+1}\big(|q(x,y)|+\sqrt{q(x,y)^2-1}\big)^{-n}}{(1-n^2-z)}.
\end{align*}
For the integrals in  \eqref{eq:firsttilderz} we find using \eqref{eq:TrP}: They vanish  if $|q(x,y)|< 1$, and if $|q(x,y)|> 1$, $\mp q(x,y)>0$, then we have
\begin{align*}
\int_{\R} \frac{\mathrm{Tr}(\mathscr{P}^{+,i{s}})(g_x^{-1}g_y)}{{s}^2+1-z}\, {s} \tanh\Big(\frac{\pi{s}}{2}\Big)\,d{s}&=\frac{1}{\sqrt{q(x,y)^2-1}}\int_{\R} \frac{\cos({s} t)}{{s}^2+1-z}\, {s} \tanh\Big(\frac{\pi{s}}{2}\Big)\,d{s},\\
\int_{\R} \frac{\mathrm{Tr}(\mathscr{P}^{-,i{s}})(g_x^{-1}g_y)}{{s}^2+1-z}\, {s} \coth\Big(\frac{\pi{s}}{2}\Big)\,d{s}&=\frac{\pm 1}{\sqrt{q(x,y)^2-1}}\int_{\R} \frac{\cos({s} t)}{{s}^2+1-z}\, {s} \coth\Big(\frac{\pi{s}}{2}\Big)\,d{s},
\end{align*}
where, as before, $t=\log\big(|q(x,y)|+\sqrt{q(x,y)^2-1}\big)$. To calculate the sum of the two integrals, we use the identities
\begin{align}\begin{split}
\tanh\Big(\frac{s}{2}\Big)- \coth\Big(\frac{s}{2}\Big)&=\frac{-2}{\sinh(s)},\\ 
\tanh\Big(\frac{s}{2}\Big)+ \coth\Big(\frac{s}{2}\Big)&=\frac{2}{\tanh(s)},\qquad s\in \R\setminus\{0\}.\label{eq:identtanhcoth}\end{split}
\end{align}
This yields in the case $-q(x,y)<-1$:
\begin{align*}
&\int_{\R} \frac{\mathrm{Tr}(\mathscr{P}^{+,i{s}})(g_x^{-1}g_y)}{{s}^2+1-z}\, {s} \tanh\Big(\frac{\pi{s}}{2}\Big)\,d{s}+\int_{\R} \frac{\mathrm{Tr}(\mathscr{P}^{-,i{s}})(g_x^{-1}g_y)}{{s}^2+1-z}\, {s} \coth\Big(\frac{\pi{s}}{2}\Big)\,d{s}\\
& =\frac{-2}{\sqrt{q(x,y)^2-1}}\int_{\R} \frac{{s}\cos({s} t)}{({s}^2+1-z)\sinh(\pi {s})}\,d{s},
\end{align*}
while in the case $-q(x,y)>1$ we get
\begin{align*}
&\int_{\R} \frac{\mathrm{Tr}(\mathscr{P}^{+,i{s}})(g_x^{-1}g_y)}{{s}^2+1-z}\, {s} \tanh\Big(\frac{\pi{s}}{2}\Big)\,d{s}+\int_{\R} \frac{\mathrm{Tr}(\mathscr{P}^{-,i{s}})(g_x^{-1}g_y)}{{s}^2+1-z}\, {s} \coth\Big(\frac{\pi{s}}{2}\Big)\,d{s}\\
& =\frac{2}{\sqrt{q(x,y)^2-1}}\int_{\R} \frac{{s}\cos({s} t)}{({s}^2+1-z)\tanh(\pi {s})}\,d{s}.
\end{align*}
Plugging these formulas back into \eqref{eq:firsttilderz}, dividing by $C_0$, and noting that for an even integrand we can replace $\frac{1}{2}\int_{\R}$ by $\int_{0}^\infty$ finishes the proof.
\end{proof}
We shall now simplify the expression of the resolvent integral kernel given in Lemma \ref{lem:rz1} and show that it can be expressed as in Proposition \ref{prop:resolventAdS}.  
First, we show that 
\begin{lemma}\label{f^pmformule}
The functions $f^\pm_{-\la(\la+2)}(\zeta)$ are given by 
\begin{align*}
& f^+_{-\la(\la+2)}(\zeta)=\frac{\pi}{2\sqrt{\zeta^2-1}}\frac{(\zeta+\sqrt{\zeta^2-1})^{-\la-1}}{\tan(\pi(\la+1))},\\
&  f^-_{-\la(\la+2)}(\zeta)=-\frac{\pi}{2\sqrt{\zeta^2-1}}\frac{(|\zeta|+\sqrt{\zeta^2-1})^{-\la-1}}{\sin(\pi(\la+1))}.
\end{align*}
\end{lemma}
\begin{proof}
Take $z=-\la(\la+2)$ with ${\rm Im}(\la)>0$ and ${\rm Re}(\la)>-1$. 
We have
\[\begin{split}
q_z(t):=\int_0^\infty \frac{{s}\cos({s} t)}{({s}^2+1-z)\tanh(\pi {s})}\,d{s}= & 
\frac{1}{2} \int_0^\infty \frac{{s}(e^{is}+e^{-ist})}{({s}^2+1-z)\tanh(\pi {s})}\,d{s}\\
=&\frac{1}{2} \int_\R  \frac{{s}e^{ist}}{({s}^2+1-z)\tanh(\pi {s})}\,d{s}  \end{split}\]
We can deform the contour of integration to $s\in ia+\R$ for any $a>{\rm Re}(\la)+1$ with $a\not \in \N$, 
use that $s^2+1-z=s^2+(\la+1)^2$. The poles of the meromorphic 
function 
\[F_{z,t}(s)= \frac{{s}e^{ist}}{2({s}^2+1-z)\tanh(\pi {s})}\] 
inside the integral are (those in ${\rm Im}(s)>0$) at $s_n=in$ with $n\in \N$ and at $s=\pm i(\la+1)$. By the residue theorem, we get, for $N$ the largest integer smaller than $a$  
\[\begin{split}
q_z(t)= & 2\pi i( \sum_{n=1}^N {\rm Res}_{s=n}F_{z,t}(s)+{\rm Res}_{s=i(\la+1)} F_{z,t}(s))\\
& +\frac{e^{-at}}{2}\int_{\R}\frac{(s+ia)e^{ist}}{((ia+t)^2+(\la+1)^2)\tanh(\pi(ia+s))}ds\\
=& -\sum_{n=1}^N \frac{ne^{-nt}}{(\la+1)^2-n^2}+\frac{\pi }{2}\frac{e^{-(\la+1)t}}{\tan(\pi(\la+1))}+\mc{O}(e^{-a t}).
\end{split}\]
Note that this can be done for any $a>0$ large. 
Thus we get 
\[\begin{split}
f_z^+(\zeta)=&\frac{1}{\sqrt{\zeta^2-1}}\bigg(\sum_{n=1}^\infty\frac{n\big(\zeta+\sqrt{\zeta^2-1}\big)^{-n}}{1-n^2-z} +\int_0^\infty  \frac{{s}\cos({s} t)}{({s}^2+1-z)\sinh(\pi {s})}\,d{s}\bigg)\\
=& \frac{\pi}{2\sqrt{\zeta^2-1}}\frac{(\zeta+\sqrt{\zeta^2-1})^{-\la-1}}{\tan(\pi(\la+1))}.
 \end{split}\]
 The same argument applies to $f^-_{z}$.
\end{proof}
We see in particular that $f^\pm(\zeta)$ decay as $|\zeta|\to \infty$ if ${\rm Re}(\la)>-1$, which matches the fact that the resolvent must map to $L^2(\AdS_3)$.

Next, we want to express also $f_z(\zeta)$ in a simpler form for $|\zeta|<1$, and to compute the constant $C_0$. The constant turns out to be related to the singularity of $f_z(\zeta)$ as $\zeta\to 1^-$, which is the light cone.
We now solve for a fundamental solution of $(\Box_{g}+\la(\la+2))$, i.e.\ we
want to solve the following equation for ${\rm Re}(\la)>0$ and $\la\notin \N$, in the distribution sense in the model $\mathbb{S}_\theta^1\times \mathbb{D}_z$ of $\AdS_3$, 
\[ (\Box_{g}+\la(\la+2))F(e^{i\theta},z)= \delta_{(1,0)}\]
where $\delta_{(1,0)}$ is the Dirac mass at $(e^{i\theta},z)=(1,0)$, with some decay as $(e^{i\theta},z)$ approaches the conformal boundary $\pl \AdS_3$ as we want $R(\la)$ to map $C_c^\infty(\AdS_3)$ to $L^2(\AdS_3)$ functions if ${\rm Re}(\la)>0$. 
As before, we use the coordinates $z=e^{i\varphi}\frac{\sinh(t)}{1+\cosh(t)}$ with $t\in \R^+$ and let 
\[\zeta(\theta,t)=\cosh(t)\cos(\theta)=-q\Big({\rm e},\tilde{\psi}^{-1}(e^{i\theta},\tfrac{\sinh(t)}{1+\cosh(t)}e^{i\varphi})\Big).\]
By Lemma \ref{lem:rz1}, since the resolvent kernel is a convolution kernel, we are searching for $F$ of the form $F(\theta,t)=f(\zeta(\theta,t))$ for some $f$ to be determined. We see from \eqref{Boxg_t_coordinate} that $f$ must solve the ODE 
\begin{equation}\label{ODEhyper}
 (1-\zeta^2)f''(\zeta) -3\zeta f'(\zeta)+\la(\la+2)f(\zeta)=0 
 \end{equation}
for $\zeta\not=1$ and we need that $f(\zeta)$ decay as $|\zeta|\to \infty$ so as to obtain an $L^2$ function outside $0$.
Making the change of coordinates $u=(1-\zeta)/2$, this reduces to the hypergeometric equation
\[ u(1-u)G''(u) -\frac{3}{2}(2u-1)G'(u)+\la(\la+2)G(u)=0.\] 
The independent solutions in $u\in (0,1)$ are the hypergeometric functions 
\[
  F_{1}(-\lambda,2+\lambda,\frac{3}{2};u), \quad u^{-1/2}F_{1}(-\frac{1}{2}-\lambda,\frac{3}{2}+\lambda, \frac{1}{2};u), \quad  \textrm{ in }u\in (0,1)\]
\[  u^{\la}{\rm F}_1(-\lambda,-\tfrac{1}{2}-\lambda,-1-2\la,\tfrac{1}{u}) ,\quad u^{-2-\la}{\rm F}_1(2+\la,\tfrac{3}{2}+\la,3+2\la,\tfrac{1}{u})
\quad \textrm{ in }u>0\]
 \[(1-u)^{\la}{\rm F}_1(-\lambda,-\tfrac{1}{2}-\lambda,-1-2\la,\tfrac{1}{1-u}) ,\quad (1-u)^{-2-\la}{\rm F}_1(2+\la,\tfrac{3}{2}+\la,3+2\la,\tfrac{1}{1-u}) \quad  \textrm{ in }u<0.\]
Only the solutions with exponents $u^{-2-\la}$ and $(1-u)^{-2-\la}$ are decaying as $|u|\to \infty$ when ${\rm Re}(\la)>1$. This means that in the region $u>1$ and $u<0$,  the resolvent kernel $R(\lambda;x,{\rm e})$ must be of the form 
\[ R_{\Box_g}(\lambda;x,{\rm e})=\left\{\begin{array}{ll}
A_0(\lambda) u(x)^{-2-\la}{\rm F}_1(\tfrac{3}{2}+\la,2+\la,3+2\la,\tfrac{1}{u(x)}), & u(x)>1\\
A_1(\lambda) (1-u(x))^{-2-\la}{\rm F}_1(\tfrac{3}{2}+\la,2+\la,3+2\la,\tfrac{1}{1-u(x)}), & u(x)<0
\end{array}\right.\]
for some constant $A_0(\lambda),A_1(\lambda)$, in order to map $C_c^\infty(\AdS_3)$ to $L^2(\AdS_3)$. 
We check that, for $u=(1-\zeta)/2$, one has 
\begin{align*}
&\forall u<0, \,\, (1-u)^{-2-\la}{\rm F}_1(\tfrac{3}{2}+\la,2+\la,3+2\la,\tfrac{1}{1-u})=2\dfrac{4^{\la+1}}{\sqrt{\zeta^2-1}}(\zeta+\sqrt{\zeta^2-1})^{-(\la+1)},\\
& \forall u>1, \,\, u^{-2-\la}{\rm F}_1(\tfrac{3}{2}+\la,2+\la,3+2\la,\tfrac{1}{u})=2\dfrac{4^{\la+1}}{\sqrt{\zeta^2-1}}(-\zeta+\sqrt{\zeta^2-1})^{-(\la+1)}\\
& \forall u\in (0,1), \,\, u^{-1/2}F_{1}(-\frac{1}{2}-\lambda,\frac{3}{2}+\lambda, \frac{1}{2};u)=\dfrac{2}{\sqrt{1-\zeta^2}}(\zeta+i\sqrt{1-\zeta^2})^{-(\la+1)}.
\end{align*}
This is achieved by checking that these functions solve the hypergeometric equation and by comparing the asymptotics at $u=0,1,\infty$. Notice that this matches (up to constant factor) with Lemma \ref{f^pmformule} for the region $|\zeta|>1$.
The proof of the expression \eqref{formuleFlambda} for the resolvent kernel follows from the following Lemma.
\begin{lemma}\label{Lemma_sol_fond}
The function $f_{-\la(\la+2)}(\zeta)$ is given in $|\zeta|<1$ by 
\begin{equation}\label{formule_for_f}
 f_{-\la(\la+2)}(\zeta)=\frac{\pi}{4}\Big(\dfrac{2}{\sqrt{1-\zeta^2}}(\zeta+i\sqrt{1-\zeta^2})^{-(\la+1)} +\frac{e^{-i\pi(\la+1)}}{\sin(\pi(\la+1))}G_\la(\zeta)\Big)
 \end{equation}
where $G_\la(\zeta)$ is defined by \eqref{def_Glambda}, and the constant $C_0$ is given by $C_0=1/(2\pi^2)$.
\end{lemma}
\begin{proof}
Consider the function $F_\la(\zeta)$ given by \eqref{formuleFlambda}, where $C_0(\la),C_1(\la),C_2(\la)$ are three free constants that we will adjust in the end to solve $\Box_{g}F_\la(\zeta)=\delta_{{\rm e}}$ is the Dirac mass at 
${\rm e}$ with respect to the volume measure ${\rm dv}_g$. In the $(t,\varphi,\theta)$ coordinates, ${\rm e}$ is given 
by $(t=0,\theta=0)$ (the parameter $\varphi$ is not important at this point as we use polar coordinates around ${\rm e}$). The volume measure is ${\rm dv}_g=\frac{1}{2}\sinh(2t)dt\wedge \d\varphi \wedge d\theta.$
Let us define for $\eps>0$ the sets 
\[\begin{gathered}
\Omega^{\pm 1}_\eps:=\{ x\in \AdS_3\,|\, |\zeta(x)\mp 1|\geq \eps\}\\
\pl_+\Omega^{\pm 1}_\eps=\{ x\in \AdS_3\,|\, \zeta(x)=\pm 1 +\eps\}, \quad \pl_-\Omega^{\pm 1}_\eps=\{ x\in \AdS_3\,|\, \zeta(x)=\pm 1 -\eps\}.\end{gathered}\]
For $\chi\in C_c^\infty(\AdS_3)$ with support in $\{\pm \zeta>-1\}$ we let 
 \[I_\eps^{\pm 1}:= \int_{\Omega_{\eps}^{\pm 1}} F_\la(\zeta(x)) (\Box_g+\la(\la+2))\chi (x)\,{\rm dv}_g(x).\] 
 By Green's formula and taking into account that $f=F_\lambda$ solves \eqref{ODEhyper}, we get 
 \[ \begin{split}
 I^{\pm 1}_\eps= &F_\la(\pm 1+\eps)\int_{\pl_+\Omega^{\pm 1}_\eps} \iota_{\nabla \chi}\,{\rm dv}_g-
 \int_{\pl_+\Omega^{\pm 1}_\eps} \chi \iota_{\nabla F_\la}\,{\rm dv}_g\\
 & +F_\la(\pm 1-\eps)\int_{\pl_-\Omega^{\pm 1}_\eps}  \iota_{\nabla \chi}\,{\rm dv}_g-
 \int_{\pl_-\Omega^{\pm 1}_\eps} \chi \iota_{\nabla F_\la}\,{\rm dv}_g
\end{split} \] 
 where the gradient $\nabla$ is with respect to $g$; here, our convention is 
 that the boundary terms are oriented in a way that $\iota_{n}{\rm dv}> 0$ if $n$ is the interior normal to $\Omega_\eps^{\pm 1}$. In particular, near $t=\theta=0$, the set $\{\zeta=1-\eps\}$ is oriented by $dt\wedge \d\varphi>0$ and 
 $\{\zeta=1+\eps\}$ is oriented by $\d\varphi\wedge d\theta>0$.
If $\eps\in \R$ is small, 
 the set $\pl_\pm\Omega^1_\eps$ are codimension $1$ submanifolds except when $\eps=0$, where they develop a singularity at $\{t=0,\theta=0\}$ corresponding to the point ${\rm e}\in \AdS_3$. Near ${\rm e}$, we can write 
 \[ \zeta-1= \frac{t^2}{2}-\frac{\theta^2}{2}+\mc{O}(\|(t,\theta)\|^4),\]
 thus we see that for $\eps>0$ small enough, near ${\rm e}$ the level set $\pl_+\Omega^1_\eps$  has one connected component and the  level set $\pl_-\Omega^1_\eps$ has two connected components corresponding to $\theta>0$ and $\theta<0$, while $\{\zeta=1\}\setminus \{ {\rm e}\}$ near ${\rm e}$ is diffeomorphic to two cones corresponding to $\theta>0$ and $\theta<0$, respectively. This is similar for $\pl_\pm \Omega^{-1}_\eps$ near $x=-{\rm e}$ which corresponds to $t=0,\theta=\pi$, and $\zeta+1=-t^2/2+(\theta-\pi)^2/2+\mc{O}(t^4+(\theta-\pi)^4)$.
 We can compute 
 \[ \nabla F_\la|_{\zeta=1\pm \eps}=F_\la '(1\pm \eps) \nabla \zeta, \quad  \nabla F_\la|_{\zeta=-1\pm \eps}=F_\la '(-1\pm \eps) \nabla \zeta\]
as well as the asymptotics for $\zeta\to 1$ (using \eqref{formuleFlambda})
\begin{equation}\label{asympt_Flambda} 
F_\lambda(\zeta)=\left\{\begin{array}{ll}
\dfrac{A_1^+(\la)}{\sqrt{\zeta^2-1}}-C_0^+(\la)(\la+1)+\mc{O}(\sqrt{\zeta^2-1}), & \zeta>1,  \\
\dfrac{A_1^-(\la)}{\sqrt{1-\zeta^2}}-(C_1(\la)+2C_2(\la))(\la+1)+\mc{O}(\sqrt{|\zeta^2-1|}), & \zeta<1,
\end{array}\right.
\end{equation}
with 
\[ A_1^+(\la)=C_0^+(\la), \quad A_1^-(\la)=\frac{C_1(\la)}{i}\]
and as $\zeta\to -1$ 
\[
F_\lambda(\zeta)=\left\{\begin{array}{ll}
\dfrac{A_{-1}^-(\la)}{\sqrt{\zeta^2-1}}-C_0^-(\la)(\la+1)+\mc{O}(\sqrt{\zeta^2-1}), & \zeta<-1,  \\
\dfrac{A_{-1}^+(\la)}{\sqrt{1-\zeta^2}}+(C_1(\la)e^{-i\pi(\la+1)}+2C_2(\la)\cos(i\pi (\la+1)))(\la+1)\\
\qquad+\;\mc{O}(\sqrt{|\zeta^2-1|}), & \zeta>-1,
\end{array}\right.
\]
with 
\[A_{-1}^+(\la)=e^{-i\pi(\la+1)}\frac{C_1(\la)}{i}-2C_2(\la)\sin(\pi(\la+1)), \quad A_{-1}^-(\la)=C_0^-(\la)\]
which gives us the asymptotics for $\eps\to 0^+$ and $s\in \{+,-\}$
\begin{equation}\label{asympt_Flambdaderiv} 
F_\lambda'(\pm 1+s\eps)=-s A_{\pm 1}^s(\la)(2\eps)^{-3/2}+\O(\eps^{-1/2}).
\end{equation}
Let us compute on $\pl_s\Omega_\eps^{\pm 1}$ for $s\in \{+,-\}$
\[ \nabla \zeta=\sinh(t)\cos(\theta)\pl_t+\frac{\sin(\theta)}{\cosh(t)}\pl_{\theta}\]
\[ \begin{split}
\iota_{\nabla \zeta}\,{\rm dv}=&\sinh(t)^2\cosh(t)\cos(\theta)\d\varphi\wedge d\theta-\sinh(t)\sin(\theta)\d\varphi \wedge dt\\
=& (\pm 1+s\eps)\sinh(t)^2 \d\varphi\wedge d\theta-\sinh(t)\sin(\theta)\d\varphi\wedge dt
\end{split}\]
and on the hypersurface $\pl_s\Omega_\eps^{\pm 1}$, we have $dt=\frac{\tan(\theta)}{\tanh(t)}d\theta$
thus
\[ \begin{split}
\iota_{\nabla \zeta}\,{\rm dv}= \frac{(\pm 1+s\eps)}{\cos(\theta)^2}(\eps^2\pm 2s\eps)\d\varphi\wedge d\theta.\end{split}\]
We see that the integral of this form on compact sets is an $\mc{O}(\eps)$. Therefore, by \eqref{asympt_Flambdaderiv} we have for $C_0^+=C_0$ and $C_0^-=C_0/i$ 
\begin{equation}\label{first_asymptotic}
\int_{\pl_s\Omega_\eps^{\pm 1}} \chi \iota_{\nabla F_\la}\,{\rm dv}=- A_{\pm 1}^s(\la)(2\eps)^{-1/2}\int_{\pl_s\Omega_\eps^{\pm 1}} \frac{1}{\cos^2(\theta)}\chi \d\varphi\wedge d\theta +\mc{O}(\eps^{1/2}).
\end{equation}
Here the integral on $\zeta=1$ is constituted of integrals on $4$ conic hypersurfaces (with a singularity at $e$): the limit cone $\mc{C}^+_\pm:=\{\theta>0, \zeta=1^\pm\}$ (the limit from $\zeta=1\pm\eps$ as $\eps\to 0^\pm$) 
is oriented in a way that $\pm \d\varphi \wedge d\theta>0$ when pulled back to $\mc{C}^+_\pm$, and similarly the limit cone $\mc{C}^-_\pm:=\{\theta<0, \zeta=1^\pm\}$ (the limit from $\zeta=1\pm\eps$ as $\eps\to 0^\pm$) is oriented in a way that $\pm \d\varphi\wedge d\theta>0$ when pulled back on $\mc{C}^-_\pm$.

Next we compute on $\pl_s\Omega_{\eps}^{\pm 1}$
\begin{equation}\label{restricted_measure}
\begin{split}
\iota_{\nabla \chi}\,{\rm dv}=& \sinh(t)\cosh(t)\pl_t\chi \d\varphi\wedge d\theta +\frac{\sinh(t)}{\cosh(t)}\pl_{\theta}\chi \d\varphi\wedge dt\\
=& \Big(\frac{\sqrt{\sin(\theta)^2\pm 2s\eps+\eps^2}(\pm 1+s\eps)}{\cos(\theta)^2}\pl_{t}\chi+ \tan(\theta)\pl_{\theta}\chi\Big) \d\varphi\wedge d\theta.
\end{split}
\end{equation}
From \eqref{restricted_measure}, we  observe that
\begin{equation}
F_\la(\pm1+s\eps)\int_{\substack{\pl_s\Omega_\eps^{\pm 1}, \\
|\cos(\theta)\mp 1|<\eps}}  \iota_{\nabla \chi}\,{\rm dv}=\mc{O}(\eps^{1/2}).
\end{equation}
Next we use that on $\pl_s\Omega_\eps^{\pm 1}\cap\{ |\cos(\theta)\mp 1|>\eps\}$ 
\[ \pl_{\theta} (\chi(t(\theta),\theta,\varphi))=(\pl_t \chi)(t,\theta,\varphi)\frac{\tan(\theta)}{\tanh(t)}+(\pl_{\theta}\chi)(t,\theta,\varphi).\]
thus, by \eqref{restricted_measure},
\[\begin{split}
\iota_{\nabla \chi}\,{\rm dv}=& \frac{(\pm 1+s\eps)\sinh(t)^2}{\cosh(t)\sin(\theta)}\pl_{\theta}(\chi(t(\theta),\theta,\varphi))+ (\tan(\theta)-\frac{\sinh(t)^2}{\tan(\theta)})\pl_{\theta}\chi  \d\varphi \wedge d\theta.
\end{split}\]
Now we have on $\pl_s\Omega_\eps^{\pm 1}\cap\{ |\cos(\theta)\mp 1|>\eps\}$
\[\frac{(\pm 1+s\eps)\sinh(t)^2}{\cosh(t)\sin(\theta)}=
\tan(\theta)(1+ \frac{\pm 2s\eps +\eps^2}{\sin(\theta)^2}),\quad (\tan(\theta)-\frac{\sinh(t)^2}{\tan(\theta)})=\frac{\mp 2s\eps -\eps^2}{\cos(\theta)\sin(\theta)}\]
We thus check that 
\[\begin{split}
\int_{\substack{\pl_s\Omega_\eps^{\pm 1}, \\
|\cos(\theta)\mp 1|>\eps}}  \iota_{\nabla \chi}\,{\rm dv}=&\int_{\substack{\pl_s\Omega_\eps^{\pm 1}, \\
|\cos(\theta)\mp 1|>\eps}} \tan(\theta)(1+ \frac{\pm 2s\eps +\eps^2}{\sin(\theta)^2})\pl_{\theta}(\chi(t(\theta),\theta,\varphi))\d\varphi \wedge d\theta\\
& + \mc{O}(\eps\log(\eps))\end{split}\]
and, by integration by parts,
\begin{equation}\label{IPP}
\begin{split}
\int_{\substack{\zeta=\pm(1+ \eps), \\
|\cos(\theta)\mp 1|>\eps}}  \iota_{\nabla \chi}\,{\rm dv}=& \int_{\substack{\zeta=\pm(1+ \eps), \\
|\cos(\theta)\mp 1|>\eps}}(-\frac{1}{\cos(\theta)^2}+ \frac{2\eps}{\sin(\theta)^2})\chi \d\varphi\wedge d\theta\\
&-8\pi \sqrt{2\eps}\chi(\pm {\rm e})+\mc{O}(\eps)
\end{split}\end{equation}
where the last term correspond to the boundary terms at $\cos(\theta)=\pm (1-\eps)$. We also have by similar argument
\begin{equation}\label{term_1-eps}
\begin{split}
\int_{\zeta=\pm(1- \eps)}  \iota_{\nabla \chi}\,{\rm dv}=& \int_{\zeta=\pm(1- \eps)}(-\frac{1}{\cos(\theta)^2}- \frac{2\eps}{\sin(\theta)^2})\chi \d\varphi \wedge d\theta
+\mc{O}(\eps).
\end{split}\end{equation}
We compute, using that $\d\varphi \wedge d\theta\geq 0$ on $\zeta=\pm(1+\eps)$ and $\d\varphi \wedge d\theta\leq 0$ on $\zeta=\pm(1-\eps)$  that
\begin{equation}\label{eps/sin^2}
\int_{\substack{\zeta=\pm(1+ \eps), \\
|\cos(\theta)\mp 1|>\eps}} \frac{2\eps}{\sin(\theta)^2} \chi \d\varphi \wedge d\theta=4\pi \sqrt{2\eps}\chi(\pm e)+\mc{O}(\eps),
\end{equation}
\begin{equation}\label{eps/sin^2again}
\int_{\zeta=\pm(1-\eps)} \frac{2\eps}{\sin(\theta)^2} \chi \d\varphi \wedge d\theta=-4\pi \sqrt{2\eps}\chi(\pm e)+\mc{O}(\eps),
\end{equation}
Combining \eqref{first_asymptotic}, \eqref{asympt_Flambda}, \eqref{IPP} and \eqref{eps/sin^2}, we finally get as $\eps\to 0$
\begin{align*}
& -F_\la(1+\eps)\int_{\substack{\zeta=1+ \eps, \\
|\cos(\theta)-1|>\eps}}  \iota_{\nabla \chi}\,{\rm dv}+\int_{\zeta=1+\eps} \chi \iota_{\nabla F_\la}\,{\rm dv}\\ 
& =\frac{C_0^+(\la)}{\sqrt{2\eps}}\int_{\substack{\zeta=1+ \eps, \\
|\cos(\theta)-1|>\eps}}\frac{1}{\cos(\theta)^2}\chi \d\varphi \wedge d\theta-\frac{C_0^+(\la)}{\sqrt{2\eps}}\int_{\zeta=1+ \eps}\frac{1}{\cos(\theta)^2}\chi \d\varphi \wedge d\theta\\
& \quad +4\pi C_0^+(\la) \chi(e)- C_0^+(\la)(\la+1)\int_{\zeta=1+ \eps}\frac{1}{\cos(\theta)^2}\chi \d\varphi \wedge d\theta+ o(1)\\
&= - C_0^+(\la)(\la+1)\int_{\zeta=1+ \eps}\frac{1}{\cos(\theta)^2}\chi \d\varphi \wedge d\theta+ o(1)
\end{align*}
Similarly, one obtains 
\begin{align*}
& -F_\la(-1-\eps)\int_{\substack{\zeta=-1- \eps, \\
|\cos(\theta)+1|>\eps}}  \iota_{\nabla \chi}\,{\rm dv}+\int_{\zeta=-1-\eps} \chi \iota_{\nabla F_\la}\,{\rm dv}\\ 
&= - C_0^-(\la)(\la+1)\int_{\zeta=-1-\eps}\frac{1}{\cos(\theta)^2}\chi \d\varphi \wedge d\theta+ o(1)
\end{align*}
Then, combining \eqref{first_asymptotic}, \eqref{asympt_Flambda}, \eqref{term_1-eps}, \eqref{eps/sin^2again} , we obtain
\begin{align*}
& -F_\la(1-\eps)\int_{\zeta=1- \eps}  \iota_{\nabla \chi}\,{\rm dv}+\int_{\zeta=1-\eps} \chi \iota_{\nabla F_\la}\,{\rm dv}\\ 
& =-4\pi A_1^-(\la) \chi(e)- (C_1(\la)+2C_2(\la))(\la+1)\int_{\zeta=1- \eps}\frac{1}{\cos(\theta)^2}\chi \d\varphi \wedge d\theta+ o(1)
\end{align*}
Similarly 
\begin{align*}
& -F_\la(-1+\eps)\int_{\zeta=-1+ \eps}  \iota_{\nabla \chi}\,{\rm dv}+\int_{\zeta=-1+\eps} \chi \iota_{\nabla F_\la}\,{\rm dv}\\ 
& =(C_1(\la)e^{-i\pi(\la+1)}+2C_2(\la)\cos(\pi(\la+1))(\la+1)\int_{\zeta=-1+ \eps}\frac{1}{\cos(\theta)^2}\chi \d\varphi \wedge d\theta\\
& \quad  -4\pi A_{-1}^+(\la) \chi(-e)+o(1)
\end{align*}
Summing everything (taking into account the orientation), we obtain 
\[ \begin{split}
\lim_{\eps\to 0}I_\eps=&\frac{4\pi C_1(\la)}{i}\chi(e)-(C_1(\la)+2C_2(\la)-C_0^+(\la))(\la+1)\int_{\zeta=1^+}\frac{1}{\cos(\theta)^2}\chi \d\varphi \wedge d\theta\\
& +4\pi(e^{-i\pi(\la+1)}\frac{C_1(\la)}{i}+2C_2(\la)\sin(\pi(\la+1)))\chi(-e)\\
&+(C_0^-(\la)+C_1(\la)e^{-i\pi(\la+1)}+2C_2(\la)\cos(\pi(\la+1))(\la+1)\int_{\zeta=-1^-}\frac{1}{\cos(\theta)^2}\chi \d\varphi \wedge d\theta. 
\end{split}\]
In order to have that $\lim_{\eps\to 0}I_\eps=\chi({\rm e})$, we then need to choose the constants such that 
\[C_1(\la)=\frac{i}{4\pi}, \quad C_2(\la)=\frac{e^{-i\pi(\la+1)}}{2i\sin(\pi(\la+1))}C_1(\la)=\frac{e^{-i\pi(\la+1)}}{8\pi \sin(\pi(\la+1))}\]
\[ C_0^-(\la)=-C_1(\la)e^{-i\pi(\la+1)}-2C_2(\la)\cos(\pi(\la+1))=-\frac{1}{4\pi \sin(\pi(\la+1))},\]
\[ C_0^+(\la)=C_1(\la)+2C_2(\la)=\frac{i}{4\pi}+\frac{e^{-i\pi(\la+1)}}{4\pi \sin(\pi(\la+1))}=\frac{1}{4\pi \tan(\pi(\la+1))},\]
which is what was claimed. Now we deduce from Lemma \ref{f^pmformule} that $C_0=1/8\pi$ and then that 
$f_{-\la(\la+2)}(\zeta)$ is given by \eqref{formule_for_f}.
\end{proof}
This shows that the resolvent has the shape $$R_{\Box_g}(\la)=F_\la(-q(x,y))$$ with $F_\la$ given by the expression \eqref{formuleFlambda}.
From this expression, we deduce directly that the integral kernel of $R_{\Box_g}(\la)$ is a convolution kernel by an  $L^1_{\rm loc}(\AdS_3)$ function, for all $\la\in \C\setminus \Z$, it is meromorphic in $\la$ with simple poles at $\la\in \Z$, and 
thus $R_{\Box_g}(\la)$ extends meromorphically to $\la\in \C$ as a continuous operator 
mapping $L^2_{\rm comp}(\AdS_3)$ to $L^2_{\rm loc}(\AdS_3)$ with simple poles at $\la\in \Z$.
We also remark that, since \eqref{R_Ktype} holds in ${\rm Re}(\la)>-1$ and since, for $f\in L^2_{\rm comp}(\R^+)$ 
both $R(\la)(fe_{k,\ell})$ and $e_{k,\ell}R_{k,\ell}(\la)f$ extend meromorphically to $\C$, they are equal. Since for each $\la_0\in \mc{Z}_{k,\ell}$ (recall \eqref{Zkell}) there are infinitely many pairs $(k,\ell)\in \N$ 
such that $\la_0=-2-k-\ell$, we see that the polar part of $R_{\Box_g}(\la)$ at $\la_0$ has infinite rank.

The proof of Proposition \ref{prop:resolventAdShol} is a direct consequence of  computation made in the proof of Lemma \ref{Lemma_sol_fond} where the function $F(\la)$ is replaced by a similar one but with constants $C_2(\la)=0$, $C_0^-(\la)=0$, $C_1(\la)=i/(4\pi)$ and $C_0^+(\la)=C_1(\la)=i/(4\pi)$.

\providecommand{\bysame}{\leavevmode\hbox to3em{\hrulefill}\thinspace}
\providecommand{\MR}{\relax\ifhmode\unskip\space\fi MR }
\providecommand{\MRhref}[2]{%
  \href{http://www.ams.org/mathscinet-getitem?mr=#1}{#2}
}
\providecommand{\href}[2]{#2}

\bigskip
\end{document}